\documentclass[11pt]{article} % use larger type; default would be 10pt
\usepackage[utf8]{inputenc} % set input encoding (not needed with XeLaTeX)

\usepackage[margin=1in]{geometry} % to change the page dimensions
\usepackage{graphicx} % support the \includegraphics command and options

\usepackage{dsfont} % for indicator \indc

\usepackage{booktabs} % for much better looking tables	
\usepackage{array} % for better arrays (eg matrices) in maths
\usepackage{paralist} % very flexible & customisable lists (eg. enumerate/itemize, etc.)
\usepackage{verbatim} % adds environment for commenting out blocks of text & for better verbatim
\usepackage{mathrsfs}
\usepackage{amssymb}
\usepackage{amsthm}
\usepackage{amsmath,amsfonts,amssymb}
\usepackage{esint}
\usepackage{graphics}
\usepackage{enumerate}
\usepackage{mathtools}
\usepackage{xfrac}
\usepackage{subcaption}
\usepackage{stmaryrd}
 \usepackage{mathabx}

\usepackage[dvipsnames]{xcolor}
\usepackage[colorlinks=true, pdfstartview=FitV, linkcolor=blue, citecolor=blue, urlcolor=blue]{hyperref}
\usepackage[normalem]{ulem}
\usepackage{accents}
\usepackage{tikz}
\usetikzlibrary{calc}
\usepackage{pgf}
\usetikzlibrary{external}
\numberwithin{equation}{section}
\numberwithin{figure}{section}

\newtheorem{theorem}{Theorem}[section]

\newtheorem{assumption}[theorem]{Assumption}
\newtheorem{corollary}[theorem]{Corollary}
\newtheorem{proposition}[theorem]{Proposition}
\newtheorem{lemma}[theorem]{Lemma}
\theoremstyle{definition}
\newtheorem{definition}[theorem]{Definition}

\newtheorem{remark}[theorem]{Remark}
\newtheorem{problem}{Problem}

\DeclarePairedDelimiter{\floor}{\lfloor}{\rfloor}

\newcommand*{\supp}{\ensuremath{\mathrm{supp\,}}}

\newcommand*{\N}{\ensuremath{\mathbb{N}}}
\newcommand*{\Q}{\mathbb{Q}}

\newcommand*{\Z}{\ensuremath{\mathbb{Z}}}

\newcommand*{\R}{\ensuremath{\mathbb{R}}}

\newcommand{\eps}{\varepsilon}

\renewcommand*{\tilde}{\widetilde}

\renewcommand{\P}{\ensuremath{\mathbb{P}}}

\newcommand{\ep}{\eps}

\DeclareMathOperator{\dist}{dist}

\DeclareSymbolFont{boldoperators}{OT1}{cmr}{bx}{n}
\SetSymbolFont{boldoperators}{bold}{OT1}{cmr}{bx}{n}

\newcommand\thickbar[1]{\accentset{\rule{.45em}{.6pt}}{#1}}
\renewcommand{\bar}{\thickbar}

\newcommand{\T}{\mathbb{T}}

\def\XXint#1#2#3{{\setbox0=\hbox{$#1{#2#3}{\int}$}
\vcenter{\hbox{$#2#3$}}\kern-.5\wd0}}

\let\originalleft\left
\let\originalright\right
\renewcommand{\left}{\mathopen{}\mathclose\bgroup\originalleft}
\renewcommand{\right}{\aftergroup\egroup\originalright}

\newcommand{\indc}{\mathds{1}}

\newcommand{\C}{\mathbb{C}}

\newcommand{\E}{\mathbb{E}}

\renewcommand{\hat}{\widehat}

\newcommand{\Var}{\mathrm{Var}}
\newcommand{\continuous}{C}

\makeatletter
\pgfmathdeclarefunction{erf}{1}{%
  \begingroup
    \pgfmathparse{#1 > 0 ? 1 : -1}%
    \edef\sign{\pgfmathresult}%
    \pgfmathparse{abs(#1)}%
    \edef\x{\pgfmathresult}%
    \pgfmathparse{1/(1+0.3275911*\x)}%
    \edef\t{\pgfmathresult}%
    \pgfmathparse{%
      1 - (((((1.061405429*\t -1.453152027)*\t) + 1.421413741)*\t 
      -0.284496736)*\t + 0.254829592)*\t*exp(-(\x*\x))}%
    \edef\y{\pgfmathresult}%
    \pgfmathparse{(\sign)*\y}%
    \pgfmath@smuggleone\pgfmathresult%
  \endgroup
}
\makeatother

\newcommand{\iop}{\mathcal{I}}
\usepackage{titlesec}

\newcommand{\addperiod}[1]{#1.}
\titleformat{\section}
   {\centering\normalfont\Large}{\thesection.}{0.5em}{}
\titleformat*{\subsection}{\bfseries}
\titleformat{\subsubsection}[runin]
  {\normalfont\bfseries}
  {\thesubsubsection.}
  {0.5em}
  {\addperiod}
\titleformat*{\subsubsection}{\normalfont\itshape}
\titleformat*{\paragraph}{\bfseries}
\titleformat*{\subparagraph}{\large\bfseries}
\newcommand{\indep}{\mathrel{\perp\!\!\!\perp}}
 
\title{Effective Lagrangian regularity and the uniqueness threshold for random H\"older velocity fields}

\author{Maria Colombo\thanks{Institute of Mathematics, \'Ecole Polytechnique F\'ed\'erale de Lausanne.
{\footnotesize \href{mailto:maria.colombo@epfl.ch}{maria.colombo@epfl.ch}.}
}
\and 
Elias Hess-Childs\thanks{Department of Mathematical Sciences, Carnegie Mellon University.
{\footnotesize \href{mailto:ehesschi@andrew.cmu.edu}{ehesschi@andrew.cmu.edu}.}
}
\and 
Keefer Rowan\thanks{Institute of Mathematics, \'Ecole Polytechnique F\'ed\'erale de Lausanne.
{\footnotesize \href{mailto:keefer.rowan@epfl.ch}{keefer.rowan@epfl.ch}.}
}
}
\date{\today}

\usepackage[nottoc,notlot,notlof]{tocbibind}

\begin{document}

\maketitle
\begin{abstract}
   We study the behavior of the ordinary differential equations, flow maps, and continuity equations associated to autonomous random velocity fields that admit a natural multiscale finite range decomposition. The velocity fields we consider are only H\"older regular in space---$C^{\alpha-}(\T^d)$ for some $\alpha \in (0,1)$---thus the associated ODE and continuity equation are not \textit{a priori} well-posed. However, above the critical threshold of $\alpha = 1/2$, due to multiscale stochastic cancellations, we prove well-posedness is almost surely restored away from the zero level set of the velocity field. This threshold marks a genuine transition, as demonstrated by examples lying below the threshold that exhibit robust ill-posedness. We additionally provide effective regularity estimates below the critical threshold and prove analogous results in the related ``refreshing'' regime.
\end{abstract}

\setcounter{tocdepth}{1} 
\tableofcontents

\section{Introduction}

In this paper, we study trajectories of the ordinary differential equation associated with a \emph{random} velocity field $U : \T^d \to \R^d$; that is, solutions to
\begin{equation}
    \label{eq:main ODE autonomous data}
    \begin{cases}
        \dot X_t = U(X_t),\\
        X_0 = y,
    \end{cases}
\end{equation}
as well as the closely related continuity equation, 
\begin{equation}
    \label{eq:continuity}
    \begin{cases}
            \partial_t f +\nabla \cdot( U f) =0,\\
            f(0,\cdot) = \phi(\cdot).
    \end{cases}
\end{equation}

The velocity fields we consider are only H\"older regular in space with some regularity index $\alpha \in (0,1)$. Thus, as is shown by the classical Peano nonuniqueness example $u(x) = |x|^\alpha$, nonuniqueness of the ODE trajectories is deterministically possible. However, the class considered here has an additional natural random structure: a multiscale, finite range decomposition described by Assumption~\ref{asmp:main for autonomous} below. Due to this random structure, there are extensive \textit{stochastic cancellations} which generate an \textit{effective regularity}---as appropriately measured from the perspective of a Lagrangian particle---strictly larger than the original \textit{bare regularity}. We emphasize that the stochastic cancellations utilized in this work are \textit{intrinsic to the velocity field}, present in almost every realization of the velocity field, in contrast to the regularization by noise literature~\cite{zvonkin_transformation_1974,veretennikov_strong_1981,davie_uniqueness_2007} in which stochastic cancellations are induced via an extrinsic stochastic forcing to the ODE.

This gain in effective regularity lowers the classical Lipschitz threshold of well-posedness down to a new critical threshold of $\alpha =1/2$, where $\alpha$ governs the bare regularity of the velocity field through~\eqref{eq:main theorem hypothesis}---in particular ensuring that $U \in C^{\alpha-}(\T^d)$ almost surely. Above this threshold, we prove almost sure Lagrangian well-posedness, while below it, we show that Lagrangian ill-posedness remains possible. The $\alpha=1/2$ threshold appears quite naturally from a fairly short heuristic computation relying on central limit theorem-type stochastic cancellation which is presented in Section~\ref{s:heuristic}.

For the random velocity fields we consider, we take the full velocity field $U$ to be the sum of many \textit{scale fields}, $U := \sum_{j=0}^\infty u^j$, where the $u^j$ are \textit{mutually independent}. The field $u^j$ models the component of the velocity field on the length scale $2^{-j}$; accordingly we require the scale-by-scale regularity hypothesis~\eqref{eq:main theorem hypothesis} consistent with $U \in C^{\alpha-}_x.$ In addition to independence across scales, we suppose that each scale field decorrelates in space beyond its associated length scale $2^{-j}$. More precisely we take $u^j$ to have the finite range of dependence $2^{-j}$, as defined in Assumption~\ref{asmp:main for autonomous} below. 

We use this finite range hypothesis as it is a particularly simple spatial decorrelation assumption, and it includes many problems of substantial interest. Single scale media with a finite range of dependence play a central role in stochastic homogenization~\cite{papanicolaou_boundary_1981,gloria_quantitative_2017,armstong_quantitative_2019,armstrong_renormalization_2025}. Multiscale sums of independent finite range fields feature prominently in the study of superdiffusion~\cite{chatzigeorgiou_gaussian_2025,morfe_critical_2025,armstrong_superdiffusive_2024,armstrong_superdiffusion_2026}, level-set percolation for strongly correlated Gaussian fields~\cite{duminil-copin_equality_2023,muirhead_severo_percolation_2024,muirhead_percolation_2024,schweiger_finite_range_2024}, and rigorous renormalization group theory~\cite{brydges_finite_range_2004,bauerschmidt_simple_2013,bauerschmidt_renormalisation_2019}. Natural examples of velocity fields satisfying our assumptions---namely Gaussian fields with power law spectra and fields constructed from Poisson point processes---are given in Section~\ref{s:examples}.

We take the following precise hypotheses. We note that the following assumption affords substantial flexibility in the large scales: we do not require $u^0$ to be finite range or centered. As such, $u^0$ can incorporate any $C^2_x$ deterministic field. Throughout, we take the convention that $0 \in \N$ and use $X \indep Y$ to mean that $X,Y$ are independent random variables.

\begin{assumption}
    \label{asmp:main for autonomous}
    We suppose $U : \T^d \to \R^d$ is such that 
    \[U(x) = \sum_{j=0}^\infty u^j(x)\]
    where $u^j$ satisfy the following properties.
    \begin{enumerate}
        \item \label{asmp:scale independence auto} The fields $u^j$ are mutually independent.
        \item \label{asmp:centered auto} For all $j \in \N$ with $j \geq 1$, $\E u^j =0.$ 
        \item \label{asmp:finite range auto} For all $j \in \N$ with $j \geq 1$ and all Borel $A,B \subseteq \T^d$ with $\mathrm{dist}(A,B) \geq 2^{-j}$, we have that $u^j|_A \indep u^j|_B$.
    \end{enumerate}
\end{assumption}

We will also require for some $p \geq 1$ depending on $d$ and $\alpha$---specified in the statements below---the regularity hypothesis
    \begin{equation}
    \label{eq:main theorem hypothesis}
   \max_{n=0,1,2} \sup_{j\in \mathbb N} 2^{(\alpha-n)j}\|\nabla^n u^j\|_{L^p_\omega \continuous^0_{x}}<\infty.\end{equation}
As noted above, this guarantees that $U \in C^{\alpha-}_x$ almost surely, but does not require that $U \in C^{\alpha}_x$; that is, this allows for genuinely H\"older rough fields.

The central mechanism of effective regularity generation will be the large scales of the velocity field ``sweeping'' ODE trajectories over the small scales, leading to stochastic averaging. As such, we only get good control of the ODE dynamics \textit{away from the (random) zero level set}, $Z$. In fact, we expect the phenomenology at the zero level set to be meaningfully different, with ill-posedness for all $\alpha \in (0,1)$, see Remark~\ref{rem:zero set}.

We now state our first main result, which almost surely gives Lipschitz flows for~\eqref{eq:main ODE autonomous data}, localized away from the zero level set. We also consider the thickened zero level set, $Z_\ep$, for all $\ep>0$:
\begin{equation}\label{eq:zero sets}
Z := \{x \in \T^d : U(x) = 0\} \quad \text{and} \quad Z_\ep :=\{x \in \T^d: d_{\T^d}(x,Z) \leq \ep\}.
\end{equation}

\begin{theorem}
    \label{thm:main ode intro}
    Let $\alpha \in (1/2,1)$. Suppose $U:\T^d \to\R^d$ is a $C^{\alpha-}(\T^d)$ multiscale finite range random velocity field; that is, $U$ satisfies Assumption~\ref{asmp:main for autonomous}, and \eqref{eq:main theorem hypothesis} holds for some $p > \frac{d}{\alpha-1/2}$.
    
    Then almost surely, the ODE~\eqref{eq:main ODE autonomous data} admits a unique Lipschitz flow away from the random zero level set of $U$.
    More precisely, for all $\ep>0$, there exists an almost surely finite random constant $M_\ep>0$ such that for any $\tau \in [0,1]$ and any two integral curves of $U$ avoiding $Z_\ep$ on $[0,\tau]$---$X^i: [0,\tau] \to \T^d\backslash Z_\ep$ with $\dot X^i = U(X^i)$---we have the bi-Lipschitz stability estimate
    \begin{equation}
    \label{eq:compression intro}
    M_\ep^{-1} d_{\T^d}(X^1_0, X^2_0) \leq d_{\T^d}(X^1_t ,X^2_t) \leq M_\ep d_{\T^d}(X^1_0, X^2_0)\qquad \mbox{for all $t \in [0,\tau]$}.
    \end{equation}
    In particular, the solution to~\eqref{eq:main ODE autonomous data} is uniquely determined up until hitting $Z$.
\end{theorem}

We emphasize that the random constant $M_\ep$ in the above result is \textit{uniform in the choice of the curves,} depending only on the distance $\eps>0$ and the realization of the field $U$. Above and throughout we take the convention that $\infty \cdot 0 =\infty$, so~\eqref{eq:compression intro} only gives \textit{almost sure} uniqueness.

\begin{remark}
     We note that for any fixed $y \in \T^d$---provided that $U(y) \ne 0$ a.s.---we get from the above that, for $\alpha>1/2$, 
        \[\lim_{h \to 0} \P\big(\text{\eqref{eq:main ODE autonomous data} has a unique solution on } [0,h]\big) =1.\]
    That is, there is almost surely uniqueness for \textit{some amount of time}.
\end{remark}

While the above theorem gives no control on ODE trajectories that cross the zero level set, it still allows for a form of essential well-posedness, provided we can ensure that almost every ODE trajectory avoids the zero level set. We can make such a guarantee provided the velocity field is divergence-free and the zero level set is sufficiently small in a Minkowski dimension sense. This is the content of our next result. We emphasize that well-posedness above $\alpha = 1/2$, in both Theorem~\ref{thm:main ode intro} and Corollary~\ref{cor:main pde ode}, is \textit{dimension independent}.

\begin{corollary}
\label{cor:main pde ode}
      Let $\alpha \in (1/2,1)$. Suppose $U:\T^d \to\R^d$ is a $C^{\alpha-}(\T^d)$ multiscale finite range random velocity field; that is, $U$ satisfies Assumption~\ref{asmp:main for autonomous}, and \eqref{eq:main theorem hypothesis} holds for some $p > \frac{d}{\alpha-1/2}$. Suppose in addition that $\nabla \cdot U=0$ and there exists a random constant $\delta>0$ such that almost surely
        \begin{equation}
        \label{eq:small zeros}
        \lim_{\ep \to 0} \ep^{-(1-\alpha +\delta)} |Z_\ep| =0.\end{equation}      
    Then almost surely, 
    \begin{enumerate}
        \item\label{item:ae  unique} for almost every $y \in \T^d$,~\eqref{eq:main ODE autonomous data} admits a unique solution on $(-\infty,\infty)$,
        \item and for every bounded $\phi : \T^d \to \R$,~\eqref{eq:continuity} admits a unique bounded solution $f$.
    \end{enumerate}
\end{corollary}

\begin{remark}
\label{rem:after cor}
    In the setting of Corollary~\ref{cor:main pde ode}, Item~\ref{item:ae unique} gives that there is an almost everywhere unique flow associated to the ODE~\eqref{eq:main ODE autonomous data}. That $U$ is divergence-free then gives that this flow is volume preserving, hence the unique flow is a regular Lagrangian flow, in the DiPerna--Lions sense. This then readily gives that the unique solution to~\eqref{eq:continuity} is renormalized---that is, for all $\beta \in C^1(\R)$, $\beta(f)$ is also a distributional solution to~\eqref{eq:continuity} with initial data $\beta(\phi)$.
\end{remark}

By Lemma~\ref{lem:sufficient condition for good zero set}, the hypothesis of the smallness of the zero level set,~\eqref{eq:small zeros}, is in particular implied if the pointwise marginal law of $U(x)$ has a density on a neighborhood of $0$ that is bounded uniformly in $x$. Thus~\eqref{eq:small zeros} is a fairly mild condition we expect to typically hold for fields that do not ``prefer to be $0$'' in some strong way. The divergence-free constraint is much more substantial, though it is a case of significant interest due in part to its connection with incompressible fluids.

We now show that $\alpha =1/2$ represents a real threshold for qualitative regime change, from guaranteed Lagrangian well-posedness to possible Lagrangian ill-posedness.

\begin{theorem}
\label{thm:autonomous intro sharp}
    For all $d \geq 4$ and $\alpha \in (0,1/2)$, there exists a $C^{\alpha-}(\T^d)$ multiscale finite range random velocity field $U : \T^d \to \R^d$ satisfying Assumption~\ref{asmp:main for autonomous} such that \eqref{eq:main theorem hypothesis} holds for all $p \geq 1$, $\nabla \cdot U =0$, and $Z = \emptyset$.
 However, almost surely,
 \begin{enumerate}
     \item \eqref{eq:main ODE autonomous data} admits nonunique solutions on $[0,1]$ for a positive Lebesgue measure set of $y \in \T^d$,
     \item and there exists a bounded $\phi$ such that~\eqref{eq:continuity} admits nonunique bounded solutions on $[0,1]$.
 \end{enumerate}
\end{theorem}

Our final result of this section gives a form of enhanced effective regularity even in the $\alpha \in (0,1/2)$ regime of (possible) ill-posedness.
\begin{theorem}
    \label{thm:sweeping Bihari} 
     Let $\alpha \in (0,1/2)$. Suppose $U:\T^d \to\R^d$ is a $C^{\alpha-}(\T^d)$ multiscale finite range random velocity field; that is, $U$ satisfies Assumption~\ref{asmp:main for autonomous}, and \eqref{eq:main theorem hypothesis} holds for all $p\geq 1$.
     
     Then for all $y \in \T^d$ and $\ep>0$, there exists an almost surely finite random constant $M>0$ such that if $U(y) \ne0$ and $X^1, X^2$ are any solutions to~\eqref{eq:main ODE autonomous data}, then we have the (quasi)-stability estimate for all $t \in [0,1]$,
    \begin{equation}
    \label{eq:autonomous quasi stability}
    d_{\T^d}(X^1_t, X^2_t) \leq Mt^{\frac{1-\alpha}{1-2\alpha}-\ep}.
    \end{equation}
\end{theorem}

We compare~\eqref{eq:autonomous quasi stability} with the best available deterministic estimate---a consequence of the Bihari--LaSalle inequality---which gives that
\[ d_{\T^d}(X^1_t, X^2_t) \leq C(\alpha) \|U\|_{C^\alpha_x}^{\frac{1}{1-\alpha}} t^{\frac{1}{1-\alpha}}.\]
Since $\frac{1-\alpha}{1-2\alpha}> \frac{1}{1-\alpha}$ for $\alpha \in (0,1/2)$, we see that Theorem~\ref{thm:sweeping Bihari} gives better estimates on the rate of particle separation than the bare regularity gives alone. This is a manifestation of our idea of effective Lagrangian regularity, discussed further in Section~\ref{ss:effective lagrangian}.

We also provide results for the deeply related ``refreshing regime'', which consists of time-dependent velocity fields satisfying a different finite range decomposition: decorrelation is imposed in time rather than in space. We introduce this regime as well as our main results for it in Section~\ref{s:refreshing intro}. For clarity, we focus the majority of our introductory discussion purely on the sweeping regime as governed by Assumption~\ref{asmp:main for autonomous}. 

We next present the central motivating heuristic for this work with a brief introduction of the main tools in Section~\ref{s:heuristic}. Section~\ref{s:discussion} contains a discussion of how this work connects with disparate pieces of the mathematical literature. Section~\ref{s:refreshing intro} introduces the refreshing regime, some open problems are stated in Section~\ref{s:open problems}, and Section~\ref{s:examples} gives examples of velocity fields satisfying both the refreshing and sweeping hypotheses. A technical overview of the proof is then given in Section~\ref{s:overview}; following that are the complete proofs of the main results.

\section{\texorpdfstring{Outline of the main ideas and the $\alpha = 1/2$ threshold}{Outline of the main ideas and the alpha = 1/2 threshold}}
\label{s:heuristic}

We now present a heuristic argument that illustrates the origin of the $\alpha = 1/2$ threshold. We then give a brief discussion of the main ideas used to turn this nonrigorous heuristic into a rigorous argument. A more extensive overview of the technical aspects of the proof is deferred to Section~\ref{s:overview}.

\subsection{A heuristic computation}

 We suppose that $U(x) = \sum_{j=0}^\infty u^j(x)$ is a velocity field satisfying Assumption~\ref{asmp:main for autonomous} such that for some $\alpha \in (0,1)$, typically $|u^j(x)| \approx 2^{-\alpha j}$. This corresponds roughly to $U \in C^{\alpha-}(\T^d)$---recalling that $u^j$ has range of dependence $2^{-j}$. We consider the separation of two particles that start close together; that is, we let $X^1, X^2$ be such that $|X^1_0 - X^2_0| \approx 10 \cdot 2^{-j}$ and
\[\dot X^i_t = U(X^i_t).\]
Our goal now is to estimate the time $\tau_j$ such that $|X^1_{\tau_j} - X^2_{\tau_j}| \approx 10 \cdot 2^{-(j-1)}$---that is, the time for $X^1_t,X^2_t$ to separate to the next dyadic scale (or slightly above, in order to ensure separation beyond the range of dependence).

We take the following (very heuristic) decomposition of $U \approx U^L + u^j + U^S$, where
\begin{itemize}
    \item we suppose that $U^L$ is living on much larger scales than $2^{-j}$, so that $U^L$ can just be treated as a constant $U^L \approx v \in \R^d$,
    \item we suppose that $U^S$ lives on much smaller scales than $2^{-j}$---thus by the H\"older regularity of $U$, it is pointwise small---so that $U^S$ can be treated as negligible: $U^S \approx 0$.
\end{itemize}

We have thus reduced the problem to considering $U \approx v+ u^j$ where $v \in \R^d$. We expect that typically $|v| \approx 1$ and $|u^j| \approx 2^{-\alpha j} \ll 1$. Thus a reasonable expansion for $X^i_t$, neglecting higher order terms, is
\[X^i_t \approx X^i_0 + vt + \int_0^t u^j(X^i_0 + vs)\,ds.\]
Then, by the finite range structure of $u^j$ and the initial separation of the $X^i_0$, we expect that typically 
\[\int_0^t u^j(X^1_0 + vs)\,ds \indep \int_0^t u^j(X^2_0 + vs)\,ds,\]
where $A \indep B$ means $A,B$ are independent random variables. Thus, since the difference of independent random variables is of approximately the same magnitude as the maximum,
\[|X^1_t - X^2_t| \approx \max\Big(|X^1_0 - X^2_0|, \Big|\int_0^t u^j(X^1_0 + vs)\,ds\Big|,\Big|\int_0^t u^j(X^2_0 + vs)\,ds\Big|\Big).\]
Consequently, the time $\tau_j$ at which $X^1, X^2$ separate to the next dyadic scale should essentially be the first time $t$ such that 
\[\Big|\int_0^t u^j(X^2_0 + vs)\,ds\Big| \approx 2^{-j}.\]
Since $|v| \approx 1$, the finite range structure of $u^j$ means that $\int_0^t u^j(X^2_0 + vs)\,ds$ is essentially the sum of $2^j t$ many independent terms, each of magnitude $2^{-j} |u^j| \approx 2^{-(1+\alpha)j}.$ The central limit theorem thus gives that 
\[\Big|\int_0^t u^j(X^2_0 + vs)\,ds\Big| \approx   2^{-(1+\alpha)j} 2^{j/2} \sqrt{t} = 2^{-(\alpha+1/2)j}\sqrt{t}.\]
Setting this to be $\approx 2^{-j}$ thus gives that
\[2^{-(\alpha+1/2)j}\sqrt{\tau_j} \approx 2^{-j} \quad \text{or} \quad \tau_j \approx 2^{2(\alpha -1/2)j}.\]

We thus see the appearance of the $\alpha =1/2$ threshold; if $\alpha>1/2$, then as $j \to \infty,\tau_j \to \infty$, but if $\alpha<1/2$, as $j \to \infty, \tau_j \to 0$. In fact for $\alpha<1/2,$ $\sum_j \tau_j <\infty$, which suggests that no matter how close together particles start, the particles will separate to the unit scale in finite time. This is precisely ODE nonuniqueness. 

For $\alpha>1/2$, this computation suggests the closer together the particles start, the longer it takes for them to separate. While consistent with ODE uniqueness, this is actually too strong a statement since this would suggest that the top Lyapunov exponent should be $\leq 0$. This is because we were overly hasty in approximating $U^L \approx v$; in the case that $\alpha>1/2$ the correct behavior of $\tau_j \approx 1$ as $j \to \infty$ would re-emerge if we instead took the next order expansion $U^L(x) \approx v + Ax$.

We note that the ansatz for the leading-order correction to ballistic motion, $\int_0^t u^j(X^i_0 + vs)\,ds$, transforms the finite range \textit{in space} of the velocity field into a finite range \textit{in time} of the integrand. Finite range in time is particularly well suited to making these heuristic computations rigorous: the linear flow of time prevents iterative feedback loops. This is a primary motivation for studying the refreshing case, introduced in Section~\ref{s:refreshing intro}. Moreover, in the rigorous argument for Theorem~\ref{thm:main ode intro}, we will more fully reduce the autonomous sweeping regime to a finite-range-in-time problem.

\subsection{Some comments on the heuristic}

In the above heuristic computation, we emphasize that the stochastic cancellation is appearing due to the large scales of the velocity field---modeled here as a constant background vector $v \in \R^d$---\textit{sweeping} the ODE trajectories over the small scale field causing an incoherent averaging and producing stochastic cancellations under the central limit scaling. We call this phenomenon, where the large scales essentially average out the small scales, the \textit{sweeping effect}. Thus we see the velocity field is in some sense self-regularizing: its own action causes the small scales to behave as if they were more regular.

It is for exactly this reason that we need to isolate away from the zero level set of $U$ in Theorem~\ref{thm:main ode intro}, since this self-regularization mechanism fails exactly at stagnation points. On the zero level set, we have no background constant moving us over the small scales, and hence our techniques get no stochastic cancellation, and so no improvement in the ODE behavior over the naive regularity.

We next note that the rate at which $\tau_j \to 0$ as $ j \to \infty$ when below the regularity threshold can be turned into a rate of separation of $X^1_t,X^2_t$. In particular, rearranging somewhat, we get the prediction for $\alpha \in (0,1/2)$,
\[|X^1_t - X^2_t| \approx t^\frac{1}{1-2\alpha}.\]
This is much smaller than the upper bound $|X^1_t-X^2_t|\lesssim t^{\frac{1}{1-\alpha}}$ provided by the bare regularity. In fact, our ill-posedness examples saturate this heuristic separation rate; see Theorem~\ref{thm:sharp autonomous} which is a more precise version of Theorem~\ref{thm:autonomous intro sharp}.

Theorem~\ref{thm:sweeping Bihari} gives exactly this sort of particle separation estimate, though with a slightly worse exponent. We conjecture in Section~\ref{s:open problems} that the true exponent is indeed $\frac{1}{1-2\alpha}$, as predicted by the heuristic; however we cannot prove this at this time. These particle separation estimates connect to our notion of effective Lagrangian regularity, discussed in Section~\ref{ss:effective lagrangian}.

Finally, we emphasize that while the heuristic gives the correct threshold---and was the original inspiration for this work---it is quite far from being directly formalizable. We already saw that the approximation $U^L \approx v$ is rather poor when above the critical thresholds, as we miss out on the linear fluctuations of $U^L$ which dominate the true separation dynamics. However \textit{a priori} we are missing much more, since the naive bound is $\|U^L\|_{C^1_x} \approx 2^{(1-\alpha)j} \to \infty$. Thus it is subtle to treat $U^L$ as a ``smooth large scale field''. Additionally $\|U^S\|_{C^1_x} = \infty$, so even though $U^S$ may be pointwise small, it is also very delicate to treat it as a perturbative error.

We conclude this section with a loose discussion of the tools we use to turn the above heuristic into a rigorous proof.

\subsection{Proving uniqueness: effective regularity and nonlinear Young integral equations}

Here we focus on the uniqueness side of the problem, when the parameters are above the critical threshold. The central technical tool we use is a measure of the effective regularity of the velocity field, using ideas from the theory of nonlinear Young integrals introduced to study regularization by noise in stochastic differential equations~\cite{catellier_averaging_2016,galeati_noiseless_2021}. As motivation, we note the following remarkable result, which is a corollary of~\cite[Theorem 4.8]{galeati_noiseless_2021}.

\begin{proposition}
\label{prop:nonlinear young intro}
Let $u : \T^d \to \R^d$ be such that $u \in C^0(\T^d)$ and $y \in \T^d$. Let $X$ be \textit{any} solution to
\[\begin{cases}
    \dot X_t = u(X_t),\\
    X_0 = y,
\end{cases}\]
and $\iop u(t,x) := \int_0^t u(x+X_s)\,ds$. Then, if $\iop u \in C^\gamma([0,1], C^1(\T^d))$ for some $\gamma>1/2$, $X$ is the \textit{unique} solution to the above ODE.
\end{proposition}

At this stage, it is likely best to consider the time regularity $\gamma>1/2$ as a purely technical constraint, but a further discussion of this will be given in Section~\ref{s:overview}; see in particular Section~\ref{sss:1/2}. We view the $C^\gamma_t C^1_x$ norm of $\iop u$ as a measure of the \textit{effective regularity} of $u$, as seen from the perspective of a solution curve. We can then view the above proposition as an \textit{effective} version of the Cauchy--Lipschitz theorem, which says that if $u$ is \textit{effectively Lipschitz from the perspective of a solution curve}, then the ODE admits a unique solution.

We thus can shift from considering \textit{bare regularities}---under which the velocities $U$ we will consider will be such that $\|U\|_{C^1_x} = \infty$---to \textit{effective regularities}, for which we can hope to prove that the $U$ are appropriately effectively regular. The above heuristic then suggests how the additional effective regularity arises: the stochastic cancellation of the velocity field as it is integrated over the solution curve. We thus prove the uniqueness, as well as the precise quantitative stability estimates, by correctly estimating the effective regularity of the velocity field---substantially exploiting stochastic cancellations---and then using variations of~\cite[Theorem 4.8]{galeati_noiseless_2021} to deduce ODE estimates.

While this is the right first idea, implementing it in practice is still fairly delicate, as the solution curve depends on the velocity field in an intricate way, thus ruining both any naive independence across increments and any naive centeredness in the integral defining $\iop u$. Even worse, $x \mapsto u^j(x+X_s)$ is not at all (naively) independent across different times $s$, since $x$ varies over the entire torus. Thus the desired central limit theorem cancellations from integrating in time are in no way direct. Further---and more technical---discussion of the proof strategy is given in Section~\ref{ss:wellposedness}.

\subsection{Proving nonuniqueness: one active scale at a time}

We now discuss briefly how we construct the velocity fields of Theorem~\ref{thm:autonomous intro sharp}. As mentioned above, a primary difficulty with translating the heuristic into a proof is that the approximation of $U^L \approx v$ and $U^S \approx 0$ is at best hard to justify (and at worst essentially untrue). However, since for Theorem~\ref{thm:autonomous intro sharp} we are able to construct our own velocity fields, we can make the approximation true by fiat.

The idea is to use the standard dimensional extension trick to embed a nonautonomous velocity field on $\T^{d-1}$ into an autonomous velocity field on $\T^d$ (for which some care is needed to maintain the correct finite range structure). Then for the time-dependent velocity field, we can split into time intervals $[T^{j+1},T^j]$ on which only the field $u^j$ is active, thus we really can neglect the contribution of the large scales and the small scales (since they are entirely absent). Under this reduction, the heuristic argument is much more straightforward to make rigorous, though still requires substantial work to precisely track the CLT-like behavior. The construction and argument given here are similar to those presented in~\cite{hess-childs_sharp_2026}. However---as elaborated in Section~\ref{ss:nonunique}, and particularly in Section~\ref{sss:comparison}---some additional ideas are needed to adapt the argument used for the shear-flow construction in~\cite{hess-childs_sharp_2026} to a velocity field with finite spatial range of dependence.

\section{Background and discussion}

\label{s:discussion}

The problem of uniqueness for ODEs and continuity equations has a long lineage. The classical result, giving well-posedness for spatially Lipschitz velocity fields, goes back to Cauchy, Lipschitz, Picard, and Lindel\"of. By extensions and counterexamples given by Osgood and Peano, the deterministic threshold for uniqueness was entirely resolved: suppose $u : \T^d \to \R^d$ is such that $|u(x) - u(y)| \leq \omega(|x-y|)$. If $\int_0^1 \frac{dr}{\omega(r)} = \infty$, then there is uniqueness for the associated ODE; if on the other hand $\int_0^1 \frac{dr}{\omega(r)} <\infty$, there exist examples for which there is nonuniqueness.

For the continuity equation---or, equivalently, the transport equation in the case that $\nabla \cdot u=0$, which we will restrict ourselves to for this part of the discussion---a richer theory is available. By the Ambrosio superposition principle~\cite{ambrosioTransportEquationCauchy2008}, nonuniqueness of the continuity equation for bounded solutions requires not just ODE nonuniqueness at a single point, but ODE nonuniqueness for a positive measure set of initial conditions. Thus well-posedness can hold under even weaker conditions than for the ODE. The theory of DiPerna--Lions~\cite{diperna_ordinary_1989}---extended by Ambrosio~\cite{ambrosio_transport_2004}---proves transport uniqueness for velocity fields at unit regularity but with less integrability, getting down to $u \in BV$. Interestingly, when $u \in W^{1,p}$ for $p<d$, transport uniqueness holds in the class of bounded solutions despite the possible presence of ODE nonuniqueness on a positive measure set of initial conditions~\cite{bruePositiveSolutionsTransport2021,kumar2023nonuniqueness}. The assumption of boundedness of the solution is essential; nonuniqueness has been shown for $u \in W^{1,p}$ when the solution is less integrable \cite{MoSz2019AnnPDE,MoSa2019,MoSz2019CalcVar,BrueColomboKumar2026}.
In contrast to integrability, the regularity threshold for the continuity equation is the same as for the ODE theory: transport nonuniqueness is possible for a velocity field having any modulus of continuity failing Osgood's criterion~\cite{colombo_sharpness_2026}.

Thus the deterministic, ``worst case'' thresholds are clearly sorted out, for both the PDE and ODE theory. We know that nonuniqueness is \textit{possible} for $C^\alpha_x$ velocity fields; the natural next question is then whether it is \textit{typical}. A first answer to this question was provided by~\cite{Orlicz1932GenericODE}, proving that velocity fields $u$ that admit unique ODE flows are generic (in the Baire category sense) in $L^1_t C^0_x$. Similarly, using ideas from~\cite{diperna_ordinary_1989}, \cite{lions_uniqueness_1998} proved that velocity fields $u$ for which the transport equation admits unique bounded solutions are Baire generic in the divergence-free subspace of $L^p_x$.\footnote{We note~\cite{galeati_2d_2026} for an interesting modern application of similar ideas to nonlinear fluid equations.} Thus it appears uniqueness is typical, at least from the Baire genericity perspective.

Baire genericity of a given property, however, depends delicately on the exact topology: Baire genericity of velocities admitting a unique ODE flow in $C^\alpha_x$ for $\alpha \in (0,1)$ remains an open problem (essentially because smooth functions are not dense in $C^\alpha$). Baire generic sets can also be quite atypical in infinite dimensions: any probability measure $\mu$ on an infinite-dimensional, separable Banach space $X$ is carried by a meager set---that is, there exists $S \subseteq X$ with $X \backslash S$ Baire generic and $\mu(S) = \mu(X) =1$---as can be seen using tightness and that compact sets have an empty interior. Thus the Baire genericity of uniqueness should not be seen as a complete resolution of the question of typicality of nonuniqueness phenomena.

In fact, from a different perspective, nonuniqueness \textit{should be} typical. Turbulent fluids are the most natural manifestation of physical velocity fields having fractional regularities (around $C^{1/3}_x$)~\cite{frisch_turbulence_1995}. In these fluid fields, the ODE trajectories also admit the physical interpretation of motion of microscopic particles. In these turbulent velocity fields, the phenomenon of Richardson dispersion and anomalous dissipation of passive scalars~\cite{richardson_atmospheric_1926,falkovichParticlesFieldsFluid2001,drivas_lagrangian_2017,drivas_anomalous_2022} predicts nonuniqueness of both the ODE and the transport equation for the fluid velocity field. This prediction is universal, in the sense that it is independent of the specifics of how the energy is injected into the fluid. Thus, we expect nonuniqueness for typical ``fluid-like'' velocity fields. This prediction is borne out by the multitude of synthetic anomalous dissipation constructions, showing nonuniqueness for velocity fields in $C^\alpha$ for all $\alpha \in (0,1)$: see~\cite{colombo_anomalous_2023,armstrong_anomalous_2025,burczak_anomalous_2023_fixed,elgindi_norm_2024,hofmanova_anomalous_2025,johansson_anomalous_2024, hess-childs_universal_2025} among others.

The natural probabilistic way of addressing the question of the typicality of ODE nonuniqueness is by constructing ``typical'' random H\"older regular velocity fields. This of course simply shifts the question of typicality onto the construction of the velocity field measures; however the Gaussian field with independent Fourier coefficients having a power law $|k|^{-d/2 -\alpha}$ coefficient scaling is essentially \textit{the} basic random $C^{\alpha-}$ field. As shown in Section~\ref{s:examples}, this field fits more broadly into the class of velocity fields we consider here, admitting a decomposition as a sum of independent finite range fields. 

We thus view this work as studying the natural probabilistic model of typical H\"older regular velocity fields. Interestingly, we get a richer phenomenology---a transition from well-posedness to (possible) ill-posedness at $\alpha=1/2$---than is present for either the Baire-category perspective or the fluid perspective.

\subsection{The Kraichnan model}

An interesting point of comparison for the model we consider is the Kraichnan model~\cite{kraichnanSmallScaleStructure1968}. The Kraichnan model also takes the velocity field to be Gaussian with independent Fourier coefficients scaling like $|k|^{-d/2-\alpha}$---hence it is spatially $C^{\alpha-}_x$. However, the Kraichnan model lies at the opposite extreme in its temporal structure; whereas we consider autonomous velocity fields, the Kraichnan model is \textit{white-in-time}. This in particular means that the velocity field has negative temporal regularity $C^{-1/2-}_t C^{\alpha-}_x$, the ODE becomes a (multiplicative noise) SDE, and the continuity equation a (multiplicative noise) SPDE.

While it may seem from this that it would be particularly technical to study the Kraichnan model, it was actually introduced into the physics literature as a simple, essentially solvable, model velocity field for fluid turbulence. There is extensive work in the physics literature precisely characterizing its phenomenology: see in particular the striking works~\cite{gawedzki_anomalous_1995,bernard_slow_1998} as well as the reviews~\cite{falkovichParticlesFieldsFluid2001,gawedzkiKrzysztofGawedzkiSoluble2008}. There are also a variety of mathematical works on the Kraichnan model, starting with~\cite{jan_integration_2002,jan_flows_2004,lototskii_passive_2004,lototsky_wiener_2006}. More recently, there have been a variety of works which exploit that the two particle dynamics is governed by a particularly simple, finite-dimensional equation to prove anomalous dissipation, particle dispersion, as well as the interesting phenomenon of anomalous regularization (which we will not further discuss here) for the Kraichnan model~\cite{rowan_anomalous_2024,galeati_anomalous_2024,drivas_anomalous_2025,rowan_obukhov--corrsin_2025}. \cite{drivas_anomalous_2025} in particular considers a family of Kraichnan models parameterized by a regularity index $\alpha \in (0,1)$ and \textit{compressibility parameter} $\eta \in [0,1]$ (with $\eta = 0$ being the pure gradient case and $\eta =1$ the divergence-free case). They show that two particles $X^1,X^2$ advected by the flow typically separate like $t^{\frac{1}{2(1-\alpha)}}$, no matter how close $X^1_0$ is to $X^2_0$, provided $\eta > \eta_{\mathrm{crit}}(\alpha,d)$. For $\eta < \eta_{\mathrm{crit}}(\alpha,d)$, there is no separation as $X^2_0 \to X^1_0.$

The Kraichnan model thus gives a well-understood point of comparison for the results here. However, in contrast to the case we consider, the Kraichnan model exhibits (a form of) ODE nonuniqueness for all $\alpha \in (0,1)$, at least in the incompressible case. The Kraichnan model is also much more tractable, due to the white-in-time property. Since the velocity field is instantaneously resampled, there are essentially no ``multiscale interactions''. Each scale contributes additively to particle separation. In contrast, for an autonomous flow, as described above, the central phenomenology is governed by a multiscale interaction: the large scales of the flow sweeping particles rapidly over the small scales, causing stochastic averaging.

\subsection{Effective Lagrangian regularity and particle dispersion}

\label{ss:effective lagrangian}

Our central results show that a random $C^{\alpha-}_x$ velocity can behave ``as if'' it is more regular than $C^{\alpha-}_x$, at least from the perspective of Lagrangian particles. This leads us to give the following definition of \textit{effective Lagrangian regularity}, which gives an intrinsic notion of regularity as it relates to particle separation. 

\begin{definition}
    Let $U : \T^d \to \R^d$ with $U \in C^0(\T^d)$ and $y \in \T^d$. Define
    \[R_t(y) := \sup_{X^1,\ X^2 \text{ solutions to~\eqref{eq:main ODE autonomous data}}}  d_{\T^d}(X^1_t, X^2_t).\]
    If, for some $t_0 >0,$ $R_t(y) = 0$ for all $t \in [0,t_0]$, then we say that $U$ is \textit{effectively regular at $y$}.
    Otherwise, we define
    \[E(y) := 1-\limsup_{t \to 0^+} \frac{\log t}{\log R_t(y)}.\]
    We then say that $U$ has \textit{effective Lagrangian regularity} $E(y)$ at $y$.
\end{definition}

We note that for any $U \in C^0_x, R_t(y) \leq Ct$, so that $E(y) \geq 0$. More generally, by the Bihari--LaSalle inequality, if $U \in C^\alpha(\T^d)$, then $R_t(y) \leq C t^{\frac{1}{1-\alpha}}$, so that $E(y) \geq \alpha$. Thus the effective regularity is always at least as large as the \textit{bare regularity}. As shown by~\cite[Theorem 1.4]{hess-childs_turbulent_2025}, it is possible to build (time-dependent) divergence-free velocity fields $U \in L^\infty_t C^\alpha_x$ such that $E(y) = \alpha$ for all $y \in \T^d$ (after appropriately extending the definition of $E$ to allow time-dependence in $U$). Thus in general, the effective regularity need not be larger than the bare regularity.

The nonlinear Young integral theory provides exactly the needed toolbox for proving a velocity field has higher effective regularity than bare regularity. Proposition~\ref{prop:nonlinear young intro} gives a sufficient condition for the field to be effectively regular, in terms of the \textit{bare regularity} of a \textit{transformed/averaged field}. Proposition~\ref{prop:effective Holder Bihari}, stated below, gives a sufficient condition for enhanced effective regularity below the effectively-regular threshold. We note, as explained in Section~\ref{sss:sweeping estimates}, passing estimates directly on the original velocity field averaged over an integral curve is insufficient in the sweeping regime; further transformations are needed. Nevertheless, as suggested by the various open problems in Section~\ref{s:open problems}, the methods currently available for establishing effective regularity remain incomplete.

In the language of effective regularity, our first main result Theorem~\ref{thm:main ode intro} gives that (provided that the zero level set is measure zero, as is typical), if the random field we consider $U$ has regularity index $\alpha >1/2$, then almost surely, $U$ is effectively regular at almost every $y \in \T^d$---that is, there is ODE uniqueness for some amount of time.

Theorem~\ref{thm:sweeping Bihari} then gives that if $U$ has regularity index $\alpha < 1/2$ and a Lebesgue negligible zero level set, then almost surely the effective Lagrangian regularity of $U$ is greater than or equal to $\frac{\alpha}{1-\alpha}$ at almost every $y \in \T^d$. Since $\frac{\alpha}{1-\alpha} > \alpha$, we see that this proves an enhanced effective regularity over the bare regularity. We note also that $\frac{\alpha}{1-\alpha}$ crosses above $1$ at $\alpha = 1/2$, corresponding to the gain of uniqueness above the $\alpha=1/2$ threshold.

We note that the heuristic presented in Section~\ref{s:heuristic} actually predicts an effective regularity of $2\alpha$, which is greater than $\frac{\alpha}{1-\alpha}$ for $\alpha \in (0,1/2)$. We also show that $E(y) \leq 2\alpha$ for the velocity fields fitting Assumption~\ref{asmp:main for autonomous} that we construct in Section~\ref{s:sharpness} (see Theorem~\ref{thm:sharp autonomous}). While we believe that $2\alpha$ is the correct effective regularity threshold, our proof technique falls short for fairly technical reasons. See Section~\ref{sss:effective regularity} for further discussion.

Interestingly, as noted above, for the incompressible Kraichnan model with spatial regularity $\alpha$, we get particle separation like $t^{\frac{1}{2(1-\alpha)}}$, hence a formal effective Lagrangian regularity of $2\alpha -1$, which is less than $\alpha$ for $\alpha \in (0,1)$. This is, however, consistent: the Kraichnan model has a negative temporal regularity and thus obeys a different scaling than classical velocity fields. We note that $2\alpha = 2(\alpha + 1/2) - 1$, thus---in some sense---the heuristic predicting particle separation like $t^{\frac{1}{1-2\alpha}}$ can be alternatively interpreted as saying that particles behave as if they were in a Kraichnan model with spatial regularity $\alpha+1/2$.

\subsubsection{Superdiffusion and an additional regime change}

The problem of superdiffusion~\cite{chatzigeorgiou_gaussian_2025,morfe_critical_2025,armstrong_superdiffusive_2024,armstrong_superdiffusion_2026} considers exactly the autonomous divergence-free Gaussian field with independent power-law scaling coefficients, giving rise to a $C^{\alpha}_x$ velocity field; however the results of~\cite{chatzigeorgiou_gaussian_2025,armstrong_superdiffusive_2024} are for $\alpha = -1$ and~\cite{armstrong_superdiffusion_2026} are for $\alpha \in (-1,-1+\ep)$. Let us focus on the latter case as it is somewhat clearer for our purposes. We can thus see superdiffusion as essentially the same problem---controlling the rate of particle separation in random fields---but in a meaningfully different regime. In our language,~\cite[Theorem A]{armstrong_superdiffusion_2026} morally gives that for $\alpha \in (-1,-1+\ep)$, the almost sure (formal) effective Lagrangian regularity of $U$ is exactly $\alpha$. We note that since $\alpha <0$, one needs to be a bit careful in interpreting the effective regularity, and in particular there is no trivial lower bound on the effective regularity: both sides of this bound are highly non-trivial.

We see though that~\cite[Theorem A]{armstrong_superdiffusion_2026} together with our Theorem~\ref{thm:sweeping Bihari} (and Theorem~\ref{thm:main ode intro}) suggest an additional regime change, this time at $\alpha=0.$ We know for $\alpha>0$ that the effective Lagrangian regularity is greater than $\alpha$ (we conjecture it to be $2\alpha$, though can only prove $\frac{\alpha}{1-\alpha}$) but for $\alpha$ close to $-1$, we know by~\cite[Theorem A]{armstrong_superdiffusion_2026} that it is \textit{exactly} $\alpha$. Thus there must be a transition between these two regimes. The natural place is at $\alpha =0$, for the simple reason that $\alpha=0$ marks the transition between the \textit{large scales being larger than the small scales} and the \textit{small scales being larger than the large scales}. The sweeping phenomenon---which is the origin of the increased effective regularity in this work---is precisely driven by the large scales dominating the pointwise value of the velocity field. As such, it is perfectly natural for the phenomenology to change discretely at $\alpha =0$. 

\subsubsection{Analogies with renormalization}

We note that the emphasis on effective regularity over bare regularity takes inspiration from the theory of renormalization, and more precisely the Wilsonian notion of renormalization group coming from the effective field theory literature~\cite{wilson_renormalization_1974,bauerschmidt_renormalisation_2019}. A central pillar of the renormalization group method is to study the \textit{effective physics} on \textit{large scales} through \textit{coarse-grained coordinates}. We similarly want to study the ``effective physics'', but instead of working with coarse-grained coordinates to study large scales, we essentially work in \textit{averaged quasi-Lagrangian coordinates}---that is, the averaged velocity field from the perspective of a single particle. There is then a form of coarse-graining happening---since by integrating in time we smooth out fast temporal fluctuations---however, it is of a rather different form than more traditional renormalization coarse-grainings. Further, we note that there is no ``flow along scales'' present in our argument: it is purely non-iterative. While renormalization group techniques progressively coarse-grain to flow from the microscopic theory to the macroscopic theory along many small steps, we control the effective regularity in a single averaging step (containing many different components).

\subsection{Self-advection vs.\ self-regularization: discrepancies with fluid velocities}

Fluid velocity fields are \textit{self-advecting}; in particular, the large scales of the field push around the small scales. The fields we consider, being both autonomous and the sum of \textit{independent} scale fields, are very much not self-advecting. Self-advecting fields do not seem to exhibit the sweeping effect, which is central to the self-regularization present in the fields we consider. Since the small scales move with the large background field, a particle is not swept over them, but rather moves with them. As such, very loosely (neglecting the important contribution due to the pressure) we should imagine the small scales of a fluid to be \textit{frozen} from the perspective of an ODE trajectory. On the other hand, for the random autonomous fields we consider, the small scales are typically rapidly fluctuating from the perspective of an ODE trajectory.

Thus we see the self-advection basic to fluids works directly to prevent self-regularization due to sweeping. We in particular expect in fluid-like fields that particles separate at the maximal possible rate, $\approx t^{\frac{1}{1-\alpha}}$ for a $C^\alpha_x$ field. Thus we expect the effective Lagrangian regularity to be exactly $\alpha$ for a $C^\alpha_x$ fluid-like velocity field, without any gain of effective regularity. The most physically relevant case is $\alpha=1/3$, the natural regularity scale for fluids following K41 theory~\cite{frisch_turbulence_1995}. The classical Richardson dispersion prediction is that particles separate in a turbulent fluid velocity like $t^{\frac{3}{2}} = t^{\frac{1}{1- 1/3}}$, that is, their effective Lagrangian regularity and bare regularity agree at $1/3$.

We thus see by Theorem~\ref{thm:sweeping Bihari} that a typical $C^{1/3}_x$ random field has meaningfully different particle dispersion statistics than a true fluid. This shows another advantage of the concept of effective Lagrangian regularity, as it is a more precise description of the phenomenology than a purely qualitative statement such as the presence of anomalous dissipation of passive scalars or the presence of nonunique ODE trajectories---both of which are expected for both the autonomous $C^{1/3}_x$ Gaussian field and a typical fluid field.

The iterative homogenization works~\cite{armstrong_anomalous_2025,burczak_anomalous_2023_fixed,burczak_scalar_2026} crucially rely on a (one-sided) self-advecting structure to ensure that the small scales are essentially frozen from the perspective of a particle moving in the background flow of the large scale field. \cite[Appendix A]{armstrong_anomalous_2025} makes particularly clear how their central argument would fail precisely due to (a version of) the sweeping effect. The argument presented there can be interpreted as the law-of-large-numbers scaling: an incoherent average of a zero-mean object tends towards zero. In contrast, we are analyzing the next order: the central-limit scaling. This clearly separates the phenomenology of the iterative homogenization anomalous dissipation works---as well as the rest of the deterministic anomalous dissipation literature more broadly---from the phenomenology studied in this work.

\subsubsection{Behavior at the zero level set}

As noted above, self-advection causes the small scales to essentially be frozen from the perspective of an ODE trajectory. In a random autonomous field, a particle originating from a stagnation point---that is, with initial data on the zero level set---will also have the small scales (as well as the large scales) be frozen. This is just another version of the sweeping regularization argument failing at the zero level set. However, it is suggestive that the zero level set truly does not gain any effective regularity. We phrase this as an open problem in Section~\ref{s:open problems}.

\subsection{Lagrangian chaos}

A different topic for which random velocity fields can be shown to typically exhibit Lagrangian properties that are difficult to access deterministically is the study of \textit{Lagrangian chaos and exponential mixing}. \cite{bedrossian_lagrangian_2022} proves that suitable (highly viscous) stochastically forced fluid models exhibit Lagrangian chaos (or equivalently, have a positive top Lyapunov exponent), that is, there is exponential growth of the derivative of the flow map $\Phi_t$ for the ODE associated to the velocity field solving the stochastic fluid equation. Building on~\cite{dolgopyat_sample_2004}, in~\cite{bedrossian_almost-sure_2022} this is then leveraged to prove exponential mixing for the associated transport equation. This inspired a large body of interesting work---see in particular~\cite{gess_stabilization_2021,blumenthal_exponential_2023,luo_elementary_2024,bedrossian_negative_2024,cooperman_exponential_2026,navarro-fernandez_exponential_2026}---which strongly suggests that Lagrangian chaos is quite generic for ``typical'' divergence-free velocity fields. However, proving Lagrangian chaos---which is weaker than exponential mixing---in even the rather simple deterministic case of the Chirikov standard map is an extremely difficult and long-standing open problem~\cite{blumenthal_lyapunov_2017}. 

Thus, in both the study of Lagrangian chaos and the problem we consider, randomness in the velocity field makes generic a very delicate deterministic property. In both cases, the randomness is preventing ``conspiracies'', but in somewhat different directions. In the case of Lagrangian chaos, the randomness ensures that there are no subtle cancellations preventing exponential growth of the derivative. This is similar to our example demonstrating nonuniqueness, as stated by Theorem~\ref{thm:autonomous intro sharp}, which uses the randomness to give a lower bound, ensuring an incoherent sum does not perfectly cancel. In contrast, in our main uniqueness theorem, Theorem~\ref{thm:main ode intro}, the randomness is ensuring that there are \textit{enough} cancellations.

\subsection{Level-set preserving fields}

Two-dimensional, autonomous, divergence-free velocity fields $u : \T^2 \to \R^2$ have the very special property of admitting a \textit{stream function}---that is a function $\psi : \R^2 \to \R$ such that $u = \nabla^\perp \psi$. Thus for any $\dot X_t = u(X_t)$, we have that 
\[\frac{d}{dt}\psi(X_t) = \dot X_t \cdot \nabla \psi(X_t) = u(X_t) \cdot \nabla \psi(X_t) = \nabla^\perp \psi(X_t) \cdot \nabla \psi(X_t) =0.\]
Thus, the ODE flow of $u$ preserves the level sets of $\psi$. This allows for a particularly geometric treatment of the ODE problem in this special case, through the study of the level sets of $\psi$. An elementary result is then that if $u$ has the above properties and $u \in C^0(\T^2)$, then $u$ admits unique ODE solutions, away from the zero level set $\{u=0\}$~\cite[Theorem 2.1]{bouchut_two-dimensional_2001} (see also~\cite[Theorem 5.1]{silvestre_loss_2013} for an interesting alternate proof).

The striking result of~\cite{alberti_uniqueness_2014} interprets the uniqueness problem using the geometric structure of the level sets of $\psi$ and gives a precise---necessary and sufficient---condition for uniqueness of bounded solutions to the associated continuity equation: the weak Sard property. The weak Sard property is in particular implied if $|\{u=0\}| =0$. This is the central observation used in~\cite{bagnara_regularity_2026} to prove uniqueness for bounded solutions to the transport equation for suitable random, continuous, divergence-free, autonomous velocity fields in two dimensions.

The special structure of 2D autonomous divergence-free flows has a natural extension to higher dimensions; we will focus on the 3D case for simplicity. If $u(x) = \nabla \phi(x) \times \nabla \psi(x)$, for $\phi, \psi : \T^3 \to \R$, then---since $A \times B \perp A, B$---we have that for any solution $\dot X_t = u(X_t)$, $\dot \phi(X_t) = \dot \psi(X_t) =0$. Such a decomposition of a 3D divergence-free field---similar to the streamfunction representation in 2D---allows one to largely reduce the ODE dynamics to the level set geometry, now the level sets of $(\phi,\psi): \R^3 \to \R^2$. In contrast to the 2D case, having a representation of this form is a very special property, not a universal property of divergence-free fields. Building on ideas from~\cite{alberti_uniqueness_2014,alberti_structure_2013}, \cite{bagnara_regularity_2026} proves uniqueness for bounded solutions to the transport equation for suitable random divergence-free autonomous velocity fields in three dimensions, assuming they admit such a representation.

In contrast to the above results, our hypotheses are dimension-independent and assume no special geometric structure. As such, unlike for 2D autonomous divergence-free fields, it seems that the powerful tools from geometric measure theory cannot be used in this setting.

In both the 2D autonomous divergence-free case and the general sweeping regime of Theorem~\ref{thm:main ode intro}, the zero level set of the velocity field plays a special role as a place of possible nonuniqueness. However, this is due to rather different mechanisms. In the case of the stream function, it is essentially due to the implicit function theorem---which guarantees a simple structure of the level curves away from critical points---together with the special 1D fact of ODE uniqueness away from zeros. In our sweeping regime, the zero level set is the unique location where the sweeping mechanism fails entirely. In the stream function case, there is always uniqueness away from zeros, while in our case, we only expect uniqueness for $\alpha>1/2$. Both cases suggest that zeros are particularly unstable locations in ODEs. This is sensible since nonuniqueness at a zero of the velocity field only requires a rough cusp at a single point, while nonuniqueness away from a zero requires sufficiently coherent roughness over an entire curve of points, as demonstrated by Proposition~\ref{prop:nonlinear young intro}.

\subsection{Well-posedness by random data}

Another interesting point of comparison for the present work is the literature on probabilistic well-posedness for nonlinear dispersive PDE. For many such equations, scaling identifies a natural critical Sobolev index, and the deterministic Cauchy problem becomes ill-posed below the corresponding threshold. For example, the cubic wave equation in three dimensions is locally well-posed for initial data in $H^s\times H^{s-1}$ for $s\geq \frac{1}{2}$ but strongly ill-posed below this threshold~\cite{ChristCollianderTao2003Illposedness,Lebeau2005PerteRegularite}. On the other hand, appropriate randomization of the initial data can yield almost sure well-posedness below the deterministic threshold. Returning to the cubic wave equation, independently randomizing the spectral coefficients of the initial datum yields almost sure local well-posedness in $H^s\times H^{s-1}$ with $s=\frac{1}{4}$~\cite{BurqTzvetkov2008RandomDataWaveI}. Analogous probabilistic well-posedness has been developed for many other dispersive equations, including the cubic nonlinear Schr\"odinger equation~\cite{BenyiOhPocovnicu2015CubicNLS} and energy-critical defocusing quintic wave equation~\cite{OhPocovnicu2016QuinticWave}. Thus, similarly to the phenomena described in the present work, the addition of randomness lowers a well-posedness threshold.

This lowered regularity threshold is again due to the randomness preventing ``conspiracies''. More precisely, the randomization makes anomalously coherent frequency superpositions probabilistically atypical, yielding better space-time integrability estimates for the free linear evolution of the equation. These improved estimates then allow the closure of a perturbative argument. In contrast, in the present work the stochastic cancellation is generated by the random dynamics itself and yields---in the divergence-free case---almost sure well-posedness for bounded solutions to the corresponding transport equation.

\subsection{Regularization by noise}

Finally, we discuss the relationship between the (extensive) literature on regularization by noise and our work. The general setting for regularization by noise is a deterministic ODE (or PDE, which can of course be regarded as an ODE in infinite dimensions) $\dot X_t = U(X_t)$ that is not \textit{a priori} well-posed. For example, $U : \T^d \to \R^d$ is a fixed velocity field, and we only know that $U \in C^0(\T^d)$. The noise is then added \textit{extrinsically}---e.g.\ as an additive white noise---that is, one considers the perturbed problem $\dot X_t= U(X_t) + \dot W_t$, where $W_t$ is a standard Brownian motion. The classical results of~\cite{zvonkin_transformation_1974,veretennikov_strong_1981} study exactly this situation, proving well-posedness of strong solutions (a certain kind of stochastically adapted solution) to this problem, requiring no additional regularity of $U$. The groundbreaking work~\cite{davie_uniqueness_2007} proves well-posedness in the classical solution class: there almost surely exists a unique ODE solution, with no adaptedness requirements (this is now known as path-by-path uniqueness). These results show that $U$ is ``regularized by noise'': adding a noise perturbation makes the problem well-posed at $C^0_x$ instead of the deterministic threshold of $C^1_x$.

The phenomenon of regularization by noise has now been extensively studied: relaxing the necessary integrability conditions for well-posedness~\cite{krylov_strong_2005,rockner_sdes_2023,rockner_sdes_2025, anzeletti_path-by-path_2025}; proving extensions to negative regularity velocity fields down to $C^{- 1/2}_x$~\cite{flandoli_multidimensional_2017,zhang_heat_2018}; giving extensions to non-Brownian noises, such as fractional Brownian motion~\cite{nualart_regularization_2002,catellier_averaging_2016,galeati_noiseless_2021,galeati_solution_2025,butkovsky_weak_2025} and $\alpha$-stable processes~\cite{priola_pathwise_2012,priola_davies_2018,athreya_strong_2020,kremp_rough_2025}; and proving an enhanced well-posedness for PDE~\cite{flandoli_well-posedness_2010,beck_stochastic_2019,flandoli_delayed_2021}.  

Regularization by noise is closely related to the phenomenology we consider---as is perhaps most evident from our essential utilization of the nonlinear Young integral tools developed in~\cite{catellier_averaging_2016,galeati_noiseless_2021} for the purposes of studying regularization by noise. Additionally, both the problem we consider and the problem of regularization by noise rely on stochastic cancellations damping small scales sufficiently to generate enough effective regularity for well-posedness.

Despite the similarities, the problems are fundamentally quite distinct. In regularization by noise, the noise is an external perturbation that acts to regularize the ill-posed deterministic object. In the problem we study, there is only one object---the random velocity field---and it regularizes itself, through the sweeping effect. The nonuniqueness/ill-posedness sides of both problems are perhaps more closely related, as can be seen in the similarities between the construction of Section~\ref{s:sharpness} and that of~\cite{hess-childs_sharp_2026}. 

\section{The refreshing case}

\label{s:refreshing intro}

The above discussion focused exclusively on the sweeping-driven phenomenology, as governed by Assumption~\ref{asmp:main for autonomous}. This is largely for expositional clarity. In the body of this paper, we consider also the related---but mechanistically distinct---\textit{refreshing-driven regime}. While for the sweeping regime, we consider only autonomous velocity fields, in the refreshing setting, we vitally use the time-dependence of the velocity field. We make \textit{no assumption} on the range of dependence \textit{in space}. Rather, we suppose only a finite range of dependence \textit{in time}. More precisely, we take the following assumption.
\begin{assumption}
    \label{asmp:main for refreshing}
     For some $\beta>0$, we suppose $U :[0,1]\times \T^d \to \R^d$ is a random velocity field such that 
    \[U(t,x) = \sum_{j=0}^\infty u^j(t,x)\]
    where $u^j$ satisfy the following properties.
    \begin{enumerate}
        \item\label{asmp:scale independence} The fields $u^j$ are mutually independent.
        \item\label{asmp:centered} For all $j \in \N$ with $j\geq 1$, $\E u^j =0.$ 
        \item\label{asmp:finite range} For all $j \in \N$ with $j \geq 1$ and all Borel $A,B \subseteq [0,1]$ with $\mathrm{dist}(A,B) \geq 2^{-\beta j}$, we have that $u^j|_{A \times \T^d} \indep u^j|_{B \times \T^d}$.
    \end{enumerate}
\end{assumption}

We require that the scale fields in the previous assumption obey regularity estimates that ensure, almost surely, that $U \in C^0([0,1],C^{\alpha-}(\T^d))$. In particular, we will require for some $p \geq 1$ depending on $d,\alpha$ and $\beta$---specified in the statements below---that
    \begin{equation}
    \label{eq:refreshing theorem hypothesis}
    \max_{n=0,1,2}\sup_{j\in \mathbb N} 2^{(\alpha-n)j}\|\nabla^n u^j\|_{L^p_\omega \continuous^0_{t,x}}<\infty.\end{equation}

We note that $\beta$ is now an additional free parameter governing the rate of refreshing in time. In the refreshing regime, we consider the (non-autonomous) ODE, started from an arbitrary time $s \in [0,1]$,
\begin{equation}
    \label{eq:main refreshing ODE s}
    \begin{cases}
        \dot X_t = U(t,X_t)\\
        X_s = y,
    \end{cases}
\end{equation}
as well as the two-parameter flow maps, which solve
\begin{equation}
    \label{eq:main refreshing flow}
    \begin{cases}
     \frac{d}{dt} \Phi_{s,t}(y) = U(t,\Phi_{s,t}(y)),\\
        \Phi_{s,s}(y) = y.
    \end{cases}
\end{equation}

Since the decorrelation is taken explicitly in time, we no longer need to utilize a background sweeping flow to move us over small scales, causing stochastic cancellation. Instead, the stochastic cancellation comes directly from the small scales rapidly fluctuating in time. As such, stagnation points of the velocity field no longer pose a problem, leading to the cleaner form of Theorem~\ref{thm:refreshing intro} below. Since well-posedness in the refreshing setting does not rely on the ``self-regularization'' of the large scales sweeping the small scales, it allows for substantially simpler proofs as well as more quantitative estimates. 

One can perform a precisely analogous heuristic argument to that of Section~\ref{s:heuristic}, which predicts ODE uniqueness when $\alpha +\beta/2>1$ and particle separation at the rate $t^{\frac{1}{2(1-\alpha-\beta/2)}}$ when $\alpha + \beta/2<1$.\footnote{We note that a similar heuristic prediction appears in the physics literature~\cite{chaves_lagrangian_2003}, however, in the slightly different ``quasi-Lagrangian'' setting.} Thus in the refreshing case, we get a conjectural effective Lagrangian regularity of $2\alpha + \beta - 1$ in place of the $2\alpha$ of the sweeping regime. This is extremely natural, as the sweeping heuristic is precisely that the large scales move a particle at essentially unit speed; this leads to the assumed finite range of dependence \textit{in space} turning into a form of finite range of dependence \textit{in time}, as the particle slides over the small scales. Since the spatial range is $2^{-j}$, the unit speed translates (morally) to a temporal range of dependence of $2^{-j}$. That is: \textit{the sweeping regime behaves like the refreshing regime with $\beta=1$}, as is also evinced by the heuristic effective regularities agreeing at this parameter.

The refreshing case is then worth engaging with for several reasons. The most important is that it exhibits phenomenology of independent interest, showing the flexibility of our tools and the effective regularity concept. The second is that it allows for quantitative stochastic integrability estimates, which are essentially unavailable in the autonomous case; see Theorem~\ref{thm:refreshing ODE quantitative} and Theorem~\ref{thm:refreshing ODE qualitative}. Finally, as discussed more in Section~\ref{ss:wellposedness} and particularly Section~\ref{sss:sweeping estimates}, the sweeping proof relies on reducing to a more complicated version of the refreshing setting. Thus the proofs in the refreshing setting serve to demonstrate the primary ideas in a simpler setting.

We now give clean statements of our main results in the refreshing regime; more precise and quantitative versions of Theorem~\ref{thm:refreshing intro} and Theorem~\ref{thm:refreshing intro sharp} are given in Section~\ref{s:overview}. Our first result essentially gives that if the scale fields $u^j$ of $U$ are such that $U \in C^0_t C^{\alpha-}_x$ for $\alpha\in(0,1)$ and $\alpha + \beta/2>1$, then the problems~\eqref{eq:main refreshing ODE s},~\eqref{eq:main refreshing flow}, and~\eqref{eq:continuity} are \textit{well-posed.} We note that the following result, as well as Theorem~\ref{thm:refreshing Bihari}, holds for any $d \geq 1$.

\begin{theorem}
    \label{thm:refreshing intro} Let $\beta >0$ and $\alpha \in (0,1)$ such that $\alpha + \beta/2 > 1$. Suppose $U:[0,1] \times \T^d \to\R^d$ is a $C^0([0,1],C^{\alpha-}(\T^d))$ multiscale finite range random velocity field; that is, $U$ satisfies Assumption~\ref{asmp:main for refreshing}, and \eqref{eq:refreshing theorem hypothesis} holds for some $p >\frac{ \beta}{\alpha + \beta/2-1} \lor \frac{\beta}{\alpha}$.
    
    Then almost surely:
    \begin{enumerate}
        \item \label{item:intro refresh ode good} for all $y \in \T^d$ and $s \in [0,1]$, the ODE~\eqref{eq:main refreshing ODE s} admits a unique solution on $[0,1]$, and thus the two-parameter flow $\Phi_{s,t}$ given by~\eqref{eq:main refreshing flow} is uniquely defined,
        \item \label{item:intro refresh lipschitz} $\Phi_{s,t}$ is uniformly-in-time Lipschitz-in-space, $\sup_{0 \leq s,t \leq 1} \|\Phi_{s,t}\|_{W^{1,\infty}_y} < \infty$,
        \item \label{item:intro refresh pde} and the continuity equation~\eqref{eq:continuity} admits unique positive measure-valued solutions, that is, if $\phi \in TV(\T^d)$ is a positive measure, then there exists a unique distributional solution to~\eqref{eq:continuity} such that for all $t \in [0,1], f(t,\cdot) \in TV(\T^d)$ is a positive measure.
    \end{enumerate}
\end{theorem}

Our next result shows that $\alpha+\beta/2 =1$ represents a real threshold for qualitative regime change, from guaranteed Lagrangian well-posedness to possible Lagrangian ill-posedness.

\begin{theorem}
\label{thm:refreshing intro sharp}
    For all $d \geq 3,$ $\beta>0$, and $\alpha \in (0,1)$
 with $\alpha + \beta/2<1$, there exists a multiscale random velocity field $U : [0,1] \times \T^d \to \R^d$ in $C^0([0,1],C^{\alpha-}(\T^d))$; $U$ satisfies Assumption~\ref{asmp:main for refreshing} for $\beta$ and is such that \eqref{eq:refreshing theorem hypothesis} holds for all $p \geq 1$. However, almost surely,
 \begin{enumerate}
     \item for almost every $y\in \T^d$,~\eqref{eq:main refreshing ODE s} with $s=0$ admits nonunique solutions on $[0,1]$,
     \item and there exists a positive bounded $\phi$ such that \eqref{eq:continuity} admits nonunique positive bounded solutions on $[0,1]$.
 \end{enumerate}
 \end{theorem}

We finally give a refreshing analog of Theorem~\ref{thm:sweeping Bihari}, proving an enhanced effective Lagrangian regularity below the critical threshold $\alpha + \beta/2 =1$.
\begin{theorem}
    \label{thm:refreshing Bihari}
    Let $\beta >0$ and $\alpha \in (0,1)$ such that
    \[\alpha + \beta >1 \quad \text{and} \quad \alpha + \beta/2 <1.\]
    Suppose $U:[0,1]\times\T^d\rightarrow \R^d$ is a $C^0([0,1],C^{\alpha-}(\T^d))$ multiscale finite range random velocity field; that is, $U$ satisfies Assumption~\ref{asmp:main for refreshing}, and~\eqref{eq:refreshing theorem hypothesis} holds for all $p \geq 1$.
    
    Then for all $\ep>0$ and $y \in \T^d$ there exists a random constant $M(y)$ such that $M(y) \in L^p_\omega$ for all $p \geq1$ and if $X^1, X^2$ are any solutions to~\eqref{eq:main refreshing ODE s} with $s=0$, then we have the (quasi)-stability estimate for all $t \in [0,1]$,
    \[d_{\T^d}(X^1_t, X^2_t) \leq M(y)t^{\frac{1-\alpha}{\beta(2 - 2\alpha - \beta)}-\ep}.\]
\end{theorem}

\begin{remark}
    The restriction $\alpha +\beta>1$---which will appear also in the more precise version of Theorem~\ref{thm:refreshing intro sharp} given by Theorem~\ref{thm:sharp refreshing}---is natural to the problem. When $\alpha + \beta<1$, the contribution of $u^j$, the velocity field on scale $2^{-j}$, on the time of decorrelation $2^{-\beta j}$, is to move the particle $\approx 2^{-(\alpha +\beta) j} \gg 2^{-j}.$ Thus $u^j$ moves a particle over many of its own length scales in the correlated time interval; in contrast, if $\alpha + \beta >1$, $2^{-(\alpha+\beta)j} \ll 2^{-j}$, so the effect of $u^j$ relative to its length scale $2^{-j}$ is (morally) negligible. This allows for the possibility of additional phenomenology when $\alpha + \beta<1$ and prevents the same argument as the $\alpha + \beta > 1$ case from working. We mostly care about $\beta \geq 1$ in any case---as otherwise one should really account for the possibility of the sweeping effect causing even further regularization---so this is a fairly harmless restriction.
\end{remark}

\section{Open problems}

\label{s:open problems}
We now state a variety of interesting open problems. We believe resolving either of the two following problems would represent a substantial advance in the understanding of the generic Lagrangian phenomenology of random H\"older velocity fields.

\begin{problem}
\label{prob:generic ill-posedness}
    Prove a \textit{typical} version of the transition to ill-posedness, of which Theorem~\ref{thm:autonomous intro sharp} and Theorem~\ref{thm:refreshing intro sharp} are examples. That is, show for a suitable class of random velocity fields---including at least the Gaussian fields defined in Section~\ref{s:examples}---there is ODE ill-posedness for $\alpha <1/2$ in the sweeping case or $\alpha + \beta/2<1$ in the refreshing case.
\end{problem}

\begin{remark}
    As for our own results, the refreshing case should be substantially easier to resolve in the above problem. Due to the external refreshing in time, proving central limit theorem-like lower bounds is conceivable. However, the challenges are numerous. A prototypical difficulty is that for particles separated to scale $2^{-j}$, we expect the scale field $u^j$ to be generating the majority of the separation, following heuristic computations. That said, demonstrating this to be true would require passing sharp effective regularity estimates on all of the remaining scales, both smaller and larger; the ``bare'' size of the ``non-resonant'' scales would allow particles to separate at much faster rates than the $u^j$ field is expected to separate the particles.

    The divergence-free case will likely be easier to tackle; the presence of a large potential term could potentially shift the phenomenology away from nonuniqueness, e.g.\ as seen in the Kraichnan model~\cite{drivas_anomalous_2025} (see also the heuristic arguments of~\cite{chaves_lagrangian_2003}).

    A suitable model class for ill-posedness should essentially just be Assumption~\ref{asmp:main for autonomous} or Assumption~\ref{asmp:main for refreshing} (for the sweeping and refreshing regimes respectively) together with a suitable non-degeneracy assumption---such as $\E u^j(x) \otimes u^j(x) \gtrsim 2^{-2\alpha j} I$ for all $x \in \T^d$---and control on the divergence, e.g.\ $\nabla\cdot U=0$.

    Finally, we note that allowing for hypergeometric scale separations should make the problem a lot easier---analogous to the results in multiscale homogenization~\cite{armstrong_anomalous_2025,burczak_anomalous_2023_fixed,burczak_scalar_2026}---essentially as it makes the ``one active scale at a time'' picture much closer to being true. This would constitute a very interesting step in the right direction, though not quite a full resolution of the problem.
\end{remark}

 Our main well-posedness result gives well-posedness only away from the zero level set, and the ill-posedness example has a trivial zero level set. As such, the behavior of particles started on the zero level set is not at all understood. The next problem is to address this gap.

\begin{problem}
    Characterize the generic behavior of ODEs in autonomous random H\"older velocity fields \textit{started on the zero level set} for a suitable class of fields---including at least the Gaussian fields defined in Section~\ref{s:examples}.
\end{problem}

\begin{remark}\label{rem:zero set}
    Once again, the divergence-free case should be both more tractable (as well as being a case of particular interest). A natural way to study this problem is to consider particles started at $y=0$ for the velocity field $U(x) := \tilde U(x) - \tilde U(0)$, where $\tilde U$ is a field satisfying hypotheses like Assumption~\ref{asmp:main for autonomous}. For some relevant physics heuristics, see~\cite{chaves_lagrangian_2003}.

    It seems likely that there is \textit{no effective regularity gain} at the zero level set; in particular, there is ODE ill-posedness for all $\alpha \in (0,1)$. To ``zeroth order'', that is because there is no sweeping effect to cause the regularization. However, to compute the effective regularity gain, one needs to take heuristics to the next order: as a particle leaves the zero level set, sweeping-like behavior will become nonnegligible (though still much weaker than macroscopically away from the zero level set, by the regularity of the velocity field). A heuristic argument however suggests that this reduced sweeping effect is insufficient to meaningfully impact the effective regularity.
\end{remark}

The next problem is of independent interest, and the estimates needed to resolve it should also prove useful for Problem~\ref{prob:generic ill-posedness}.

\begin{problem}
    Prove a sharp effective Lagrangian regularity lower bound---$2\alpha$ when $2\alpha <1$ in the sweeping case and $2\alpha +\beta-1$ when $2\alpha +\beta -1 <1$ in the refreshing case---under essentially the same hypotheses as this work.
\end{problem}

\begin{remark}
    This problem amounts to a sharpening of Theorem~\ref{thm:sweeping Bihari} and Theorem~\ref{thm:refreshing Bihari} to the rates established by the ill-posedness examples of Theorem~\ref{thm:sharp autonomous} and Theorem~\ref{thm:sharp refreshing}. As discussed in Section~\ref{sss:effective regularity}, the tools and ideas of stochastic sewing should be useful for this problem, which may also require developing a more suitable probabilistic solution notion to ensure only ``adapted'' trajectories are considered.
\end{remark}

The next three problems would constitute significant developments in the tools and theory of this work. The first is to give a higher regularity version of the effective regularity theory. For $\alpha \geq 1$, a $\mathcal{C}^\alpha_x$ velocity field generically has no better than a $\mathcal{C}^\alpha_x$ ODE flow map. The general idea of this work is that stochastic cancellations should cause the velocity field to gain additional effective regularity, though currently this is only studied when $\alpha \in (0,1)$. It is natural then to ask whether this is visible even at \textit{high regularities}. The spatial regularities of the flow map $\alpha+1/2$ and $\alpha+\beta/2$ are expected through analogy with the Kraichnan model.

\begin{problem}
    Show that, for general $\alpha >0$ (including $\alpha \gg 1$) above the critical threshold, random velocity fields obeying a suitable multiscale decomposition into finite range fields give rise to ODE flows with spatial regularity $\alpha + 1/2$ in the sweeping regime (which would require localizing away from the zero level set) and $\alpha + \beta/2$ in the refreshing regime. 
\end{problem}

The next problem is to extend the refreshing regime results to the negative regularity setting, i.e.\ $\alpha + \beta/2 >1$ but $\alpha <0$. While we expect this to work, at least to some extent, the proof would need a fairly substantial modification. The most important failure is that we no longer have access to any good \textit{a priori} (quasi)-stability of the ODE trajectories, which is used in the Bihari--LaSalle-type argument in the proof of Lemma~\ref{lem:abstract averaging}. That is, we can never resort to the bare regularity of the field and must use the effective regularity throughout the proof. We note also that some care is necessary in interpreting solutions for negative regularity vector fields, though appropriate notions of solution have been introduced in, e.g., \cite[Definition 4.2]{galeati_noiseless_2021} and~\cite[Definition 2.2]{butkovsky_stochastic_2025}.

\begin{problem}
    Prove ODE well-posedness in the refreshing case for general negative $\alpha<0$ provided $\alpha + \beta/2>1$.
\end{problem}

The final problem works to demonstrate that the full multiscale decomposition obeying such strong hypotheses is actually necessary to capture the sweeping phenomenology. It is in fact rather difficult to construct a deterministic velocity field $u \in C^\alpha_x$ for $\alpha>1/2$ such that $u + Z$ has nonunique ODE solutions with positive probability, where $Z$ is a standard normal on $\R^d$. This is essentially because a version of the sweeping effect seemingly circumvents all usual ODE nonuniqueness constructions for a.e.\ choice of $Z$. However, it seems quite unlikely that a random background constant is sufficient to ensure a.s.\ ODE uniqueness.

\begin{problem}
    Construct for some $d \in \N$, $\alpha > 1/2$ a single $u \in C^\alpha(\T^d)$ such that the ODE
    \[\begin{cases}\dot X_t = u(X_t) + z,\\
    X_0 = 0,
    \end{cases}
    \]
    admits instantaneous nonuniqueness (that is nonuniqueness on the time interval $[0,h]$ for every $h >0$) for a positive Lebesgue measure set of $z$ in $\R^d$.
\end{problem}

\begin{remark}
    There are of course many natural modifications of the above problem, such as whether nonuniqueness can be made generic in the initial data, whether the nonuniqueness is possible for every $z \in \R^d$, whether $u$ is taken to be divergence-free or not, whether $u$ is taken to be time-dependent or not, whether $\alpha$ can be taken arbitrarily close to $1$, etc. The central mathematical question is whether the sweeping effect is so strong as to prevent ODE nonuniqueness even without strong decorrelation assumptions baked into the velocity field.
\end{remark}

\section{Examples}
\label{s:examples}

In this section, we give natural examples---Gaussian fields with power law spectra and Poisson-driven fields---that satisfy the hypotheses of the sweeping regime, Assumption~\ref{asmp:main for autonomous}, and those of the refreshing regime, Assumption~\ref{asmp:main for refreshing}.

Gaussian fields with power law spectra are very commonly used $\alpha$-H\"older velocity fields, naturally arising as the inverse fractional Laplacian of spatial white noise. In the refreshing case, we add an (appropriately scaled) Ornstein--Uhlenbeck time dependence, which is also a classical decorrelating-in-time Gaussian process. We note that both the power law spectrum Gaussian field and the OU process are not naturally given in a finite range decomposition; the existence of such a decomposition needs to be proved.

Poisson point processes provide a natural way to generate finite range fields, using the Poisson points as the base of a random compactly supported velocity field. For the sweeping case, we use spatial Poisson points; for the refreshing case, we use temporal Poisson points. For these fields the finite range property is essentially direct from the construction. However, proving the density bounds necessary to invoke Lemma~\ref{lem:sufficient condition for good zero set}---which in turn implies the assumption~\eqref{eq:small zeros} of Corollary~\ref{cor:main pde ode}---requires a bit more care.

\subsection{Gaussian fields}

For $b\in\{0,1\},$ we denote
\[J^{k,b}:=I-b\frac{k\otimes k}{|k|^2}.\]
Thus $b=1$ corresponds to Leray projection onto divergence-free fields.

\subsubsection{The sweeping case}

We first construct autonomous Gaussian fields that obey the sweeping hypotheses, Assumption~\ref{asmp:main for autonomous}. Fixing $\alpha\in(0,1)$ and $b\in\{0,1\}$, we define the $\R^d$-valued, autonomous, Gaussian vector field
\begin{equation}\label{eq:autonomous gaussian}
U^{\alpha,b}(x):=\sum_{k \in \Z^d \backslash \{0\}} |k|^{-d/2 -\alpha} e^{-2\pi i k\cdot x} J^{k,b}\zeta^k,    
\end{equation}
where the $\zeta^k$ are standard $\C^d$-valued Gaussian random variables, $\zeta^{-k}=\overline{\zeta^k}$, and $\zeta^k$ is independent of $\zeta^\ell$ unless $\ell=\pm k$. The condition $\zeta^{-k}=\overline{\zeta^k}$ ensures that $U^{\alpha,b}$ is real-valued.

We then note that $U^{\alpha,b}$ satisfies the following properties, the proof of which is deferred to Appendix~\ref{appen:examples}. The finite range decomposition follows ideas developed in renormalization theory, in particular~\cite{bauerschmidt_simple_2013} and~\cite[Chapter 3]{bauerschmidt_renormalisation_2019}.

\begin{proposition}\label{prop:autonomous gaussian}
Let $\alpha\in(0,1)$ and $b\in\{0,1\}$. Then there exists a mutually independent family of smooth, centered, autonomous, Gaussian fields $\{u^j\}_{j\geq 0}$ such that
\begin{enumerate}
    \item\label{item: autonomous scale decomp} $U^{\alpha,b}$ is equal to $\sum_{j= 0}^\infty u^j$ in law,
    \item\label{item:autonomous finite range} for each $j\geq 1$, $u^j$ has a finite spatial range of dependence $2^{-j}$,
    \item\label{item: autonomous moment bounds} and there exists $C(d,\alpha)>0$ such that for all $j\geq 0$, $p\geq 1$, and $n\in\{0,1,2\}$,
    \[\|\nabla^n u^j\|_{L^p_\omega C^0_x}\leq  C (p^{1/2}+j^{1/2})2^{(n-\alpha)j}.\]
\end{enumerate}
Moreover, if $b=1$ and $d\geq 2$, then
\begin{enumerate}
    \item[(i)]\label{item:div free} $\nabla\cdot U^{\alpha,b}=0$,
    \item[(ii)]\label{item:zero level set} and almost surely, letting $Z_\eps$ be defined as in~\eqref{eq:zero sets}, for all $0 \leq \alpha'<\alpha$,
    \[\lim_{\eps\rightarrow 0} \eps^{-\alpha'd}|Z_\eps|=0.\]
\end{enumerate}
\end{proposition}

 Items~\ref{item: autonomous scale decomp}-\ref{item: autonomous moment bounds} show that $U^{\alpha,b}$ satisfies the assumptions of Theorem~\ref{thm:main ode intro} for any $\alpha'$ such that $\frac{1}{2}<\alpha'<\alpha$. If in addition $b=1$ and $d\geq 2$, then Items~\hyperref[item:div free]{(i)} and~\hyperref[item:zero level set]{(ii)} show that $U^{\alpha,b}$ satisfies the assumption of Corollary~\ref{cor:main pde ode}, using that $\alpha d \geq \alpha >1-\alpha$ for $\alpha \in (1/2,1)$.

\begin{remark}
In fact, the local moduli of continuity for $U^{\alpha,b}$ are no better than $C^{\alpha-}_x$. When $d\geq 2$, for every index $1\leq i\leq d$, it almost surely holds that for almost every $x\in\T^d$,
\begin{equation}\label{eq:local modulus}
\limsup_{y\rightarrow x} \frac{|U^{\alpha,b}_i(x)-U^{\alpha,b}_i(y)|}{|x-y|^{\alpha}}=\infty.   \end{equation}
Thus the local bare regularity of $U^{\alpha,b}$ is strictly worse than $C^{\alpha}_x$ almost everywhere. The blow up~\eqref{eq:local modulus} follows from the law of the iterated logarithm for $U^{\alpha,b}_i$, see for example~\cite[Theorem 5.6]{meerschaert_fernique_2013}.
\end{remark}

\subsubsection{The refreshing case} Next we construct Gaussian fields that obey the refreshing hypotheses, Assumption~\ref{asmp:main for refreshing}. Fixing $\alpha\in(0,1)$, $\beta>0$ and $b\in\{0,1\}$, we define
\begin{equation}\label{eq:refreshing gaussian}
U^{\alpha,\beta,b}(t,x):=\sum_{k \in \Z^d \backslash \{0\}} |k|^{-d/2 -\alpha} e^{-2\pi i  k\cdot x} J^{k,b}\zeta^{k,\beta}_t,  
\end{equation}
where $\zeta^{k,\beta}_t$ are stationary solutions to the Ornstein--Uhlenbeck SDE
\[
d\zeta^{k,\beta}_t=-|k|^\beta\zeta^{k,\beta}_t+\sqrt{2}|k|^{\beta/2}dW_t^{k,\beta},
\]
for a standard $\C^d$-valued Brownian motion $W_t^{k,\beta}$, $\zeta^{-k,\beta}_t=\overline{\zeta_t^{k,\beta}}$, and $\zeta^{k,\beta}$ is independent of $\zeta^{\ell,\beta}$ unless $\ell=\pm k$. These conditions similarly enforce that $U^{\alpha,\beta,b}$ is real-valued.

Thus, for any fixed $t$, $U^{\alpha,\beta, b}(t,\cdot)$ has the same law as $U^{\alpha,b}$. We now note the following properties of $U^{\alpha,\beta,b}$, also proved in Appendix~\ref{appen:examples}. The finite range decomposition in time requires developing a different finite range decomposition---this time for the OU process---than that used for Proposition~\ref{prop:autonomous gaussian}.

\begin{proposition}\label{prop:refreshing gaussian}
Let $\alpha\in(0,1)$, $\beta>0$, and $b\in\{0,1\}$. Then there exists a mutually independent family of smooth in space, continuous in time, centered, Gaussian fields $\{u^j\}_{j\geq 0}$ such that
\begin{enumerate}
    \item\label{item: refreshing scale decomp} $U^{\alpha,\beta,b}$ is equal to $\sum_{j= 0}^\infty u^j$ in law,
    \item\label{item:refreshing finite range} for each $j\geq 1$, $u^j$ has a finite temporal range of dependence $2^{-\beta j}$,
    \item\label{item: refreshing moment bounds} and there exists $C(d,\alpha,\beta)>0$ such that for all $j\geq 0$, $p\geq 1$, and $n\in\{0,1,2\}$,
    \[\|\nabla^n u^j\|_{L^p_\omega C^0_{t,x}}\leq  C (p^{1/2}+j^{1/2})2^{(n-\alpha)j}.\]
\end{enumerate}
\end{proposition}

Proposition~\ref{prop:refreshing gaussian} then shows that $U^{\alpha,\beta,b}$ satisfies Assumption~\ref{asmp:main for refreshing} for $\beta>0$, as well as the relevant moment bounds of the theorems in Section~\ref{s:refreshing intro} for $\alpha' <\alpha$. Additionally, the good $p$ moment scaling of Item~\ref{item: refreshing moment bounds} satisfies the quantitative moment bound hypotheses of Theorem~\ref{thm:refreshing ODE quantitative}, for appropriate $\alpha,\beta$ (e.g.\ $\beta =1, \alpha \in (1/2,1)$).

\subsection{Poisson-driven fields}

\subsubsection{The sweeping case}

We first define a large family of autonomous Poisson driven fields that satisfy Assumption~\ref{asmp:main for autonomous}. Fix $\lambda>0$, and let $\{\Pi^j\}_{j\geq 0}$ be a family of independent Poisson point processes (see e.g.~\cite{last_poisson_2018} for a friendly introduction) on $\T^d$, such that for every $j\geq 0$, $\Pi^j$ has intensity $\lambda|B_{2^{-(j+2)}}|^{-1}$. This ensures that the expected number of points in $\Pi^j$ in each ball of radius $2^{-(j+2)}$ is equal to $\lambda$. Then, let $\{f^j_y\}_{y\in\T^d}$ be a family of independent $C^2(\T^d)$ valued random variables\footnote{For those concerned about taking a continuum of random variables, this construction can be clearly completed using a countable family of $f^j_y$ and numbering the Poisson points.}  such that $\supp(f^j_y)\subset B_{2^{-(j+1)}}(y)$, and $\E f^j_y=0$. We then define the random scale field
\begin{equation}\label{eq:autonomous poisson}
u^j(x):=\sum_{y\in \Pi^j} f^j_y(x),
\end{equation}
and let $\mu^j_{y,x}$ denote the law of $f^j_y(x)$. We defer the proof to Appendix~\ref{appen:examples}.

\begin{proposition}\label{prop: autonomous poisson}
Fix $\alpha\in(0,1)$, $\lambda>0$, and $p\geq 1$, and let, for all $j\geq 0$, $u^j$ be defined by~\eqref{eq:autonomous poisson}. Suppose there exists $M\geq 1$ such that for all $j\geq 0$, $n\in\{0,1,2\}$, and $y\in\T^d,$
\[\|\nabla^n f^j_y\|_{L^p_\omega C^0_x}\leq M 2^{(n-\alpha-d/p)j}.\]
Then 
\begin{enumerate}
    \item\label{item:poisson independence} the $u^j$ are mutually independent with $\E u^j =0$,
\item\label{item: poisson autonomous finite range} for each $j\geq 1$, $u^j$ has a finite spatial range of dependence $2^{-j}$,
\item\label{item:poisson autonomous moment bounds} and there exists a constant $C(d)>0$ such that for all $j\geq 0$ and $n\in\{0,1,2\}$,
\[\|\nabla^n u^j\|_{L^p_\omega C^0_x}\leq CM(\lambda+p)2^{(n-\alpha)j}.\]
\end{enumerate}
Letting $U:=\sum_{j= 0}^\infty u^j$, we further have that
\begin{enumerate}
    \item[(i)]\label{item:poisson div free} if for all $j\geq 0$ and $y\in\T^d$, $\nabla\cdot f^j_y=0$, then $\nabla\cdot U=0$,
    \item[(ii)]\label{item: poisson zero level set} if there exists $r\in(0,1)$ and $C>0$ such that for all $j\geq 0$
    \begin{equation}\label{eq:density condition}
    \sup_{\substack{x,y\in\T^d\\d_{\T^d}(x,y)\leq 2^{-(j+2)}}}\bigg\|\frac{d\mu^j_{y,x}(z)}{dz}\bigg\|_{L^\infty(\R^d)}\leq Cr^j e^{\lambda j},
    \end{equation}
    then it almost surely holds that for all $\alpha'<\alpha$
    \[\lim_{\eps\rightarrow 0} \eps^{-\alpha'd}|Z_\eps|=0,\]
    where $Z_\eps$ is defined as in~\eqref{eq:zero sets}.
\end{enumerate}
\end{proposition}

The above proposition thus shows that for $\alpha>\frac{1}{2}$, given the random fields $f^j_y$ satisfy reasonable moment estimates, $U:=\sum_{j= 0}^\infty u^j$ satisfies Assumption~\ref{asmp:main for autonomous} and the hypotheses of Theorem~\ref{thm:main ode intro}, for $p$ suitably chosen. Additionally, if the $f^j_y$ are divergence-free, and their laws have suitably bounded densities, then $U$ satisfies the hypotheses of Corollary~\ref{cor:main pde ode}, using that $\alpha d \geq \alpha > 1-\alpha$ for $\alpha \in (1/2,1)$.

\begin{remark}
Fixing $p\geq 1$, the most natural examples of Poisson-driven fields satisfying the hypotheses of Proposition~\ref{prop: autonomous poisson} are the 
homogeneous self similar fields, where for all $j\geq 0$ and $y\in\T^d$, $f^j_y$ is equal in distribution to $2^{-(\alpha+\frac{d}{p})j}f(2^j(\cdot -y))$ for some fixed random field $f$ with $\supp(f)\subset B_{1/2}(0)$. It then holds that
\[\|\nabla^n f^j_y\|_{L^p_\omega C^0_x}\leq 2^{(n-\alpha-\frac{d}{p})j}\|\nabla^n f\|_{L^p_\omega C^0_x},\]
and 
\[\sup_{\substack{x,y\in\T^d\\d_{\T^d}(x,y)\leq 2^{-(j+2)}}}\bigg\|\frac{d\mu^j_{y,x}(z)}{dz}\bigg\|_{L^\infty(\R^d)}=2^{d(\alpha+\frac{d}{p})j}\sup_{x\in B_{1/4}}\Big\|\frac{d\mu_x(z)}{dz}\Big\|_{L^\infty(\R^d)},\]
where $\mu_x$ denotes the law of $f(x)$. Thus, if $\max_{0\leq n\leq 2}\|\nabla^n f\|_{L^p_\omega C^0_x}<\infty$ then Items~\ref{item:poisson independence}-\ref{item:poisson autonomous moment bounds} hold with $M=1 \lor \max_{0\leq n\leq 2}\|\nabla^n f\|_{L^p_\omega C^0_x}$, if $\nabla\cdot f=0$ then Item~\hyperref[item:poisson div free]{(i)} holds, and if 
\[\sup_{x\in B_{1/4}}\Big\|\frac{d\mu_x(z)}{dz}\Big\|_{L^\infty(\R^d)}<\infty,\]
 and $2^{d(\alpha+\frac{d}{p})}<e^{\lambda}$, then Item~\hyperref[item: poisson zero level set]{(ii)} holds.
\end{remark}

\begin{remark} One could similarly define a \textit{checkerboard field} by replacing the Poisson point process $\Pi^j$ in the definition above by a deterministic lattice of points separated at length-scale $2^{-j}$. Given appropriate moment and density estimates on the $f^j_y$, these random fields would also satisfy Assumption~\ref{asmp:main for autonomous}, the conditions of Theorem~\ref{thm:main ode intro}, and Corollary~\ref{cor:main pde ode}. As the construction of and estimates for this type of field are simpler than for the Poisson driven fields above, we do not include them.
\end{remark}

\subsubsection{The refreshing case}
Finally, we define Poisson driven random fields that satisfy the refreshing hypotheses of Assumption~\ref{asmp:main for refreshing}. Fixing $\beta>0$ and $\lambda>0$, suppose that $\{\Pi^j\}_{j\geq 0}$ is a family of independent Poisson point processes on $[0,1]$ such that for every $j\geq 0$, $\Pi^j$ has intensity $\lambda 2^{\beta j}$. Then, let $\{f^j_s\}_{s\in[0,1]}$ be a family of independent $C^2(\R\times\T^d)$ valued random variables such that $\supp(f^j_s)\subset [s-2^{-\beta j-1},s+2^{-\beta j-1}]\times \T^d$ and $\E f^j_s=0$.\footnote{The issue with a continuum of random variables can be resolved exactly as above.} We then define the random scale field
\begin{equation}\label{eq:refreshing poisson}
u^j(t,x):=\sum_{s\in \Pi^j} f^j_s (t,x).
\end{equation}

\begin{proposition}\label{prop: refreshing poisson}
Fix $\alpha\in(0,1)$, $\beta>0$, $\lambda>0$, $p\geq 1$ and let, for all $j\geq 0$, $u^j$ be defined by~\eqref{eq:refreshing poisson}. Suppose there exists $M\geq 1$ such that for all $j\geq 0$, $n\in\{0,1,2\}$, and $s\in[0,1]$,
\[\|\nabla^n f^j_s\|_{L^p_\omega C^0_{t,x}}\leq M 2^{(n-\alpha-\beta /p)j}.\]
Then 
\begin{enumerate}
\item the $u^j$ are mutually independent with $\E u^j =0$,
\item for each $j\geq 1$, $u^j$ has a finite temporal range of dependence $2^{-\beta j}$,
\item and there exists a constant $C>0$ such that for all $j\geq 0$ and $n\in\{0,1,2\}$,
\[\|\nabla^n u^j\|_{L^p_\omega C^0_{t,x}}\leq CM(\lambda+p)2^{(n-\alpha)j}.\]
\end{enumerate}
\end{proposition}

The above proposition then shows that $U := \sum_{j=0}^\infty u^j$ satisfies Assumption~\ref{asmp:main for refreshing} for $\beta>0$, as well as the relevant moment bounds of the theorems in Section~\ref{s:refreshing intro}, with $p$ appropriately chosen. The $p$ moment scaling---which is typically (almost) sharp due to the $\frac{p}{\log p}$ scaling of a Poisson random variable---shows the necessity of Theorem~\ref{thm:refreshing ODE qualitative}, since the hypotheses of Theorem~\ref{thm:refreshing ODE quantitative} will not typically be satisfied for this field.

\section{Overview of the argument and notation}

\label{s:overview}

We now present an overview of the argument: explaining some of the technical ideas needed for the proofs as well as the architecture of the paper. We will also present more precise---but also more technical---versions of Theorem~\ref{thm:autonomous intro sharp}, Theorem~\ref{thm:refreshing intro}, and Theorem~\ref{thm:refreshing intro sharp}. In the body of the paper, we will prove these more precise versions; we explain in this section how to deduce the versions stated above.

We cover first the well-posedness side of the theory, for both the refreshing and sweeping regimes, before turning to the ill-posedness side. Before that though, we give some notation and classical facts.

\subsection{Notation and basic facts}

We often somewhat abuse notation, writing $|x-y|$ for $d_{\T^d}(x,y)$ for $x,y \in \T^d$. We use the following function space notation.

\begin{definition}
    For a function $g : X \to Y$, where $X = \T^d$ or $\R^d$ and $Y$ is a Banach space, we define for $s \geq 0$ the (semi)norm space $\mathcal{C}^s$ as follows. For $s=0$, we let
    \[\|g\|_{\mathcal{C}^0(X) Y} := \sup_{x \in X} \|g(x)\|_Y.\]
    For $s \in (0,\infty)$---letting $s = n + r$ where $r \in (0,1]$ and $n \in \N$---we define $\mathcal{C}^s(X)$ as a seminorm on a subspace of $C^n(X)$ by
    \[\|g\|_{\mathcal{C}^s(X) Y} := \sup_{x,x' \in X, x \ne x'} \frac{\|\nabla^n g(x) - \nabla^ng(x')\|_Y}{|x-x'|^r}.\]
    We note in particular that $\mathcal{C}^1$ is the Lipschitz seminorm and that $\mathcal{C}^0$ is the natural norm on the space of bounded functions---rather than $C^1$ and $C^0$ which have additional qualitative regularity hypotheses that we do not want to require, as we will frequently be working with Lipschitz functions that are not $C^1$.

    We also consider for $p \in [1,\infty), s\geq 0$ the fractional Sobolev (semi)norm spaces $W^{s,p}$. For $s \in \N$, we just let $\|g\|_{W^{s,p}(X)Y} := \|\nabla^s g\|_{L^p(X)Y}$, and for $s>0$ and $s \not \in \N$, we again write $s = n+r$ where $r \in (0,1)$ and $n \in \N$ and define
    \[\|g\|_{W^{s,p}(X) Y} := \Big(\int_X \int_X \frac{\|\nabla^n g(x) - \nabla^n g(x')\|^p_Y}{|x-x'|^{d + rp}}\,dx\,dx'\Big)^{1/p}.\]
\end{definition}

We often consider functions taking many arguments, such as $g(\omega, y,t,x)$ where $g : \Omega \times \T^d \times [0,1]  \times \T^d \to \R$. We almost always suppress the $\omega$ argument, viewing $g$ as a random function $g(y,t,x)$. For functions of this form, we often want to consider iterated norms, such as $\|g\|_{\mathcal{C}^0_y L^p_\omega \mathcal{C}^{1/2}_t \mathcal{C}^1_x}$, which is interpreted by peeling off the norms from left to right, viewing it first as a function $\T^d \to L^p_\omega \mathcal{C}^{1/2}_t \mathcal{C}^1_x$, which we take the $\mathcal{C}^0_y$ norm of. To interpret the codomain space, $L^p_\omega \mathcal{C}^{1/2}_t \mathcal{C}^1_x$, we view it as a function $\Omega \to \mathcal{C}^{1/2}_t \mathcal{C}^1_x$ and take the $L^p_\omega$ norm; and so on.

We note the elementary, but important, facts that 
\begin{align*}
    \|g\|_{\mathcal{C}^{s^1}_x \mathcal{C}^{s^2}_y} &=\|g\|_{\mathcal{C}^{s^2}_y\mathcal{C}^{s^1}_x },\\
    \|g\|_{L^p_x W^{s,p}_y} &=    \|g\|_{ W^{s,p}_y L^p_x},\\
    \|g\|_{\mathcal{C}^{s^1}_x W^{s^2,p}_y} &\leq     \|g\|_{ W^{s^2,p}_y\mathcal{C}^{s^1}_x},
\end{align*}
that is, norms at equal integrability commute and there is a one sided embedding when moving a higher integrability norm inside of a lower integrability norm; these are versions of the Fubini theorem and the Minkowski integral inequality respectively.

The first simple observation follows straightforwardly from the definitions and the comments above. The conditions on the exponents are generically overly restrictive but are sufficient for our purposes and straightforwardly handle possible issues with components of $f$ that are not controlled by the seminorms (e.g.\ constants). 

\begin{lemma}\label{lem:alpha-holder_embedding} 
For $p\in[1,\infty)$, $1 \geq r>r_0> 0$, and $2 \geq s> s_0 >0$, there exists $C(d,r,r_0,s,s_0)>0$ such that for all $f : \Omega \times [0,1] \times \T^d \to \R$,
\[\|f\|_{L^p_\omega W^{r_0,p}_t W^{s_0,p}_x}\leq C\|f\|_{ \mathcal{C}^r_t\mathcal{C}^s_x L^p_\omega}.\]
\end{lemma}

We next record the following version of a Sobolev embedding, which follows by the standard theory as in~\cite{triebel_theory_1983}.

\begin{lemma}\label{lem:sobolev} For $p>1$, $r,\gamma,s,\nu\geq 0$ such that $ 1 \geq r > r-1/p>\gamma> 0$ and $ 2 \geq s> s-d/p>\nu> 0$, there exists a constant $C(d,r,s,\gamma,\nu,p)>0$ such that for all $f : [0,1] \times \T^d\to \R$,
\[ \|f\|_{\mathcal{C}^{\gamma}_t\mathcal{C}^{\nu}_x}\leq C\|f\|_{W^{r,p}_t W^{s,p}_x}.\]
\end{lemma}

\subsection{Well-posedness through effective regularity}

\label{ss:wellposedness}

We will first discuss the refreshing case. The sweeping case will proceed essentially by a reduction to a refreshing-like case, so understanding the argument in the refreshing case is a useful first step to understanding the sweeping argument.

\subsubsection{Nonlinear Young equation theory}

As suggested in Section~\ref{s:heuristic}, the central tool we will be using to quantify effective regularity is the theory of nonlinear Young integrals and nonlinear Young integral equations. The loose idea of this theory is that, if the velocity field integrated over a solution curve has some additional spatial regularity, then one can do a fractional integration by parts in time to exploit the additional spatial regularity. The name comes from the Young integration theory~\cite{young_inequality_1936} that inspired the sewing lemma~\cite{gubinelli_controlling_2004,feyel_curvilinear_2006} which provided a unified treatment of Lyons' theory of rough paths~\cite{lyons_differential_1998}. See~\cite{friz_rough_2020} for a very readable introduction. The theory of nonlinear Young integration, which also relies on sewing, was developed in~\cite{catellier_averaging_2016,hu_nonlinear_2017,galeati_noiseless_2021}, primarily to study regularization by noise. 

Whereas traditional Young integrals and rough paths primarily study integrals of the form $\int f(X_t) dY_t$ where $X,Y$ are H\"older regular paths, nonlinear Young integrals are essentially of the form 
\[\int \partial_t A(t,X_t)\,dt,\]
where $A$ may only be H\"older regular in time, thus requiring some care to interpret. This is the natural nonlinear generalization of Young integrals as the two H\"older regular objects are paired nonlinearly as opposed to the classic bilinear pairing.

We now state the nonlinear Young results we need, after some notation.

\begin{definition} 
For any velocity field $v\in \mathcal{C}^0([0,1]\times \T^d)$ and curve $X\in C^0([0,1])$ we let $\iop^Xv:[0,1]\times \T^d\rightarrow \R^d$ be defined by
\[\iop^X v(t,x)=\int_0^t v(s,X_s+x)\,ds.\]
\end{definition}

We now give the essential nonlinear Young tool that we use to prove the well-posedness results above the critical thresholds: $\alpha+\beta/2>1$ in the refreshing case and $\alpha>1/2$ in the sweeping case. The argument is a modification of~\cite[Theorem 4.8]{galeati_noiseless_2021} and will be given in Appendix~\ref{appen:young integrals}. The utility of splitting into two components $I^1, I^2$ will be explained in Section~\ref{sss:abstract averaging}.
\begin{proposition}
    \label{prop:effective Lipschitz Gronwall}
Let $\frac{1}{2}<\gamma_1\leq \gamma_2\leq 1$ and suppose that $v\in C^0([0,1]\times \R^d)$\footnote{We really mean here the usual meaning of $C^0$, that is, bounded continuous, as compared to $\mathcal{C}^0$, which means only bounded.} and $w^1,w^2\in \mathcal{C}^{\gamma_1}([0,1])$ with $w^1_0=w^2_0=0$. For $i=1,2$, let $X^i$ be any solution to
\[X^i_t=y^i+\int_0^t v(s,X_s^i)\,ds+w_t^i,\]
for $y^i\in\R^d$. Then, for any splitting $\mathcal{I}^{X^1}v(t,x)=I^1(t,x)+I^2(t,x)$, there exists a constant $C(\gamma_1,\gamma_2)>0$ such that
\[\|X^1-X^2\|_{\mathcal{C}^{\gamma_1}_t}\leq C\exp\Big(C \big(\|I^1\|_{\mathcal{C}^{\gamma_1}_t \mathcal{C}^1_x}^{1/\gamma_1}\vee \|I^2\|_{\mathcal{C}^{\gamma_2}_t \mathcal{C}^1_x}^{1/\gamma_2}\big)\Big)\big(|y^1-y^2|+\|w^1-w^2\|_{\mathcal{C}^{\gamma_1}_t}\big).\]
\end{proposition}

We now state the nonlinear Young tool that we use to prove the effective regularity estimates below the critical threshold. The result is a version of the Bihari--LaSalle inequality---which is used to bound separation rates of ODE trajectories in H\"older regular velocity fields---that uses the effective H\"older regularity in place of the bare H\"older regularity. Its proof, given in Appendix~\ref{appen:young integrals}, follows essentially as a combination of the classical Bihari--LaSalle argument and the argument of~\cite[Theorem 4.8]{galeati_noiseless_2021}.
\begin{proposition}
    \label{prop:effective Holder Bihari}
    Let $\gamma,\nu\in (0,1)$ with $\gamma(1+\nu)>1$ and suppose that $v \in C^0([0,1] \times \R^d)$ and $w^1,w^2\in \mathcal{C}^\gamma([0,1])$ with $w^1_0=w^2_0=0$. For $i=1,2$, let $X^i$ be any solution to
    \[ X^i_t =y^i+ \int_0^tv(s,X^i_s)\,ds+w_t^i,\qquad t\in[0,1],\]
    for $y^i\in\R^d$. Then there exists $C(\gamma,\nu)>0$ such that if $\iop^{X^1}v\in \mathcal{C}^\gamma([0,1],\mathcal{C}^\nu(\R^d))$, then for all $t \in [0,1],$
    \[\big|X^1_t-X^2_t\big| \leq C(|y^1-y^2|+\|w^1-w^2\|_{\mathcal{C}^\gamma_t}t^{\gamma}+\|\iop^{X^1}v\|_{\mathcal{C}^\gamma_t\mathcal{C}^\nu_x}^{\frac{1}{1-\nu}} t^{\frac{\gamma}{1-\nu}}).\]
\end{proposition}

\subsubsection{High frequency cutoff}

By Proposition~\ref{prop:effective Lipschitz Gronwall} and Proposition~\ref{prop:effective Holder Bihari}, in order to deduce the desired ODE stability estimates, we need to get $\mathcal{C}^\gamma_t \mathcal{C}^\nu_x$ control of $\iop^X U$ where $X$ is a solution curve. Let us focus for now on the case of $\nu=1$, which will be useful for applying Proposition~\ref{prop:effective Lipschitz Gronwall}. The first technical trick we use is a \textit{high frequency (or UV) cutoff}. That is, letting $U = \sum_{j=0}^\infty u^j$ be the decomposition in terms of finite-range-in-time fields given by Assumption~\ref{asmp:main for refreshing}, we prove \textit{quantitative} estimates for $U^N := \sum_{j=0}^N u^j$. It is important for these quantitative estimates that we allow for a forcing to the ODE, since a solution to the true ODE with velocity field $U$ solves the ODE with velocity field $U^N$ with a forcing that is controlled by the tail sum $\sum_{j>N} \|u^j\|$, which will be suitably vanishing as $N \to \infty$.

Working with the $U^N$ fields with high frequencies removed means that $U^N$ is qualitatively $C^0_t C^2_x$, thus we get qualitative ODE uniqueness for free. This is helpful in particular for ensuring suitable adaptedness of the ODE trajectories, ensuring they cannot depend on the unrevealed future properties of $U$. This adaptedness is essential for proving the stochastic cancellations.

\subsubsection{What would happen if $X^N \indep U^N$}

By the above discussion, we want to control the $\mathcal{C}^\gamma_t \mathcal{C}^1_x$ norm of $\iop^{X^N} U^N$ where $X^N$ is the unique solution curve from fixed initial data $y$ of the velocity field $U^N$. That is, we need control like
\[\Big|\sum_{j=0}^N \int_s^t \nabla u^j(r, X^N_r+x)\,dr\Big| \leq C |t-s|^\gamma,\]
uniformly in $t,s,x$. We will prove this scale field by scale field, that is, we just apply the triangle inequality to the sum $\sum_{j=0}^N$. We now discuss the control of a fixed $j \geq 1$ term (the $j=0$ control follows by the bare regularity of $u^0$).

Since $X^N$ solves the ODE for $U^N$, it is essentially coupled (in the probabilistic sense) to $U^N$, hence to $u^j$. Let's however temporarily suppose that $X^N \indep u^j$. Then, conditioning on $X^N$, we have that
\[\Big|\int_s^t \nabla u^j(r, X^N_r+x)\,dr\Big|\]
is just the integral of the process $r \mapsto \nabla u^j(r, X^N_r+x)$, which has a finite range of dependence in time with range $2^{-\beta j}$. Under our regularity assumption, the pointwise value of $\nabla u^j(r, X^N_r+x)$ is controlled by $2^{(1-\alpha)j}$. Thus, we get essentially $2^{\beta j} |t-s|$ many independent hits, each with magnitude $2^{(1-\alpha - \beta)j}.$ So, if $|t-s| \gg 2^{-\beta j}$, the central limit-type scaling we get for finite range processes would easily give that
\[\Big\|\int_s^t \nabla u^j(r, X^N_r+x)\,dr\Big\|_{L^p_\omega} \leq C 2^{(1-\alpha - \beta) j} \sqrt{2^{\beta j} |t-s|} = C2^{(1-\alpha -\beta/2)j} |t-s|^{1/2}.\]
A simple argument using the bare regularity of $u^j$ gives the same estimate for $|t-s| \leq 2^{-\beta j}$.

Dividing by $|t-s|^{1/2}$ and taking the supremum over $t,s,x$, we then get that 
\[\big\|\iop^{X^N} u^j\|_{\mathcal{C}^{1/2}_t \mathcal{C}^1_x L^p_\omega} \leq C 2^{(1-\alpha -\beta/2)j}.\]
This is \textit{almost} what we want, except for two problems: 1) the norms are in the wrong order---as we want almost sure regularity and this is giving regularity in $L^p_\omega$---and  2) we need $\mathcal{C}^\gamma_t$ regularity with $\gamma >1/2$ in order to apply Proposition~\ref{prop:effective Lipschitz Gronwall}. Both of these problems are resolved using that, for $\alpha + \beta/2 >1$, there is a bit of room in $2^{(1-\alpha-\beta/2)j}$. This allows us to interpolate with the bare regularity bounds on the $u^j$, gaining a bit of regularity in both time and space while still maintaining exponential decay in $j$. We can then use essentially a version of the Kolmogorov continuity theorem---or equivalently Sobolev embeddings and Fubini---to swap the norm order at the cost of an arbitrarily small amount of regularity (provided we can take $p$ arbitrarily large). This then allows us to conclude, provided $\alpha + \beta/2>1$, that
\[\big\|\iop^{X^N} u^j\|_{L^p_\omega \mathcal{C}^{1/2+\ep}_t \mathcal{C}^1_x} \leq C 2^{-C^{-1} j}.\]
Summing over $j$, we then get $\big\|\iop^{X^N} U^N\|_{L^p_\omega \mathcal{C}^{1/2+\ep}_t \mathcal{C}^1_x}  \leq C$, setting us up to apply Proposition~\ref{prop:effective Lipschitz Gronwall}.

\subsubsection{Abstract averaging}
\label{sss:abstract averaging}

In reality, we do not have $X^N \indep U^N$, so the above argument doesn't quite work. We prove abstractly the estimates essentially hold as if $X^N \indep U^N$ in Lemma~\ref{lem:abstract averaging}, which is the central tool we will use to prove estimates on $\iop^{X^N} U^N$. For Lemma~\ref{lem:abstract averaging}, proved in Section~\ref{s:abstract averaging}, we want to estimate the contribution due to the scale field $u^j$. We essentially work on a time mesh with spacing $2^{-\beta j}$; any fluctuations below that scale can be dealt with purely using the bare regularity bounds. For the time discretized process, we want to decompose it in terms of a martingale term, which has centered conditional expectations, and a ``finite variation'' remainder term. The martingale term we will then control using sharp martingale CLT upper bounds, provided by Corollary~\ref{cor:martingale upper bound}. We note that in order to get the correct scaling in the probabilistic moment order, we need to use this sharper bound---which requires good control on all of the conditional $L^q$ moments---as opposed to something like the BDG inequality, which would give a rate like $p$ (compared to $p^{\lambda \lor 1/2}$). This $p^{\lambda \lor 1/2}$ scaling will be necessary for establishing the well-posedness of the flow.

The martingale term $M$ is essentially just the Doob martingale decomposition, where we subtract off the conditional means to make it correctly centered. Since the finite range property in time still allows any two adjacent time blocks to be correlated, we want to introduce gaps between the terms to maximally exploit the independence; it is for this reason we actually decompose the ``martingale term'' into a sum of two martingales. Once this has been set up correctly, estimating the martingale in the desired way is straightforward, using the bare regularity estimates and finite range property to get the right bounds on the conditional expectations and then applying Corollary~\ref{cor:martingale upper bound} to get the CLT scaling upper bound.

The primary difficulty is to correctly estimate the remainder term $F$, which is made up of the conditional expectations. The conditional expectations are essentially measuring the defect from the independence of $X^N$ and $u^j$, since if they were independent, the conditional expectation would be zero, by the finite range structure and the centering hypothesis of $u^j$. The idea is to build a replacement curve $\tilde X^N$ that is conditionally independent of $u^j$ on the relevant time interval, hence the conditional expectation of the field over that curve is identically zero. We then show that the conditional expectations are close, using the bare regularity of $u^j$, provided $X^N, \tilde X^N$ are suitably close. We can then show that the (suitably constructed) replacement curve is close to the original curve using a Bihari--LaSalle-type argument, since, up to a small forcing, the two curves essentially solve the same ODE with $C^\alpha_x$ velocity field $U^N$.

This argument estimates $\iop^{X^N} u^j$ by splitting it into two terms, one of which is $\mathcal{C}^{1/2}_t \mathcal{C}^1_x$ and the other $\mathcal{C}^{1}_t \mathcal{C}^1_x$. The latter term typically has worse stochastic integrability, due to the power of $\frac{1}{1-\alpha + \ep}$ on $\|U^N\|$ that appears from the Bihari--LaSalle argument. This is why it is useful to allow the two different components, gaining different powers in the exponential depending on their regularity, in Proposition~\ref{prop:effective Lipschitz Gronwall}.

Lemma~\ref{lem:abstract averaging} completes this argument in a rather abstract setting, allowing us to flexibly apply it in the various cases we will consider.

\subsubsection{The significance of $\mathcal{C}^{1/2}_t$}
\label{sss:1/2}

The time regularity $\mathcal{C}^{1/2}_t$ appears in two distinct ways in this argument. First, we need $\mathcal{C}^\gamma_t$ regularity in Proposition~\ref{prop:effective Lipschitz Gronwall} with $\gamma >1/2$, which is essentially just the standard restriction from Young integration, which in turn is effectively the usual requirement of having at least a full derivative between $f,g$ to make sense of the pairing $\int f \nabla g$. Since $\gamma$ will give the regularity of both terms of the pairing, we thus need $\gamma + \gamma >1$.

For rather distinct reasons, $\mathcal{C}^{1/2}_t$ arises as the natural regularity scale from a CLT/Donsker invariance argument, which is why it arises as the regularity on the martingale term in the abstract averaging lemma, Lemma~\ref{lem:abstract averaging}. This is related to the fact that if we take a centered finite range process $A^\ep_t$ with range $\ep$ and define $B^\ep_t:= \int_0^t A^\ep_s\,ds$, then $\|B^\ep\|_{\mathcal{C}^1_t} \leq \|A^\ep\|_{\mathcal{C}^0_t}$, which uses no stochastic cancellation. Stochastic cancellation gives improvement down to $\mathcal{C}^{1/2}_t$, $\|B^\ep\|_{\mathcal{C}^{1/2}_t} \lesssim C \sqrt{\ep} \|A^\ep\|_{\mathcal{C}^0_t}$, but typically offers no additional improvement for lower regularities: we expect $\|B^\ep\|_{ \mathcal{C}^{0}_t L^p_\omega} \approx \|B^\ep\|_{ \mathcal{C}^{1/2}_t L^p_\omega} \approx \sqrt{\ep} \|A^\ep\|_{\mathcal{C}^0_t L^p_\omega }$.

It is because these two thresholds perfectly coincide (up to a harmless arbitrary $\ep>0$, which can always be included via interpolation with the bare regularity estimates) that we are able to prove sharp results. When the thresholds fail to coincide---which is the case for the effective regularity estimates requiring Proposition~\ref{prop:effective Holder Bihari}, as there we are forced to take $\gamma > \frac{1}{1+\nu} > 1/2$---we prove (seemingly) non-sharp estimates, as discussed more below.

\subsubsection{ODE well-posedness in the refreshing case}

In Section~\ref{s:refreshing}, we prove our ODE well-posedness results in the refreshing case. We prove two different well-posedness results, under different moment hypotheses. The first one---which requires strong moment bounds, stronger than having any exponential moment---is the following. We will want to take $s=0$ in~\eqref{eq:main refreshing ODE s}, that is, we consider the following initial value problem
\begin{equation}
    \label{eq:main refreshing ODE 0}
    \begin{cases}
        \dot X_t = U(t,X_t),\\
        X_0 = y.
    \end{cases}
\end{equation}

The precise result we prove then is the following.

\begin{theorem}
    \label{thm:refreshing ODE quantitative}
    Suppose $U = \sum_{j=0}^\infty u^j$ satisfies Assumption~\ref{asmp:main for refreshing} for some $\beta >0$ and let $\alpha \in (0,1)$ such that $\alpha + \beta/2 > 1$. Suppose that for some $K \geq 1$ and 
    \begin{equation}
    \label{eq:lambda hypothesis}
    0 \leq \lambda < \frac{\beta}{2-\alpha} \land \Big( 1- \frac{1-\alpha}{\beta}\Big), 
    \end{equation}
    we have that for all $j \in \N, n \in \{0,1,2\}$, and $1 \leq p<\infty$,
     \begin{equation*}
         \|\nabla^n u^j\|_{L^p_\omega \continuous^0_{t,x}} \leq K p^\lambda 2^{(n-\alpha)j}.
     \end{equation*}
     Then there exist $\delta(\alpha,\beta,\lambda)>0$ and a random field $M : \T^d \to [1,\infty)$ such that 
     \begin{enumerate}
         \item almost surely $\sup_{y \in \T^d} |M(y)| < \infty$; in fact for all $p \geq 1$, $\E \big[\big(\log\big(\sup_{y \in \T^d} |M(y)| + 1\big)\big)^p\big]<\infty$.
         \item for all $p \geq 1$, $M \in L^p_\omega L^p_y$,
         \item and for any $y,z \in \T^d$ and $w \in C^{1/2 + \delta}([0,1])$ with $w_0=0$, if we have any solution $X$ to~\eqref{eq:main refreshing ODE 0}, and any solution $\tilde X$ to 
          \begin{equation}
         \label{eq:forced ODE theorem}
         \tilde X_t=z+\int_0^t U(s,\tilde X_s)\,ds+w_t,\qquad t\in[0,1],
         \end{equation}
        then we have the stability estimates on the time interval $[0,1]$
        \begin{equation}
        \label{eq:refreshing stability estimate}
        \|X - \tilde X\|_{C^{1/2 +\delta}_t} \leq M(y)\big(|y-z| + \|w\|_{C^{1/2+\delta}_t}\big),\end{equation}
         and, in the case that $w=0$, the expansion/compression estimate for all $t \in [0,1],$
         \[\frac{1}{M(y)} |y-z| \leq |X_t - \tilde X_t| \leq M(y) |y-z|.\]
     \end{enumerate}
\end{theorem}

Note that Gaussian tails give $p$ moments scaling like $p^{1/2}$, so Gaussian fields will satisfy~\eqref{eq:lambda hypothesis} for most $\beta,\alpha$. In particular, if $\beta \geq 1$, then any $\alpha + \beta/2>1$ works. We also note that~\eqref{eq:refreshing stability estimate} gives a stability estimate for both the initial data and forcings (since $w$ is what appears in the integral equation, the norm of $w$ is a negative regularity norm from the perspective of the differential equation), which is the natural form of a robust stability estimate. 

The very strong moments required for the above theorem buy us fairly strong stochastic integrability. In particular, the first item gives all polylogarithmic moments for the Lipschitz norm of the flow map and the second item gives $L^p$ moments on the $W^{1,p}_y$ norm of the flow map for all $p\geq 1$. We are losing (at least) an exponential from the assumed moments because Gr\"onwall, as well as its effective version Proposition~\ref{prop:effective Lipschitz Gronwall}, exponentiates the field in the estimates.

The proof of Theorem~\ref{thm:refreshing ODE quantitative} proceeds by using Lemma~\ref{lem:abstract averaging} to get good $L^p_\omega \mathcal{C}^\gamma_t \mathcal{C}^1_x$ estimates on $\iop^{X^N} u^j$, uniformly in the initial point $y$ for the integral curve $X^N$ and with good scaling in $p$. This is then used with Proposition~\ref{prop:effective Lipschitz Gronwall} to give the second and third items of Theorem~\ref{thm:refreshing ODE quantitative}. The first item is more subtle, as we need to perform a Kolmogorov continuity/Sobolev argument in the initial point $y$, which requires losing a bit of regularity. For this we ``bootstrap'' using the already proved $W^{1,p}_y$-type estimate from the second and third items to gain the necessary regularity. Doing this causes us to lose another $\log$ of stochastic integrability, hence why we only have polylogarithmic moments on the Lipschitz norm but $L^p$ moments on the $W^{1,p}$ norm. The sharp moment tracking, coming from Lemma~\ref{lem:abstract averaging}, is necessary to get this bootstrap argument to close. We needed $W^{1,p}_y$ estimates, which in turn required $L^p_\omega$ moments on the exponential of a power of the averaged field, coming from applying Proposition~\ref{prop:effective Lipschitz Gronwall}. 

The above then proves the desired stability estimate, but uniformly in $N$ for the field $U^N$ with high frequencies removed, in place of the true field $U$. An approximation argument then concludes Theorem~\ref{thm:refreshing ODE quantitative}.

Beyond being a nice quantitative version of the main stability estimate, Theorem~\ref{thm:refreshing ODE quantitative} is also an essential step in proving the result under the weaker moment hypotheses of Theorem~\ref{thm:refreshing intro}. Proposition~\ref{prop:refreshing truncation approximation} shows that, under the weak $L^p$ moment hypothesis, we can build a sequence of coupled approximating velocity fields that also satisfy Assumption~\ref{asmp:main for refreshing} and are almost surely bounded; that is, they satisfy~\eqref{eq:lambda hypothesis} with $\lambda=0$ (at a slightly worse regularity). Further, these approximating fields are \textit{exactly equal} to the original velocity field with probability going to one (of course the almost sure bounding constant diverges in this limit as well). As such, we can apply Theorem~\ref{thm:refreshing ODE quantitative} to the approximating fields to conclude the well-posedness under the weaker moment hypothesis. Keeping track of constants and optimizing, we get the following result, which has the same moment hypothesis as Theorem~\ref{thm:refreshing intro} but also provides a more precise stability estimate and gives moment bounds on the regularity modulus. We note that Theorem~\ref{thm:refreshing intro} is a direct corollary of Theorem~\ref{thm:refreshing ODE qualitative}, using the bi-Lipschitz estimates on the ODE trajectories to construct the two-parameter flow map.

\begin{theorem}
    \label{thm:refreshing ODE qualitative}
    Suppose $U = \sum_{j=0}^\infty u^j$ satisfies Assumption~\ref{asmp:main for refreshing} for some $\beta >0$ and let $\alpha \in (0,1)$ such that $\alpha + \beta/2 > 1$. Let
    \[p >\frac{ \beta}{\alpha + \beta/2-1} \lor \frac{\beta}{\alpha}\]
    and suppose for some $K \geq 1$, 
     we have that for all $j \in \N,$ and  $n \in \{0,1,2\},$
     \begin{equation*}
         \|\nabla^n u^j\|_{L^p_\omega \continuous^0_{t,x}} \leq K 2^{(n-\alpha)j}.
     \end{equation*}
     Then there exist $\delta(\alpha,\beta,p)>0, \gamma(\alpha,\beta,p)>0$, and a random field $M : \T^d \to [1,\infty)$ such that $\gamma \to \infty$ as $p \to \infty$ and
        \begin{enumerate}
         \item almost surely $\sup_{y \in \T^d} |M(y)| < \infty$; in fact $\E \big[\big(\log \log\big(\sup_{y \in \T^d} |M(y)| + e\big)\big)^{\gamma}\big]<\infty$,
         \item $\E \big[\big(\log\big(\|M\|_{L^p_y} + 1\big)\big)^{\gamma}\big]<\infty$,
         \item and for any $y,z \in \T^d$, any $w \in C^{1/2 + \delta}([0,1])$ with $w_0=0$, if we have any solution $X$ to~\eqref{eq:main refreshing ODE 0}, and any solution $\tilde X$ to 
          \begin{equation*}
         \tilde X_t=z+\int_0^t U(s,\tilde X_s)\,ds+w_t,\qquad t\in[0,1],
         \end{equation*}
         then we have the stability estimate on the time interval $[0,1]$,
         \[\|X - \tilde X\|_{C^{1/2 +\delta}_t} \leq M(y)\big(|y-z| + \|w\|_{C^{1/2+\delta}_t}\big),\]
        and, in the case that $w=0$, the expansion/compression estimate for all $t \in [0,1],$
        \[\frac{1}{M(y)} |y-z| \leq |X_t - \tilde X_t| \leq M(y) |y-z|.\]
     \end{enumerate}
\end{theorem}

The cost of the approximation/optimization argument is a loss of an additional logarithm in the moment estimates.

\subsubsection{Effective regularity in the refreshing case}
\label{sss:effective regularity}

We now turn to the proof of Theorem~\ref{thm:refreshing Bihari}. The proof similarly proceeds by using Lemma~\ref{lem:abstract averaging} to first get good $L^p_\omega \mathcal{C}^{1/2}_t \mathcal{C}^\nu_x$ control on $\iop^{X^N} u^j$ (let us focus on the contribution due to the ``martingale'' term in Lemma~\ref{lem:abstract averaging}; the remainder term is easier to deal with and is consistent with a strictly better estimate, hence ``subcritical''). The issue, as noted above, is that we need control on the $\mathcal{C}^{\gamma}_t \mathcal{C}^\nu_x$ norm with $\gamma(1+\nu) >1$ in order to apply Proposition~\ref{prop:effective Holder Bihari}. This is achieved by interpolating with the $\mathcal{C}^1_t \mathcal{C}^1_x$ regularity, which follows directly from the bare regularity of $u^j$. However, unlike the previous interpolations where we could interpolate arbitrarily little, this one is macroscopic: we have to substantially use the bare regularity. Since the bare regularity makes no use of the stochastic cancellations, this is seemingly fairly suboptimal. As discussed before, the effective regularity we obtain is strictly worse than that predicted by the heuristic of Section~\ref{s:heuristic}, giving us further reason to believe this argument is suboptimal. In fact, the argument could be further optimized somewhat (treating more carefully the two terms of decomposition coming from Lemma~\ref{lem:abstract averaging}) to give a slightly better (but more complicated) effective regularity. Since this still does not agree with the heuristic effective regularity, we refrain from including this optimization.

The obstruction to getting the heuristic effective regularity is precisely the requirement---natural from the perspective of nonlinear Young integration---that $\gamma(1+\nu) >1$. This is precisely the sort of regularity requirement that can often be bypassed by using the ideas of stochastic sewing~\cite{le_stochastic_2020}, used e.g.\ to this effect in \cite{butkovsky_stochastic_2025}. As such, we believe that one can likely use similar ideas to improve this bound. However, additional care will be needed, even on the level of the definition of a solution, as stochastic sewing requires appropriate adaptedness of the trajectories, so one will likely need a more stochastic solution notion similar to a weak solution in SDE theory.

\subsubsection{ODE stability estimates in the sweeping case}
\label{sss:sweeping estimates}

The sweeping case---proved in Section~\ref{s:autonomous}---uses all of the ideas from the refreshing case as well as some new ones. We will focus only on the proof of Theorem~\ref{thm:main ode intro}; the proof of Theorem~\ref{thm:sweeping Bihari} follows similarly by combining the ideas used for Theorem~\ref{thm:main ode intro} and Theorem~\ref{thm:refreshing Bihari}. We will also only discuss the proof of Theorem~\ref{thm:main ode intro} under the hypothesis that we have almost sure bounds in place of $p$ moment bounds; the general theorem only utilizing a moment hypothesis follows from this special case using a similar approximation argument to the refreshing case, with Proposition~\ref{prop:autonomous truncation approximation} in place of Proposition~\ref{prop:refreshing truncation approximation}. We neglect to provide quantitative bounds in the sweeping regime as there are several places where the quantitative bounds become \textit{much} worse than are available in the refreshing regime.

The main difficulty in the sweeping case is that there is no linear flow of time that cleanly ``foliates'' space into leaves that obey finite range estimates. We instead want to use the sweeping of the large scales of the flow to drag the particles through space in order to exploit the spatial finite range. 

The way to overcome this obstacle is to work in a transformed coordinate system, turning a spatial direction into a time direction. The basic point is that, if $U(y) \ne 0$, then, locally at least, all integral curves passing near $y$ can be reparameterized in terms of the coordinate $r \mapsto y + \frac{U(y)}{|U(y)|} r$. More precisely, if we have $\dot X_t = U(X_t)$, we can locally write $X_t = y + r(t) \eta + \tilde Z(t)$, where $\eta := \frac{U(y)}{|U(y)|}$ and $\tilde Z \cdot \eta =0$. Then, near enough to $y$, $\dot r>0$, so we have an inverse $\tau(r)$. We can then write the evolution of $\tilde Z$ in terms of $r$: $Z_r := \tilde Z(\tau(r))$. Then one can derive that
\[\dot Z_r = \frac{\Pi_{\eta^\perp} U(y+\eta r + Z_r)}{\eta \cdot U(y + \eta r + Z_r)}.\]
We then show the well-posedness of this ODE system, from which we can deduce the well-posedness of the original system.
The transformed ODE looks a lot more like the refreshing case, since the ``time'' coordinate $r$ in the transformed system directly maps to spatial displacements in $U$. 

There are, however, three immediate problems. The first is that we get essentially no good bounds on the velocity field on the right hand side, since if $\eta \cdot U =0$, the field explodes. We resolve this by introducing a regularization $\phi^\delta : \R \to [\delta,\infty)$, where $\phi^\delta(x) = x$ for $x \geq 2 \delta$ and considering the modified velocity field 
\[\frac{\Pi_{\eta^\perp} U(y+\eta r + z)}{\phi^\delta\big(\eta \cdot U(y + \eta r + z)\big)}.\]
This velocity field will then obey good \textit{a priori} bounds, as we take the scale field of $U$ to be almost surely bounded and $\phi^\delta$ is now explicitly lower bounded. Once we have well-posedness for the regularized, transformed system, we will prove that, provided we start sufficiently near $y$ and that $U(y) \ne 0$, then as $\delta \to 0^+$ the regularization from $\phi^\delta$ will act trivially (for some amount of time) with asymptotically full probability. Thus the well-posedness of the regularized, transformed system will imply the well-posedness of the original transformed system. This is a place where the quantitative estimates would become quite poor, as all estimates diverge as $\delta \to 0.$

The next problem is that as $z$, which is an element of $\eta^\perp$, the orthogonal subspace to $\eta$, explores $\eta^\perp$, $y+\eta r + z$ will potentially densely fill the torus as it ``wraps around''. Thus to know $z \mapsto \frac{\Pi_{\eta^\perp} U(y+\eta r + z)}{\phi^\delta\big(\eta \cdot U(y + \eta r + z)\big)}$ at a fixed $r$, we would need to know all of $U$, which destroys the finite range structure in $r$. This is easily resolved by adding another modification, $\chi(r,z)$, which is just a smooth cutoff function between $B_{1/8}$ and $B_{1/4}$. We then consider the field
\[V^{\delta}(r,z) := \chi(r,z)\frac{\Pi_{\eta^\perp} U(y+\eta r + z)}{\phi^\delta\big(\eta \cdot U(y + \eta r + z)\big)}.\]
Since $V^\delta$ is just $0$ for $z$ large enough, this resolves the problem of wrapping around the torus breaking the finite range structure in $r$. Similarly to $\phi^\delta$, we again can show that---at least for small enough times---the cutoff is harmless. In fact the $\chi$ cutoff is much simpler to handle than the $\phi^\delta$ regularization.

The final immediate problem is the nonlinear pairing of numerator and the denominator. If we had instead the velocity field $ \chi(r,z) {\Pi_{\eta^\perp} U(y+\eta r + z)}$, we could then just take the scale field of $U$ guaranteed by Assumption~\ref{asmp:main for autonomous} to immediately obtain a finite range in $r$ multiscale decomposition satisfying the refreshing hypothesis Assumption~\ref{asmp:main for refreshing}, with $\beta =1$. However, the division by $\phi^\delta\big(\eta \cdot U(y + \eta r + z)\big)$, as well as the fact that $\phi^\delta$ is a nonlinear function, both make it so the scale decomposition of $U$ from Assumption~\ref{asmp:main for autonomous} does not immediately give the desired multiscale finite range in $r$ decomposition satisfying Assumption~\ref{asmp:main for refreshing}.

This is partly resolved by~\eqref{eq:V big decomp}, where we take a particular decomposition into fields $v^j$ that are analogous to the scale fields for Assumption~\ref{asmp:main for refreshing}, together with a bunch of remainder terms. The remainder terms will be sufficiently small just using the bare regularity bounds, so can be directly absorbed without having to prove any stochastic cancellations. However, the $v^j$ are still not quite in the form of Assumption~\ref{asmp:main for refreshing}. They are close enough for morally the same proof as the refreshing case to go through. An example term in $v^j$ is 
\[\chi(r,z) \frac{\Pi_{\eta^\perp} u^j(y + \eta r + z)}{\phi^\delta\big(\eta \cdot U^{j-1}(y + \eta r + z)\big)},\]
where $U^{j-1} = \sum_{k=0}^{j-1} u^k$. The reason we can get estimates on this as in the refreshing case---even though it is not finite range in time with range $2^{-j}$ due to the long range correlation present in $U^{j-1}$---is that 
\begin{itemize}
    \item it obeys the same bare regularity estimates as the $j$th scale field should: $\|\nabla^n v^j\|_{\mathcal{C}^0_{r,z}} \lesssim 2^{(n-\alpha)j},$
    \item and it is \textit{conditionally finite range and centered}; that is, after conditioning on $(u^k)_{k \ne j}$ as well as $u^j$ in a neighborhood of $y$, it has the right finite range property in time and centeredness. The reason we need $u^k$ with $k>j$ and $u^j$ near $y$ is that $\eta$ is itself random and needs to be conditioned away. This breaks the finite range of $v^j$ in a small neighborhood of $r=0$, but this is allowed by Lemma~\ref{lem:abstract averaging}.
\end{itemize}
It turns out these two properties are sufficient to get the necessary estimates on $v^j$, using again Lemma~\ref{lem:abstract averaging} on the appropriate conditional probability measure.

The consequence is that we are able to get suitably good estimates on $\iop^Z V^\delta$---analogous to those of $\iop^{X^N} U^N$ in the refreshing case---allowing us to conclude the well-posedness (a.s.\ Lipschitzness) of the transformed ODE system. Let us note two technical wrinkles to the argument that could cause confusion on first reading. The first is that, as in the refreshing case, we actually truncate the high frequencies to make a qualitatively smooth $V^{N,\delta}$, pass estimates uniformly in $N$, and then send $N \to \infty$ to conclude about the $V^\delta$ field. The second is that we have only discussed the component of the ODE living in $\eta^\perp$, which dictates the graph of the integral curve as a subset of $\T^d$, but not how it is \textit{timed}. The timing of the integral curve also satisfies an ODE of a similar form, and the stability in the timing will be needed in order to conclude the stability of the original system. As such, the actual velocity field we consider is on the larger system $(z_\tau, z_{\eta^\perp})$, where $z_\tau$ tracks the time aspect and $z_{\eta^\perp}$ the part discussed above.

We then pass the (local) well-posedness of the transformed system to a (local) well-posedness of the original ODE. We note that we get a \textit{different transformed system} for each basepoint $y \in \T^d$ around which we take the transformation. Proposition~\ref{prop:local lipschitz flow} gives that for each $y,\delta$ we get a subset $G_{y,\delta} \subseteq \T^d$ on which we get the desired stability estimate---that is, the estimate holds for ODE trajectories which live entirely in $G_{y,\delta}$---with a random constant $M_{y,\delta}$. Then performing a rather soft covering/compactness argument, we can pass this local stability to the claimed Lipschitz flow away from the zero level set. This argument would further dramatically worsen any quantitative estimates. The zero level set arises as an obstruction since the $\phi^\delta$ regularization will \textit{always} be active at the zero level set, preventing us from passing our well-posedness for the regularized transformed system back to the original problem. This then concludes the proof of Theorem~\ref{thm:main ode intro}.

\subsubsection{ODE to PDE}

We now discuss the proof of Corollary~\ref{cor:main pde ode}, given in Section~\ref{s:pde}. The section starts with a lemma which gives a simple sufficient condition---requiring only reasonable uniform bounds on $\P(|U(x)| \leq r)$ together with the regularity of the field---for the smallness of the zero level set that is an assumption of Corollary~\ref{cor:main pde ode}. 

Theorem~\ref{thm:main ode intro} already gives the primary necessary input for the corollary: ODE uniqueness away from the zero level set. What remains is to establish that almost every trajectory never hits the zero level set---hence we have the first item of the corollary, a.e.\ uniqueness---and then to show this implies the renormalization/continuity uniqueness. The second step, given by Proposition~\ref{prop:ODE to continuity uniqueness}, is essentially an exercise in applying the Ambrosio superposition principle~\cite{ambrosioTransportEquationCauchy2008} and utilizing the volume preserving property of divergence-free velocity fields.

The first step, proving a.e.\ ODE trajectory never hits the zero level set, is given by Proposition~\ref{prop:abstract uniqueness away from bad plus small bad}, essentially as a modification of the argument of~\cite{aizenman_sufficient_1978}, which importantly utilizes both that $u$ is divergence-free and that the zero level set has a sufficiently small Minkowski dimension. We get a slightly better requirement on the Minkowski dimension than~\cite{aizenman_sufficient_1978} alone, utilizing the regularity to show that points slow down near the zero level set, making it harder for a large measure set to hit the zero level set.

\subsection{Nonuniqueness through central limit scaling}

\label{ss:nonunique}

We now turn to the proofs of Theorem~\ref{thm:autonomous intro sharp} and Theorem~\ref{thm:refreshing intro sharp}, given in Section~\ref{s:sharpness}---and present more refined versions of these results.

We will use the following simple definition.
\begin{definition}
    For a set $A \subseteq \T^d$, we denote its diameter by $\mathrm{diam}\,A$ so that
    \[\mathrm{diam}\,A := \sup_{x,y \in A} |x-y|.\]
\end{definition}

Theorem~\ref{thm:refreshing intro sharp} is an essentially direct corollary of the following result.

\begin{theorem}
\label{thm:sharp refreshing}
    For all $d\geq 3$, $\beta >0,$ and $\alpha \in (0,1)$ with 
    \[\alpha + \beta/2 < 1 \quad \text{and} \quad \alpha + \beta >1,\]
    there exists $U : [0,1] \times \T^d \to \R^d$ with decomposition $U = \sum_{j=0}^\infty u^j$ such that
    \begin{enumerate}
        \item[(i)]  $U = \sum_{j=0}^\infty u^j$ satisfies Assumption~\ref{asmp:main for refreshing} for $\beta,d$.
        \item[(ii)] For all $n \in \N, \ep>0$, there exists $K(d,n,\alpha,\beta,\ep)>0$ such that for all $j \in \N, p \geq 1$,
        \[\|\nabla^n u^j\|_{L^p_\omega \continuous^0_{t,x}} \leq K \sqrt{p}2^{(n-\alpha+\ep)j}.\]
        \item[(iii)] For all $j \in \N, \nabla \cdot u^j =0.$
    \end{enumerate}
    However, 
    \begin{enumerate}
        \item \label{item:ode nonuniqueness refreshing} For all $y \in \T^d$ and $h\in (0,1]$,
        \[\P\big(\text{\eqref{eq:main refreshing ODE 0} has a unique solution on } [0,h]\big) = 0.\]
        Further, for all $y \in \T^d$ and for all $\ep>0$ sufficiently small, almost surely in $U$, 
        \[\limsup_{t\to 0} t^{-\frac{1}{2 -2\alpha - \beta -\ep}} \mathrm{diam}\{X_t : X \text{ is a solution to~\eqref{eq:main refreshing ODE 0}}\} = \infty.\]
        \item 
        \label{item:pde nonuniqueness refreshing}
        Almost surely, $U$ admits nonunique bounded solutions on $[0,1]$ to~\eqref{eq:continuity} for some bounded initial data.
    \end{enumerate}
\end{theorem}

We note that though Theorem~\ref{thm:sharp refreshing} requires $\alpha+\beta>1$ while Theorem~\ref{thm:refreshing intro sharp} has no such requirement, if we have any $\alpha +\beta/2<1$, then we can always choose $\beta' \geq \beta$ such that we still have that $\alpha + \beta'/2 <1$ and $\alpha + \beta'>1$. We can then take the velocity field Theorem~\ref{thm:sharp refreshing} gives for $(\alpha,\beta')$, and note that it obeys the hypotheses of Theorem~\ref{thm:refreshing intro sharp} for $(\alpha,\beta)$, since Assumption~\ref{asmp:main for refreshing} is easier to satisfy for smaller $\beta$. We note similarly that the slight mismatch in the stated regularity can be resolved by increasing $\alpha$ somewhat. The ``almost sure in $\omega$, for almost every $y$'' quantifiers in Theorem~\ref{thm:refreshing intro sharp} are implied by Item~\ref{item:ode nonuniqueness refreshing} of Theorem~\ref{thm:sharp refreshing} and Fubini's theorem. The solutions in Theorem~\ref{thm:refreshing intro sharp} can be taken to be positive as $\nabla \cdot U =0$, so we can shift a general bounded solution to be a positive bounded solution by adding a constant.

We see this quite robustly shows ill-posedness: the scale fields are all divergence-free, they have Gaussian stochastic tails, but nonetheless for every $y\in \T^d$, there is almost surely instantaneous nonuniqueness. We note also the quantitative rate of growth of the diameter of the solution set, in agreement with the heuristic effective regularity prediction and in some sense complementary to Theorem~\ref{thm:refreshing Bihari}. 

The random velocity fields constructed for Theorem~\ref{thm:sharp refreshing} will have the further properties of being 1) compactly supported in time, 2) finite range in space with $u^j$ having range $2^{-j}$, and 3) having the ``expected'' spacetime regularity $\|\partial_t^\ell \nabla^n u^j\|_{L^p_\omega C^0_{t,x}} \leq C2^{(\beta \ell + n - \alpha+\ep)j}$. These additional properties allow us to prove the following more precise version of Theorem~\ref{thm:autonomous intro sharp} using the construction of Theorem~\ref{thm:sharp refreshing}.

\begin{theorem}
\label{thm:sharp autonomous}
    For all $d \geq 4$ and $\alpha \in (0,1/2)$, there exists $U: \T^d \to \R^d$ with decomposition $U(x) = \sum_{j=0}^\infty u^j(x)$ such that 
    \begin{enumerate}
        \item[(i)]  $U = \sum_{j=0}^\infty u^j$ satisfies Assumption~\ref{asmp:main for autonomous}.
        \item[(ii)] For all $n \in \N, \ep>0$, there exists $K(d,n,\alpha,\ep)>0$ such that for all $j \in \N, p \geq 1$,
        \[\|\nabla^n u^j\|_{L^p_\omega \continuous^0_{x}} \leq K \sqrt{p}2^{(n-\alpha+\ep)j}.\]
        \item[(iii)] For all $j \in \N, \nabla \cdot u^j =0.$
        \item[(iv)] $\P(\forall x \in \T^d, |U(x)| \geq 1) = 1.$
    \end{enumerate}
    However, 
     \begin{enumerate}
        \item \label{item:ode nonunique instantly and spreading auto}
        There exists $y \in \T^d$ such that for all $h>0$
         \[\P\big(\text{\eqref{eq:main ODE autonomous data} has a unique solution on } [0,h]\big) = 0.\]
        Further, there exists $y \in \T^d$ such that almost surely in $U$, for all $\ep>0$ sufficiently small, 
        \[\limsup_{t\to 0} t^{-\frac{1}{1 -2\alpha  -\ep}} \mathrm{diam}\{X_t : X \text{ is a solution to~\eqref{eq:main ODE autonomous data}}\} = \infty.\]
        \item \label{item:pde nonunique auto}
        Almost surely, $U$ admits nonunique bounded solutions on $[0,1]$ to~\eqref{eq:continuity} for some bounded initial data.
    \end{enumerate}
\end{theorem}

Item~\ref{item:ode nonunique instantly and spreading auto} shows a robust almost sure instantaneous nonuniqueness for some $y \in \T^d$. By the Ambrosio superposition principle~\cite[Theorem 3.2]{ambrosioTransportEquationCauchy2008}, Item~\ref{item:pde nonunique auto} gives that almost surely there is ODE nonuniqueness for~\eqref{eq:main ODE autonomous data} on the time interval $[0,1]$ for a positive measure set of initial data, thus Item~\ref{item:pde nonunique auto} gives Theorem~\ref{thm:autonomous intro sharp} as a corollary, after increasing $\alpha$ somewhat to deal with the mismatch in claimed regularities.

Theorem~\ref{thm:sharp autonomous} follows from the construction of Theorem~\ref{thm:sharp refreshing}---together with the properties of the construction noted above---using the standard dimensional embedding trick. Working in dimension $d \geq 4$, we take $\beta =1$ and $\alpha <1/2$, so that $\alpha +\beta/2 <1$. We then see the velocity field supplied by Theorem~\ref{thm:sharp refreshing} for dimension $d-1$---using the second property noted above---$\tilde U :[0,1] \times \T^{d-1} \to \R^{d-1}$ has finite range $2^{-j}$ in \textit{spacetime}. By the first property, it is compactly supported in time, so can be taken to be a function $\T^d \to \R^{d-1}$---treating time periodically---without losing regularity. By the spacetime regularity of $\tilde u$, this field has exactly the desired spatial regularity on $\T^d$. We then define our autonomous field $U : \T^d \to \R^d$ by $U(x) = (1,\tilde U(x))$. This embedding has a one-to-one correspondence between ODE trajectories on $\T^{d-1}$ for the time dependent velocity field $\tilde U$ and ODE trajectories on $\T^d$ for the autonomous field $U$. This allows us to conclude Theorem~\ref{thm:sharp autonomous} from Theorem~\ref{thm:sharp refreshing}.

\subsubsection{Construction of the velocity field}

We thus focus on the proof of Theorem~\ref{thm:sharp refreshing} for the remainder. We build the velocity field in the following way. We construct a decreasing sequence of times $T^j \to 0$. On each time interval $[T^{j+1},T^j]$, only the scale field $u^j$ will be ``active''; all the other fields will be identically $0.$ This resolves the least justified step of the heuristic argument, which was neglecting the contributions of the large scales and the small scales. The scale field $u^j$, on the time interval $[T^{j+1},T^j]$ is taken to be refreshing---that is, iid resampled---on the time horizon of $2^{-\beta j} \ll T^j - T^{j+1}$, thus giving the desired finite range of time dependence. We take $T^j - T^{j+1} = N^j 2^{-\beta j}$ with $N^j \in \N$, so that there are $N^j$ many independent resamplings of $u^j$ on its period of activity. On each such sample, we take $u^j$ to be a smooth, divergence-free Gaussian field with finite range of spatial dependence $2^{-j}$ and typical magnitude $2^{-\alpha j}$. This then verifies the desired regularity estimates as well as the claimed finite range of spatial dependence above.

Following the heuristic computation, we expect two particles---already separated at length scale $2^{-j}$ at $T^{j+1}$---to have their relative displacement get $N^j$ hits of magnitude $2^{-(\alpha + \beta)j}$ over the time interval $[T^{j+1},T^j]$. We want the particles to separate to length scale $2^{-(j-1)}$ by $T^j$. Thus, by the CLT scaling, we want $2^{-(\alpha+\beta)j} \sqrt{N^j} \gg 2^{-(j-1)}$, or $N^j \gg 2^{2(\alpha +\beta -1)j}.$ We take an asymptotically strict separation $N^j \approx 2^{2(\alpha +\beta -1)j + j^{1/2}}.$ This then essentially fully determines the velocity field.

\subsubsection{Spontaneous stochasticity}

The central quantitative estimate we wish to prove, from which the two conclusions of Theorem~\ref{thm:sharp refreshing} follow, is a \textit{spontaneous stochasticity-type} estimate. We perturb the (ill-posed) ODE by a Brownian noise with intensity $\sqrt{\kappa}$, thus giving a (well-posed, by the regularization by noise results cited above) SDE:
\begin{equation*}
    \begin{cases}
        dX^\kappa_t = U(t,X^\kappa_t) dt + \sqrt{\kappa} dW_t,\\
        X^\kappa_0 =y.
    \end{cases}
\end{equation*}
The spontaneous stochasticity estimate is then given by Proposition~\ref{prop:main separation estimate}, which states for all $C \leq j \leq j_*$, the variance in $W$ (that is, the conditional variance given $U$) of $X^{\kappa^{j_*}}_{T^j}$ is $> 2^{-2j}$ with probability $1- \ep_j$, with $\kappa^{j_*} \to 0$ as $j_* \to \infty$ and $\ep_j \to 0$ as $j \to \infty$. Sending $j_* \to \infty$, we see that this gives that---uniformly along a vanishing sequence of diffusivities---$X^\kappa_{T^j}$ is spread to scale $2^{-j}$, with asymptotically full probability in $U$, as we consider earlier and earlier times. This then essentially immediately gives the first conclusion of Theorem~\ref{thm:sharp refreshing}, using that the solutions to the SDE concentrate on ODE solutions as $\kappa \to 0$ and controlling $T^j$ to compute the quantitative rate of spreading.

The name spontaneous stochasticity comes from the fluids literature, where it is also known that this sort of uniform lower bound on variances in the vanishing noise limit is equivalent to anomalous dissipation~\cite{drivas_lagrangian_2017}. Anomalous dissipation in turn implies nonuniqueness for the associated transport equation. This chain of implications is given in Section~\ref{ss:anomalous}, to prove Proposition~\ref{prop:variance lower bound implies nonunique transport}. Proposition~\ref{prop:variance lower bound implies nonunique transport}, together with some further soft arguments contained in Proposition~\ref{prop:soft fubini variance bound} to properly commute the integrals/expectations, then concludes the second item of Theorem~\ref{thm:sharp refreshing} as a consequence of the separation estimate of Proposition~\ref{prop:main separation estimate}.

While the above sequence of implications may seem a bit over-complicated to deduce the ill-posedness of Theorem~\ref{thm:sharp refreshing}, it is not clear there is a simpler route for proving the nonuniqueness for the continuity equation. This is because, as shown in~\cite{ambrosio_continuity_2014,bruePositiveSolutionsTransport2021}, a.e.\ ODE nonuniqueness is consistent with uniqueness for bounded solutions to the continuity equation. Thus it is rather difficult to conclude the desired PDE nonuniqueness from a pure ODE argument. The addition of the Brownian noise and the connection with anomalous dissipation provide exactly the selection mechanism that allows for the construction of distinct PDE solutions.

\subsubsection{Proving the separation estimate}

What remains is to prove the separation estimate. We lower bound the variance by the expected squared distance between coupled copies of the law of $X^\kappa_t$: $X^1_t, X^2_t$. We take the coupling given by taking independent Brownian noises up until a $\kappa$ dependent time, and then take the Brownian noises to be identical after that. The initial independence ensures---with high probability---that $X^1_t$ and $X^2_t$ are sufficiently far at the $\kappa$ dependent time, using the Nash estimate (Lemma~\ref{lem:nash estimate}).

Following this initial noise-driven separation, the separation is driven by the action of the velocity field, and we treat the noise as a perturbation creating (small enough) error. We prove this separation growth through inductively applying Lemma~\ref{lem:one scale separation}, which gives that if the particles have the right initial separation at time $T^{j+1}$, then they will have the right separation at time $T^j$, with high probability.

The primary technical challenge is then to prove Lemma~\ref{lem:one scale separation}. By the explicit construction of the velocity field, on the time interval $[T^{j+1},T^j]$, the relative particle displacement $X^2 - X^1$ is essentially undergoing a random walk where it gets $N^j \approx 2^{2(\alpha+\beta-1)j +j^{1/2}}$ many hits, each of magnitude $2^{-(\alpha+\beta)j}.$ The hits however only have this magnitude provided that $|X^2 -X^1|\geq 2^{-j}$; if $|X^1 - X^2| \ll 2^{-j}$, the hits are much smaller. 

Thus we essentially need to study the long time behavior of a random walk, where we lose good control on the random walk if it gets too close to the origin. However, for $d \geq 3,$ we know Brownian motion (and a suitably isotropic random walk more generally) is transient, and so does not particularly ``want to'' get close to the origin. 

The argument therefore proceeds as follows. We couple the $(X^1,X^2)$ process to a process $(Y^1,Y^2)$, where $(Y^1,Y^2) =(X^1,X^2)$ up until $|X^2- X^1| \leq C2^{-j}$, but $(Y^1,Y^2)$ are undergoing a random walk that is effectively homogeneous, that is, we do not lose control of the size of the hits if $|Y^1 - Y^2| \ll 2^{-j}$. Then we compare $(Y^1,Y^2)$ to a Gaussian random walk, using a quantitative, process-level central limit theorem (or Donsker invariance principle), given by Theorem~\ref{thm:quantitative donsker}.

For the Gaussian random walk, we will have uniform control on the covariances of the steps. This allows us to prove (in Lemma~\ref{lem:gaussian random walks}) 1) that the components of the Gaussian random walk never get too close together (with high probability), using that we are in $d \geq 3$, and 2) the final time marginal (corresponding to the time $T^j$) of the Gaussian random walk is suitably spread out, thus the two points of the Gaussian random walk are spread to the right scale with high probability. Since the Gaussian walk is close to $(Y^1,Y^2)$, these properties transfer to $(Y^1,Y^2)$. Thus, with high probability $\inf_{t \geq 0} |Y^1_t - Y^2_t| \geq C 2^{-j}$, and so $(X^1,X^2) = (Y^1,Y^2)$ for all times (with high probability). Since $(Y^1,Y^2)$ is spread to the right scale at time $T^j$, then so is $(X^1,X^2)$, with high probability, allowing us to conclude.

There are two points worth remarking on for making the above argument rigorous. We need to ``overshoot'' the scale $2^{-j}$ on each step, since we want it to be an asymptotically trivial probability that the random walk points get within $2^{-j}$ of each other. Since the probability depends on the ratio of $2^{-j}$ and the initial separation, we want the initial separation to be $\gg 2^{-j}$; this is why we aim for $2^{-j + j^{1/2}/4}$ in Lemma~\ref{lem:one scale separation}. The choice of $N^j$ will allow exactly this sort of separation to be attained with high probability. The other point is that we have entirely neglected the action of the noise $\sqrt{\kappa} dW_t$ in the discussion around Lemma~\ref{lem:one scale separation}. The noise is treated as a pure error term, though exactly how and when this error is taken is fairly subtle. The key point is that the oscillation of $\sqrt{\kappa} W_t$ on the time scale $2^{-\beta j}$ will be $\ll 2^{-j}$ for the $j$ we consider; this, together with the regularity of the active scale, ensures that the noise contributes very little to each independent ``hit''. We emphasize that, while Lemma~\ref{lem:one scale separation} is being used, the noises acting on each particle $X^1, X^2$ are identical; thus the noise can neither drive separation nor compression directly, only indirectly through its interactions with the velocity field.

\subsubsection{Comparison with~\cite{hess-childs_sharp_2026}}
\label{sss:comparison}

The above construction shares many commonalities with that of~\cite{hess-childs_sharp_2026}, which has the rather different objective of constructing velocity fields demonstrating a lack of regularization by noise for a variety of driving noises. Thus, while in both cases SDEs are considered, in~\cite{hess-childs_sharp_2026} the SDE itself remains ill-posed, while here, since the velocity fields $U \in C^0_{t,x}$, the SDE is well-posed and only the ODE is ill-posed. \cite{hess-childs_sharp_2026} only considers the ODE/SDE problem, so does not consider connections with the continuity/transport equation.

In proving the ODE/SDE separation estimates, there are two primary new arguments that are needed in this work that were not needed in \cite{hess-childs_sharp_2026}. The first---and simpler---one is using the noise for a short amount of time to drive an initial separation before switching to using the velocity field to drive the separation. The second---and much more substantial---new component is the significantly more involved arguments needed to derive an estimate giving the single scale separation: Lemma~\ref{lem:one scale separation} in this work compared to~\cite[Proposition 4.2]{hess-childs_sharp_2026}. That is because~\cite{hess-childs_sharp_2026} does not need to have a spatial finite range of dependence, allowing the scale fields to be constructed using shear flows, essentially reducing the random walk step to a one-dimensional problem, and in particular, avoiding the difficulty of particles getting too close together. This allows the use of the classical 1D Berry--Esseen theorem instead of the more complicated process-level CLT given by Theorem~\ref{thm:quantitative donsker}.

\section{An abstract averaging lemma}

\label{s:abstract averaging}

We now prove the central technical lemma that will be repeatedly used to pass estimates on the averaged scale fields, which we will then interpolate, Sobolev embed, and plug into the nonlinear Young theory. We will first need a sharp (in moment scaling) martingale CLT upper bound.

To establish this sharp CLT upper bound, we first recall the following consequence of~\cite[Definition 6.1.3, Proposition 6.1.5, Theorem 7.3.1]{pena_decoupling_1999}, which are based on the works~\cite{hitczenko_domination_1994,kwapien_semimartingale_1991}.

\begin{theorem}[Decoupling inequality]
    There exists a universal constant $C>0$ such that for any sequence of real-valued random variables $A_i$ adapted to a filtration $\mathcal{F}_i$, there exists (on perhaps a larger probability space) a sequence of real-valued random variables $B_i$ and a sigma algebra $\mathcal{G}$ such that $B_i$ are conditionally independent given $\mathcal{G}$, and for all $i \geq 1$, 
    \[\mathrm{Law}\big(B_i \mid \mathcal{G}\big) = \mathrm{Law}\big(A_i \mid \mathcal{F}_{i-1}\big)\]
    and such that for all $p \geq 1, n \geq 1$
    \[\Big\|\sum_{i=1}^n A_i \Big\|_{L^p_\omega} \leq C \Big\|\sum_{i=1}^n B_i \Big\|_{L^p_\omega}.\]
\end{theorem}

Additionally, we recall the following optimal bounds on sums of independent random variables.

\begin{theorem}[{\cite[Corollary 2, Remark 2]{latala_estimation_1997}}]
    There exists a universal constant $C>0$ such that for any sequence of independent mean-zero real-valued random variables $X_i$, we have the bound for $p \geq 2, n \geq 1$
    \[\Big\|\sum_{i=1}^n X_i\Big\|_{L^p_\omega} \leq C \sup_{ 2  \lor \frac{p}{n} \leq q \leq p} \frac{p}{q} \Big(\frac{n}{p}\Big)^{1/q} \max_{1 \leq i \leq n} \|X_i\|_{L^q_\omega}.\]
\end{theorem}

Together these directly imply the following sharp martingale CLT estimate that we will use for the averaging lemma.

\begin{corollary}
\label{cor:martingale upper bound}
    There exists a universal constant such that for any martingale difference sequence $D_i$ adapted to a filtration $\mathcal{F}_i$---so $D_i$ is $\mathcal{F}_i$ measurable and $\E[D_i \mid \mathcal{F}_{i-1}] =0$---we have for all $p \geq 2, n \geq 1$, the bounds
    \begin{equation}
        \label{eqn:quantCLT}
        \Big\|\sum_{i=1}^n D_i\Big\|_{L^p_\omega} \leq C \sup_{ 2  \lor \frac{p}{n} \leq q \leq p} \frac{p}{q} \Big(\frac{n}{p}\Big)^{1/q} \max_{1 \leq i \leq n} \big\|\E[|D_i|^q \mid \mathcal{F}_{i-1}]^{1/q}\big\|_{L^\infty_\omega}.
    \end{equation}
\end{corollary}

We now prove the averaging lemma. It is stated rather abstractly so as to be applied in the distinct refreshing and sweeping cases. We note that~\eqref{eq:independent and centered} only requires centering away from the initial $[0,h]$ time interval, which is useful in Section~\ref{s:autonomous}, as $\eta$ depends on the scale field $v^j$ at the initial time.

\begin{lemma}
    \label{lem:abstract averaging}
    There exists $C(d)>0$ such that the following holds. Let $V,v : [0,1] \times \R^d \to \R^d$ be random vector fields such that $V \in L^1_\omega\mathcal{C}^0_t\mathcal{C}^1_x$ and $v \in \mathcal{C}^0([0,1] \times \R^d)$. Let $\mathcal{F}_t$ be a filtration such that for all $t \in [0,1]$, $V|_{[0,t] \times \R^d}, v|_{[0,t] \times \R^d}$ are $\mathcal{F}_t$-measurable. Let $h \in (0,1]$ and suppose that for all $t \in [0,1]$, 
    \begin{equation}
    \label{eq:independent and centered}
    v|_{[t+h,1] \times \R^d} \indep \mathcal{F}_t \quad\text{and}\quad \E v|_{[h,1] \times \R^d} =0.
    \end{equation}
    For a fixed $y \in \R^d,$ let $Z_t$ be the (random, almost surely uniquely defined) solution to 
    \[\begin{cases}\dot Z_t = V(t,Z_t),\\ Z_0 =y.\end{cases}\]
    We split
    \begin{equation}
        \label{eq:splitting in to components}
        \int_0^t v(r,Z_r+x)\,dr = M(t,x) + F(t,x),
    \end{equation}
    with
    \begin{equation}
        \label{eq:F t define}
        F(t,x):= \sum_{\ell=2}^{\infty} \E\Big[\int_{(\ell-1) h}^{\ell h \land t} v(r,Z_r+x)\,dr \,\Big|\, \mathcal{F}_{(\ell-2)h}\Big],
    \end{equation}
    where $\int_a^b f(r)\,dr$ is taken to be zero when $b<a$. Then $M$ and $F$ obey the following estimates: for all $\nu \in [0,1)$ there exists a constant $C(\nu)>0$ such that for all $p \geq 2$, sigma algebras $\mathcal{G} \subseteq \mathcal{F}_0$, $\lambda \in [0,1],$ and $n \in \N$,
    \begin{align}
    \label{eq:M averaged bound}
        \|M\|_{\mathcal{C}^{1/2}_t \mathcal{C}^0_x L^p_\omega} &\leq  C h^{1/2} \|v\|_{L^p_\omega \mathcal{C}^0_{t,x}} +  C h^{1/2} \min\big(p^{1/2} \|v\|_{L^p_\omega \mathcal{C}^0_{t,x}}, p^{\lambda \lor 1/2} \sup_{q \geq 2} q^{-\lambda} \|v\|_{L^q_\omega \mathcal{C}^0_{t,x}}\big),
        \\
        \label{eq:F averaged bound}
        \|F\|_{\mathcal{C}^1_t \mathcal{C}^0_x L^p_\omega} &\leq C h \|\nabla v\|_{L^{2}_\omega \mathcal{C}^0_{t,x}}\Big(  \big\|V - \E[V \mid \mathcal{G}]\big\|_{L^{p}_\omega \mathcal{C}^0_{t,x}} \lor h^{\frac{\nu}{1-\nu}} \|V\|_{L^{\frac{p}{1-\nu}}_\omega \mathcal{C}^0_t \mathcal{C}^\nu_z}^{\frac{1}{1-\nu}} \Big),\\
          \|M\|_{L^p_\omega \mathcal{C}^1_t \mathcal{C}^n_x }&\leq C \|\nabla^n v\|_{L^p_\omega \mathcal{C}^0_{t,x}},\label{eq:M a priori}
          \\\|F\|_{L^p_\omega \mathcal{C}^1_t \mathcal{C}^n_x }&\leq C \|\nabla^n v\|_{L^p_\omega \mathcal{C}^0_{t,x}}.
          \label{eq:F a priori}
    \end{align}
\end{lemma}

\begin{proof}
    By the definition of $F(t,x)$,
    \begin{equation}
    \label{eq:M define}
    M(t,x) =  \int_0^t v(r,Z_r+x)\,dr - \sum_{\ell=2}^{\infty} \E\Big[\int_{(\ell-1) h}^{\ell h \land t} v(r,Z_r+x)\,dr \,\Big|\, \mathcal{F}_{(\ell-2)h}\Big].
    \end{equation}
    We first note that~\eqref{eq:M a priori} and~\eqref{eq:F a priori} are essentially direct from the definitions and the (conditional) Minkowski inequality.

    Fix now $0 < s \leq t, x \in \R^d$ ($s=0$ follows by continuity). Let 
    \[m = \lceil s/h \rceil +1\quad\text{and}\quad n = \lfloor t/h\rfloor.\]
    We note that $m \geq 2$. We consider first~\eqref{eq:M averaged bound}, and note that
    \begin{align*}
    \|M(t,x) - M(s,x)\|_{L^p_\omega} &= \Big\|\int_s^t v(r,Z_r+x)\,dr - \sum_{\ell=2}^{\infty} \E\Big[\int_{(\ell-1) h \lor s}^{\ell h \land t} v(r,Z_r+x)\,dr\,\Big|\, \mathcal{F}_{(\ell-2)h}\Big]\Big\|_{L^p_\omega}
    \\&\leq \Big\| \sum_{\ell=m}^{n } D_\ell \Big\|_{L^p_\omega} + C(t-s) \land h \|v\|_{L^p_\omega \mathcal{C}^0_{t,x}},
    \end{align*}
    where we denote
    \[D_\ell := \int_{(\ell-1)h}^{\ell h} v(r,Z_r+x)\,dr - \E\Big[\int_{(\ell-1) h}^{\ell h} v(r,Z_r+x)\,dr\,\Big|\, \mathcal{F}_{(\ell-2)h}\Big].\]

    If $m >n$, then we're done. Otherwise, $m \leq n$ and we further split
    \[\Big\| \sum_{\ell=m}^{n } D_\ell \Big\|_{L^p_\omega} \leq \Big\| \sum_{\ell=m, \ell \in 2 \N}^{n } D_\ell \Big\|_{L^p_\omega}+ \Big\| \sum_{\ell=m, \ell \in 2 \N+1}^{n } D_\ell \Big\|_{L^p_\omega}.\]
    We only bound the first term as the bound on the second follows by an identical argument. We note that $D_\ell$ is $\mathcal{F}_{\ell h}$ measurable and that 
    \[\E[D_\ell \mid \mathcal{F}_{(\ell-2)h}] =0,\]
    due to~\eqref{eq:independent and centered}. Thus---after reindexing time---$\sum_{\ell=m, \ell \in 2 \N}^{n } D_\ell$ is a sum of martingale differences, allowing us to apply Corollary~\ref{cor:martingale upper bound}, to get
    \[\Big\| \sum_{\ell=m, \ell \in 2 \N}^{n } D_\ell \Big\|_{L^p_\omega} \leq  C \sup_{ 2  \lor \frac{p}{n-m+1} \leq q \leq p} \frac{p}{q} \Big(\frac{n-m+1}{p}\Big)^{1/q} \max_{m \leq \ell \leq n} \big\|\E[|D_\ell|^q \mid \mathcal{F}_{(\ell-2)h}]^{1/q}\big\|_{L^\infty_\omega}.\]
    We then note that by the definition of $D_\ell$, the (conditional) Minkowski inequality, and the independence of $v|_{[(\ell-1)h,\ell h]}$ and $\mathcal{F}_{(\ell-2)h}$ by~\eqref{eq:independent and centered}, we have that for $\ell \geq m$,
    \begin{align*}
        \E[|D_\ell|^q \mid \mathcal{F}_{(\ell-2)h}]^{1/q} &= \E\Big[\Big|\int_{(\ell-1)h}^{\ell h} v(r,Z_r+x)\,dr - \E\Big[\int_{(\ell-1) h}^{\ell h} v(r,Z_r+x)\,dr\,\Big|\, \mathcal{F}_{(\ell-2)h}\Big]\Big|^q \,\Big|\, \mathcal{F}_{(\ell-2)h}\Big]^{1/q}
        \\&\leq 2h \E\big[\|v|_{[(\ell-1)h,\ell h]}\|_{\mathcal{C}^0_{t,x}}^q \,\big|\, \mathcal{F}_{(\ell-2)h}\big]^{1/q}
        \\&= 2h \|v|_{[(\ell-1)h,\ell h]}\|_{L^q_\omega \mathcal{C}^0_{t,x}}
        \\&\leq 2h \|v\|_{L^q_\omega \mathcal{C}^0_{t,x}}.
    \end{align*}
    Combining the above displays, we get
  \[\|M(t,x) - M(s,x)\|_{L^p_\omega} \leq  C h \sup_{ 2  \lor \frac{p}{n-m+1} \leq q \leq p} \frac{p}{q} \Big(\frac{n-m+1}{p}\Big)^{1/q}\|v\|_{L^q_\omega \mathcal{C}^0_{t,x}}
+  C(t-s) \land h \|v\|_{L^p_\omega \mathcal{C}^0_{t,x}}.
    \]
    Bounding $(t-s) \land h \leq (t-s)^{1/2} h^{1/2}$ and $n-m+1 \leq h^{-1} (t-s)$, then dividing by $(t-s)^{1/2}$ and taking the supremum over $t,s$, and $x$, we get that
     \[\|M\|_{\mathcal{C}^{1/2}_t \mathcal{C}^0_x L^p_\omega} \leq  C h^{1/2} \|v\|_{L^p_\omega \mathcal{C}^0_{t,x}} +  C h \sup_{\tau>0} \tau^{-1/2} \sup_{ 2  \lor \frac{p}{h^{-1} \tau} \leq q \leq p} \frac{p}{q} \Big(\frac{h^{-1}\tau}{p}\Big)^{1/q}  \|v\|_{L^q_\omega \mathcal{C}^0_{t,x}}.\]
    Using the elementary inequality
    \[
    \sup_{2\vee \frac{p}{n}\leq q\leq p} \frac{p}{q}\Big(\frac{n}{p}\Big)^{1/q} q^\lambda\leq p^{\lambda\vee \frac{1}{2}}\sqrt{n}
    \]
    for $n\geq 1$ and $\lambda\in[0,1]$ to bound the supremum, we conclude~\eqref{eq:M averaged bound}.
    
    We now consider~\eqref{eq:F averaged bound}, keeping $s,t,x$ as above. Then by Minkowski's integral inequality
    \begin{align}
        \|F(t,x) - F(s,x)\|_{L^p_\omega} &= \Big\|\sum_{\ell=2}^{\infty} \E\Big[\int_{(\ell-1) h \lor s}^{\ell h \land t} v(r,Z_r+x)\,dr\,\Big|\, \mathcal{F}_{(\ell-2)h}\Big]\Big\|_{L^p_\omega}
        \notag\\&\leq \sum_{\ell=2}^\infty \int_{(\ell-1) h \lor s}^{\ell h \land t}\big\|\E[v(r,Z_r+x)\mid \mathcal{F}_{(\ell-2)h}]\big\|_{L^p_\omega}\,dr
    \notag\\&\leq (t-s) \sup_{\ell \geq 2} \sup_{r \in [(\ell-1)h,\ell h] \cap [0,1]} \big\|\E[v(r,Z_r+x)\mid \mathcal{F}_{(\ell-2)h}] \big\|_{L^p_\omega}.
    \label{eq:F difference bound}
    \end{align}
    We then fix $\ell \geq 2, r \in [(\ell-1)h,\ell h] \cap [0,1]$, and note that for any $\mathcal{F}_{(\ell-2)h}$-measurable curve $\tilde Z :[0,\infty) \to \R^d$, we have by the independence and centering of~\eqref{eq:independent and centered} (note that we only assumed centering for $t\geq h$, hence why we started the sum from $\ell=2$):
    \[\E[v(r,\tilde Z_r+x)\mid \mathcal{F}_{(\ell-2)h}]=0.\]
    Thus we have that 
    \begin{align}\big\|\E[v(r,Z_r+x)\mid \mathcal{F}_{(\ell-2)h}] \big\|_{L^p_\omega} &= \big\|\E[v(r,Z_r+x) - v(r, \tilde Z_r+x)\mid \mathcal{F}_{(\ell-2)h}] \big\|_{L^p_\omega}
    \notag\\& \leq  \big\|\E\big[\|\nabla v(r,\cdot)\|_{\mathcal{C}^0_x} |Z_r - \tilde Z_r| \,\big| \, \mathcal{F}_{(\ell-2)h}\big] \big\|_{L^p_\omega}
    \notag\\&\leq    \big\|\E\big[\|\nabla v(r,\cdot)\|_{\mathcal{C}^0_x}^2\,\big| \, \mathcal{F}_{(\ell-2)h}\big]^{1/2} \E\big[|Z_r - \tilde Z_r|^2\,\big| \, \mathcal{F}_{(\ell-2)h}\big]^{1/2}  \big\|_{L^p_\omega}
    \notag\\&\leq \|\nabla v\|_{L^2_\omega \mathcal{C}^0_{t,x}} \|Z - \tilde Z\|_{L^{p}_\omega \mathcal{C}^0([(\ell-1)h,\ell h])},
    \label{eq:bound the conditional average}
    \end{align}
    where for the second to last inequality we use conditional Cauchy--Schwarz and for the final inequality we use that $v(r,\cdot)$ is independent of $\mathcal{F}_{(\ell-2)h}$ to factor $\|\nabla v\|_{L^2_\omega \mathcal{C}^0_{t,x}}$ out of the $L^p_\omega$ norm and then the (conditional) Minkowski integral inequality and that $p\geq 2$.

    Let $\tilde V(t,z) := \E[V(t,z) \mid \mathcal{G}]$ and then define $\tilde Z_t$ to be the unique solution to
    \[\begin{cases} \dot{\tilde Z}_t = \tilde V(t,\tilde{Z}_t) \\ \tilde Z_{(\ell-2)h} =Z_{(\ell -2)h}.\end{cases} \]
    Note that $\tilde Z$ is (by construction) $\mathcal{F}_{(\ell-2)h}$ measurable. We then note that
    \[\frac{d}{dt} (Z_t - \tilde Z_t) = V(t,Z_t) - V(t,\tilde Z_t) + V(t,\tilde Z_t) - \tilde V(t,\tilde Z_t),\]
    so that
    \[\frac{d}{dt} |Z_t - \tilde Z_t| \leq \|V\|_{\mathcal{C}^0_t \mathcal{C}^\nu_z} |Z_t - \tilde Z_t|^\nu + \|V - \tilde V\|_{\mathcal{C}^0_{t,x}}.\]
    Thus integrating in time, we have that for $R := \sup_{r \in [(\ell-2)h, \ell h]} |Z_r - \tilde Z_r|$, 
    \[R \leq  2h \|V\|_{\mathcal{C}^0_t \mathcal{C}^\nu_z} R^\nu + 2h \|V - \tilde V\|_{\mathcal{C}^0_{t,x}}. \]
    Then either $R \leq 4h \|V - \tilde V\|_{\mathcal{C}^0_{t,x}}$ or we get that
    \[R \leq 2h \|V\|_{\mathcal{C}^0_t \mathcal{C}^\nu_z} R^\nu + R/2.\]
    Thus, after rearranging, we see that in either case we have
    \[R \leq \big(4h \|V\|_{\mathcal{C}^0_t \mathcal{C}^\nu_z}\big)^{\frac{1}{1-\nu}} \lor  4h \|V - \tilde V\|_{\mathcal{C}^0_{t,x}}.\]
    Thus 
    \[\|Z - \tilde Z\|_{L^{p}_\omega \mathcal{C}^0([(\ell-2)h,\ell h])} \leq \big(8h \|V\|_{L^{\frac{p}{1-\nu}}_\omega \mathcal{C}^0_t \mathcal{C}^\nu_z}\big)^{\frac{1}{1-\nu}} \lor  \big(8h \|V - \tilde V\|_{L^{p}_\omega \mathcal{C}^0_{t,x}}\big).\]
    So by~\eqref{eq:F difference bound},~\eqref{eq:bound the conditional average}, and recalling the definition of $\tilde V$, there exists $C(\nu)>0$ such that
    \[ \|F(t,x) - F(s,x)\|_{L^p_\omega}  \leq C(t-s) h \|\nabla v\|_{L^{2}_\omega \mathcal{C}^0_{t,x}}\Big(  \big\|V - \E[V \mid \mathcal{G}]\big\|_{L^{p}_\omega \mathcal{C}^0_{t,x}} \lor h^{\frac{\nu}{1-\nu}} \|V\|_{L^{\frac{p}{1-\nu}}_\omega \mathcal{C}^0_t \mathcal{C}^\nu_z}^{\frac{1}{1-\nu}} \Big).\]
    Dividing by $t-s$ and taking a supremum in $t,s,x$, we get~\eqref{eq:F averaged bound}.
\end{proof}

\section{Stability estimates in the refreshing regime}
\label{s:refreshing}

We now prove the stability estimates for the refreshing case: Theorem~\ref{thm:refreshing ODE quantitative}, Theorem~\ref{thm:refreshing ODE qualitative}, and Theorem~\ref{thm:refreshing Bihari}. To that end, we fix for this section a dimension $d\geq 1$, some $\beta > 0,$ and some $U = \sum_{j=0}^\infty u^j$ satisfying Assumption~\ref{asmp:main for refreshing} for $\beta,d$. We allow constants $C>0$ in this section to freely depend on $\beta,d$. We assume without loss of generality that the $u^j$ are in $\mathcal{C}^0_t \mathcal{C}^2_x$ almost surely, as will in particular be implied by hypotheses of the propositions below. This allows for qualitative well-posedness of the objects defined below.

We introduce the following notation for the UV-truncated velocity field, associated flow, and associated averaged velocity fields.

\begin{definition}
\label{def:iop refreshing}
    For each $N \in \N$, we let $U^N : [0,1] \times \T^d \to \R^d$ denote the truncated sum
    \[U^N(t,x) := \sum_{j=0}^N u^j(t,x).\]
    We then let $\Phi^N: [0,1]  \times \T^d \to \T^d$ denote the (uniquely defined) flow solution of the ODE
    \begin{equation}\label{eq:truncated ODE flow}\begin{cases}
        \dot \Phi^{N}_t(y) = U^N(t,\Phi^{N}_t(y)),\\ \Phi^{N}_0(y) = y.
    \end{cases}
    \end{equation}
    For any $v : [0,1] \times \T^d \to \R^d$ with $v \in \mathcal{C}^0([0,1]\times \T^d)$, we let $\iop^{N} v : \T^d \times [0,1] \times \T^d \to \R^d$ be defined by
    \[\iop^{N} v(y,t,x) := \int_0^t v\big(s, \Phi^N_s(y) + x\big)\,ds.\]
\end{definition}

We now state estimates---uniformly in the truncation parameter $N$---on $\iop^N U^N$, which will then be combined with Proposition~\ref{prop:effective Lipschitz Gronwall} in order to prove a uniform-in-truncation effective regularity. We defer the proof, which is essentially a careful application of Lemma~\ref{lem:abstract averaging} and Sobolev embeddings, to Section~\ref{ss:effective regularity refreshing}.

\begin{proposition}
    \label{prop:refreshing stochastic bound lipschitz flow}
    Let $\alpha \in (0,1)$ with $\alpha + \beta/2>1$. Suppose for some $\lambda \in [0,1]$ and $K \geq 1$, we have that for all $j \in \N, n \in \{0,1,2\}, p \geq 1$,
    \[\|\nabla^n u^j\|_{L^p_\omega \mathcal{C}^0_{t,x}} \leq K p^\lambda 2^{(n-\alpha)j}.\]
    Then for any $N \in \N$, there exists a splitting
    \[\iop^{N} U^N := \sum_{j=0}^N \big(I^{1,j,N} + I^{2,j,N}\big)\]
    such that for all $0< \ep < C(\alpha,\beta)^{-1}$,  there exists $C(\ep,\alpha,\beta,d)>0$ such that for all $p \geq 1$,
    \begin{align}
        \|I^{1,j,N}\|_{\mathcal{C}^0_y L^p_\omega \mathcal{C}^{1-\frac{1-\alpha}{\beta} -\ep}_t \mathcal{C}^{1}_x } &\leq C K p^{\lambda \lor 1/2} 2^{- C^{-1} j}, \label{eq:I1 bound}\\
        \|I^{2,j,N}\|_{\mathcal{C}^0_y L^p_\omega \mathcal{C}^{1-\ep}_t \mathcal{C}^1_x} &\leq C K^{4(\frac{2-\alpha}{\beta} \lor 1)}  p^{((\frac{2-\alpha}{\beta} \lor 1) + \ep) \lambda} 2^{- C^{-1} j},
            \label{eq:I2 bound}\\
            \|I^{1,j,N}\|_{ W^{1,2p}_yL^p_\omega  \mathcal{C}^{1}_t \mathcal{C}^{1}_x } &\leq CK p^\lambda 2^{2j}  \|\nabla \Phi^N\|_{L^{2p}_\omega L^{2p}_y\mathcal{C}^0_t},
            \label{eq:I1 deriv bound}
            \\
            \|I^{2,j,N}\|_{ W^{1,2p}_yL^p_\omega  \mathcal{C}^{1}_t \mathcal{C}^{1}_x } &\leq CK p^\lambda 2^{2j}  \|\nabla \Phi^N\|_{L^{2p}_\omega L^{2p}_y\mathcal{C}^0_t},
                           \label{eq:I2 deriv bound}
    \end{align}
    where $\Phi^N$ is given by~\eqref{eq:truncated ODE flow}.
\end{proposition}

We now give a version of Theorem~\ref{thm:refreshing ODE quantitative} for the UV-truncated problem, uniformly in the truncation parameter. We also keep more explicit track of the constants, as we will need to optimize with another truncation scheme in order to prove Theorem~\ref{thm:refreshing ODE qualitative}.

\begin{proposition}\label{prop:refreshing cutoff estimates}
    Let $\alpha \in (0,1)$ with $\alpha + \beta/2>1$. Suppose for some $K \geq 1$ and
        \[0 \leq \lambda < \frac{\beta}{2-\alpha} \land \Big( 1- \frac{1-\alpha}{\beta}\Big),\]
    we have that for all $j \in \N, n \in \{0,1,2\}, p \geq 1$,
    \[\|\nabla^n u^j\|_{L^p_\omega \mathcal{C}^0_{t,x}} \leq K p^\lambda 2^{(n-\alpha)j}.\]   
    Then there exists $\delta(\alpha,\beta,\lambda)>0$ and  $C(\alpha,\beta,\lambda)>0$ such that for all $N \in \N$, there exists a random field $M^N : \T^d \to [0,\infty)$ such that
    \begin{enumerate}
        \item
        \label{item:lipschitz truncated modulus} For all $p \geq 1$ and $t \geq 2,$
        $\P(\|M^N\|_{\mathcal{C}^0_y} \geq t) \leq  C e^{C p^C K^C} (\log t)^{-p}$,
        \item \label{item:lp truncated modulus}
        for all $p \geq 1$, $\|M^N\|_{L^p_\omega L^p_y} \leq C e^{Cp^C K^C}$,
        \item 
        \label{item:stability estimate refreshing truncated}
        and for any $y,z \in \T^d$, any $w \in \mathcal{C}^{1/2 + \delta}([0,1])$ with $w_0 =0$, if we let $X_t = \Phi^N_t(y)$, and take the solution $\tilde X$ to 
          \begin{equation}
         \label{eq:forced ODE proposition}
         \tilde X_t=z+\int_0^t U^N(s,\tilde X_s)\,ds+w_t,\qquad t\in[0,1],
         \end{equation}
         then we have the stability estimate on the time interval $[0,1]$,
         \[\|X - \tilde X\|_{\mathcal{C}^{1/2 +\delta}_t} \leq M^N(y)\big(|y-z| + \|w\|_{\mathcal{C}^{1/2+\delta}_t}\big),\]
         and the expansion/compression estimate for all $t\in[0,1]$
         \[ \frac{1}{M^N(y)}|y-z|-\|w\|_{C^{1/2+\delta}_t}\leq |X_t-\tilde X_t|\leq M^N(y)\big(|y-z|+\|w\|_{C^{1/2+\delta}_t}\big).\]
    \end{enumerate}
\end{proposition}

\begin{proof}
    Fix $N \in \N$. Then the first two bounds of Proposition~\ref{prop:refreshing stochastic bound lipschitz flow} together with the bounds on $\lambda$ and the assumed bound on the $u^j$ imply that there exists $\eps(\lambda,\alpha,\beta)>0$ and $C(\alpha,\beta,\lambda)>0$ such that $\iop^{N}U^N= \sum_{j=0}^N (I^{1,j,N}+I^{2,j,N})$ with
    \begin{align}
    \notag
        \|I^{1,j,N}\|_{\mathcal{C}^0_y L^p_\omega \mathcal{C}^{1-\frac{1-\alpha}{\beta} -\ep}_t \mathcal{C}^{1}_x } &\leq C  Kp^{\lambda \lor 1/2} 2^{-C^{-1}j},\\
        \|I^{2,j,N}\|_{\mathcal{C}^0_y L^p_\omega \mathcal{C}^{1-\ep}_t \mathcal{C}^1_x} &\leq C K^{4(\frac{2-\alpha}{\beta} \lor 1)}  p^{((\frac{2-\alpha}{\beta} \lor 1) + \ep) \lambda} 2^{-C^{-1} j},
        \label{eq:I N j bounds}
    \end{align}
    where
    \begin{equation}\label{eq:p paramters}
     \frac{1}{2}\vee \lambda <1-\frac{1-\alpha}{\beta} -\ep =: \gamma\quad \text{and}\quad{\Big(\ep + \Big(\frac{2-\alpha}{\beta} \lor 1\Big) \Big) \lambda}<(1-\eps),  
    \end{equation}
    where we also use that, as $\alpha + \beta/2>1$, $1- \frac{1-\alpha}{\beta} > 1/2$.  

    Let then $I^{1,N} = \sum_{j=0}^N I^{1,j,N}$ and $I^{2,N} = \sum_{j=0}^N I^{2,j,N}$, so we have that
        \begin{align}
    \notag
        \|I^{1,N}\|_{\mathcal{C}^0_y L^p_\omega \mathcal{C}^{\gamma}_t \mathcal{C}^{1}_x } &\leq C  K p^{\lambda \lor 1/2} ,\\
        \|I^{2,N}\|_{\mathcal{C}^0_y L^p_\omega \mathcal{C}^{1-\ep}_t \mathcal{C}^1_x} &\leq C K^{4(\frac{2-\alpha}{\beta} \lor 1)} p^{((\frac{2-\alpha}{\beta} \lor 1) + \ep) \lambda},
        \label{eq:I N bounds}
    \end{align}

    Let $y, z \in \T^d, w \in \mathcal{C}^{\gamma}([0,1])$, $X_t:= \Phi^N_t(y)$, and $\tilde X$ solve~\eqref{eq:forced ODE proposition}. Then Proposition~\ref{prop:effective Lipschitz Gronwall} with $\gamma_1=\gamma$ and $\gamma_2=1-\eps$, letting $\delta = \gamma -1/2$, implies that there exists $C(\lambda,\alpha,\beta)>0$ such that 
    \begin{align}
      \|X-\tilde X\|_{\mathcal{C}^{1/2+\delta}_t}&\leq C\exp\big(C \|I^{1,N}(y,\cdot,\cdot)\|_{\mathcal{C}^\gamma_t \mathcal{C}^1_x}^{\frac{1}{\gamma}}\vee\|I^{2,N}(y,\cdot,\cdot)\|_{\mathcal{C}^{1-\ep}_t \mathcal{C}^1_x}^{\frac{1}{1-\ep}}\big)(|y-z|+\|w\|_{\mathcal{C}^{1/2+\delta}_t})
      \notag\\&=: M^N(y) (|y-z|+\|w\|_{\mathcal{C}^{1/2+\delta}_t}).  
      \label{eq:X difference bound}
    \end{align}
    This then gives the stability estimate in Item~\ref{item:stability estimate refreshing truncated}. Additionally, since $|X_t-\tilde X_t|\leq \|X-\tilde X\|_{C^\gamma_t} + |y-z|$, it also implies the second inequality in the compression estimate.

    For the first inequality in the compression estimate, fixing $t\in[0,1]$, and letting $Y_s:=X_{t-s}$, $\tilde Y_s:=\tilde X_{t-s}$, and $V^N(s,x)=-U^N(t-s,x)$, we have that $Y_s$ solves the ODE
    \[\begin{cases}
        \dot Y_s= V^N(s,Y_s),\\
        Y_0=X_t,
    \end{cases}\]
    for $s\in[0,t]$ and $\tilde Y_s$ is a solution to the integral equation
    \[\tilde Y_s=\tilde X_t+\int_{0}^s V^N(r,\tilde Y_r)\,dr+\tilde w_s,\qquad s\in[0,t],\]
    where $\tilde w_s:=w_{t-s}-w_t$. Additionally, we easily see that 
    \[\iop^YV^N(s,x)=\iop^N U^N(y,t-s,x)-\iop^NU^N(y,t,x).\]
    We thus have that $\iop^Y V^N$ can be broken up into two terms $J^{1,N}$ and $J^{2,N}$ such that
    \[\|J^{1,N}\|_{\mathcal{C}^\gamma_t \mathcal{C}^1_x}\leq \|I^{1,N}(y,\cdot,\cdot)\|_{\mathcal{C}^\gamma_t \mathcal{C}^1_x},\]
    \[\|J^{2,N}\|_{\mathcal{C}^{1-\eps}_t \mathcal{C}^1_x}\leq \|I^{2,N}(y,\cdot,\cdot)\|_{\mathcal{C}^{1-\eps}_t \mathcal{C}^1_x},\]
    Proposition~\ref{prop:effective Lipschitz Gronwall}, now applied to $Y$ and $\tilde Y$, thus implies that
    \[|y-z|=|Y_t-\tilde Y_t|\leq M^N(y)\Big(|Y_0-\tilde Y_0|+\|\tilde w\|_{\mathcal{C}^{1/2+\delta}_t}\Big)\leq M^N(y)\big(|X_t-\tilde X_t|+\|w\|_{\mathcal{C}^{1/2+\delta}_t}\big).\]
    Rearranging, this implies the lower compression estimate, and concludes the proof of Item~\ref{item:stability estimate refreshing truncated}.
    
    We now need to bound $M^N$. To that end, we note that for $n \in \N$ and any fixed $y \in \T^d$, by~\eqref{eq:I N bounds} and~\eqref{eq:p paramters},
    \[\E \|I^{1,N}(y,\cdot,\cdot)\|_{\mathcal{C}^\gamma_t \mathcal{C}^1_x}^{\frac{n}{\gamma}} \leq \|I^{1,N}\|_{\mathcal{C}^0_y L^{n/\gamma}_\omega \mathcal{C}^\gamma_t \mathcal{C}^1_x}^{n/\gamma} \leq C^{n/\gamma} K^{Cn} \Big(\frac{n}{\gamma}\Big)^{n \gamma^{-1} (\lambda \lor 1/2)} \leq C^n K^{Cn} n^{(1-C^{-1})n},\]
    and
    \begin{align*}\E \|I^{2,N}(y,\cdot,\cdot)\|_{\mathcal{C}^{1-\ep}_t \mathcal{C}^1_x}^{\frac{n}{1-\ep}} &\leq \|I^{2,N}\|_{\mathcal{C}^0_y L^{n/(1-\ep)}_\omega \mathcal{C}^{1-\ep}_t \mathcal{C}^1_x}^{\frac{n}{1-\ep}} 
    \\&\leq  C^{\frac{n}{1-\ep}}  K^{Cn}\Big(\frac{n}{1-\ep}\Big)^{{n((\frac{2-\alpha}{\beta} \lor 1) + \ep) \lambda}/(1-\ep)} 
    \\&\leq C^n K^{Cn} n^{(1-C^{-1})n}.
    \end{align*}
    Thus Taylor expanding the exponential in~\eqref{eq:X difference bound}, using the two displays above, and the Stirling approximation of the factorial, we see that
    \begin{align*} \|M^N\|_{L^p_\omega L^p_y} &\leq  \|M^N\|_{ \mathcal{C}^0_y L^p_\omega}
    \\&\leq C \sup_{y \in \T^d} \Big(\E \exp\big(2Cp \|I^{1,N}(y,\cdot,\cdot)\|_{\mathcal{C}^\gamma_t \mathcal{C}^1_x}^{\frac{1}{\gamma}}\big) \Big)^{1/2p}  \Big(\E \exp\big(2Cp \|I^{2,N}(y,\cdot,\cdot)\|_{\mathcal{C}^{1-\ep}_t \mathcal{C}^1_x}^{\frac{1}{1-\ep}}\big) \Big)^{1/2p} 
    \\&\leq C\sup_{y \in \T^d} \bigg(\sum_{n=0}^\infty \frac{C^n p^n \E \|I^{1,N}(y,\cdot,\cdot)\|_{\mathcal{C}^\gamma_t \mathcal{C}^1_x}^{\frac{n}{\gamma}} }{n!} \bigg)^{1/2p} \bigg(\sum_{n=0}^\infty \frac{C^n p^n \E \|I^{2,N}(y,\cdot,\cdot)\|_{\mathcal{C}^{1-\ep}_t \mathcal{C}^1_x}^{\frac{n}{1-\ep}} }{n!} \bigg)^{1/2p}
    \\&\leq C\bigg(\sum_{n=1}^\infty \Big(\frac{C p K^{C}}{n^{C^{-1}}}\Big)^n\bigg)^{1/p}.
    \end{align*}
    Then we note that
    \[\sum_{n=1}^\infty \Big(\frac{C p K^{C}}{n^{C^{-1}}} \Big)^n\leq \sum_{n=1}^{(CpK^C)^C} (Cp K^C)^n + \sum_{n=(CpK^C)^C}^\infty 2^{-n} \leq (CpK^C)^{Cp^C K^C} +1\leq Ce^{C p^C K^C},\]
    thus in total we have that 
    \[\|M^N\|_{L^p_\omega L^p_y} \leq C e^{Cp^C K^C}.\]
    We thus have Item~\ref{item:lp truncated modulus}.

    We note that, pointwise in $\omega,y$, $\|\nabla \Phi^N(y)\|_{\mathcal{C}^{0}_t} \leq M^N(y).$ Thus by Item~\ref{item:lp truncated modulus},
    \[\|\nabla \Phi^N\|_{L^p_\omega L^p_y \mathcal{C}^{0}_t} \leq \|M^N\|_{L^p_\omega L^p_y} \leq C e^{Cp^C K^C}.\]
    Then~\eqref{eq:I1 deriv bound} and~\eqref{eq:I2 deriv bound} of Proposition~\ref{prop:refreshing stochastic bound lipschitz flow} together with the above display give that for all $p \geq 2$,
    \begin{align*}
                    \|I^{1,j,N}\|_{ W^{1,p}_yL^p_\omega  \mathcal{C}^{1}_t \mathcal{C}^{1}_x } &\leq  C e^{Cp^C K^C}2^{2j},
            \\
            \|I^{2,j,N}\|_{ W^{1,p}_yL^p_\omega  \mathcal{C}^{1}_t \mathcal{C}^{1}_x } &\leq  C e^{Cp^C K^C} 2^{2j}.
    \end{align*}
    Interpolating these bounds with~\eqref{eq:I N j bounds} gives
    \begin{align*}
        \|I^{1,j,N}\|_{L^p_\omega W^{C^{-1},p}_y  \mathcal{C}^{\gamma}_t \mathcal{C}^{1}_x }  = \|I^{1,j,N}\|_{W^{C^{-1},p}_y L^p_\omega \mathcal{C}^{\gamma}_t \mathcal{C}^{1}_x } &\leq C e^{Cp^C K^C} 2^{-C^{-1} j},\\
        \|I^{2,j,N}\|_{L^p_\omega  W^{C^{-1},p}_y \mathcal{C}^{1-\ep}_t \mathcal{C}^1_x}=\|I^{2,j,N}\|_{W^{C^{-1},p}_y L^p_\omega \mathcal{C}^{1-\ep}_t \mathcal{C}^1_x} &\leq C e^{Cp^C K^C} 2^{-C^{-1}j}.
    \end{align*}
    Then Sobolev embedding in $y$ with Lemma~\ref{lem:sobolev} for $p \geq C$ and using the pointwise control given by~\eqref{eq:I1 bound} and~\eqref{eq:I2 bound} to pass from $\mathcal{C}^{C^{-1}}_y$ to $\mathcal{C}^0_y$, and then summing in $j$, we see that for all $p \geq 1$,
    \begin{align*}
        \|I^{1,N}\|_{L^p_\omega \mathcal{C}^0_y  \mathcal{C}^{\gamma}_t \mathcal{C}^{1}_x } &\leq C e^{Cp^C K^C},
        \\   \|I^{2,N}\|_{L^p_\omega  \mathcal{C}^0_y \mathcal{C}^{1-\ep}_t \mathcal{C}^1_x} &\leq  C e^{Cp^C K^C}.
    \end{align*}
    Then we compute for $t \geq 2$, using Chebyshev and the above moment bounds,
    \begin{align*}
        \P(\|M^N\|_{\mathcal{C}^0_y} \geq t) &\leq \P\big( \|I^{1,N}\|_{\mathcal{C}^0_y \mathcal{C}^\gamma_t \mathcal{C}^1_x}^{1/\gamma} \geq C^{-1} \log t - C\big) +  \P\big( \|I^{2,N}\|_{\mathcal{C}^0_y \mathcal{C}^{1-\ep}_t \mathcal{C}^1_x}^{1/(1-\ep)} \geq C^{-1} \log t - C\big)
        \\&\leq C e^{C p^C K^C} (\log t)^{-p},
    \end{align*}
    proving Item~\ref{item:lipschitz truncated modulus}, and hence concluding the proof.
\end{proof}

We are now ready to prove Theorem~\ref{thm:refreshing ODE quantitative} by carefully passing to a limit in Proposition~\ref{prop:refreshing cutoff estimates}.

\begin{proof}[Proof of Theorem~\ref{thm:refreshing ODE quantitative}]
    Let
    \[M(y) := \liminf_{N \to \infty} M^N(y),\]
    where $M^N(y)$ are from Proposition~\ref{prop:refreshing cutoff estimates}. We then note that by Fatou's lemma and Items~\ref{item:lipschitz truncated modulus} and~\ref{item:lp truncated modulus} of Proposition~\ref{prop:refreshing cutoff estimates} we have the following:
    \begin{enumerate}
     \item for all $p \geq 1$ and $t \geq 2,$
        $\P(\|M\|_{\mathcal{C}^0_y} \geq t) \leq  C e^{C p^C K^C} (\log t)^{-p}$,
        \item and for all $p \geq 1$, $\|M\|_{L^p_\omega L^p_y} \leq C e^{Cp^C K^C}$.
    \end{enumerate}
    From these, the first two items of the theorem are direct. All that is left is to prove the stability and compression estimates with modulus $M$. 
    
    We first prove the stability estimate. Let $y ,z \in \T^d$, $w \in \mathcal{C}^{1/2 + \delta}([0,1])$, $X$ a solution to~\eqref{eq:main refreshing ODE 0}, $\tilde X$ a solution to~\eqref{eq:forced ODE theorem}. Let $N \in \N$ and let $X^N_t := \Phi^N_t(y)$. We view $X, \tilde X$ as forced solutions to the ODE with velocity field $U^N$; that is, $X, \tilde X$ solve~\eqref{eq:forced ODE proposition} with forcings $w^1_t := \int_0^t\sum_{j>N} u^j(s,X_s)\,ds$ and $w^2_t := \int_0^t \sum_{j>N}u^j(s,\tilde X_s)\,ds + w_t$ respectively. Thus the stability estimate in Proposition~\ref{prop:refreshing cutoff estimates} implies that
    \begin{align}
    \label{eq:stability compare to N}
        \|X - \tilde X\|_{\mathcal{C}^{1/2+\delta}_t} &\leq   \|X - X^N\|_{\mathcal{C}^{1/2+\delta}_t} + \|X^N - \tilde X\|_{\mathcal{C}^{1/2+\delta}_t} 
       \notag \\&\leq M^N(y) \big(|y-z| + \|w^1\|_{\mathcal{C}^{1/2+\delta}_t} + \|w^2\|_{\mathcal{C}^{1/2+\delta}_t}\big)
       \notag \\&\leq  M^N(y) \big(|y-z| + \|w\|_{\mathcal{C}^{1/2+\delta}_t} + 2 \sum_{j>N} \|u^j\|_{\mathcal{C}^0_{t,x}}\big),
    \end{align}
    Then we note that
    \[\Big\|\sum_{j=0}^\infty \|u^j\|_{\mathcal{C}^0_{t,x}}\Big\|_{L^1_\omega} \leq \sum_{j=0}^\infty \|u^j\|_{L^1_\omega \mathcal{C}^0_{t,x}} \leq K\sum_{j=0}^\infty 2^{-\alpha j} <\infty.\]
    Thus  $\sum_{j=0}^\infty \|u^j\|_{\mathcal{C}^0_{t,x}} < \infty$ almost surely, and so, on a universal full probability set (that is, it doesn't depend on $X,\tilde X$), we have that
    \[\lim_{N \to \infty} \sum_{j>N} \|u^j\|_{\mathcal{C}^0_{t,x}} =0.\]
    Thus taking a $\liminf_{N \to \infty}$ of both sides of~\eqref{eq:stability compare to N}, we see that---on the same universal full probability set---
    \[ \|X - \tilde X\|_{\mathcal{C}^{1/2+\delta}_t} \leq M(y) \big(|y-z| + \|w\|_{\mathcal{C}^{1/2+\delta}_t}\big),\]
    as claimed. This concludes the proof of the stability estimate, after harmlessly modifying $M$ on a zero probability set. Additionally, when $w=0$, it implies the second inequality of the stability estimate.
    
    Continuing to the first inequality of the compression estimate, now assuming that $w=0$, for any $t\in[0,1]$, Proposition~\ref{prop:refreshing cutoff estimates} implies that
    \begin{equation}\label{eq:lower stability estimate}|X_t-\tilde X_t|\geq |\tilde X_t-X^N_t|-\|X-X^N\|_{\mathcal{C}^0_t}\geq \frac{1}{M^N(y)}|y-z|-\sum_{j>N} \|u^j\|_{\mathcal{C}^0_{t,x}}-\|X-X^N\|_{\mathcal{C}^0_t}.\end{equation}
    However, by the stability estimate of Proposition~\ref{prop:refreshing cutoff estimates}, by similar arguments to the previous paragraph, we have that 
    \[\|X-X^N\|_{\mathcal{C}^0_t} \leq M^N(y)\sum_{j > N} \|u^j\|_{\mathcal{C}^0_{t,x}}.\]
    Thus since $\sum_{j>N} \|u^j\|_{\mathcal{C}^0_{t,x}}\rightarrow 0$ on a universal full probability set, letting $N_k \to \infty$ be such that $M^{N_k}(y) \to M(y)$, combining the above displays and taking $k \to \infty$ (noting the bound is trivial in the case that $M(y) =\infty$), we have that on a universal full probability set,
    \[|X_t-\tilde X_t|\geq  \frac{1}{M(y)}|y-z|.\]
    This concludes the proof of the third item, and thus the theorem.
\end{proof}

The next proposition provides, under the assumptions of Theorem~\ref{thm:refreshing ODE qualitative}, a sequence of approximating velocity fields which satisfy the hypotheses of Theorem~\ref{thm:refreshing ODE quantitative} and exactly agree with the original field with high probability. This will allow us to prove Theorem~\ref{thm:refreshing ODE qualitative} using Theorem~\ref{thm:refreshing ODE quantitative} and an optimization argument.  We defer the proof to Section~\ref{ss:refreshing truncation}.

\begin{proposition} 
\label{prop:refreshing truncation approximation}
Suppose $U=\sum_{j=0}^\infty u^j$ satisfies Assumption~\ref{asmp:main for refreshing} for some $\beta>0$. Additionally, suppose that for some $\alpha\in(0,1)$ and $p>1$, we have a constant $K\geq 1$ such that for $n\in\{0,1,2\}$ and all $j\in\N$,
\begin{equation}\label{eq:refreshing truncation moments}
\|\nabla^n u^j\|_{L^p_\omega C^0_{t,x}}\leq K 2^{(n-\alpha)j}.
\end{equation}
    Then for all $\eps>\frac{\beta}{p}$ and $A\geq 1$, there exist vector fields $\{u^{j,A}\}_{j \geq 0}$ (on possibly a larger probability space, though coupled to the original $u^j$) also satisfying Assumption~\ref{asmp:main for refreshing} such that for all $j\geq 0$, $n\in\{0,1,2\}$,
\[\|\nabla^n u^{j,A}\|_{C^0_{t,x}} \leq A K 2^{(n-\alpha+\eps)j},
\]
and, letting $U^A:=\sum_{j=0}^\infty u^{j,A}$, there exists $C(\beta,p,\eps)>0$ such that
\[\P(U^A\neq U)\leq C A^{1-p}.\]
\end{proposition}

We now are ready to prove Theorem~\ref{thm:refreshing ODE qualitative}.

\begin{proof}[Proof of Theorem~\ref{thm:refreshing ODE qualitative}]
    As $p > \frac{\beta}{\alpha+\beta/2-1} \lor \frac{\beta}{\alpha}$ we can choose $\ep>0$ such that $\alpha + \beta/2 > 1+\ep$, $\alpha > \ep$, and $p > \beta/\ep$. Then by Proposition~\ref{prop:refreshing truncation approximation} we have that for all $A \geq 1$, we get a $\{u^{j,A}\}_{j\geq 0}$ satisfying Assumption~\ref{asmp:main for refreshing} with the same $\beta$ and such that for all $j\geq 0$, $n\in\{0,1,2\}$,
\[\|\nabla^n u^{j,A}\|_{C^0_{t,x}} \leq A K 2^{(n-\alpha+\eps)j},
\]
and, letting $U^A:=\sum_{j=0}^\infty u^{j,A}$, there exists $C(\alpha,\beta,p)>0$ such that
\begin{equation}
\label{eq:bad set estimate A}
\P(U^A\neq U)\leq C A^{1-p}\le C A^{-p/2},\end{equation}
where we use that $p \geq 2$. Note that $U^A$ then satisfies the hypotheses of Theorem~\ref{thm:refreshing ODE quantitative} with $\lambda=0$ (though with prefactor constant $AK$ and regularity exponent $\alpha -\ep$), thus by (the proof of) Theorem~\ref{thm:refreshing ODE quantitative}, for each $A \geq 1$, we have a random field $M^A : \T^d \to \R$ such that for any $y,z \in \T^d$, any $w \in C^{1/2 + \delta}([0,1])$ with $w_0=0$, if we have any solution $X$ to
\[\begin{cases} \dot X_t = U^A(t,X_t),\\
X_0 = y,\end{cases}\]
and any solution $\tilde X$ to 
          \begin{equation*}
         \tilde X_t=z+\int_0^t U^A(s,\tilde X_s)\,ds+w_t,\qquad t\in[0,1],
         \end{equation*}
         then we have the stability estimate on the time interval $[0,1]$,
         \begin{equation}
         \label{eq:truncated stability}
         \|X - \tilde X\|_{C^{1/2 +\delta}_t} \leq M^A(y)\big(|y-z| + \|w\|_{C^{1/2+\delta}_t}\big),
         \end{equation}
         and, in the case that $w=0$, the compression estimate for $t\in[0,1]$
         \[\frac{1}{M^A(y)}|y-z|\leq |X_t-\tilde X_t|\leq M^A(y)|y-z|,\]
    where for some $R(\alpha,\beta,p)>0$ (noting that $R$ can be taken to be uniform in $p$ for $p$ large enough),
    \begin{enumerate}
     \item for all $t \geq 2,$
        $\P(\|M^A\|_{\mathcal{C}^0_y} \geq t) \leq  R e^{R A^R} (\log t)^{-1}$,
        \item and for all $q \geq 1$, $\|M^A\|_{L^q_\omega L^q_y} \leq R e^{Rq^R A^R}$.
    \end{enumerate}
    Let
    \[B^A := \{\omega : U^A \ne U\}\]
    then define
    \[M(y,\omega) :=\inf\{M^n(y,\omega) : n\geq 1\text{ such that }\omega \not \in B^n\}.\]
    Then, by~\eqref{eq:truncated stability} and the definition of $B^n$, we have the third item. What remains is to bound $M$. We note that for any $n \in \N, t \geq 2$, by the definition of $M$,~\eqref{eq:bad set estimate A}, and the tail bounds on $\|M^n\|_{\mathcal{C}^0_y}$,
\[
    \P(\|M\|_{\mathcal{C}^0_y} \geq t) \leq  \P(\|M^n\|_{\mathcal{C}^0_y} \geq t) + \P(B^n)
    \leq R e^{R  n^R} (\log t)^{-1} +  C n^{-p/2}.
\]
    Taking that $n = \floor{\big((2R)^{-1}\log \log t\big)^{R^{-1}}}$, we get that for $t \geq C$,
    \[\P(\|M\|_{\mathcal{C}^0_y} \geq t) \leq C (\log t)^{-1/2} + C \big(\log \log t\big)^{-\frac{p}{2R}} \leq  C \big(\log \log t\big)^{-\frac{p}{2R}}.\]    Similarly, using Chebyshev with the moment bounds for $\|M^A\|_{L^q_y}$, we have that 
    \[\P(\|M\|_{L^q_y} \geq t) \leq \P(\|M^n\|_{L^q_y} \geq t) + \P(B^n) \leq R e^{Rq^R n^R} t^{-1} + C n^{-p/2}.\]
     Taking that $n = \floor{q^{-1}\big((2R)^{-1} \log t\big)^{R^{-1}}}$, we get that for $t \geq C$,
     \[\P(\|M\|_{L^q_y} \geq t) \leq C t^{-1/2} + C_q (\log t)^{-\frac{p}{2R}}.\]
     These bounds then allow us to conclude, letting $\gamma = \frac{p}{4R}$ and layer-caking to get the moment bounds.
\end{proof}

The following proposition will provide the primary estimates on $\iop^N U^N$ needed to prove Theorem~\ref{thm:refreshing Bihari}. We defer the proof---which is fairly similar to that of Proposition~\ref{prop:refreshing stochastic bound lipschitz flow}---to Section~\ref{ss:effective regularity refreshing}.

\begin{proposition}
    \label{prop:refreshing stochastic bound holder}
    Let $\alpha \in (0,1)$ with $\alpha + \beta>1$ and $\alpha + \beta/2 <1$. Suppose for all $p \geq 1,$ there exists $K(p)>0$ such that for all $j \in \N$ and $n \in \{0,1,2\}$,
    \[\|\nabla^n u^j\|_{L^p_\omega \mathcal{C}^0_{t,x}} \leq K(p) 2^{(n-\alpha)j}.\]
    Then for all $\ep<C(\alpha,\beta)^{-1}$ and $p\geq 1$, there exists $C(K,\alpha,\beta,\ep,p)>0$ such that, letting
    \[\nu := 2\alpha + \beta -1 -\ep \quad \text{and} \quad \gamma := \frac{1 - \alpha + \ep}{\beta} - \ep,\]
    for all $y\in \T^d$ and $N \in \N$,
    \[\|\iop^{N} U^N\|_{ \mathcal{C}^0_y L^p_\omega \mathcal{C}^{\gamma }_t \mathcal{C}^\nu_x} \leq C.\]
\end{proposition}

\begin{remark}
    Note that for $\nu(\ep),\gamma(\ep)$ as defined in Proposition~\ref{prop:refreshing stochastic bound holder}, we have that 
    \[\lim_{\ep \to 0} \gamma(1+\nu)>1\]
    using that $\alpha +\beta/2<1$ and $\alpha+\beta>1$. This will be useful for applying Proposition~\ref{prop:effective Holder Bihari}.
\end{remark}

We are now ready to prove Theorem~\ref{thm:refreshing Bihari} by combining Proposition~\ref{prop:refreshing stochastic bound holder}, Proposition~\ref{prop:effective Holder Bihari}, and an approximation argument sending $N \to\infty$.

\begin{proof}[Proof of Theorem~\ref{thm:refreshing Bihari}]
Fix $y \in \T^d$, $\ep>0$ and let
\[\nu := 2\alpha + \beta -1 -\ep \quad \text{and} \quad \gamma := \frac{1 - \alpha + \ep}{\beta} - \ep.\]
Take $\ep>0$ sufficiently small such that Proposition~\ref{prop:effective Holder Bihari} holds, that is, $\gamma(1+\nu)>1$. We then define
\[M := \liminf_{N \to\infty} \|\iop^NU^N(y,\cdot,\cdot)\|_{\mathcal{C}^\gamma_t \mathcal{C}^\nu_x}^{\frac{1}{1-\nu}}.\]
By Proposition~\ref{prop:refreshing stochastic bound holder} and Fatou's lemma, $M \in L^p_\omega$ for all $p \geq 1$.

Let $X,\tilde X$ solve~\eqref{eq:main refreshing ODE 0}. We then, as in the proof of Theorem~\ref{thm:refreshing ODE quantitative}, view $X,\tilde X$ as solutions to
\[\bar X_t = y + \int_0^t U^N(s,\bar X_s)\,ds + \bar w_t,\]
with forcings $w_t := \int_0^t \sum_{j > N} u^j(s,X_s)\,ds$ and $\tilde w_t :=  \int_0^t \sum_{j > N} u^j(s,\tilde X_s)\,ds$ respectively. Letting $X^N_t := \Phi^N_t(y)$, Proposition~\ref{prop:effective Holder Bihari} then gives that for $t \in [0,1],$
\begin{align*}
    |X_t - \tilde X_t| &\leq |X_t-X^N_t| + |X^N_t - \tilde X_t| 
    \\&\leq C \big((\|w\|_{\mathcal{C}^\gamma_t} + \|\tilde w\|_{\mathcal{C}^\gamma_t})t^\gamma + \|\iop^N U^N(y,\cdot,\cdot)\|_{\mathcal{C}^\gamma_t \mathcal{C}^\nu_x}^{\frac{1}{1-\nu}} t^{\frac{\gamma}{1-\nu}}\big)
    \\&\leq C \big( \sum_{j > N} \|u^j\|_{\mathcal{C}^0_{t,x}} + \|\iop^N U^N(y,\cdot,\cdot)\|_{\mathcal{C}^\gamma_t \mathcal{C}^\nu_x}^{\frac{1}{1-\nu}} t^{\frac{\gamma}{1-\nu}}\big).
\end{align*}
Taking a $\liminf_N$ of both sides, and using as in the proof of Theorem~\ref{thm:refreshing ODE quantitative} that $\sum_{j > N} \|u^j\|_{\mathcal{C}^0_{t,x}} \to 0$ almost surely, we get that (after harmlessly modifying $M$ on a zero probability set),
\[|X_t - \tilde X_t| \leq CM t^{\frac{\gamma}{1-\nu}}.\]
We then conclude, after redefining $M = CM$, plugging in the definitions of $\gamma,\nu$, and decreasing $\ep>0$ somewhat.
\end{proof}

\subsection{Proof of Proposition~\ref{prop:refreshing stochastic bound lipschitz flow} and Proposition~\ref{prop:refreshing stochastic bound holder}: stochastic averaging by refreshing}

\label{ss:effective regularity refreshing}

We now pass the necessary estimates on the averaged velocity fields. The following interpolation lemma will be useful.

\begin{lemma}
\label{lem:iterated holder interpolation}
    Let $\theta \in [0,1], a_1,a_2, b_1,b_2 \in [0,1]$, and $Y$ a Banach space. Let $f : [0,1] \times \R^d \to Y$, then we have the interpolation inequality
    \[\|f\|_{\mathcal{C}^{\theta a_1 + (1-\theta)a_2}_t \mathcal{C}^{\theta b_1 + (1-\theta)b_2}_x Y} \leq  4\|f\|_{\mathcal{C}^{ a_1}_t \mathcal{C}^{ b_1}_x Y}^\theta  \|f\|_{\mathcal{C}^{ a_2}_t \mathcal{C}^{ b_2}_x Y}^{1-\theta}.\]
\end{lemma}

\begin{proof}
We have that for any $t \ne s$ and $x\neq y$
 \begin{align*}
    &\frac{\|f(t,x)-f(s,x)-(f(t,y)-f(s,y))\|_{Y}}{|t-s|^{\theta a_1 + (1-\theta)a_2}|x-y|^{\theta b_1 + (1-\theta)b_2}}
    \\&\qquad= \Bigg(\frac{\|f(t,x)-f(s,x)-(f(t,y)-f(s,y))\|_{Y}}{|t-s|^{a_1}|x-y|^{b_1 }}\Bigg)^{\theta}\Bigg(\frac{\|f(t,x)-f(s,x)-(f(t,y)-f(s,y))\|_{Y}}{|t-s|^{a_2}|x-y|^{b_2 }}\Bigg)^{1-\theta}
    \\&\qquad\leq 4\|f\|_{\mathcal{C}^{a_1}_t \mathcal{C}^{b_1}_xY}^\theta \|f\|_{\mathcal{C}^{a_2}_t \mathcal{C}^{b_2}_xY}^{1-\theta},
 \end{align*}
 where the factor of $4$ appears from handling the case that some of $a_1,b_1,a_2,b_2 = 0$. 
 Taking the supremum over $t\neq s$ and $x\neq y$ concludes the bound.
\end{proof}

\begin{proof}[Proof of Proposition~\ref{prop:refreshing stochastic bound lipschitz flow}]
    Throughout we assume that $p$ is sufficiently large and $\eps$ is sufficiently small (both thresholds depending only on $\alpha,\beta$ and $d$). This suffices to prove the estimates for all $p$ due to the monotonicity of $L^p_\omega$. We fix $N \in \N$. We first claim for each $0\leq j\leq N$, there is a splitting 
    \[\iop^{N} u^j = I^{1,j,N} + I^{2,j,N}\]
    where for $n = 1,2$ and all $0 < \ep$ small enough, there exists $C(\alpha,\beta,\ep)>0$ such that
    \begin{align} 
\notag
    \|I^{1,j,N}\|_{\mathcal{C}^0_y \mathcal{C}^{1/2}_t \mathcal{C}^1_x L^p_\omega}&\leq  CK p^{\lambda \lor 1/2} 2^{(1-\alpha-\beta/2)j},
    \\\notag
\|I^{2,j,N}\|_{\mathcal{C}^0_y\mathcal{C}^{1}_t \mathcal{C}^1_x L^p_\omega} &\leq  CK^{4 (\frac{2-\alpha}{\beta} \lor 1)}  2^{-C^{-1} j}   p^{((\frac{2-\alpha}{\beta} \lor 1) + \ep) \lambda},\\\notag
        \|I^{1,j,N}\|_{\mathcal{C}^0_y\mathcal{C}^{1}_t \mathcal{C}^n_x L^p_\omega} &\leq C K p^\lambda 2^{(n-\alpha)j},\\
        \|I^{2,j,N}\|_{\mathcal{C}^0_y\mathcal{C}^{1}_t \mathcal{C}^n_x L^p_\omega} &\leq C K p^\lambda 2^{(n-\alpha)j}.
        \label{eq:before interpolation sharp bound}
    \end{align}
    For $j=0$, this holds trivially, taking $I^{1,0,N} =0 $ and $I^{2, 0,N} = \iop^N u^0$. Otherwise, we take $j \geq 1$, in which case we will use Lemma~\ref{lem:abstract averaging}, which we must put ourselves in the setting of. We fix $y \in \T^d$ and work pointwise in $y$. Let $V = U^N, v = \nabla u^j$, and for $t \geq 0$,
    \begin{align*}\mathcal{F}_t &:= \sigma\big(u^k_s : k \ne j, k \in \N, s \in [0,1]\big) \lor \sigma\big(u^j_s : s \in [0,t]\big),\\
    \mathcal{G} &:= \sigma\big(u^k_s : k \ne j, k \in \N, s \in [0,1]\big) \subseteq \mathcal{F}_0.
    \end{align*}
    Finally, let $h = 2^{-\beta j}$ and
    \[I^{2,j,N}(y,t,x):=\sum_{\ell=2}^\infty \E\bigg[ \int_{(\ell-1)h}^{\ell h\wedge t} u^j(r,\Phi^N_r(y)+x)\,dr\mid \mathcal{F}_{(\ell-2)h}\bigg],\]
    so that $I^{1,j,N}=\iop^{N} u^j-I^{2,j,N}$. Then by Assumption~\ref{asmp:main for refreshing}, we are exactly in the setting of Lemma~\ref{lem:abstract averaging}, with $\nabla I^{1,j,N} =M$ and $\nabla I^{2,j,N}=F$, except that we are on $\T^d$ instead of $\R^d$; this however is easily fixed by ``lifting'' the fields to periodic fields on $\R^d$. 

    By Lemma~\ref{lem:abstract averaging}, the latter two inequalities of~\eqref{eq:before interpolation sharp bound} are direct consequences of~\eqref{eq:M a priori},~\eqref{eq:F a priori}, Minkowski's integral inequality, and the assumed bounds on $u^j$.

    For the first inequality of~\eqref{eq:before interpolation sharp bound}, we use~\eqref{eq:M averaged bound} to get
    \begin{equation*}\|I^{1,j,N}\|_{\mathcal{C}^0_y \mathcal{C}^{1/2}_t \mathcal{C}^1_x L^p_\omega} \leq C h^{1/2} \|\nabla u^j\|_{L^p_\omega \mathcal{C}^0_{t,x}} + C h^{1/2} p^{\lambda \lor 1/2} \sup_{q \geq 2} q^{-\lambda} \|\nabla u^j\|_{L^q_\omega \mathcal{C}^0_{t,x}} \leq CK p^{\lambda \lor 1/2} 2^{(1-\alpha-\beta/2)j}, 
    \end{equation*}
    as claimed.

    Finally, for $\nu\in(0,\alpha)$ to be determined,  for the second inequality of~\eqref{eq:before interpolation sharp bound}, we use~\eqref{eq:F averaged bound} to get that for some $C(\alpha,\beta)>0$,
    \begin{align}\|I^{2,j,N}\|_{\mathcal{C}^0_y \mathcal{C}^{1}_t \mathcal{C}^1_x L^p_\omega} &\leq  C h \|\nabla^2 u^j\|_{L^{2}_\omega \mathcal{C}^0_{t,x}} \big(\|u^j\|_{L^{p}_\omega \mathcal{C}^0_{t,x}} \lor h^{\frac{\nu}{1-\nu}}\|U^N\|_{L^{\frac{p}{1-\nu}}_\omega \mathcal{C}^0_t \mathcal{C}^{\nu}_z}^{\frac{1}{1-\nu}} \big)
    \notag\\&\leq CK 2^{(2 - \alpha-\beta) j} \big( K p^\lambda 2^{-\alpha j} + 2^{-\frac{\beta\nu}{1-\nu}j} K^{\frac{1}{1-\nu}} p^{\frac{\lambda}{1-\nu}} \big)
    \notag\\&\leq  CK^2 p^{\lambda}  2^{(2 - 2\alpha-\beta) j}+ CK^{\frac{2}{1-\nu}}  2^{(2 - \alpha-\frac{\beta}{1-\nu}) j}   p^{\frac{\lambda}{1-\nu}}.
    \label{eq:I 2 j intermediate}
    \end{align}
    Here we use that $\nu<\alpha$ as with our assumed bounds on $u^j$ it guarantees that
    \[\|U^N\|_{L^{\frac{p}{1-\nu}}_\omega \mathcal{C}^0_t \mathcal{C}^{\nu}_z}\leq C K p^{\lambda}\sum_{k\geq 0}^\infty2^{(\nu-\alpha)k}\leq C K p^\lambda.\]
    We then take for some $\ep>0$
    \[\frac{1}{1-\nu} = \Big(\frac{2-\alpha}{\beta} \lor 1\Big) + \ep .\]
    We note that since $\alpha +\beta/2 \geq 1$,
    \[1-\frac{\beta}{2-\alpha}  \leq 1 - \frac{2 - 2\alpha}{2-\alpha} = \frac{\alpha}{2-\alpha} <\alpha,\]
    hence $\nu \in (0, \alpha)$ for $\eps$ sufficiently small. Plugging in this choice of $\nu$ to~\eqref{eq:I 2 j intermediate}, bounding the first term brutally, and simplifying, we get the second inequality of~\eqref{eq:before interpolation sharp bound}. We have thus finished showing~\eqref{eq:before interpolation sharp bound}.

    For the remainder, we let implicit constants depend on $\ep>0$ (taken sufficiently small) arbitrarily. We now interpolate the second and fourth inequalities of~\eqref{eq:before interpolation sharp bound} with $n=2$, using Lemma~\ref{lem:iterated holder interpolation} to get that,
    \[\|I^{2,j,N}\|_{\mathcal{C}^0_y \mathcal{C}^1_t \mathcal{C}^{1+ C^{-1}}_x L^p_\omega} \leq CK^{4 (\frac{2-\alpha}{\beta} \lor 1)}  2^{-C^{-1} j}   p^{((\frac{2-\alpha}{\beta} \lor 1) + \ep) \lambda}.\]
    Using then Lemma~\ref{lem:alpha-holder_embedding} and Lemma~\ref{lem:sobolev} (applied with $p \geq C$), we have that
    \[\|I^{2,j,N}\|_{\mathcal{C}^0_y L^p_\omega \mathcal{C}^{1- \ep}_t \mathcal{C}^1_x } \leq CK^{4 (\frac{2-\alpha}{\beta} \lor 1)}  2^{-C^{-1} \ep j}   p^{((\frac{2-\alpha}{\beta} \lor 1) + \ep) \lambda}.\]
    We thus get~\eqref{eq:I2 bound}.

     Next, we interpolate the first and third inequality of~\eqref{eq:before interpolation sharp bound} with $n=1$ using Lemma~\ref{lem:iterated holder interpolation} to give that for $\theta \in [0,1],$
     \[\|I^{1,j,N}\|_{\mathcal{C}^0_y\mathcal{C}^{1-\theta/2}_t \mathcal{C}^{1}_x L^p_\omega} \leq C K p^{\lambda \lor 1/2} 2^{(1-\alpha -\beta\theta/2) j}.\]
     Taking $\theta = \frac{2(1-\alpha)}{\beta} + 2\ep\in (0,1)$, we have that 
     \[\|I^{1,j,N}\|_{\mathcal{C}^0_y\mathcal{C}^{1-\frac{1-\alpha}{\beta} - \ep}_t \mathcal{C}^{1}_x L^p_\omega} \leq C K p^{\lambda \lor 1/2} 2^{-C^{-1} j}.\]
    Then interpolating this with the third inequality of~\eqref{eq:before interpolation sharp bound} with $n=2$ using Lemma~\ref{lem:iterated holder interpolation}, we have that 
      \[\|I^{1,j,N}\|_{\mathcal{C}^0_y\mathcal{C}^{1-\frac{1-\alpha}{\beta} - \ep + C^{-1}}_t \mathcal{C}^{1 + C^{-1}}_x L^p_\omega} \leq C K p^{\lambda \lor 1/2} 2^{-C^{-1} j}.\]
     Using then Lemma~\ref{lem:alpha-holder_embedding} and Lemma~\ref{lem:sobolev} for $p \geq C$, we have that
    \[\|I^{1,j,N}\|_{\mathcal{C}^0_yL^p_\omega \mathcal{C}^{1-\frac{1-\alpha}{\beta} -\ep}_t \mathcal{C}^{1}_x } \leq C K p^{\lambda \lor 1/2} 2^{-C^{-1}  j},\]
    thus giving~\eqref{eq:I1 bound}.

    Finally we turn our attention to~\eqref{eq:I1 deriv bound} and~\eqref{eq:I2 deriv bound}. We prove~\eqref{eq:I2 deriv bound}, which just uses the bare regularity of the $u^j$. The bound follows similarly (and more simply) for $\iop^N u^j$, and hence---by the triangle inequality---for $I^{1,j,N}$, thus giving~\eqref{eq:I1 deriv bound}. We note that by the definition of $I^{2,j,N}$, the conditional Minkowski integral inequality, and the chain rule 
    \begin{align*}\|I^{2,j,N}\|_{L^p_\omega W^{1,2p}_y \mathcal{C}^1_t\mathcal{C}^1_x} &\leq \Big\|\|\nabla \Phi^N\|_{L^{2p}_y \mathcal{C}^0_t} \|\nabla^2 u^j\|_{\mathcal{C}^0_{t,x}} \Big\|_{L^p_\omega} 
    \\&\leq \|\nabla \Phi^N\|_{L^{2p}_\omega L^{2p}_y \mathcal{C}^0_t} \|\nabla^2 u^j\|_{L^{2p}_\omega \mathcal{C}^0_{t,x}}
    \\&\leq CK p^\lambda 2^{2j} \|\nabla \Phi^N\|_{L^{2p}_\omega L^{2p}_y \mathcal{C}^0_t},
    \end{align*}
    where we use H\"older's inequality and then the assumed bounds on $u^j$. We thus conclude the proposition.
\end{proof}

\begin{proof}[Proof of Proposition~\ref{prop:refreshing stochastic bound holder}]
     We fix $N \in \N$ and $y \in \T^d$. In this argument, we allow constants to depend freely on $p \geq 1$. We first claim for each $j \in \N$, there is a splitting 
    \[\iop^{N} u^j(y,\cdot,\cdot) = I^{1,j} + I^{2,j}\]
    where for $n=0,1,2,$
    \begin{align} 
\notag
    \|I^{1,j}\|_{\mathcal{C}^{1/2}_t \mathcal{C}^0_x L^p_\omega}&\leq  C 2^{(-\alpha-\beta/2)j},\\\notag
\|I^{2,j}\|_{\mathcal{C}^{1}_t \mathcal{C}^0_x L^p_\omega} &\leq  C 2^{( 1- 2\alpha-\beta) j},\\\notag
        \|I^{1,j}\|_{\mathcal{C}^{1}_t \mathcal{C}^n_x L^p_\omega} &\leq C 2^{(n-\alpha)j},\\
        \|I^{2,j}\|_{\mathcal{C}^{1}_t \mathcal{C}^n_x L^p_\omega} &\leq C 2^{(n-\alpha)j}.
        \label{eq:before interpolation bihari bound}
    \end{align}
    For $j=0$, this holds trivially, taking $I^{1,0} =0 $ and $I^{2, 0} = \iop^N u^0(y,\cdot,\cdot)$. Otherwise, we take $j \geq 1$, in which case we will use Lemma~\ref{lem:abstract averaging}, which we must put ourselves in the setting of.  Let $V = U^N, v = u^j$. For $t \geq 0$, we let 
    \begin{align*}\mathcal{F}_t &:= \sigma\big(u^k_s : k \ne j, k \in \N, s \in [0,1]\big) \lor \sigma\big(u^j_s : s \in [0,t]\big),\\
    \mathcal{G} &:= \sigma\big(u^k_s : k \ne j, k \in \N, s \in [0,1]\big) \subseteq \mathcal{F}_0.
    \end{align*}
    Finally, let $h = 2^{-\beta j}$. Then by Assumption~\ref{asmp:main for refreshing}, we are exactly in the setting of Lemma~\ref{lem:abstract averaging}, except that we are on $\T^d$ instead of $\R^d$; this however is easily fixed by ``lifting'' the fields to periodic fields on $\R^d$. 

    Then we have by Lemma~\ref{lem:abstract averaging}, for any $\ep>0$, letting $\nu = \alpha-\ep$, letting $I^{1,j} =M$ and $I^{2,j}=F$, we first see the latter two inequalities of~\eqref{eq:before interpolation bihari bound} are direct consequences of~\eqref{eq:M a priori},~\eqref{eq:F a priori}, Minkowski's integral inequality, and the assumed bounds on $u^j$.

    For the first inequality of~\eqref{eq:before interpolation bihari bound}, we use~\eqref{eq:M averaged bound} to get
    \begin{equation*}\|I^{1,j}\|_{\mathcal{C}^{1/2}_t \mathcal{C}^0_x L^p_\omega} \leq C h^{1/2} \|u^j\|_{L^p_\omega \mathcal{C}^0_{t,x}} + C h^{1/2} \|u^j\|_{L^p_\omega \mathcal{C}^0_{t,x}} \leq C 2^{(-\alpha-\beta/2)j}, 
    \end{equation*}
    as claimed.

    Finally, for the second inequality of~\eqref{eq:before interpolation bihari bound}, we use~\eqref{eq:F averaged bound} to get for some constant $C(\ep,p)>0$,
    \begin{align*}\|I^{2,j}\|_{\mathcal{C}^{1}_t \mathcal{C}^0_x L^p_\omega} &\leq  C h \|\nabla u^j\|_{L^{2}_\omega \mathcal{C}^0_{t,x}} \big(\|u^j\|_{L^{p}_\omega \mathcal{C}^0_{t,x}} \lor h^{\frac{\nu}{1-\nu}}\|U^N\|_{L^{\frac{p}{1-\nu}}_\omega \mathcal{C}^0_t \mathcal{C}^{\nu}_z}^{\frac{1}{1-\nu}} \big)
    \\&\leq C 2^{(1 - \alpha-\beta) j} \big( 2^{-\alpha j} + 2^{-\beta\frac{\alpha-\ep}{1-\alpha+\ep}j}\big)
    \end{align*}
    Then since $\alpha + \beta >1$, we choose $\ep>0$ sufficiently small depending on $\alpha,\beta$ to get that the second term is controlled by the first, thus finishing the proof of~\eqref{eq:before interpolation bihari bound}.

    Now we want to apply Lemma~\ref{lem:abstract averaging} again, this time with $v = \nabla u^j$. We keep $\mathcal{F}_t, \mathcal{G}, V,h$ as above. Let $\tilde M, \tilde F$ be the pieces of the splitting obtained by Lemma~\ref{lem:abstract averaging} applied with this new $v$. Note by inspecting the definitions~\eqref{eq:M define} and~\eqref{eq:F t define}, we have that $\tilde M = \nabla I^{1,j}$ and $\tilde F = \nabla I^{2,j}$. Thus Lemma~\ref{lem:abstract averaging} gives, following exactly the computations of the previous two displays, that
    \begin{align}
        \notag
    \|I^{1,j}\|_{\mathcal{C}^{1/2}_t \mathcal{C}^1_x L^p_\omega}&\leq  C 2^{(1-\alpha-\beta/2)j},\\
\|I^{2,j}\|_{\mathcal{C}^{1}_t \mathcal{C}^1_x L^p_\omega} &\leq  C 2^{( 2- 2\alpha-\beta) j}.
        \label{eq:before interpolation bihari bound 2}
    \end{align}
    Interpolating $n=0$ and $n=1$ of the third inequality of~\eqref{eq:before interpolation bihari bound} using Lemma~\ref{lem:iterated holder interpolation}, we have for any $\nu \in (0,1)$,
    \begin{equation}
        \|I^{1,j}\|_{\mathcal{C}^{1}_t \mathcal{C}^\nu_x L^p_\omega} \leq C 2^{(\nu-\alpha)j}.
        \label{eq:interpolated a priori bihari}
    \end{equation}
    Interpolating the first and second inequalities of~\eqref{eq:before interpolation bihari bound} with the bounds of~\eqref{eq:before interpolation bihari bound 2} using Lemma~\ref{lem:iterated holder interpolation}, we have for any $\nu \in (0,1)$,
        \begin{align}
        \notag
    \|I^{1,j}\|_{\mathcal{C}^{1/2}_t \mathcal{C}^\nu_x L^p_\omega}&\leq  C 2^{(\nu-\alpha-\beta/2)j},\\
\|I^{2,j}\|_{\mathcal{C}^{1}_t \mathcal{C}^\nu_x L^p_\omega} &\leq  C 2^{( 1+\nu- 2\alpha-\beta) j}.
        \label{eq:interpolated bihari bound averaged}
    \end{align}
    Interpolating~\eqref{eq:interpolated a priori bihari} with the first inequality of~\eqref{eq:interpolated bihari bound averaged} using Lemma~\ref{lem:iterated holder interpolation}, we have for any $\gamma \in [1/2,1]$,
    \begin{equation}
        \|I^{1,j}\|_{\mathcal{C}^{\gamma }_t \mathcal{C}^\nu_x L^p_\omega} \leq  C 2^{(-\alpha -\beta +\nu +  \gamma \beta)j}.
    \end{equation}

    We now want to choose $\gamma,\nu \in [0,1]$ to maximize $\frac{\gamma}{1-\nu}$ (the separation rate given by Proposition~\ref{prop:effective Holder Bihari}) subject to $\gamma ( 1+\nu)>1$ (a restriction appearing in  Proposition~\ref{prop:effective Holder Bihari}) and $\nu + \gamma \beta < \alpha + \beta$ (in order to ensure decay in $j$). 

    We make the (slightly suboptimal, but helpful for making the argument go through simply) choice of $\nu = 2\alpha + \beta - 1 -\ep$ for some $\ep>0$. We then take
    \[ \gamma  =\frac{1 - \alpha + \ep}{\beta} - \ep.\]
    
    We note that for $\ep$ sufficiently small, since $\alpha + \beta/2 <1$ and $\alpha + \beta > 1$, $\gamma \in (1/2,1)$. We then have that under these choices, for some $C(\alpha,\beta,\ep)>0,$
    \begin{equation}
    \label{eq:I 1 bihari almost final}
    \|I^{1,j}\|_{\mathcal{C}^{\gamma }_t \mathcal{C}^\nu_x L^p_\omega} \leq C2^{-C^{-1} j}.
     \end{equation}

    We also have from~\eqref{eq:interpolated bihari bound averaged} and the definition of $\nu$ and that $\gamma \leq 1$,
    \begin{equation}
        \label{eq:I 2 bihari almost final}
    \|I^{2,j}\|_{\mathcal{C}^{\gamma}_t \mathcal{C}^\nu_x L^p_\omega} \leq \|I^{2,j}\|_{\mathcal{C}^{1}_t \mathcal{C}^\nu_x L^p_\omega} \leq  C 2^{-\ep j} \leq C 2^{-C^{-1} j}.
    \end{equation}

    Interpolating~\eqref{eq:I 1 bihari almost final} and~\eqref{eq:I 2 bihari almost final} respectively with the third and fourth inequalities of~\eqref{eq:before interpolation bihari bound} for $n=1$ using Lemma~\ref{lem:iterated holder interpolation}, then using Lemma~\ref{lem:alpha-holder_embedding} and Lemma~\ref{lem:sobolev} for $p \geq C$ large enough, we get for $\ep$ sufficiently small, 
    \[\|\iop^{N} u^j(y,\cdot,\cdot)\|_{ L^p_\omega \mathcal{C}^{\gamma }_t \mathcal{C}^\nu_x} \leq \|I^{1,j}\|_{ L^p_\omega \mathcal{C}^{\gamma }_t \mathcal{C}^\nu_x}+ \|I^{2,j}\|_{ L^p_\omega \mathcal{C}^{\gamma }_t \mathcal{C}^\nu_x}  \leq  C 2^{-C^{-1}j}.\]
    Summing over $j$ and taking the supremum in $y$, we conclude.
\end{proof}

\subsection{Proof of Proposition~\ref{prop:refreshing truncation approximation}: finite range truncation}
\label{ss:refreshing truncation}

\begin{proof}[Proof of Proposition~\ref{prop:refreshing truncation approximation}]

By homogeneity, we may assume that $K=1$. For each $j\geq 0$, let $\{\varphi^{j,\ell}\}_{\ell=1}^{L^j}$ be a partition of unity for $[0,1]$ such that $0\leq \varphi^{j,\ell}\leq 1$, $\sum_{\ell=1}^{L^j}\varphi^{j,\ell}=1$, $\mathrm{diam}(\supp\varphi^{j,\ell})\leq 2^{-\beta j}$,  and $L^j\leq C 2^{\beta j}$ where $C>0$ is some universal constant. 

Fixing $\eps$ and $A$, we first construct intermediate velocity fields from which we will construct the $u^{j,A}$. For each $1\leq \ell \leq L^j$ let
\[b^{j,\ell}:=\sup_{t\in \supp\varphi^{j,\ell}} \sum_{n=0}^2 2^{-(n-\alpha)j}\|\nabla^nu^j(t)\|_{\mathcal{C}^0_{x}},\]
so that
\[\Big\|\sup_{1\leq \ell\leq L^j}b^{j,\ell}\Big\|_{L^p_\omega}\leq \sum_{n=0}^2 2^{-(n-\alpha)j}\|\nabla^n u^j\|_{L^p_\omega \mathcal{C}^0_{t,x}}\leq 3,\]
by~\eqref{eq:refreshing truncation moments}. Let $\chi:[0,\infty)\rightarrow [0,1]$ be a smooth function such that $\chi(x)=1$ for $x\leq 1$ and $\chi(x)=0$ for $x\geq 2$. We then let
\[v^{j,A}(t,x):=\bigg(\sum_{\ell=1}^{L^j}\varphi^{j,\ell}(t) \chi\Big(\frac{b^{j,\ell}}{A2^{\eps j}}\Big) \bigg)u^j(t,x).\]
We thus have that $v^{j,A}=u^j$ if $b^{j,\ell}\leq A2^{\eps j}$ for all $1\leq \ell \leq L^j$, as well as the sure bound
\[\|\nabla^n v^{j,A}\|_{\mathcal{C}^0_{t,x}}\leq 2 A 2^{(n-\alpha+\eps)j},\]
for $n=0,1,2.$

Next, for $j\geq 1$, since $\E u^j=0$, it holds that
\[ \mathbb{E}[v^{j,A}(t,x)]=\sum_{\ell=1}^{L^j}\varphi^{j,\ell}(t)\E\Big[\Big(\chi\Big(\frac{b^{j,\ell}}{A2^{\eps j}}\Big)-1\Big)u^j(t,x)\Big].\]
We then note that, by Chebyshev's inequality 
\[\mathbb{E}[\indc_{b^{j,\ell}\geq A2^{\eps j}}b^{j,\ell}]\leq (A2^{\eps j})^{1-p}\|b^{j,\ell}\|_{L^p_\omega}^p\leq 3^p A^{1-p} 2^{(1-p)\eps j},\]
thus letting
\[p^{j,\ell,A}:= 3^{-p}2^{-\eps j}\mathbb{E}[\indc_{b^{j,\ell}\geq A2^{\eps j}}b^{j,\ell}],\]
we find that $0\leq p^{j,\ell,A}\leq  A^{1-p} 2^{-p\eps j}\leq 1$.

On an enlarged probability space, define independent Bernoulli random variables (also independent of $u^j$) $\xi^{j,\ell,A}$ with
\[\E[\xi^{j,\ell,A}]=\P(\xi^{j,\ell,A}=1)=p^{j,\ell,A}.\]
We then let $u^{0,A}:=v^{0,A}$ and for $j\geq 1$
\[u^{j,A}(t,x):= v^{j,A}(t,x)-\sum_{\ell=1}^{L^j}\varphi^{j,\ell}(t)\frac{\xi^{j,\ell,A}}{p^{j,\ell,A}}\E\Big[\Big(\chi\Big(\frac{b^{j,\ell}}{A2^{\eps j}}\Big)-1\Big)u^j(t,x)\Big],\]
with the convention that $\frac{\xi^{j,\ell,A}}{p^{j,\ell,A}}=0$ if $p^{j,\ell,A}=0$. We must thus verify that $u^{j,A}$ has the desired properties.

By the definition of $p^{j,\ell,A}$, $\xi^{j,\ell,A}$ and $u^{j,A}$ it immediately holds that $u^{j,A}$ is centered for $j\geq 1$. Additionally, we note that $b^{j,\ell}$ is independent of $b^{j,\ell'}$ if $\mathrm{dist}(\supp\varphi^{j,\ell},\supp\varphi^{j,\ell'})\geq 2^{-\beta j}$. By our assumption on the diameter of the supports for $\varphi^{j,\ell}$, this implies that $u^{j,A}$ has a finite temporal range of dependence of $3\cdot 2^{-\beta j}$. Re-indexing and changing constants, we can thus make it so that $u^{j,A}$ has a finite temporal range of dependence of $2^{-\beta j}$.

On the other hand, for arbitrary $j$ and $n=0,1,2$, by the product rule 
\[|\nabla^n u^{j,A}(t,x)|\leq |\nabla^{n} v^{j,A}(t,x)|+\sum_{\ell=1}^{L^j}\varphi^{j,\ell}(t) \frac{1}{p^{j,\ell,A}}\mathbb{E}\Big[\Big|\Big(\chi\Big(\frac{b^{j,\ell}}{A2^{\eps j}}\Big)-1\Big)\nabla^n u^j(t,x)\Big|\Big].\]
For $t\in \supp \varphi^{j,\ell}$,
\begin{align*}
\frac{1}{p^{j,\ell,A}}\mathbb{E}\Big[\Big|\Big(\chi\Big(\frac{b^{j,\ell}}{A2^{\eps j}}\Big)-1\Big)\nabla^n u^j(t,x)\Big|\Big]\leq \frac{2^{(n-\alpha)j}}{p^{j,\ell,A}}\mathbb{E}[\indc_{\{b^{j,\ell}\geq A2^{\eps j}\}} b^{j,\ell}]\leq 3^p2^{(n-\alpha+\eps)j}.
\end{align*}
Thus using the sure bound on $v^{j,A}$ and that $\varphi^{j,\ell}$ are a partition of unity, the two displays above imply that 
\begin{equation}\label{eq:pre a.s. refreshing}
\|\nabla^n u^{j,A}\|_{\mathcal{C}^0_{t,x}}\leq C A 2^{(n-\alpha+\eps)j},    
\end{equation}
for some $C$ depending only on $p$.

Finally, we have that
\begin{align*}
\P(U^{A}\neq U)&\leq \P(b^{j,\ell}> A 2^{\eps j}\text{ or } \xi^{j,\ell,A}= 1\text{ for some $j,\ell$})
\\&\leq \sum_{j\geq 0}\P(\max_{1\leq \ell\leq L^j}b^{j,\ell}> A 2^{\eps j})+\sum_{j\ge 1}\sum_{\ell=1}^{L^j}p^{j,\ell, A}
\\&\leq C\sum_{j\geq 0} A^{-p} 2^{-p\eps j}+C\sum_{j\ge 1}2^{(\beta-p\eps)j}A^{1-p}
\\&\leq C A^{1-p}
\end{align*}
for a constant $C(\beta,p,\eps)>0$, where we have used Chebyshev's inequality, the $p$ moment bounds on $\|\nabla^n u^j\|_{\mathcal{C}^0_{t,x}}$, and that $L^j\leq C 2^{\beta j}$ in the second last inequality. The final inequality follows as the sums are finite since $\beta<p\eps$ by assumption.

Changing $A$ somewhat to absorb the constant in~\eqref{eq:pre a.s. refreshing}, we conclude.
\end{proof}

\section{Stability estimates in the sweeping regime}
\label{s:autonomous}

In this section we prove the stability estimates for the sweeping case: Theorem~\ref{thm:main ode intro} and Theorem~\ref{thm:sweeping Bihari}. To that end, we fix for this section a dimension $d$ and a velocity field $U(x)= \sum_{j=0}^\infty u^j(x)$ satisfying Assumption~\ref{asmp:main for autonomous}. Due to the following proposition, whose proof we defer to Section~\ref{ss:autonomous truncation}, throughout the majority of this section we will work with velocity fields for which the $u^j$ obey sure bounds; the general result will then follow by approximation.

\begin{proposition} 
\label{prop:autonomous truncation approximation}
Suppose $U=\sum_{j=0}^\infty u^j$ satisfies Assumption~\ref{asmp:main for autonomous}. Additionally, suppose that for some $\alpha\in(0,1)$ and $p>1$, we have a constant $K\geq 1$ such that for $n\in\{0,1,2\}$ and all $j\in\N$,
\begin{equation}\label{eq:autonomous truncation moments}
\|\nabla^n u^j\|_{L^p_\omega \mathcal{C}^0_{x}}\leq K 2^{(n-\alpha)j}.
\end{equation}
    Then for all $\eps>\frac{d}{p}$ and $A\geq 1$ there exist vector fields $\{u^{j,A}\}_{j \geq 0}$ (on possibly a larger probability space, though coupled to the original $u^j$) also satisfying Assumption~\ref{asmp:main for autonomous} such that for all $j\geq 0$, $n\in\{0,1,2\}$,
\[\|\nabla^n u^{j,A}\|_{\mathcal{C}^0_{x}} \leq A K 2^{(n-\alpha+\eps)j},
\]
and, letting $U^A:=\sum_{j=0}^\infty u^{j,A}$, there exists $C(d,p,\eps)>0$ such that
\[\P(U^A\neq U)\leq C A^{1-p}.\]
\end{proposition}

We will use the following linear algebra notation.

\begin{definition}
    For a vector $v \in \R^d \backslash\{0\}$, we let $v^\perp \subseteq \R^d$ denote the orthogonal complement of the span of $v$. We let $\Pi_v$ denote orthogonal projection onto the subspace spanned by $v$ and $\Pi_{v^\perp}$ the orthogonal projection onto $v^\perp.$
\end{definition}

For fixed $y \in \T^d$, we define the random direction 
\begin{equation*}
    \eta := \begin{cases}
    \frac{U(y)}{|U(y)|} & U(y) \ne 0,\\
    (1,0,...,0) & U(y) =0.
    \end{cases}
\end{equation*}
We work on the space $\R \times \eta^\perp$, for which we use the following notation: we take $z \in \R \times \eta^\perp$ and decompose $z = (z_\tau, z_{\eta^\perp})$, where $z_\tau \in \R$ and $z_{\eta^\perp} \in \eta^\perp$. For $\delta>0$, we define $V^{\delta,y}: [0,\infty) \times \R \times \eta^\perp \to \R \times \eta^\perp$ by
\begin{equation}\label{eq:V delta def}
V^{\delta,y}(t,z)= \chi(t,z_{\eta^\perp})\frac{\big(1,\Pi_{\eta^\perp} U(y+ \eta t + z_{\eta^\perp})\big)}{\phi^\delta\big(\eta \cdot U(y +\eta t + z_{\eta^\perp})\big)},\end{equation}
where $\chi \in \mathcal{C}^\infty_{t,x}$ and $1_{[0,1/8] \times B_{1/8}} \leq \chi \leq 1_{[0,1/4] \times B_{1/4}}$ and $\phi^\delta : \R \to [\delta,\infty)$ is such that $\phi^\delta \in \mathcal{C}^\infty,$ $\phi^\delta|_{(-\infty,\delta]}(x) = \delta$, $\phi^\delta|_{[2\delta,\infty)}(x) = x$, and for $n \geq 1$, $\big|\frac{d^n}{dx^n} \phi^\delta\big| \leq C(n) \delta^{1-n}$. We then study the well-posedness of the following ODE,
\begin{equation}\label{eq:V delta ODE}
\begin{cases}
    \dot Z^{\delta,y}_t = V^{\delta,y}(t,Z^{\delta,y}_t),\\
    Z^{\delta,y}_0 = a,
\end{cases}
\end{equation}
as well as the two-parameter flow problem,
\begin{equation}\label{eq:V delta flow}
\begin{cases}
    \frac{d}{dt} \Psi^{\delta,y}_{s,t}(a) = V^{\delta,y}(t,\Psi^{\delta,y}_{s,t}(a)),\\
    \Psi^{\delta,y}_{s,s}(a) = a.
\end{cases}
\end{equation}

The stability of this transformed, regularized problem will then imply the stability of the original problem, following our sketch in Section~\ref{sss:sweeping estimates}.

In order to control~\eqref{eq:V delta ODE} and~\eqref{eq:V delta flow}, we will work with ``UV-cutoff'' versions of the problem and prove uniform estimates, as in the refreshing case. For $N \in \N, \delta>0$, we then let 
\begin{equation}\label{eq:V delta N def}
V^{N,\delta,y}(t,z):= \chi(t,z_{\eta^\perp})\frac{\big(1,\Pi_{\eta^\perp} U^N(y+ \eta t + z_{\eta^\perp})\big)}{\phi^\delta\big(\eta \cdot U^N(y +\eta t + z_{\eta^\perp})\big)}.\end{equation}
We then have the (well-posed due to the truncation of the high frequencies) flow given by 
\begin{equation}\label{eq:V delta N flow}
\begin{cases}
    \dot \Psi^{N,\delta,y}_t(a) = V^{N,\delta,y}(t,\Psi^{N,\delta,y}_t(a)),\\
    \Psi^{N,\delta,y}_0(a) = a.
\end{cases}
\end{equation}

We then define, similar to Definition~\ref{def:iop refreshing}, for $v \in \mathcal{C}^0([0,1] \times \R \times \eta^\perp)$, $\iop^{N,\delta,y} v$ on $(\R \times \eta^\perp)\times [0,1] \times (\R \times \eta^\perp)$ by
\[\iop^{N,\delta,y} v(a,t,z) := \int_0^t v(s,\Psi^{N,\delta,y}_s(a)+z)\,ds.\]
We defer the proof of the following result giving the needed bounds on the averaged velocity field when $\alpha>1/2$---which is analogous to Proposition~\ref{prop:refreshing stochastic bound lipschitz flow}---to Section~\ref{ss:effective regularity autonomous}.
\begin{proposition}
\label{prop:fixed p sweeping averaging}
Let $\alpha \in (\frac{1}{2},1)$, $\delta>0$, and $V^{N,\delta,y}(t,z)$ defined by~\eqref{eq:V delta N def}. If for some $K >0$ and all $j \in \N, n \in \{0,1,2\}$, we have the almost sure bound
\[\|\nabla^n u^j\|_{\mathcal{C}^0_{x}} \leq K 2^{(n-\alpha)j},\]
then there exists $C(d,\alpha,K,\delta) >0$ such that for all $N, p\geq 1, y\in \T^d$, there exists a splitting
\[\iop^{N,\delta,y} V^{N,\delta,y} = \sum_{j=0}^N I^{j,N,\delta}\]
such that
\begin{align}
\label{eq:Ij bound auto}
    \|I^{j,N,\delta}\|_{\mathcal{C}^0_a L^p_\omega \mathcal{C}^{1/2 + (2\alpha-1)/64}_t \mathcal{C}^1_z} \leq C p^{1/2} 2^{- C^{-1}j},\\
    \label{eq:Ij deriv bound auto}
      \|\nabla_a I^{j,N,\delta}\|_{\mathcal{C}^0_a L^p_\omega \mathcal{C}^{1}_t \mathcal{C}^1_z} \leq C 2^{2j}\|\nabla \Psi^{N,\delta,y}\|_{\mathcal{C}^0_a L^p_\omega\mathcal{C}^0_t}.
\end{align}
\end{proposition}

We now combine the above proposition with Proposition~\ref{prop:effective Lipschitz Gronwall} in order to prove well-posedness of the transformed, regularized flow $\Psi^{\delta,y}$ when $\alpha>1/2$.

\begin{proposition}
    \label{prop:transformed flow}
    Let $\alpha \in (\frac{1}{2},1)$, $\delta>0, y \in \T^d$, and $V^{\delta,y}(t,z)$ defined by~\eqref{eq:V delta def}. If for some $K >0$ and all $j \in \N, n \in \{0,1,2\}$, we have the almost sure bound
\[\|\nabla^n u^j\|_{\mathcal{C}^0_{x}} \leq K 2^{(n-\alpha)j},\]
    then almost surely there exists a unique solution to~\eqref{eq:V delta ODE} for all $a$, thus~\eqref{eq:V delta flow} almost surely has a unique solution. Further, the unique solution $\Psi^{\delta,y}$ to~\eqref{eq:V delta flow} is almost surely spatially Lipschitz uniformly in $s$ and $t$. That is,
    \[\P\Big(\sup_{s,t\in[0,1]} \|\nabla \Psi^{\delta,y}_{s,t}\|_{\mathcal{C}^0_a} < \infty\Big) = 1.\]
\end{proposition}

\begin{proof}
First we prove the compression estimates for~\eqref{eq:V delta ODE}. Let 
\[\gamma:=\frac{1}{2}+(2\alpha-1)/64>\frac{1}{2}\]
and fix some $a,b\in \R\times \eta^\perp$. Then, let $Z^1$ and $Z^2$ be any solutions to~\eqref{eq:V delta ODE} with $Z^1_0=a$ and $Z^2_0=b$, and $Z^N_t:=\Psi^{N,\delta,y}_t(a)$. We now view $Z^i$ as forced solutions to the ODE with velocity field $V^{N,\delta,y}$ with forcings $w^1_t:=\int_0^t (V^{\delta,y}-V^{N,\delta,y})(s,Z_s^1) \,ds$ and $w^2_t:=\int_0^t (V^{\delta,y}-V^{N,\delta,y})(s,Z_s^2)\,ds$ respectively. The stability estimate in Proposition~\ref{prop:effective Lipschitz Gronwall} then implies that
\begin{align*}
\|Z^1-Z^2\|_{\mathcal{C}^\gamma_t}&\leq \|Z^1-Z^N\|_{\mathcal{C}^\gamma_t}+\|Z^N-Z^2\|_{\mathcal{C}^\gamma_t}
\\&\leq  C\exp(C\|\iop^{N,\delta,y} V^{N,\delta,y}(a,\cdot,\cdot)\|_{\mathcal{C}^\gamma_t \mathcal{C}^1_z}^{1/\gamma} )\big(|a-b|+\|w^1\|_{\mathcal{C}^\gamma_t}+\|w^2\|_{\mathcal{C}^\gamma_t}\big)
\\&\leq C\exp(C\|\iop^{N,\delta,y} V^{N,\delta,y}(a,\cdot,\cdot)\|_{\mathcal{C}^\gamma_t \mathcal{C}^1_z}^{1/\gamma} )\big(|a-b|+2\|V^{\delta,y}-V^{N,\delta,y}\|_{C^0_{t,z}}\big),
\end{align*}
where we have used the crude estimate
\[\|w^1\|_{\mathcal{C}^\gamma_t}+\|w^2\|_{\mathcal{C}^\gamma_t}\leq 2\|V^{\delta,y}-V^{N,\delta,y}\|_{\mathcal{C}^0_{t,z}}.\]
Letting
\begin{equation}\label{eq:autonomous M N def}
M^N(a):=C\exp(C\|\iop^{N,\delta,y} V^{N,\delta,y}(a,\cdot,\cdot)\|_{\mathcal{C}^\gamma_t \mathcal{C}^1_z}^{1/\gamma}),
\end{equation}
the above estimate thus implies that for all $t\in[0,1]$,
\begin{equation}\label{eq:transformed upper compression}|Z^1_t- Z^2_t|\leq M^N(a)(|a-b|+2\|V^{\delta,y}-V^{N,\delta,y}\|_{C^0_{t,z}}). \end{equation}
We now give the matching lower estimate
\begin{equation}\label{eq:transformed lower compression} \frac{1}{M^N(a)} |a-b|-( M^N(a)+1)\|V^{\delta,y}-V^{N,\delta,y}\|_{\mathcal{C}^0_{t,z}}\leq |Z^1_t-Z^2_t|.\end{equation}

Fixing some $t\in[0,1]$, let $W^i_s:=Z^i_{t-s}$ and $\tilde V(s,z)=-V^{N,\delta,y}(t-s,z)$ so that $W^i_s$ is a solution to the forced ODE
\[W^i_s=Z^i_t+\int_0^s \tilde V(r,W^i_r)\,dr+\tilde w^i_s,\qquad s\in[0,t],\]
with $\tilde w^i_s:=w^i_{t-s}-w^i_t$. Then, letting $W^N_s:=Z^N_{t-s}$, we have that
\[\dot W^N_s=\tilde V(s,W^N_s),\]
for $s\in[0,t]$ and
\[\iop^{W^N}\tilde V(s,z)=\iop^{N,\delta,y} V^{N,\delta,y}(a,t-s,z)-\iop^{N,\delta,y} V^{N,\delta,y}(a,t,z).\]
This implies that
\[\|\iop^{W^N}\tilde V\|_{\mathcal{C}^\gamma_t\mathcal{C}^1_z}\leq \|\iop^{N,\delta,y} V^{N,\delta,y}(a,\cdot,\cdot)\|_{\mathcal{C}^\gamma_t \mathcal{C}^1_z}.\]
Proposition~\ref{prop:effective Lipschitz Gronwall} thus implies that
\begin{align*}
|a-b|=|W^N_t-W^2_t|&\leq M^N(a)\big(|W^N_0-W^2_0|+\|\tilde w^2\|_{\mathcal{C}^\gamma_t}\big)
\\&= M^N(a)\big(|Z^N_t-Z^2_t|+\|V^{\delta,y}-V^{N,\delta,y}\|_{\mathcal{C}^0_{t,z}}\big)
\\&\leq M^N(a)\big(|Z^1_t-Z^2_t| + |Z^N_t - Z^1_t|+\|V^{\delta,y}-V^{N,\delta,y}\|_{\mathcal{C}^0_{t,z}}\big)
\\&\leq M^N(a)\big(|Z^1_t-Z^2_t| + (M^N(a)+1)\|V^{\delta,y}-V^{N,\delta,y}\|_{\mathcal{C}^0_{t,z}}\big),
\end{align*}
where we bound $ |Z^N_t - Z^1_t|$ as above. Rearranging, this is~\eqref{eq:transformed lower compression}.

As we intend to take $N\rightarrow \infty$, we now argue that there exists a universal (not depending on $Z^i$) set on which
\begin{equation}\label{eq:V N to zero}
\lim_{N\rightarrow \infty} \|V^{\delta,y}-V^{N,\delta,y}\|_{\mathcal{C}^0_{t,z}}=0.
\end{equation}
To this end, we note that
\begin{align*}
&V^{\delta,y}(t,z)-V^{N,\delta,y}(t,z)
\\&\quad= \chi(t,z_{\eta^\perp})\frac{\big(0,\Pi_{\eta^\perp}\sum_{j>N}u^j(y+\eta t+z_{\eta^\perp})\big)}{\phi^\delta(\eta\cdot U(y+\eta t+z_{\eta^\perp}))}
\\&\qquad+\chi(t,z_{\eta^\perp})\frac{\big(1,\Pi_{\eta^\perp}U^N(y+\eta t+z_{\eta^\perp})\big)\big(\phi^\delta(\eta\cdot U^N(y+\eta t+z_{\eta^\perp}))-\phi^\delta(\eta\cdot U(y+\eta t+z_{\eta^\perp}))\big)}{\phi^\delta(\eta\cdot U^N(y+\eta t+z_{\eta^\perp}))\phi^\delta(\eta\cdot U(y+\eta t+z_{\eta^\perp}))}.
\end{align*}
Thus, since $\chi\leq 1$ and $\phi^\delta \geq \delta$, it holds that
\[\|V^{\delta,y}-V^{N,\delta,y}\|_{\mathcal{C}^0_{t,z}}\leq \Big(\delta^{-1}+\delta^{-2}\|\phi^\delta\|_{\mathcal{C}^1}\big(1+\sum_{j=0}^\infty \|u^j\|_{\mathcal{C}^0_x}\big)\Big)\sum_{j>N}\|u^j\|_{\mathcal{C}^0_x}.\]
Then we note that
    \[\Big\|\sum_{j=0}^\infty \|u^j\|_{\mathcal{C}^0_{x}}\Big\|_{L^1_\omega} \leq \sum_{j=0}^\infty \|u^j\|_{L^1_\omega \mathcal{C}^0_{x}} \leq K\sum_{j=0}^\infty 2^{-\alpha j} <\infty.\]
    Thus  $\sum_{j=0}^\infty \|u^j\|_{\mathcal{C}^0_{x}} < \infty$ almost surely, and so, on a universal full probability set, we have that
    \[\lim_{N \to \infty} \sum_{j>N} \|u^j\|_{\mathcal{C}^0_{x}} =0.\]
    This immediately implies~\eqref{eq:V N to zero} holds on the same set.

    Now, let
    \[M(a):=\liminf_{N\rightarrow \infty} M^N(a).\]
    Then, taking $N \to \infty$ along a ($a$-dependent) subsequence such that $M^{N_k}(a) \to M(a)$ and using \eqref{eq:transformed upper compression}, \eqref{eq:transformed lower compression}, and~\eqref{eq:V N to zero}, we have that on the same universal full probability set, for all $t\in[0,1]$,
    \[\frac{1}{M(a)}|a-b|\leq |Z^1_t-Z^2_t|\leq M(a)|a-b|.\]
    We now argue that $M=\sup_{a\in \R\times \eta^\perp} M(a)<\infty$ almost surely.

    For any $a$, the estimate~\eqref{eq:Ij bound auto} implies that
    \begin{align*}
        \mathbb{E}\|\iop^{N,\delta,y} V^{N,\delta,y}(a,\cdot,\cdot)\|_{\mathcal{C}^\gamma_t\mathcal{C}^1_z}^{n/\gamma}\leq \|\iop^{N,\delta,y} V^{N,\delta,y}\|_{\mathcal{C}^0_aL^{n/\gamma}_\omega\mathcal{C}^\gamma_t\mathcal{C}^1_z}^{n/\gamma}&\leq\Big(\sum_{j=0}^N\|I^{j,N,\delta}\|_{\mathcal{C}^0_aL^{n/\gamma}_\omega\mathcal{C}^\gamma_t\mathcal{C}^1_z} \Big)^{n/\gamma}
        \\&\leq C^{n/\gamma}\Big(\frac{n}{\gamma}\Big)^{n/(2\gamma)}
        \\&\leq C^n n^{(1-C^{-1})n}.
    \end{align*}
    Thus Taylor expanding the exponential in~\eqref{eq:autonomous M N def}, and using Stirling's approximation, for all $p\geq 1$
    \begin{align*}
    \|M^N\|_{\mathcal{C}^0_aL^p_\omega}&\leq C\sup_{a\in\R\times \eta^\perp} \Big(\mathbb{E}\exp(Cp\|\iop^{N,\delta,y} V^{N,\delta,y}(a,\cdot,\cdot)\|_{\mathcal{C}^\gamma_t\mathcal{C}^1_z}^{1/\gamma})\Big)^{1/p}
    \\&\leq C \sup_{a\in\R\times \eta^\perp} \Big(\sum_{n=0}^\infty \frac{C^np^n}{n!}\E\|\iop^{N,\delta,y} V^{N,\delta,y}(a,\cdot,\cdot)\|_{\mathcal{C}^\gamma_t\mathcal{C}^1_z}^{n/\gamma} \Big)^{1/p}
    \\&\leq C\Big(\sum_{n=0}^\infty \Big(\frac{Cp}{n^{C^{-1}}}\Big)^n \Big)^{1/p}
    \\&\leq C,
    \end{align*}
    for a constant $C>0$ depending on $d,\alpha, K,\delta$ and $p$.

    We next claim that for all $p\geq 1$, there exists $C_p>0$ such that for all $j \in \N$
    \begin{equation}
    \label{eq:I j bounded}
    \|I^{j,N,\delta}\|_{L^p_\omega \mathcal{C}^0_a \mathcal{C}^\gamma_t \mathcal{C}^1_z}  \leq C 2^{-C^{-1} j}.
    \end{equation}
    We note first that $I^{j,N,\delta}$ does not depend on $a_\tau$---since $V^{N,\delta,y}$ doesn't depend on $z_\tau$. We then note that, by the cutoff in $z_{\eta^\perp}$ in $V^{N,\delta,y}$, $\Psi^{N,\delta,y}_t(a) =a$ for all $|a_{\eta^\perp}| \geq 1/4$ and $t \geq 0$. Thus
    \[\|I^{j,N,\delta}\|_{L^p_\omega \mathcal{C}^0_a \mathcal{C}^\gamma_t \mathcal{C}^1_z} \leq \|I^{j,N,\delta}\|_{L^p_\omega \mathcal{C}^0_a(B_{1/2}) \mathcal{C}^\gamma_t \mathcal{C}^1_z}.\]
    Let $\indc_{B_{1/2}} \leq \psi \leq \indc_{B_1}$ and $\tilde I^{j,N,\delta}(a,\cdot,\cdot):= \psi(a)I^{j,N,\delta}(a,\cdot,\cdot)$. We note Proposition~\ref{prop:effective Lipschitz Gronwall} gives that, pointwise in $\omega$ and $a$, $\|\nabla \Psi^{N,\delta,y}(a)\|_{\mathcal{C}^0_t}\leq M^N(a),$ thus we have the estimate
    \[\|\nabla\Psi^{N,\delta,y}\|_{\mathcal{C}^0_a L^p_\omega\mathcal{C}^0_t}\leq \|M^N\|_{\mathcal{C}^0_a L^p_\omega}\leq C.\]
    Then~\eqref{eq:Ij deriv bound auto} implies that
    \[\|\nabla_a I^{j,N,\delta}\|_{\mathcal{C}^0_a L^p_\omega\mathcal{C}^1_t \mathcal{C}^1_z}\leq C 2^{2j}.\]
    Then by the above and~\eqref{eq:Ij bound auto}, we have that for all $p \geq 1$,
    \begin{align}
            \|\tilde I^{j,N,\delta}\|_{\mathcal{C}^1_a L^p_\omega \mathcal{C}^\gamma_t \mathcal{C}^1_z}&\leq C 2^{2j},\notag\\
            \|\tilde I^{j,N,\delta}\|_{\mathcal{C}^0_a L^p_\omega \mathcal{C}^\gamma_t \mathcal{C}^1_z}&\leq C 2^{-C^{-1}j}.
            \label{eq:tilde Ij pointwise}
    \end{align}
    Interpolating using Lemma~\ref{lem:iterated holder interpolation}, then using Lemma~\ref{lem:alpha-holder_embedding} and Lemma~\ref{lem:sobolev}---which we can do even though $a \in \R^d$ instead of $a \in \T^d$ by periodizing, using that the $\tilde I^{j,N,\delta}$ vanishes outside of $B_1$---we get for $p$ large enough, 
    \[\|\tilde I^{j,N,\delta}\|_{L^p_\omega \mathcal{C}^{C^{-1}}_a \mathcal{C}^\gamma_t \mathcal{C}^1_z} \leq \|\tilde I^{j,N,\delta}\|_{L^p_\omega W^{C^{-1}, p}_a \mathcal{C}^\gamma_t \mathcal{C}^1_z} \leq C 2^{-C^{-1}j}.\]
    We can then use the pointwise control given by~\eqref{eq:tilde Ij pointwise} to pass from $\mathcal{C}^{C^{-1}}_a$ to $\mathcal{C}^0_a$, thus giving that 
    \[ \|I^{j,N,\delta}\|_{L^p_\omega \mathcal{C}^0_a \mathcal{C}^\gamma_t \mathcal{C}^1_z} \leq \|I^{j,N,\delta}\|_{L^p_\omega \mathcal{C}^0_a(B_{1/2}) \mathcal{C}^\gamma_t \mathcal{C}^1_z} \leq \|\tilde I^{j,N,\delta}\|_{L^p_\omega \mathcal{C}^0_a \mathcal{C}^\gamma_t \mathcal{C}^1_z} \leq C 2^{-C^{-1} j},\]
    allowing us to conclude~\eqref{eq:I j bounded}. Then summing over $j$, we have that
    \[\|\iop^{N,\delta,y} V^{N,\delta,y}\|_{L^p_\omega \mathcal{C}^0_a\mathcal{C}^\gamma_t \mathcal{C}^1_z}\leq C.\]
    We thus have that 
    \[\P(\|M^N\|_{\mathcal{C}^0_a}\geq t)\leq \P(\|\iop^{N,\delta,y} V^{N,\delta,y}\|_{\mathcal{C}^0_a\mathcal{C}^\gamma_t \mathcal{C}^1_z}\geq C^{-1}(\log t)^\gamma-C)\leq C(\log(t))^{-\gamma p},\]
    for all sufficiently large $t$ by Chebyshev's inequality. Fatou's lemma then implies that
    \[\P\Big(\sup_{a\in \R\times \eta^\perp}M(a)\geq t\Big)\leq C(\log(t))^{-\gamma p},\]
    as well, hence $M=\sup_{a\in\R\times\eta^\perp} M(a)<\infty$ almost surely as claimed.

    Collecting the statements above, we have shown that there exists an almost surely finite random variable $M$ such that for all $a,b\in\R\times \eta^{\perp}$, if $Z^i$ are solutions to~\eqref{eq:V delta ODE} with initial conditions $Z^1=a$ and $Z^2=b$, then for all $t\in[0,1]$
    \[\frac{1}{M}|a-b|\leq |Z^1_t-Z^2_t|\leq M|a-b|.\]
    The upper bound implies that $\Psi^{\delta,y}_{0,t}$ is well-defined for all $t\in[0,1]$ and satisfies
    \[\sup_{t\in[0,1]}\|\nabla \Psi^{\delta,y}_{0,t}\|_{\mathcal{C}^0_a}\leq M.\]
    On the other hand, the lower bound (together with continuity) implies that $\Psi^{\delta,y}_{0,t}$ is bijective and its inverse satisfies
    \[\sup_{t\in[0,1]}\|\nabla \Psi^{\delta,y,-1}_{0,t}\|_{\mathcal{C}^0_a}\leq M.\]
    Together, these imply that $\Psi^{\delta,y}_{s,t}$ is uniquely defined by $\Psi^{\delta,y}_{s,t}:=\Psi^{\delta,y}_{0,t}\circ \Psi^{\delta,y,-1}_{0,s}$, and
    \[\sup_{s,t\in[0,1]}\|\nabla\Psi^{\delta,y}_{s,t}\|_{\mathcal{C}^0_a}\leq M^2.\]
    Since $M<\infty$ almost surely, this concludes the claim.
\end{proof}

Using Proposition~\ref{prop:transformed flow} to get the well-posedness of the transformed and regularized flow, we then transform back to the original coordinates to give a localized form of well-posedness for the original problem on the spatial sets $G_{y,\delta}$ for which we can ensure the cutoffs $\chi,\phi^\delta$ have not ``activated''. 

\begin{proposition}
\label{prop:local lipschitz flow}
     Let $\alpha \in (\frac{1}{2},1)$ and suppose for some $K >0$ and all $j \in \N, n \in \{0,1,2\}$, we have the almost sure bound
\[\|\nabla^n u^j\|_{\mathcal{C}^0_{x}} \leq K 2^{(n-\alpha)j}.\]
    Then for any $y \in \T^d$, $\delta>0$,
    \[G_{y,\delta} := \begin{cases} \Big\{x \in \T^d : |x-y| < 1/8,\, U(y) \cdot (x-y) > 0, \text{ and } \frac{U(y)}{|U(y)|} \cdot U(x) >  2\delta\Big\} & U(y) \ne 0,\\ \emptyset & U(y) =0,\end{cases}\]
    where in the above we mean $x-y \in \R^d$ to be the smallest-norm representative of the equivalence class. Then for all $y \in \T^d$, $\delta>0$, there exists a random constant $M_{y,\delta}>0$ such that almost surely $M_{y,\delta}<\infty$ and for any curves $X^1_t, X^2_t$ solving
    \begin{equation}
        \label{eq:main autonomous ODE}
        \dot X_t = U(X_t),
    \end{equation}
and any $T \in (0,1]$, if for all $t \in [0,T], X^i_t \in G_{y,\delta}$, we have that
    \[\sup_{t \in [0,T]} |X^1_t - X^2_t| \leq M_{y,\delta}|X^1_0 - X^2_0|.\]
\end{proposition}

\begin{proof}
    Fix $y\in \T^d, \delta >0$. Let $T \in (0,1]$ and $X^1,X^2$ be solutions to~\eqref{eq:main autonomous ODE} such that for all $t \in [0,T], i=1,2, X^i_t \in G_{y,\delta}$. We can suppose without loss of generality that $U(y) \ne 0$, as otherwise the statement is vacuous. In the following, we always take $t \in [0,T].$ We let $\eta := \frac{U(y)}{|U(y)|}$ and write
    \[X^i_t = y + r^i_t \eta+ \tilde Z^i_{t},\]
    noting such a decomposition is canonically defined even though $X^i_t \in \T^d$, by taking $(r^i_t,\tilde Z^i_{t}) \in \R \times \eta^\perp$ to be the smallest norm element such that the above equality holds; this will keep $(r^i_t,\tilde Z^i_{t})$ continuous in $t$ by the restriction $|x-y| < 1/8$ in $G_{y,\delta}$. We note that since $\eta \cdot U(x) > 2\delta$ for $x \in G_{y,\delta}$,
    \begin{equation}
            \label{eq:r div lower bound}
                \dot r^i_t = \eta \cdot \dot X^i_t = \eta \cdot U(X^i_t) > 2\delta.
    \end{equation}

    Thus $r^i : [0,T] \to [r^i_0, r^i_T]$ is a bijection, which admits a bijective inverse $\tau^i : [r^i_0, r^i_T] \to [0,T]$ such that
    \begin{equation}
    \label{eq:tau i deriv}
    \dot \tau^i_r = \frac{1}{\dot r^i(\tau^i_r)} = \frac{1}{\eta \cdot U(X^i_{\tau^i_r})}.\end{equation}
    We then define $Z^i : [r^i_0, r^i_T] \to \R \times \eta^\perp$ by 
    \[Z^i_r := \big(\tau^i_r, \tilde Z^i_{\tau^i_r}\big).\]
    Noting that for $r \in [r^i_0, r^i_T]$,
    \[\frac{d}{dr} \tilde Z^i_{\tau^i_r} = \dot \tau^i_r \Pi_{\eta^\perp} U(X^i_{\tau^i_r}) = \frac{\Pi_{\eta^\perp} U(X^i_{\tau^i_r})}{\eta \cdot U(X^i_{\tau^i_r})},\]
    and 
    \[X^i_{\tau^i_r} = y + r^i(\tau^i_r) \eta + \tilde Z^i_{\tau^i_r} = y +  \eta r+ Z^i_{r,\eta^\perp},\]
    we then have
    \begin{equation}
    \label{eq:Z solves transformed equation}
    \dot Z^i_r = \frac{\big(1, \Pi_{\eta^\perp} U( y +  \eta r+ Z^i_{r,\eta^\perp}))}{\eta \cdot U( y +  \eta r+ Z^i_{r,\eta^\perp})}  = V^{\delta,y}(r,Z^i_r),
    \end{equation}
    where for the final equality we use the conditions of $G_{y,\delta}$ to ensure the cutoffs $\chi, \phi^\delta$ in the definition of $V^{\delta,y}$ are ``inactive''.

    Let $\Psi^{\delta,y}$ be as in Proposition~\ref{prop:transformed flow}. Let $\bar M>0$ be the random constant given by
    \[\bar M :=\sup_{s,t \in [0,1]} \|\nabla \Psi^{\delta,y}_{s,t}\|_{\mathcal{C}^0_z},\]
    so that $\bar M<\infty$ almost surely.
    
    We then take the extension of $Z^i_r$ from $[r^i_0, r^i_T]$ to $[0,1]$ by
    \[Z^i_r :=  \Psi^{\delta,y}_{r^i_0,r}(Z^i_{r^i_0}),\]
    noting that this extension agrees with the original definition of $Z^i_r$ on $[r^i_0, r^i_T]$ since the two functions agree at $r = r^i_0$ and by~\eqref{eq:Z solves transformed equation}, both solve~\eqref{eq:V delta ODE}, which has unique solutions by Proposition~\ref{prop:transformed flow}.

    Then we note that for any $r \in [0,1]$, supposing without loss of generality that $r^1_0 \leq r^2_0,$
    \begin{align*} |Z^1_r - Z^2_r| &=  \big| \Psi^{\delta,y}_{r^1_0,r}(Z^1_{r^1_0}) - \Psi^{\delta,y}_{r^2_0,r}(Z^2_{r^2_0})\big| 
    \\&\leq \big| \Psi^{\delta,y}_{r^1_0,r}(Z^1_{r^1_0}) - \Psi^{\delta,y}_{r^1_0,r}(Z^2_{r^2_0})\big| +\big| \Psi^{\delta,y}_{r^1_0,r}(Z^2_{r^2_0}) -  \Psi^{\delta,y}_{r^2_0,r}(Z^2_{r^2_0})\big|
    \\&\leq \bar M|Z^1_{r^1_0} - Z^2_{r^2_0}| + \big| \Psi^{\delta,y}_{r^2_0,r}\circ\Psi^{\delta,y}_{r^1_0,r^2_0}(Z^2_{r^2_0}) -  \Psi^{\delta,y}_{r^2_0,r}(Z^2_{r^2_0})\big|
    \\&\leq  \bar M  |X^1_0 - X^2_0| + \bar M \big|\Psi^{\delta, y}_{r^1_0,r^2_0} (Z^2_{r^2_0}) -Z^2_{r^2_0}\big|
    \\&\leq C \bar M |X^1_0 - X^2_0|,
    \end{align*}
    where for the final line we use that $\Psi^{\delta,y}_{r^1_0,r^2_0}$ just gives the flow of~\eqref{eq:V delta ODE} on $[r^1_0,r^2_0]$, so moves particles at most $\|V^{\delta,y}\|_{\mathcal{C}^0_{t,z}} |r^2_0 - r^1_0| \leq C |X^1_0 - X^2_0|.$

    We note that since $\|U\|_{\mathcal{C}^0_x} \leq C$, $X^i_t,$ and hence $r^i_t,\tilde Z^i_t$, are Lipschitz in $t$ with a deterministic constant. By~\eqref{eq:r div lower bound} and~\eqref{eq:tau i deriv}, we also have that $\tau^i_r$ is Lipschitz in $r$ on $[r^i_0, r^i_T]$ with a deterministic constant. Then for $t \in [0,T]$, noting that $r^1_t, r^2_t \in [0,1]$,
    \begin{align}
        |X^1_t - X^2_t| &\leq |r^1_t - r^2_t|  + |Z^1_{\eta^\perp}(r^1_t) - Z^2_{\eta^\perp}(r^2_t)|
        \notag\\&\leq |r^1_t - r^2_t| +|Z^2_{\eta^\perp}(r^2_t) - Z^2_{\eta^\perp}(r^1_t)| + |Z^1_{\eta^\perp}(r^1_t) - Z^2_{\eta^\perp}(r^1_t)|
        \notag\\&\leq C |r^1_t - r^2_t|+ |Z^1_{\eta^\perp}(r^1_t) - Z^2_{\eta^\perp}(r^1_t)|
        \notag\\&\leq C |r^1_t - r^2_t|+ C \bar M |X^1_0 - X^2_0|,
        \label{eq:X controlled by r and X0}
    \end{align}
    using the above bound on $|Z^1_r - Z^2_r|.$ We now want to bound $|r^1_t - r^2_t|.$ We suppose that $r^1_t \leq r^2_t$; the case $r^1_t \geq r^2_t$ follows symmetrically. If also $r^1_t \leq r^2_0$, then since $\dot r^1_t > 2 \delta$, we have that $t \leq \delta^{-1} |r^1_0 -r^2_0| \leq \delta^{-1} |X^1_0 - X^2_0|$. Thus 
    \[|r^1_t - r^2_t| \leq |r^1_t - r^1_0| + |r^2_t -r^2_0| + |r^1_0 - r^2_0| \leq C t + |X^1_0 - X^2_0| \leq C |X^1_0 - X^2_0|.\]
    Otherwise, $r^2_0 \leq r^1_t \leq r^2_t$. Thus $r^1_t$ is in the domain of $\tau^2$, so we can bound     
    \begin{align}|r^1_t - r^2_t| &= |r^2 \circ \tau^2 (r^1_t) - r^2 \circ \tau^2(r^2_t)|
    \notag\\&\leq C |\tau^2 (r^1_t) - \tau^2(r^2_t)|
   \notag \\&= C |\tau^2(r^1_t) - \tau^1(r^1_t)|
    \notag\\&= C|Z^2_\tau(r^1_t) - Z^1_\tau(r^1_t)|
    \notag\\&\leq C \bar M |X^1_0 - X^2_0|,
    \label{eq:r controlled by X0}
    \end{align}
    where we use that for $t \in [0,T]$, $\tau^1(r^1_t) = t = \tau^2(r^2_t)$ and the bound on $|Z^1_r - Z^2_r|.$ Thus in total, we have shown 
    \[ |X^1_t - X^2_t| \leq C \bar M |X^1_0 - X^2_0|,\]
    so letting $M_{y,\delta} := C\bar M$, which is almost surely finite, we conclude.
\end{proof}

With the localized well-posedness proved, we are now ready to prove Theorem~\ref{thm:main ode intro} by gluing the well-posedness sets using a compactness argument.

\begin{proof}[Proof of Theorem~\ref{thm:main ode intro}]
    We fix $\delta>0$. Let $\theta >0$ such that $\alpha - \theta >1/2$ and $p > \frac{d}{\theta}$, which is possible by the assumptions on $p,\alpha$. Since the result is a qualitative almost sure statement, we see that the general case follows directly from Proposition~\ref{prop:autonomous truncation approximation} and proving the almost sure finiteness of $M_\ep$ under the hypothesis that for some $K >0$ and for all $j \in \N, n \in \{0,1,2\}$, we surely have that
    \[\|\nabla^n u^j\|_{\mathcal{C}^0_{x}} \leq K 2^{(n-\alpha + \theta)j}.\]
    We thus take this stronger hypothesis for the remainder of the argument. Note in particular that we surely have that $U : \T^d \to \R^d$ is continuous and that surely $\|U\|_{\mathcal{C}^0_x} + \|U\|_{\mathcal{C}^{\alpha/2}_x} \leq \bar K < \infty$.

    We define the random compact set
    \[S_\delta := \{x \in \T^d: |U(x)| \geq \delta/2\}.\]
    For $y \in \T^d$, we let $G_{y,\delta/6}, M_{y, \delta/6}$ be as in Proposition~\ref{prop:local lipschitz flow}. We note then that---by the continuity of $U$---$(G_{y,\delta/6})_{y \in \Q^d \cap \T^d}$ is an open cover for $S_\delta$, thus we have a finite subcover; i.e.\ there are $y_1,...,y_n$ such that
    \[\bigcup_{j=1}^n G_{y_j,\delta/6} \supseteq S_\delta.\]
    By the Lebesgue number lemma, there then exists $\ep>0$ such that for any $A \subseteq S_\delta$, if $\mathrm{diam}\,A \leq \ep$, then there exists $j$ such that $A \subseteq G_{y_j,\delta/6}.$ We suppose without loss of generality that $\ep^{\alpha/2} \leq \frac{\delta}{2\bar K}$. 

    Let then $M_{*,\delta/6} := \max_{1 \leq j \leq n} M_{y_j,\delta/6} \lor 1$ and then 
    \[\bar M_\delta := (4 \ep^{-1} \sqrt{d} + 1) (2 M_{*,\delta/6})^{\lceil 4 \ep^{-1} \bar K\rceil}.\]
    We note that $\bar M_\delta<\infty$ almost surely. Let $X^1,X^2$ be solutions to~\eqref{eq:main autonomous ODE} and $T \in (0,1]$ such that for all $t \in [0,T]$, $|U(X^1_t)|, |U(X^2_t)| \geq \delta$. We first prove that for all $t \in [0,T]$,
    \[|X^1_t - X^2_t| \leq \bar M_\delta |X^1_0 - X^2_0|.\]

    We let then
    \[ \gamma := 2^{-2}\ep(2 M_{*,\delta/6})^{-\lceil 4 \ep^{-1} \bar K\rceil}\]
    and split into cases according to whether $|X^1_0 -X^2_0| \leq \gamma$ or $|X^1_0 - X^2_0| > \gamma$. If $|X^1_0 - X^2_0| > \gamma$, using simply the bounded diameter of $\T^d$, we have that
    \[\sup_{t \in [0,T]} |X^1_t - X^2_t| \leq \sqrt{d} \leq \sqrt{d} \gamma^{-1} |X^1_0 - X^2_0| \leq \bar M_\delta |X^1_0 - X^2_0|. \]
    
    Thus we can suppose $|X^1_0 -X^2_0| \leq \gamma$. Split $[0,T]$ into $\lceil 4\ep^{-1} \bar K\rceil =: N$ many intervals of equal size, so we have $0 = t_0 < t_1 < \cdots < t_{N-1} < t_N = T$ with $t_{j} - t_{j-1} \leq \frac{\ep}{4 \bar K}$. We then let
    \[A_j := \{x \in \T^d : \exists t_{j-1} \leq s \leq t_j, |x - X^1_s| \leq \ep/4\}.\]
    Note by $\|U\|_{\mathcal{C}^0_x} \leq \bar K$ and the bound on $t_j - t_{j-1}$, we have that $\mathrm{diam}\, A_j \leq \ep$. We note also that for any $x \in A_j$, we have for some $s \in [0,T]$ such that $|x - X^1_s| \leq \ep/4$, thus
    \[|U(x)| \geq |U(X^1_s)| - |U(x) - U(X^1_s)| \geq \delta - \|U\|_{\mathcal{C}^{\alpha/2}_x} |x - X^1_s|^{\alpha/2} \geq \delta - \bar K \ep^{\alpha/2} \geq \delta/2.\]
    Therefore we also have that $A_j \subseteq S_\delta$. Thus, by the construction of $\ep$, we have that there exists for each $1 \leq j \leq N$, $\ell_j \in \{1,...,n\}$ such that $A_j \subseteq G_{y_{\ell_j}, \delta/6}$.

    We claim inductively that
    \[\sup_{t \in [0,t_j]} |X^1_t - X^2_t| \leq 2^j M_{*,\delta/6}^j |X^1_0 - X^2_0|.\]
    The base case of $j=0$ follows trivially. We then let $1 \leq j \leq N$ and suppose that $|X^1_{t_{j-1}} - X^2_{t_{j-1}}| \leq 2^{j-1}M_{*,\delta/6}^{j-1} |X^1_0 - X^2_0|$. Suppose for the sake of contradiction there exists some time $s \in [t_{j-1}, t_j]$ such that $|X^1_s - X^2_s|\geq 2^j M_{*,\delta/6}^{j} |X^1_0 - X^2_0|$; we can without loss of generality let $s$ be the first such time. Then for $t \in [t_{j-1}, s]$, we have that 
    \[|X^1_t - X^2_t| \leq  2^j M^j_{*,\delta/6}|X^1_0 - X^2_0|\leq  2^N M^N_{*,\delta/6} \gamma  \leq \ep/4.\]
    Thus $X^1_t, X^2_t \in A_j \subseteq G_{y_{\ell_j},\delta/6}$. Thus by the definition of $G_{y,\delta/6}, M_{y,\delta/6}$, and Proposition~\ref{prop:local lipschitz flow}---using that $\alpha -\theta>1/2$---we have that 
    \[\sup_{t \in [t_{j-1},s]} |X^1_t -X^2_t|\leq M_{y_{\ell_j}, \delta/6} |X^1_{t_{j-1}} -X^2_{t_{j-1}}| \leq M_{*,\delta/6} 2^{j-1} M_{*,\delta/6}^{j-1} |X^1_0 - X^2_0| < 2^j M_{*,\delta/6}^j |X^1_0 - X^2_0|,\]
    thus giving the desired contradiction.

    Thus we get the claim, and hence have that
    \[\sup_{t \in [0,T]} |X^1_t - X^2_t| \leq 2^N M_{*,\delta/6}^N |X^1_0 - X^2_0| \leq \bar M_\delta |X^1_0 - X^2_0|,\]
    which is the second inequality of the item. To conclude, we now show the first inequality.

    We note that $\tilde U(x):=-U(x)$ satisfies the same conditions as $U$. Thus, by the above argument, there exists an almost surely finite $\tilde M_\delta$ such that if $\tilde X^1$, $\tilde X^2$ are any two solutions to the ODE
    \begin{equation}\label{eq:reversed ODE}
    \dot{\tilde{X}}_t=\tilde U(\tilde X_t)
    \end{equation}
    and $\tilde{T}\in(0,1]$ is such that for all $s\in[0,\tilde T]$, $|\tilde U(\tilde X^1_s)|,|\tilde U(\tilde X^2_s)|\geq \delta$, then 
    \[\sup_{s\in[0,\tilde{T}]}|\tilde X^1_s-\tilde X^2_s|\leq \tilde M_{\delta}|\tilde X^1_0-\tilde X^2_0|.\]
    For any fixed $t\in[0,T]$, letting $\tilde X^i$ be defined by $\tilde X^i_s=X^i_{t-s}$, then $\tilde X^i$ solves~\eqref{eq:reversed ODE}, and for all $s\in[0,t]$, $|\tilde U(\tilde X^i_s)|=|U(X^i_{t-s})|\geq \delta$. We thus find that
    \[|X^1_0-X^2_0|=|\tilde X^1_t-\tilde X^2_t|\leq \tilde M_\delta|\tilde X^1_0-\tilde X^2_0|=\tilde M_\delta|X^1_t-X^2_t|.\]  
    Letting $M_\delta=\max(\bar M_\delta, \tilde M_\delta)$ and rearranging, we have in total shown that for all $t\in[0,T],$
    \[M_\delta^{-1}|X^1_0-X^2_0|\leq |X^1_t-X^2_t|\leq M_\delta |X^1_0-X^2_0|.\]

    We have thus shown that for all $\delta>0$, there is an almost surely finite $M_\delta>0$ such that for any solution curves $X^1, X^2$ living in $S_{2\delta}$, we have the desired stability/compression estimate. In order to conclude, we want to replace $S_{2\delta}$ with $\T^d \backslash Z_\ep$ with $Z_\ep$ as defined in the theorem statement. This however follows by the continuity of $U$ and a simple compactness argument. Thus we conclude the result.
\end{proof}

We now turn our attention to the proof of Theorem~\ref{thm:sweeping Bihari}. We need the following estimates on the averaged velocity field---which are analogous to those of Proposition~\ref{prop:refreshing stochastic bound holder}. The proof is also deferred to Section~\ref{ss:effective regularity autonomous}---and is mostly completed simultaneously with the proof of Proposition~\ref{prop:fixed p sweeping averaging}.

\begin{proposition}
\label{prop:autonomous bihari estimate}
   Let $\alpha \in (0,\frac{1}{2}]$, $\delta>0$, and $V^{N,\delta,y}(t,z)$ defined by~\eqref{eq:V delta N def}. If for some $K >0$ and all $j \in \N, n \in \{0,1,2\}$, we have the almost sure bound
    \[\|\nabla^n u^j\|_{\mathcal{C}^0_{x}} \leq K 2^{(n-\alpha)j}.\]
    Then for all $\ep>0$ and $p\geq 1$, there exists $C(d,\alpha,K,p, \ep,\delta)>0$ such that for all $y \in \T^d, N \in \N$,
    \[\|\iop^{N,\delta,y} V^{N,\delta,y}\|_{\mathcal{C}^0_a L^{p}_\omega \mathcal{C}^{1-\alpha-\ep}_t \mathcal{C}^{2\alpha - \ep}_z} \leq C.\]
\end{proposition}

With these estimates in hand, we are ready to prove Theorem~\ref{thm:sweeping Bihari}. We prove the desired stability for the transformed problem using Proposition~\ref{prop:effective Holder Bihari}. We then have to transform back to the original domain to conclude. This transformation is easier than for Theorem~\ref{thm:main ode intro}, since we only need a local estimate near time $0$ as opposed to the more global estimate of Theorem~\ref{thm:main ode intro}. As such, we do not need to ``glue'' together many coordinate charts.

\begin{proof}[Proof of Theorem~\ref{thm:sweeping Bihari}]

We note that---since the statement is a qualitative almost sure result---the general case of the $p$-moment hypothesis of the theorem follows from the special case where we have the sure bound, for some $K \geq 1$, for all $j \in \N, n \in \{0,1,2\},$
\begin{equation}
\label{eq:autonomous bihari a.s. bound}
\|\nabla^n u^j\|_{\mathcal{C}^0_x} \leq K 2^{(n-\alpha)j},\end{equation}
by using Proposition~\ref{prop:autonomous truncation approximation} and that $\ep$ can be taken arbitrarily small as $p \to\infty$. We suppose for the remainder that we have the bound~\eqref{eq:autonomous bihari a.s. bound}. We note that in particular we have that $\|U\|_{\mathcal{C}^0_x}, \|U\|_{\mathcal{C}^{\alpha/2}_x} \leq C$.

We fix $y \in \T^d$. We note also that it suffices to prove for all $\delta>0$, there exists an almost surely finite $M_\delta$ such that if $|U(y)| \geq 3\delta$, then for all $t \in [0,1]$,~\eqref{eq:autonomous quasi stability} holds. Taking $\delta \to 0$ recovers the general result, given the hypothesis that $U(y) \ne 0.$

As such, we fix $\delta>0$ and let $\gamma=1-\alpha-\eps$ and $\nu=2\alpha-\eps$ where $\ep>0$ is sufficiently small such that $\gamma(1+\nu)>1$ (which is possible as $\alpha \in (0,1/2)$). We work on the event $|U(y)| \geq 3\delta$ for the remainder.

We then define
\[M_\delta := \liminf_{N\to \infty} \|\iop^{N, \delta,y} V^{N,\delta,y}(0,\cdot,\cdot)\|_{\mathcal{C}^\gamma_t \mathcal{C}^\nu_x}^{\frac{1}{1-\nu}} + 1.\]
By Proposition~\ref{prop:autonomous bihari estimate} and Fatou's lemma, $M_\delta<\infty$ a.s.

Let then $X^1, X^2$ be solutions to~\eqref{eq:main ODE autonomous data}. As in the proof of Proposition~\ref{prop:local lipschitz flow}, we let $\eta := \frac{U(y)}{|U(y)|},$ and we can canonically decompose
\[X^i_t = y + r^i_t \eta + \tilde Z^i_t.\]
Since $|U(y)| \geq 3 \delta$ and $\|U\|_{\mathcal{C}^0_x}, \|U\|_{\mathcal{C}^{\alpha/2}_x} \leq C$, we have for some $t_\delta >0$ (independent of the choice of $X^1,X^2$), for all $t \in [0,t_\delta]$, 
\[\dot r^i_t = \eta \cdot U(X^i_t) >2\delta \quad \text{and}\quad |X^i_t - y| <1/8.\]
Thus, exactly as in the proof of Proposition~\ref{prop:local lipschitz flow}, noting that $r^i_{t_\delta} \geq 2\delta t_\delta =: r_\delta$, we have that the $r^i: [0,t_\delta] \to [0,r^i_{t_\delta}]$ admit bijective inverses $\tau^i: [0,r^i_{t_\delta}] \to [0,t_\delta]$ such that if we define $Z^i_r : [0,r_\delta] \to \R \times \eta^\perp$ by
\[Z^i_r := (\tau^i_r, \tilde Z^i_{\tau^i_r}),\]
we have that
\[\begin{cases}\dot Z^i_r = V^{\delta,y}(r,Z^i_r),\\
Z^i_0= 0.\end{cases}\]
We now claim that for all $r \in [0,r_\delta]$, 
\begin{equation}
    \label{eq:bihari stability for Z}
    |Z^1_r - Z^2_r| \leq M_\delta r^{\frac{1-\alpha}{1-2\alpha} -\ep}.
\end{equation}

We first see that~\eqref{eq:bihari stability for Z} suffices to conclude. Following exactly~\eqref{eq:X controlled by r and X0} and~\eqref{eq:r controlled by X0} in the proof of Proposition~\ref{prop:local lipschitz flow}, we have from~\eqref{eq:bihari stability for Z} that for all $t \in [0,C^{-1}]$
\[|X^1_t - X^2_t| \leq C M_\delta t^{\frac{1-\alpha}{1-2\alpha} -\ep}.\]
Using that $\frac{d}{dt} |X^1_t - X^2_t| \leq 2 \|U\|_{\mathcal{C}^0_x} \leq C$, increasing $C$ somewhat we then get the estimate for all $t \in [0,1]$. We thus conclude the desired estimate from~\eqref{eq:bihari stability for Z}, after harmlessly redefining $M_\delta$ to include the deterministic constant.

Thus, to conclude, we just need to show~\eqref{eq:bihari stability for Z}. To that end, for each $N \in \N$ we view $Z^1, Z^2$ as forced solutions to the truncated equation~\eqref{eq:V delta N flow} and apply Proposition~\ref{prop:effective Holder Bihari}, giving that for all $r \in [0,r_\delta]$
\begin{align}
\label{eq:Z bihari stability}
|Z^1_r - Z^2_r| &\leq C (|Z^1_r - \Psi^{N,\delta,y}_r(0)| + |Z^2_r - \Psi^{N,\delta,y}_r(0)|) 
\notag\\&\leq C \big( \|w^1\|_{\mathcal{C}^\gamma_r} +\|w^2\|_{\mathcal{C}^\gamma_r}  + \|\iop^{N,\delta,y} V^{N,\delta,y}(0,\cdot,\cdot)\|_{\mathcal{C}^\gamma_r \mathcal{C}^\nu_z}^{\frac{1}{1-\nu}} r^{\frac{\gamma}{1-\nu}}\big),
\end{align}
where
\[w^i_r := \int_0^r (V^{\delta,y} - V^{N,\delta,y})(s,Z^i_s)\,ds,\]
thus 
\[ \|w^1\|_{\mathcal{C}^\gamma_r} +\|w^2\|_{\mathcal{C}^\gamma_r}  \leq 2\|V^{\delta,y} - V^{N,\delta,y}\|_{\mathcal{C}^0_{r,z}}.\]
Then we note that, as in~\eqref{eq:V N to zero} of the proof of Proposition~\ref{prop:transformed flow}, we almost surely have that $\lim_{N \to \infty} \|V^{\delta,y} - V^{N,\delta,y}\|_{\mathcal{C}^0_{r,z}} =0$.

Thus taking $\liminf_N$ of both sides of~\eqref{eq:Z bihari stability} and using the definition of $M_\delta$, we see that for all $r \in [0,r_\delta]$, 
\[|Z^1_r - Z^2_r| \leq CM_\delta r^{\frac{\gamma}{1-\nu}}.\]
Plugging in the definition of $\gamma,\nu$, and perhaps taking $\ep>0$ smaller, we get the desired exponent $\frac{1-\alpha}{1-2\alpha} -\ep.$ Finally, after harmlessly redefining $M_\delta$, this then is~\eqref{eq:bihari stability for Z}, thus concluding the proof.
\end{proof}

\subsection{Proof of Proposition~\ref{prop:fixed p sweeping averaging} and Proposition~\ref{prop:autonomous bihari estimate}: stochastic averaging by sweeping}
\label{ss:effective regularity autonomous}

In this section we prove Proposition~\ref{prop:fixed p sweeping averaging} and Proposition~\ref{prop:autonomous bihari estimate}. We fix $\delta>0$, and $K>0$ such that for all $j \in \N, n \in \{0,1,2\}$, we have the almost sure bound
    \[\|\nabla^n u^j\|_{\mathcal{C}^0_{x}} \leq K 2^{(n-\alpha)j}.\] 
    We let all implicit constants depend on $K,\delta$ (as well as $\alpha,d$). We also fix $N \in \N, y \in \T^d$, but all constants are \textit{uniform in $N,y$}. We will suppress superscripts on $N,\delta,y$ on $V^{N,\delta,y}, \iop^{N,\delta,y}, \phi^\delta$, etc.

\subsubsection{\texorpdfstring{Decomposing $V$}{Decomposing V}}

We note the following direct computation. Let $a^j,b^j$ be sequences, and denote  $A^j := \sum_{k=0}^j a^k, B^j := \sum_{k=0}^j b^k$. Suppose $B^j >0$ for $j\geq 0$, then for all $N,$
\begin{equation}
\label{eq:ratio decomposition}
\frac{A^N}{B^N} = \frac{a^0}{b^0} + \sum_{j=1}^N \Big(\frac{a^j}{B^{j-1}} - A^{j-1} \frac{b^j}{(B^{j-1})^2} - \frac{a^j b^j}{B^j B^{j-1}} + A^{j-1} \frac{(b^j)^2}{B^j (B^{j-1})^2}\Big).
\end{equation}

We recall that
\[V (t,z) = \chi(t,z_{\eta^\perp})\frac{\big(1,\Pi_{\eta^\perp} U^N(y+ \eta t + z_{\eta^\perp})\big)}{\phi\big(\eta \cdot U^N(y +\eta t + z_{\eta^\perp})\big)},\]
so letting 
\[A^j := \chi(t,z_{\eta^\perp})\big(1, \Pi_{\eta^\perp} U^j(y + \eta t + z_{\eta^\perp})\big) \quad \text{and}\quad B^j := \phi \big(\eta \cdot U^j(y + \eta t + z_{\eta^\perp}) \big),\]
defining then $a^j := A^j - A^{j-1}$ with $a^0 = A^0$, and similarly for $b^j$. Then we get that $V(t,z)$ is equal to:
\begin{align}
    &\chi(t,z_{\eta^\perp})\bigg(\sum_{j=1}^N  \frac{\big(0,\Pi_{\eta^\perp} u^j(y + \eta t + z_{\eta^\perp})\big)}{\phi \big(\eta \cdot U^{j-1}(y + \eta t + z_{\eta^\perp}) \big)} 
    \notag\\&\quad- \sum_{j=1}^N    \frac{\big(1,\Pi_{\eta^\perp} U^{j-1}(y + \eta t + z_{\eta^\perp})\big) \partial_r \phi\big(\eta \cdot U^{j-1}(y + \eta t + z_{\eta^\perp}) \big)}{\big(\phi \big(\eta \cdot U^{j-1}(y + \eta t + z_{\eta^\perp}) \big)\big)^2}  \eta \cdot u^j(y + \eta t + z_{\eta^\perp})
    \notag\\&\quad - \sum_{j=1}^N  \frac{\big(1,\Pi_{\eta^\perp} U^{j-1}(y + \eta t + z_{\eta^\perp})\big) }{\big(\phi \big(\eta \cdot U^{j-1}(y + \eta t + z_{\eta^\perp}) \big)\big)^2} \Big(\phi \big(\eta \cdot U^j(y + \eta t + z_{\eta^\perp}) \big) - \phi \big(\eta \cdot U^{j-1}(y + \eta t + z_{\eta^\perp}) \big)
    \notag\\&\qquad\qquad\qquad\qquad \qquad\qquad\qquad\qquad \qquad\qquad- \partial_r \phi\big(\eta \cdot U^{j-1}(y + \eta t + z_{\eta^\perp}) \big) \eta \cdot u^j(y + \eta t + z_{\eta^\perp}) \Big)
    \notag\\&\quad - \sum_{j=1}^N  \frac{\big(0, \Pi_{\eta^\perp} u^j(y + \eta t + z_{\eta^\perp})\big) \big(\phi \big(\eta \cdot U^j(y + \eta t + z_{\eta^\perp}) \big) - \phi \big(\eta \cdot U^{j-1}(y + \eta t + z_{\eta^\perp})\big) }{\phi \big(\eta \cdot U^{j}(y + \eta t + z_{\eta^\perp}) \big)\phi \big(\eta \cdot U^{j-1}(y + \eta t + z_{\eta^\perp}) \big)}
   \notag \\&\quad + \sum_{j=1}^N    \frac{\big(1,  \Pi_{\eta^\perp} U^{j-1}(y + \eta t + z_{\eta^\perp})\big)\big(\phi \big(\eta \cdot U^j(y + \eta t + z_{\eta^\perp}) \big) - \phi \big(\eta \cdot U^{j-1}(y + \eta t + z_{\eta^\perp}) \big)\big)^2}{\phi \big(\eta \cdot U^j(y + \eta t + z_{\eta^\perp}) \big) \big(\phi \big(\eta \cdot U^{j-1}(y + \eta t + z_{\eta^\perp}) \big)\big)^2}
   \notag \\&\quad +   \frac{ \big(1,\Pi_{\eta^\perp} U^0(y + \eta t + z_{\eta^\perp})\big)}{\phi \big(\eta \cdot U^0(y + \eta t + z_{\eta^\perp}) \big)}\bigg). 
    \label{eq:V big decomp}
\end{align}
We then define
\begin{align*} v^{j}(t,z) &= \chi(t,z_{\eta^\perp})\bigg(\frac{\big(0,\Pi_{\eta^\perp} u^j(y + \eta t + z_{\eta^\perp})\big)}{\phi \big(\eta \cdot U^{j-1}(y + \eta t + z_{\eta^\perp}) \big)} 
    \\&\qquad-  \frac{\big(1,\Pi_{\eta^\perp} U^{j-1}(y + \eta t + z_{\eta^\perp})\big) \partial_r \phi\big(\eta \cdot U^{j-1}(y + \eta t + z_{\eta^\perp}) \big)}{\big(\phi \big(\eta \cdot U^{j-1}(y + \eta t + z_{\eta^\perp}) \big)\big)^2}  \eta \cdot u^j(y + \eta t + z_{\eta^\perp})\bigg),
\end{align*}
and define the four remaining terms of~\eqref{eq:V big decomp} as $R(t,z) := R^{1}(t,z) + R^{2}(t,z)+ R^{3}(t,z)+ R^{4}(t,z).$
We further decompose for $i=1,2,3:$
\[R^{i} = \sum_{j=1}^N R^{i,j}.\]
The $v^j$ will act similarly to the $u^j$ in Section~\ref{s:refreshing}; we will use Lemma~\ref{lem:abstract averaging} to gain additional regularity through stochastic cancellation. The $R^j$ are more directly ``error terms'', which are estimated solely using the bare regularity of the $u^j$.

\subsubsection{\texorpdfstring{Controlling the remainder $R$}{Controlling the remainder R}}

We now almost surely bound the $\mathcal{C}^0_t \mathcal{C}^0_z$ and ${\mathcal{C}^0_t \mathcal{C}^1_z}$ norms of $ R^{1,j}, R^{2,j}, R^{3,j},$ and $R^{4}$. We first have that
\begin{align*}
    \|R^{1,j}\|_{\mathcal{C}^0_{t,z}} &\leq C\Big\|\phi \big(\eta \cdot U^j \big) - \phi \big(\eta \cdot U^{j-1} \big)- \partial_r \phi\big(\eta \cdot U^{j-1} \big) \eta \cdot u^j\Big\|_{\mathcal{C}^0_x}
    \\&\leq C \|u^j\|_{\mathcal{C}^0_{x}}^2 
    \\&\leq C  2^{-2\alpha j},
\end{align*}
and
\begin{align*}
    \|R^{1,j}\|_{\mathcal{C}^0_t \mathcal{C}^1_z} &\leq C\Big\|\phi \big(\eta \cdot U^j \big) - \phi \big(\eta \cdot U^{j-1} \big)- \partial_r \phi\big(\eta \cdot U^{j-1} \big) \eta \cdot u^j\Big\|_{\mathcal{C}^1_x}
    \\&\qquad + \Big\|\frac{\chi(t,z) \Pi_{\eta^\perp} U^{j-1}(y + \eta t + z) }{\big(\phi \big(\eta \cdot U^{j-1}(y + \eta t + z) \big)\big)^2}\Big\|_{\mathcal{C}^0_t\mathcal{C}^1_z} 
    \\&\qquad\qquad\qquad\times \Big\|\phi \big(\eta \cdot U^j \big) - \phi \big(\eta \cdot U^{j-1} \big)- \partial_r \phi\big(\eta \cdot U^{j-1} \big) \eta \cdot u^j\Big\|_{\mathcal{C}^0_x}
\end{align*}
We note that for the second term, the first factor is bounded by $C2^{(1-\alpha)j}$ while the second factor is the same in the $\mathcal{C}^0_{t,z}$ bound, and so is controlled by $C2^{-2\alpha j}$. Thus the second term is bounded by $C 2^{(1-3\alpha)j}$. Letting $f = \eta \cdot U^j$, $g = \eta \cdot U^{j-1}$, the first term can be written as
\begin{align*}&\big\| \nabla \big(\phi \circ f - \phi \circ g - \phi' \circ g (f-g)\big)\big\|_{ \mathcal{C}^0_x} 
\\&\qquad=\big\| \nabla f \phi' \circ f -  \nabla g \phi' \circ g -\nabla g \phi'' \circ g (f-g) - \phi' \circ g (\nabla f-\nabla g)\big\|_{ \mathcal{C}^0_x} 
\\&\qquad\leq\big\| \nabla f \big(\phi' \circ f - \phi' \circ g - \phi'' \circ g (f-g)\big)\big\|_{ \mathcal{C}^0_x} +\big\|(\nabla f-\nabla g) \phi'' \circ g (f-g)\big\|_{\mathcal{C}^0_x} 
\\&\qquad\leq C \|\nabla U^j\|_{\mathcal{C}^0_x} \|u^j\|_{\mathcal{C}^0_x}^2 + \|\nabla u^j\|_{\mathcal{C}^0_x} \|u^j\|_{ \mathcal{C}^0_x}
\\&\qquad\leq C 2^{(1-3\alpha)j} + C 2^{(1-2\alpha)j}.
\end{align*}
Thus together we have 
\[
 \|R^{1,j}\|_{\mathcal{C}^0_t \mathcal{C}^1_z}  \leq C 2^{(1-2\alpha)j}.
\]
Interpolating the bound on $ \|R^{1,j}\|_{\mathcal{C}^0_t \mathcal{C}^0_z}$ and on $ \|R^{1,j}\|_{ \mathcal{C}^0_t \mathcal{C}^1_z}$, we get that
\begin{equation}
\label{eq:R 1 j bound}
\|R^{1,j}\|_{ \mathcal{C}^0_t \mathcal{C}^\nu_z} \leq C 2^{(\nu-2\alpha)j}.
\end{equation}

For $R^{2,j}$ and $R^{3,j}$, we proceed similarly but more straightforwardly. A direct computation verifies that
\[\|R^{2,j}\|_{\mathcal{C}^0_{t,z}} + \|R^{3,j}\|_{ \mathcal{C}^0_{t,z}} \leq C 2^{-2\alpha j},\]
and 
\[\|R^{2,j}\|_{ \mathcal{C}^0_t \mathcal{C}^1_z} + \|R^{3,j}\|_{ \mathcal{C}^0_t \mathcal{C}^1_z}  \leq C 2^{(1-2\alpha) j}.\]
Interpolating these bounds, we get that
\begin{equation}\label{eq:R 2 3 j bound}
\|R^{2,j}\|_{\mathcal{C}^0_t \mathcal{C}^\nu_z} + \|R^{3,j}\|_{\mathcal{C}^0_t \mathcal{C}^\nu_z}
\leq C 2^{(\nu-2\alpha)j}.
\end{equation}
Finally, $R^{4}$ straightforwardly satisfies $\|R^{4}\|_{\mathcal{C}^0_t \mathcal{C}^\nu_z} \leq C$. Thus combining this with~\eqref{eq:R 1 j bound} and~\eqref{eq:R 2 3 j bound} and summing over $j$, we have for any $\nu \in [0,2\alpha) \cap [0,1]$, there exists $C>0$ (depending on $\nu$) such that
\begin{equation}
    \label{eq:remainder bound autonomous}
    \|R\|_{\mathcal{C}^0_t \mathcal{C}^\nu_z} \leq C.
\end{equation}

\subsubsection{\texorpdfstring{Controlling averaged fields: $\iop v^{j}$}{Controlling averaged fields: iop v j}}

With the remainder term handled, we turn our attention to controlling $\|\iop v^{j}\|_{\mathcal{C}^0_a L^p_\omega \mathcal{C}^\gamma_t \mathcal{C}^\nu_z}$ with $j \geq 1$, for which we will use Lemma~\ref{lem:abstract averaging}. We fix $j \geq 1$. For applying Lemma~\ref{lem:abstract averaging}, we will work conditionally on the fields $(u^k)_{k \ne j}$ as well as $u^j(y)$ and work only in the remaining $u^j$ probability, which we denote $\tilde \P$. We denote the norms with respect to this probability measure by $L^q_{\tilde \omega}$ and the (conditional) expectations by $\tilde \E$. We do this conditioning as $v^j$ isn't finite range in time in the full probability space---the presence of the $U^{j-1}$ field and $\eta$ create long time correlations---but for fixed $(u^k)_{k \ne j}$ and $u^j(y)$, $v^j$ is finite range in time (with range $2^{-j}$, following from Item~\ref{asmp:finite range auto} of Assumption~\ref{asmp:main for autonomous}) for the $\tilde \P$ probability and will also obey the correct centering hypothesis due to the explicit linear in $u^j$ structure of $v^j$. 

To put ourselves in the setting of Lemma~\ref{lem:abstract averaging}, we let $V = V, v = \nabla^i v^j$ for $i =0,1$, and
\begin{align*}
\mathcal{F}_t &:= \sigma\big(u^k : k \ne j\big) \lor \sigma\big(u^j(y+ \eta  s + z) : 0 \leq s \leq t \land 1/4, z \in \eta^\perp \cap B_{1/4}\big),\\
\mathcal{G} &:= \mathcal{F}_0.
\end{align*}

Finally we let $h = 2^{-j}$. The adaptedness of $V,v$ to $\mathcal{F}_t$ is direct by construction (using that $\chi(t,z)$ cuts off the velocity fields outside of $[0,1/4] \times B_{1/4}$); we note also that $\eta$ is $\mathcal{G}$ measurable, which will be useful below. The main hypothesis we need to verify is~\eqref{eq:independent and centered} (noting that we now intend independence and centeredness with respect to only the $u^j$ probability, for any fixed choice of $u^k,k \ne j$ and $u^j(y)$). These however are both direct from the finite range hypothesis: Item~\ref{asmp:finite range auto} of Assumption~\ref{asmp:main for autonomous} (note we use the cutoff function $\chi$ in space and time to make sure we never wrap around the torus, which would induce long range correlations in time). Note also here we use that Lemma~\ref{lem:abstract averaging} only requires centeredness away from $[0,h]$, as the choice of $\eta$ depends on $u^j(y)$ and hence could possibly destroy the centeredness of $u^j$ for a short amount of time; for $t\geq h$ though we have entered an independent region so are centered once again.

Applying Lemma~\ref{lem:abstract averaging}, (using that by~\eqref{eq:M define} and~\eqref{eq:F t define}, the components $M,F$ for $v= \nabla v^j$ are simply the gradient of the components for $v= v^j$) we get that 
\begin{equation}
\label{eq:splitting iop vj into components}
\iop v^j = I^{j,1} + I^{j,2}
\end{equation}
where for $i=0,1$ and $n =0,1,2$ and any $\nu \in (0,1)$,
\begin{align*}
    \|I^{j,1}\|_{\mathcal{C}^0_a\mathcal{C}^{1/2}_t \mathcal{C}^i_z L^{p}_{\tilde \omega}} &\leq C 2^{-j/2} p^{1/2} \|\nabla^i v^j\|_{L^{\infty}_{\tilde{\omega}} \mathcal{C}^0_{t,z}} ,
    \\\|I^{j,2}\|_{\mathcal{C}^0_a\mathcal{C}^{1}_t \mathcal{C}^i_z L^{p}_{\tilde \omega}} &\leq C 2^{-j} \|\nabla^{1+i} v^j\|_{L^\infty_{\tilde{\omega}} \mathcal{C}^0_{t,z}} \Big(\big\|V  - \tilde \E[V  \mid \mathcal{G}]\big\|_{L^\infty_{\tilde \omega} \mathcal{C}^0_{t,z}} \lor 2^{-\frac{\nu}{1-\nu}j} \|V \|^{\frac{1}{1-\nu}}_{L^\infty_{\tilde \omega} \mathcal{C}^0_t \mathcal{C}^\nu_z} \Big),
    \\ \|I^{j,1}\|_{\mathcal{C}^0_aL^p_{\tilde \omega} \mathcal{C}^1_t \mathcal{C}^n_z} &\leq C  \|\nabla^n v^j\|_{L^\infty_{\tilde \omega} \mathcal{C}^0_{t,z}}, 
    \\ \|I^{j,2}\|_{\mathcal{C}^0_aL^p_{\tilde \omega} \mathcal{C}^1_t \mathcal{C}^n_z} &\leq C  \|\nabla^n v^j\|_{L^\infty_{\tilde \omega} \mathcal{C}^0_{t,z}}.
\end{align*}
Taking the $L^{p}_\omega$ norm (that is, integrating over the full probability space now) of both sides of the above inequalities and applying Minkowski's integral inequality, we get for $i=0,1,$ $n =0,1,2$, and $\nu \in (0,1)$,
\begin{align}
    \|I^{j,1}\|_{\mathcal{C}^0_a\mathcal{C}^{1/2}_t \mathcal{C}^i_z L^{p}_{\omega}} &\leq  C 2^{-j/2} p^{1/2} \|\nabla^i v^j\|_{L^{\infty}_{{\omega}} \mathcal{C}^0_{t,z}},
    \notag\\\|I^{j,2}\|_{\mathcal{C}^0_a\mathcal{C}^{1}_t \mathcal{C}^i_z L^{p}_{\omega}} &\leq C 2^{-j} \|\nabla^{1+i} v^j\|_{L^\infty_{\omega} \mathcal{C}^0_{t,z}} \Big(\big\|V  - \tilde \E[V  \mid \mathcal{G}]\big\|_{L^\infty_{\omega} \mathcal{C}^0_{t,z}} \lor 2^{-\frac{\nu}{1-\nu}j} \|V \|^{\frac{1}{1-\nu}}_{L^\infty_{\omega} \mathcal{C}^0_t \mathcal{C}^\nu_z} \Big),
    \notag\\ \|I^{j,1}\|_{\mathcal{C}^0_aL^{p}_{\omega} \mathcal{C}^1_t \mathcal{C}^n_z} &\leq C  \|\nabla^n v^j\|_{L^{\infty}_{\omega} \mathcal{C}^0_{t,z}}, 
   \notag \\ \|I^{j,2}\|_{\mathcal{C}^0_aL^{p}_{ \omega} \mathcal{C}^1_t \mathcal{C}^n_z} &\leq C  \|\nabla^n v^j\|_{L^{\infty}_{\omega} \mathcal{C}^0_{t,z}}.
   \label{eq:autonomous averaged integrated}
\end{align}

We first note that for $n =0,1,2$,
\[\|\nabla^n v^j\|_{ \mathcal{C}^0_{t,z}} \leq C 2^{(n-\alpha)j}.\]

Then we note that for any $\nu <\alpha$, we have that by~\eqref{eq:V big decomp} and~\eqref{eq:remainder bound autonomous}, we have for some $C>0$ (depending on $\nu$),
\[ \|V \|_{\mathcal{C}^0_t \mathcal{C}^\nu_z} \leq C + \sum_{j=1}^N  \|v^j\|_{\mathcal{C}^0_t \mathcal{C}^\nu_z} \leq C + \sum_{j=1}^N C 2^{(\nu -\alpha)j} \leq C.
\]
Note that we can always choose $\nu< \alpha$ so that $\frac{\nu}{1-\nu} > \alpha$, thus there exists $\nu \in (0,1)$ such that 
\[2^{-\frac{\nu}{1-\nu}j} \|V \|^{\frac{1}{1-\nu}}_{L_{\omega}^{\infty} \mathcal{C}^0_t \mathcal{C}^\nu_z} \leq C 2^{-\alpha j}.\]
Finally, we compute
\begin{align*}
   \big\|V  - \tilde \E [V \mid \mathcal{G}] \big\|_{L^{\infty}_{\omega} \mathcal{C}^0_{t,z}} &\leq  \sum_{k=0}^N\Big\|\frac{ (\indc_{k=0},\Pi_{\eta^\perp} u^k) }{\phi \big(\eta \cdot U^N \big)}  -\tilde \E \Big[\frac{ (\indc_{k=0},\Pi_{\eta^\perp} u^k) }{\phi \big(\eta \cdot U^N\big)}\,\Big|\, \mathcal{G}\Big]\Big\|_{L^{\infty}_{\omega} \mathcal{C}^0_x}
    \\& \leq \sum_{k=0}^{j-1}  \Big\| \frac{ (\indc_{k=0},\Pi_{\eta^\perp} u^k)  }{\phi \big(\eta \cdot U^{j-1}\big)}  - \tilde \E \Big[\frac{ (\indc_{k=0},\Pi_{\eta^\perp} u^k) }{\phi \big(\eta \cdot U^{j-1}\big)}\,\Big|\, \mathcal{G}\Big]\Big\|_{L^{\infty}_{\omega} \mathcal{C}^0_x}
    \\&\qquad + 2 \sum_{k=0}^{j-1}  \Big\| \frac{ (\indc_{k=0},\Pi_{\eta^\perp} u^k) }{\phi \big(\eta \cdot U^{N}\big)}  - \frac{(\indc_{k=0},\Pi_{\eta^\perp} u^k) }{\phi \big(\eta \cdot U^{j-1}\big)}\Big\|_{L^{\infty}_{\omega} \mathcal{C}^0_x}
   \\&\qquad + 2\sum_{k \geq j} \Big\| \frac{ \Pi_{\eta^\perp} u^k }{\phi \big(\eta \cdot U^{N}\big)}\Big\|_{L^{\infty}_{\omega} \mathcal{C}^0_x}
   \\&\leq  C \sum_{k=0}^{j-1} (\indc_{k=0}+\|u^k\|_{L^{\infty}_\omega \mathcal{C}^0_x}) \Big\|{\phi \big(\eta \cdot U^{j-1} \big)}- {\phi \big(\eta \cdot U^N \big)} \Big\|_{L^{\infty}_{\omega} \mathcal{C}^0_{x}} + C \sum_{k\geq j} \|u^k\|_{L^{\infty}_\omega \mathcal{C}^0_x}
   \\&\leq C 2^{-\alpha j},
\end{align*}
where for the third inequality we use that for $k < j$,
\[  \tilde \E \Big[\frac{ (\indc_{k=0},\Pi_{\eta^\perp} u^k)}{\phi \big(\eta \cdot U^{j-1}\big)}\,\Big|\, \mathcal{G}\Big] =\frac{ (\indc_{k=0},\Pi_{\eta^\perp} u^k) }{\phi \big(\eta \cdot U^{j-1}\big)}.\]

Therefore, combining the displays between~\eqref{eq:autonomous averaged integrated} and here, we see that
\begin{align}
    \label{eq:I 1 final averaged before interp}
    \|I^{j,1}\|_{\mathcal{C}^0_a\mathcal{C}^{1/2}_t \mathcal{C}^i_z L^{p}_{\omega}} &\leq C p^{1/2} 2^{(i-\alpha -1/2) j},
    \\\|I^{j,2}\|_{\mathcal{C}^0_aL^{p}_{\omega}\mathcal{C}^{1}_t \mathcal{C}^i_z } &\leq C  2^{(i-2\alpha)j},
        \label{eq:I 2 final averaged before interp}
    \\ \|I^{j,1}\|_{\mathcal{C}^0_aL^{p}_{\omega} \mathcal{C}^1_t \mathcal{C}^n_z} &\leq C  2^{(n-\alpha)j}, 
        \label{eq:I 1 final a priori before interp}
    \\ \|I^{j,2}\|_{\mathcal{C}^0_aL^{p}_{ \omega} \mathcal{C}^1_t \mathcal{C}^n_z} &\leq C 2^{(n-\alpha) j}.
        \label{eq:I 2 final a priori before interp}
\end{align}

We now split into cases depending on whether $\alpha >1/2$ or $\alpha \leq 1/2$.

\subsubsection{\texorpdfstring{The effective Lipschitz case (Proposition~\ref{prop:fixed p sweeping averaging}): $\alpha >1/2$}{The effective Lipschitz case: alpha >1/2}}

If $\alpha>1/2$, we let $i=1$ in~\eqref{eq:I 1 final averaged before interp} and~\eqref{eq:I 2 final averaged before interp} and $n=2$ in~\eqref{eq:I 1 final a priori before interp} and~\eqref{eq:I 2 final a priori before interp} to give that (just applying the triangle inequality to~\eqref{eq:splitting iop vj into components})
\begin{align*}
    \|\iop v^j\|_{\mathcal{C}^0_a\mathcal{C}^{1/2}_t \mathcal{C}^1_z L^{p}_\omega} &\leq C p^{1/2} 2^{-(\alpha-1/2)j}\\
    \|\iop v^j\|_{\mathcal{C}^0_a\mathcal{C}^{1}_t \mathcal{C}^2_z L^{p}_\omega} &\leq C 2^{2j}.
\end{align*}

We then interpolate these two bounds under Lemma~\ref{lem:iterated holder interpolation} to give that 
\[ \|\iop v^j\|_{\mathcal{C}^0_a\mathcal{C}^{1/2 + (\alpha-1/2)/8}_t \mathcal{C}^{1+ (\alpha-1/2)/4}_z L^{p}_\omega} \leq C p^{1/2} 2^{-(\alpha-1/2)j/4}.\]
We claim that
\begin{equation}
\label{eq:iop vj final bound}
\|\iop v^j\|_{\mathcal{C}^0_aL^{p}_\omega \mathcal{C}^{1/2 + (\alpha-1/2)/32}_t \mathcal{C}^1_z} \leq C p^{1/2} 2^{-(\alpha-1/2)j/4}.
\end{equation}
We first note that  $\iop v^j$ does not depend on $z_\tau$. We next note that $|\Psi_t(a) -a| \leq C$ by the $\mathcal{C}^0_{t,x}$ norm on $V$, so by the explicit cutoff $\chi$ in the $v^j$, and definition of $\iop$, we have that $\iop v^j$ vanishes for $z_{\eta^\perp} \not \in B_C(-a)$. Thus to prove~\eqref{eq:iop vj final bound}, it suffices to see that for all $a \in \R \times \eta^\perp$,
\begin{equation}
\label{eq:localized estimate iop vj}
\|\iop v^j\|_{L^{p}_\omega \mathcal{C}^{1/2 + (\alpha-1/2)/32}_t \mathcal{C}^1_z(B_C(-a))} \leq C p^{1/2} 2^{-(\alpha-1/2)j/4}.
\end{equation}

We then apply Lemma~\ref{lem:alpha-holder_embedding}---which we can apply even though $z \in \R^d$ instead of $z \in \T^d$ by localizing, using that $\mathcal{I}v^j$ does not depend on $z_\tau$ and we are working on a compact set of $z_{\eta^\perp}$---to get that for all $a,$
\[\|\iop v^j\|_{L^{p}_\omega W^{1/2 + (\alpha-1/2)/16,p}_t W^{1+ (\alpha-1/2)/8,p}_z(B_C(-a))} \leq C p^{1/2} 2^{-(\alpha-1/2)j/4}.\]
Finally we use Lemma~\ref{lem:sobolev}---which we can apply for similar reasons to those of Lemma~\ref{lem:alpha-holder_embedding}---with $p$ large enough, to get~\eqref{eq:localized estimate iop vj}, and hence~\eqref{eq:iop vj final bound}.

For Proposition~\ref{prop:fixed p sweeping averaging}, we define for $j \geq 1$
\[I^{j,N,\delta} := \iop v^j + \iop (R^{1,j} + R^{2,j} + R^{3,j} + \indc_{j=1} R^4),\]
and $I^{0,N,\delta} =0$, so by~\eqref{eq:V big decomp} we have that $\iop V = \sum_{j=0}^N I^{j,N,\delta}$. Then~\eqref{eq:iop vj final bound} together with~\eqref{eq:R 1 j bound}, \eqref{eq:R 2 3 j bound}, and the bound on $R^4$ gives that
\begin{align*}\|I^{j,N,\delta}\|_{\mathcal{C}^0_aL^{p}_\omega \mathcal{C}^{1/2 + (\alpha-1/2)/32}_t \mathcal{C}^1_z} &\leq \|\iop v^j\|_{\mathcal{C}^0_aL^{p}_\omega \mathcal{C}^{1/2 + (\alpha-1/2)/32}_t \mathcal{C}^1_z} + \|R^{1,j} + R^{2,j} + R^{3,j} + \indc_{j=1} R^4\|_{L^{p}_\omega \mathcal{C}^0_t \mathcal{C}^1_z} 
\\&\leq C p^{1/2} 2^{-C^{-1} j},
\end{align*}
allowing us to conclude~\eqref{eq:Ij bound auto}.

For~\eqref{eq:Ij deriv bound auto}, we note by directly inspecting the definition of the terms of $I^{j,N,\delta}$ and using the bounds on $u^j$, we see that $I^{j,N,\delta} = \iop \tilde v_j$ for some $\tilde v_j$ such that
\[\|\tilde v_j\|_{\mathcal{C}^0_t \mathcal{C}^2_z} \leq C 2^{2j}.\]
Thus by the chain rule, pointwise in $a$,
\[\|\nabla_a I^{j,N,\delta}\|_{L^{p}_\omega \mathcal{C}^{1}_t \mathcal{C}^1_z} \leq \|\tilde v_j\|_{L^\infty_\omega \mathcal{C}^0_t \mathcal{C}^2_z} \|\nabla \Psi\|_{L^p_\omega \mathcal{C}^0_t} \leq  C 2^{2j}\|\nabla \Psi\|_{L^p_\omega \mathcal{C}^0_t}.\]
Then taking the supremum over $a$ gives~\eqref{eq:Ij deriv bound auto} and hence concludes the proof of Proposition~\ref{prop:fixed p sweeping averaging}.

\subsubsection{\texorpdfstring{The effective H\"older case (Proposition~\ref{prop:autonomous bihari estimate}): $\alpha \leq 1/2$}{The effective H\"older case: alpha leq 1/2}}

We now consider the case that $\alpha \leq 1/2$. Interpolating the $i=0$ and $i=1$ cases of~\eqref{eq:I 1 final averaged before interp}, the $i=0$ and $i=1$ cases of~\eqref{eq:I 2 final averaged before interp}, and the $n=0,n=1$ cases of~\eqref{eq:I 1 final a priori before interp}, we have for $\ep \in (0,2\alpha)$,
\begin{align}
    \label{eq:I 1 final averaged before interp nu}
    \|I^{j,1}\|_{\mathcal{C}^0_a \mathcal{C}^{1/2}_t \mathcal{C}^{2\alpha -\ep}_z L^{p}_{\omega}} &\leq Cp^{1/2} 2^{(\alpha -1/2 -\ep) j},
    \\\|I^{j,2}\|_{\mathcal{C}^0_a L^p_\omega \mathcal{C}^{1}_t \mathcal{C}^{2\alpha -\ep}_z } &\leq C  2^{-\ep j},
        \label{eq:I 2 final averaged before interp nu}
    \\ \|I^{j,1}\|_{\mathcal{C}^0_a L^{p}_{\omega} \mathcal{C}^1_t \mathcal{C}^{2\alpha -\ep}_z} &\leq C  2^{(\alpha -\ep)j}.
        \label{eq:I 1 final a priori before interp nu}
\end{align}
We then interpolate~\eqref{eq:I 1 final averaged before interp nu} with~\eqref{eq:I 1 final a priori before interp nu} to give that
\begin{equation}
\label{eq:I j 1 final final holder before interp}
    \|I^{j,1}\|_{\mathcal{C}^0_a \mathcal{C}^{1-\alpha}_t \mathcal{C}^{2\alpha -\ep}_z L^{p}_{\omega}} \leq C p^{1/2} 2^{-\ep j}.
\end{equation}
Combining~\eqref{eq:I 2 final averaged before interp nu} and~\eqref{eq:I j 1 final final holder before interp}, we then have that
\[\|\iop v^j\|_{\mathcal{C}^0_a \mathcal{C}^{1-\alpha}_t \mathcal{C}^{2\alpha -\ep}_z L^{p}_{\omega}} \leq C p^{1/2} 2^{-\ep j}.\]
We then apply Lemma~\ref{lem:alpha-holder_embedding} and Lemma~\ref{lem:sobolev}---localizing with the same argument as in the Lipschitz case---for $p$ large enough depending on $\ep$, to give that
\[\|\iop v^j\|_{\mathcal{C}^0_a L^{p}_\omega \mathcal{C}^{1-\alpha-\ep}_t \mathcal{C}^{2\alpha - 2\ep}_z} \leq C p^{1/2} 2^{-\ep j}.\]
Then by~\eqref{eq:V big decomp}, combining the above display with~\eqref{eq:remainder bound autonomous}, we have that 
\[\|\iop V\|_{\mathcal{C}^0_a L^{p}_\omega \mathcal{C}^{1-\alpha-\ep}_t \mathcal{C}^{2\alpha - 2\ep}_z} \leq \|R\|_{L^{p}_\omega \mathcal{C}^0_t \mathcal{C}^{2\alpha -2\ep}_z} + \sum_{j=1}^N \|\iop v^j\|_{\mathcal{C}^0_a L^{p}_\omega \mathcal{C}^{1-\alpha-\ep}_t \mathcal{C}^{2\alpha - 2\ep}_z} \leq Cp^{1/2}.\]
This concludes the proof of the $V$ estimate of Proposition~\ref{prop:autonomous bihari estimate}, after changing $\ep$ and letting $C$ depend on $p$.

\subsection{Proof of Proposition~\ref{prop:autonomous truncation approximation}: finite range truncation}
\label{ss:autonomous truncation}

\begin{proof}[Proof of Proposition~\ref{prop:autonomous truncation approximation}]
The proof is very similar to that of Proposition~\ref{prop:refreshing truncation approximation}, except we now partition space as opposed to time. This is the reason for the $d$ dependent lower bound on $p$.

Again, by homogeneity, we may assume that $K=1$. For each $j\geq 0$, let $\{\varphi^{j,\ell}\}_{\ell=1}^{L^j}$ be a partition of unity for $\T^d$ such that $0\leq \varphi^{j,\ell}\leq 1$, $\sum_{\ell=1}^{L^j}\varphi^{j,\ell}=1$, $\mathrm{diam}(\supp\varphi^{j,\ell})\leq 2^{-j}$, $L^j\leq C 2^{jd}$, and for $n= 0,1,2$
\[\sum_{\ell=1}^{L^j}|\nabla^n\varphi^{j,\ell}|\leq C 2^{nj},\]
where above $C>0$ is a constant depending on $d$. 

Fixing $\eps$ and $A$, we first construct intermediate velocity fields from which we will construct the $u^{j,A}$. For each $1\leq \ell \leq L^j$ let
\[b^{j,\ell}:=\sup_{x\in \supp\varphi^{j,\ell}} \sum_{n=0}^2 2^{-(n-\alpha)j}|\nabla^nu^j(x)|,\]
so that
\[\Big\|\sup_{1\leq \ell\leq L^j}b^{j,\ell}\Big\|_{L^p_\omega}\leq \sum_{n=0}^2 2^{-(n-\alpha)j}\|\nabla^n u^j\|_{L^p_\omega \mathcal{C}^0_x}\leq 3,\]
by~\eqref{eq:autonomous truncation moments}. Let $\chi:[0,\infty)\rightarrow [0,1]$ be a smooth function such that $\chi(x)=1$ for $x\leq 1$ and $\chi(x)=0$ for $x\geq 2$. We then define the intermediate vector fields
\[v^{j,A}(x):=\bigg(\sum_{\ell=1}^{L^j}\varphi^{j,\ell}(x) \chi\Big(\frac{b^{j,\ell}}{A2^{\eps j}}\Big) \bigg)u^j(x).\]
We thus have that $v^{j,A}=u^j$ if $b^{j,\ell}\leq A2^{\eps j}$ for all $1\leq \ell \leq L^j$ and the sure bound
\[\|\nabla^n v^{j,A}\|_{\mathcal{C}^0_x}\leq 2 A 2^{(n-\alpha+\eps)j},\]
for $n=0,1,2.$

Next, for $j\geq 1$, since $\E u^j=0$, it holds that
\[ \mathbb{E}[v^{j,A}(x)]=\sum_{\ell=1}^{L^j}\varphi^{j,\ell}(x)\E\Big[\Big(\chi\Big(\frac{b^{j,\ell}}{A2^{\eps j}}\Big)-1\Big)u^j(x)\Big].\]
We then note that, by Chebyshev's inequality 
\[\mathbb{E}[\indc_{b^{j,\ell}\geq A2^{\eps j}}b^{j,\ell}]\leq (A2^{\eps j})^{1-p}\|b^{j,\ell}\|_{L^p_\omega}^p\leq 3^p A^{1-p} 2^{(1-p)\eps j},\]
thus letting
\[p^{j,\ell,A}:= 3^{-p}2^{-\eps j}\mathbb{E}[\indc_{b^{j,\ell}\geq A2^{\eps j}}b^{j,\ell}],\]
we find that $0\leq p^{j,\ell,A}\leq  A^{1-p} 2^{-p\eps j}\leq 1$.

On an enlarged probability space, define independent Bernoulli random variables (also independent of $u^j$) $\xi^{j,\ell,A}$ with
\[\E[\xi^{j,\ell,A}]=\P(\xi^{j,\ell,A}=1)=p^{j,\ell,A}.\]
We then let $u^{0,A}:=v^{0,A}$ and for $j\geq 1$
\[u^{j,A}(x):= v^{j,A}(x)-\sum_{\ell=1}^{L^j}\varphi^{j,\ell}(x)\frac{\xi^{j,\ell,A}}{p^{j,\ell,A}}\E\Big[\Big(\chi\Big(\frac{b^{j,\ell}}{A2^{\eps j}}\Big)-1\Big)u^j(x)\Big],\]
with the convention that $\frac{\xi^{j,\ell,A}}{p^{j,\ell,A}}=0$ if $p^{j,\ell,A}=0$. We must thus verify that $u^{j,A}$ has the desired properties.

By the definition of $p^{j,\ell,A}$, $\xi^{j,\ell,A}$ and $u^{j,A}$ it immediately holds that $u^{j,A}$ is centered for $j\geq 1$. Additionally, we note that $b^{j,\ell}$ is independent of $b^{j,\ell'}$ if $\mathrm{dist}(\supp\varphi^{j,\ell},\supp\varphi^{j,\ell'})\geq 2^{-j}$. By our assumption on the diameter of the supports for $\varphi^{j,\ell}$, this implies that $u^{j,A}$ has a finite range of dependence of $3\cdot 2^{-j}$. Re-indexing and changing constants, we can thus make it so that $u^{j,A}$ has a finite range of dependence of $2^{-j}$.

On the other hand, for arbitrary $j$ and $n=0,1,2$, by the product rule 
\[|\nabla^n u^{j,A}|\leq |\nabla^{n} v^{j,A}|+C\sum_{\ell=1}^{L^j}\sum_{k=0}^n|\nabla^{n-k}\varphi^{j,\ell}(x)| \frac{1}{p^{j,\ell,A}}\mathbb{E}\Big[\Big|\Big(\chi\Big(\frac{b^{j,\ell}}{A2^{\eps j}}\Big)-1\Big)\nabla^k u^j(x)\Big|\Big].\]
For $x\in \supp \varphi^{j,\ell}$,
\begin{align*}
\frac{1}{p^{j,\ell,A}}\mathbb{E}\Big[\Big|\Big(\chi\Big(\frac{b^{j,\ell}}{A2^{\eps j}}\Big)-1\Big)\nabla^k u^j(x)\Big|\Big]\leq \frac{2^{(k-\alpha)j}}{p^{j,\ell,A}}\mathbb{E}[\indc_{\{b^{j,\ell}\geq A2^{\eps j}\}} b^{j,\ell}]\leq 3^p2^{(k-\alpha+\eps)j}.
\end{align*}
We thus find that
\begin{align*}
\sum_{\ell=1}^{L^j}\sum_{k=0}^n|\nabla^{n-k}\varphi^{j,\ell}(x)| \frac{1}{p^{j,\ell,A}}\mathbb{E}\Big[\Big|\Big(\chi\Big(\frac{b^{j,\ell}}{A2^{\eps j}}\Big)-1\Big)\nabla^k u^j(x)\Big|\Big]&\leq C\sum_{k=0}^n 2^{(k-\alpha+\eps)j}\sum_{\ell=1}^{L^j} |\nabla^{n-k}\varphi^{j,\ell}(x)|
\\&\leq C2^{(n-\alpha+\eps)j}.
\end{align*}
Combining this with the sure bound on $v^{j,A}$ we find in total the sure bound
\begin{equation}\label{eq:pre a.s. autonomous}
\|\nabla^n u^{j,A}\|_{\mathcal{C}^0_x}\leq C A 2^{(n-\alpha+\eps)j},    
\end{equation}
for some $C$ depending only on $d$ and $p$.

Finally, we have that
\begin{align*}
\P(U^{A}\neq U)&\leq \P(b^{j,\ell}> A 2^{\eps j}\text{ or } \xi^{j,\ell,A}= 1\text{ for some $j,\ell$})
\\&\leq \sum_{j\geq 0}\P(\max_{1\leq \ell\leq L^j}b^{j,\ell}> A 2^{\eps j})+\sum_{j\ge 1}\sum_{\ell=1}^{L^j}p^{j,\ell, A}
\\&\leq C\sum_{j\geq 0} A^{-p} 2^{-p\eps j}+C\sum_{j\ge 1}2^{(d-p\eps)j}A^{1-p}
\\&\leq C A^{1-p}
\end{align*}
for a constant $C(d,p,\eps)>0$, where we have used Chebyshev's inequality, the $p$ moment bounds on $\|\nabla^n u^j\|_{\mathcal{C}^0_x}$, and that $L^j\leq C 2^{dj}$ in the second last inequality. The final inequality follows as the sums are finite since $d<p\eps$ by assumption.

Changing $A$ somewhat to absorb the constant in~\eqref{eq:pre a.s. autonomous}, we conclude.
\end{proof}

\section{Proof of Corollary~\ref{cor:main pde ode}: Aizenman avoidance and Ambrosio superposition}

\label{s:pde}

The following lemma gives an easily verifiable sufficient condition for the zero level set hypothesis,~\eqref{eq:small zeros}, of Corollary~\ref{cor:main pde ode}. The argument style is standard; we provide the proof for the reader's convenience.

\begin{lemma}
    \label{lem:sufficient condition for good zero set}
    Let $\alpha,\gamma \in (0,1]$. Suppose $U : \T^d \to \R^d$ is a random field such that
    \begin{enumerate}
        \item almost surely, $U \in C^{\alpha-}(\T^d)$,
        \item and $\limsup_{r \to 0} \sup_{x \in \T^d} r^{-\gamma d} \P( |U(x)| \leq r)<\infty$.
    \end{enumerate}
    Then almost surely, if $Z := \{ x \in \T^d : U(x) =0\}$, for any $\rho < \gamma \alpha$, 
    \[\lim_{r \to 0} r^{-\rho d} |\{x \in \T^d : d_{\T^d}(x,Z) \leq r\}| =0.\]
\end{lemma}

\begin{proof}
    Fix $\rho_1 < \rho_2 < \gamma \alpha$ arbitrary. By hypothesis, there exists $K>0$ such that for all $x \in \T^d, r>0$, $\P(|U(x)| \leq r) \leq K r^{\gamma d}$. Let
    \[Z := \{x \in \T^d : U(x) =0\}.\]
    Let $\nu < \alpha$ such that $\rho_2 < \gamma \nu$. Then let $n \in \N,M >0$. We claim for some $ C(\rho_2,M,\nu,K,d)>0$
    \begin{equation}
    \label{eq:truncated minkowski estimate}
    \P\big( |\{x \in \T^d : d_{\T^d}(x,Z) \leq 2^{-n}\}| \geq 2^{-\rho_2 dn} \text{ and } \|U\|_{\mathcal{C}^{\nu}} \leq M\big) \leq C 2^{- C^{-1} n}, 
    \end{equation}
    and that this suffices to conclude the claim. Indeed, for fixed $M >0$, by Borel--Cantelli, this gives that almost surely, if $\|U\|_{\mathcal{C}^{\nu}} \leq M$, 
    \[\limsup_{r \to 0} r^{-\rho_2 d} |\{x \in \T^d : d_{\T^d}(x,Z) \leq r\}| \leq C\limsup_{n \to \infty} 2^{\rho_2 dn} |\{x \in \T^d : d_{\T^d}(x,Z) \leq 2^{- n}\}|< \infty,\]
    and so if $\|U\|_{\mathcal{C}^{\nu}} \leq M$,
     \[\limsup_{r \to 0} r^{-\rho_1 d} |\{x \in \T^d : d_{\T^d}(x,Z) \leq r\}| =0.\]
     Taking $M \to \infty$ and using that $\|U\|_{\mathcal{C}^\nu} < \infty$ almost surely, we then get that 
        \[\lim_{r \to 0} r^{-\rho_1d} |\{x \in \T^d : d_{\T^d}(x,Z) \leq r\}| =0.\]
    Since $\rho_1 < \gamma \alpha$ was arbitrary, this allows us to conclude.

    Thus we just need to show~\eqref{eq:truncated minkowski estimate}. Let $n \in \N$, and let $A_n := 2^{-n} \Z^d \cap \T^d$ be a dyadic mesh on $\T^d$. Let
    \[B_n := \{y \in A_n: |U(y)| \leq 2M(d^{1/2} 2^{-n})^\nu\}.\]
    Then if $ \|U\|_{\mathcal{C}^{\nu}} \leq M$, we have that 
    \[ |\{x \in \T^d : d_{\T^d}(x,Z) \leq 2^{-n}\}| \leq C 2^{-dn} |B_n|.\]
    Then we note that 
    \[2^{-dn}\E |B_n| \leq   \sup_{x \in \T^d} \P\big(|U(x)| \leq 2M(d^{1/2} 2^{-n})^\nu\big) \leq C 2^{-\gamma \nu d n}.\]
    Thus
    \begin{align*}
        \P\big( |\{x \in \T^d : d_{\T^d}(x,Z) \leq 2^{- n}\}| \geq 2^{-\rho_2 d n} \text{ and } \|U\|_{\mathcal{C}^{\nu}} \leq M\big) &\leq \P(2^{-dn}|B_n| \geq C^{-1} 2^{-\rho_2 d n})
        \\& \leq  C 2^{(\rho_2-\gamma \nu) d n} 
        \\&\leq C 2^{-C^{-1}n},
    \end{align*}
    thus giving~\eqref{eq:truncated minkowski estimate}.
\end{proof}

The following argument is quite similar to that of~\cite{aizenman_sufficient_1978}, however takes additional advantage of the regularity of the velocity field $u \in C^\alpha(\T^d)$ together with the avoided set being the zero level set. This makes it ``even harder'' to hit the avoided set, somewhat strengthening the result (gaining the exponent $\alpha$ in the Minkowski dimension constraint) compared to using the result of~\cite{aizenman_sufficient_1978} directly. This proposition will be the tool we will use to ensure that almost every ODE trajectory does not hit the zero level set in the sweeping regime, hence giving a.e.\ uniqueness. We will refer to the following ODE; here we think of $u$ as a fixed deterministic field.

\begin{equation}
    \label{eq:abstract ode}
    \begin{cases}
        \dot X_t = u(X_t),\\
        X_0 = y.
    \end{cases}
\end{equation}

\begin{proposition}
\label{prop:abstract uniqueness away from bad plus small bad}
    Suppose $u : \T^d \to \R^d$, $\alpha \in (0,1)$, and $Z := \{u =0\} \subseteq \T^d$ are such that
    \begin{enumerate}
        \item $\nabla \cdot u=0,$
        \item $u \in C^\alpha(\T^d)$,
        \item $\liminf_{r \to 0} r^{-(1-\alpha)} |\{ x \in \T^d : d_{\T^d}(x,Z) \leq r\}| =0$,
        \item and for all $\tau>0, y \in \T^d$, if $X^1, X^2$ are solutions to~\eqref{eq:abstract ode} on $[0,\tau]$ such that for all $t \in [0,\tau]$, $X^1_t, X^2_t \not \in Z$, then $X^1|_{[0,\tau]} = X^2|_{[0,\tau]}$.
    \end{enumerate}
    Then for almost every $y \in \T^d$,~\eqref{eq:abstract ode} admits a unique solution on $[0,\infty)$.
\end{proposition}

\begin{proof}
    We fix some time $T \geq 1$ and prove the result on $[0,T]$. Taking $T \to \infty$ gives the general result. For each $y \in \T^d$, let $X^y_t$ be any solution to~\eqref{eq:abstract ode} and let 
    \[\tau(y):= \inf \{t \in [0,T]:  X^y_t \in Z\}.\]
    Note that $\tau(y)$, as well as $X^y|_{[0,\tau(y)]}$, is well-defined independently of which solution $X^y_t$ we took by the uniqueness of~\eqref{eq:abstract ode} up until hitting $Z$. We then define the stopped flow $\Phi : [0,T] \times \T^d \to \T^d$ by
    \[\Phi_t(y) := X^y_{t\wedge \tau(y)},\]
    and the set $S\subset \T^d$ by
    \[S := \{y \in \T^d : \Phi_T(y) \in Z\}.\]
    In order to conclude, by the uniqueness up until hitting $Z$, it suffices to show that $|S| =0$.
    
    Let $r \in (0,1)$ and define
    \[Z_r := \{x \in \T^d : d_{\T^d}(x,Z) \leq r\}.\]
    Then if $y \in S$, either $y \in Z_r$, or there exists $0 < t_1 < t_2  < \tau(y) \leq T$ such that $d_{\T^d}(\Phi_{t_1}(y), Z) = r$, $d_{\T^d}(\Phi_{t_2}(y), Z) = r/2,$ and for all $t \in [t_1,t_2],$ $\Phi_t(y) \in Z_r$. Then since $u \in C^\alpha(\T^d)$ and $u(Z) = 0$, for $t \in [t_1,t_2]$, $|\dot \Phi_t(y)| \leq C r^\alpha$. Thus $t_2 - t_1\geq C^{-1}r^{1-\alpha}$.

    In total, we therefore have for all $y\in \T^d$,
    \[ \int_0^T \indc_{\Phi_t(y) \in Z_r, t < \tau(y)}\,dt =|\{t \in [0,T] : \Phi_t(y) \in Z_r, t < \tau(y)\}| \geq C^{-1} r^{1-\alpha}\indc_{y \in S \backslash Z_r}.\]
    Thus, rearranging and integrating in $y$, we have that
    \begin{align}
    |S| &\leq Cr^{-(1-\alpha)} \int\int_0^T\indc_{\Phi_t(y) \in Z_r, t < \tau(y)}\,dt\,dy + |Z_r| 
   \notag \\&\leq CTr^{-(1-\alpha)} \sup_{t \in [0,T]} |\{y \in \T^d : \Phi_t(y) \in Z_r, t< \tau(y)\}| + |Z_r|.
    \label{eq:bad set bound S}
    \end{align}
    Thus to conclude, it suffices to show that for a Borel set $A \subseteq \T^d$ and a fixed $t \in [0,T]$,
    \begin{equation}
        \label{eq:sub volume preservation}
        |\{y \in \T^d : \Phi_t(y) \in A, t< \tau(y)\}| \leq |A|,
    \end{equation}
    since then by~\eqref{eq:bad set bound S} and~\eqref{eq:sub volume preservation}, taking a $\liminf$, we have that
    \[|S| \leq   CT\liminf_{r \to 0}r^{-(1-\alpha)} |Z_r| =0,\]
    where the final equality is direct by hypothesis.

    In order to see~\eqref{eq:sub volume preservation}, we prove it for $A$ open; the general case follows by the outer regularity of the Lebesgue measure and taking an infimum over all open $U\supseteq A$.
    
    For $\ep>0$, let $u^\ep$ be a sequence of smooth divergence-free velocity fields such that $\|u - u^\ep\|_{\mathcal{C}^0_x} \to 0.$ Let $\Phi^\ep$ be the (not stopped) flow of the ODE associated to $u^\ep$. Then we note that, by $\nabla \cdot u^\ep =0$,
    \[   |\{y \in \T^d : \Phi^\ep_t(y) \in A, t< \tau(y)\}| \leq    |\{y \in \T^d : \Phi^\ep_t(y) \in A\}| \leq |A|.\]
    However, by the ODE uniqueness before hitting $Z$, we note that if $t< \tau(y)$, $\Phi^\ep_t(y) \to \Phi_t(y)$, thus---as $A$ is open---$\Phi^\ep_t(y)$ is eventually in $A$ if $\Phi_t(y) \in A$. Thus~\eqref{eq:sub volume preservation} follows by Fatou's lemma.
\end{proof}

The following is an application of the Ambrosio superposition principle~\cite[Theorem 3.2]{ambrosioTransportEquationCauchy2008}, which we include for the reader's convenience.

\begin{proposition}\label{prop:ODE to continuity uniqueness}
    Suppose $u: \T^d \to \R^d$ is continuous, $\nabla \cdot u =0$, and for almost every $y \in \T^d$, the ODE
    \[\begin{cases}
    \dot \Phi_t(y) = u(\Phi_t(y)),\\
    \Phi_0(y) = y,
    \end{cases}\]
    admits a unique solution on $(-\infty,\infty)$. Then for any $f \in L^1(\T^d)$ with $f \geq 0$, the continuity equation
    \begin{equation}
    \label{eq:continuity abstract}
    \begin{cases}
        \partial_t \phi + \nabla \cdot (u \phi) = 0,\\
        \phi(0,\cdot) = f(\cdot),
    \end{cases}\end{equation}
    admits a unique distributional solution on $(-\infty,\infty) \times \T^d$ such that $\phi \geq 0$ and for all $T>0,$ $\phi \in \mathcal{C}^0([-T,T],L^1(\T^d))$. Further
    \[\phi(t,\cdot) = (\Phi_t)_* f = f \circ \Phi_{-t},\]
    where $(\Phi_t)_* f$ denotes the pushforward of $f$ as a measure density with respect to Lebesgue, and hence $\phi$ is a renormalized solution in the DiPerna--Lions sense: that is, for all $\beta : \R \to \R$ such that $\beta \in C^1_b(\R)$, $\beta(\phi)$ is also a distributional solution to~\eqref{eq:continuity abstract} with initial data $\beta(f)$.
\end{proposition}

\begin{proof}
    All the claims up until the claim that $(\Phi_t)_* f = f \circ \Phi_{-t}$ follow directly from the Ambrosio superposition principle~\cite[Theorem 3.2]{ambrosioTransportEquationCauchy2008}. We note that $\Phi_t$ is Lebesgue measure preserving, as can be seen by noting that $1$ is a distributional solution to~\eqref{eq:continuity abstract}, thus we must have that for all $t \in \R$, $(\Phi_t)_* 1 = 1$. Thus, by the a.e.\ uniqueness of the ODE, we have that $\Phi_{-t} \circ \Phi_t(x) = x$ a.e. Thus for any bounded $g : \T^d \to\R$,
    \[\int f \circ \Phi_{-t}(x) g(x) dx = \int  f \circ \Phi_{-t}(\Phi_t(x)) g(\Phi_t(x))\,dx  = \int f(x) g(\Phi_t(x))\,dx = \int g(x) (\Phi_t)_* (f(x)dx).\]
    Thus $(\Phi_t)_* f = f \circ \Phi_{-t}$, from which the renormalization is direct, allowing us to conclude.
\end{proof}

Corollary~\ref{cor:main pde ode}, as well as Remark~\ref{rem:after cor}, now follows essentially directly from Theorem~\ref{thm:main ode intro}, Proposition~\ref{prop:abstract uniqueness away from bad plus small bad}, and Proposition~\ref{prop:ODE to continuity uniqueness}, noting that any bounded solution $f$ to~\eqref{eq:continuity} can be translated $f + K$ so that $f +K \geq 0$, $f+K$ solves~\eqref{eq:continuity} (using that $\nabla \cdot U =0$), and $f +K \in \mathcal{C}^0_t L^1_x$.

\section{Ill-posedness below the critical threshold: CLT scaling in a model case}
\label{s:sharpness}

We now prove the main nonuniqueness examples, both in the refreshing case, Theorem~\ref{thm:sharp refreshing}, and the sweeping case, Theorem~\ref{thm:sharp autonomous}. As discussed in Section~\ref{ss:nonunique}, the sweeping case will follow from the refreshing case using a dimensional embedding. As such, the majority of this section focuses on the refreshing case. We now precisely specify the velocity field under consideration.

\subsection{Definition of the field}
In this subsection we define the random fields we consider throughout the rest of the section. To this end, fix $d \geq 3$, $\alpha \in (0,1)$, and $\beta>0$ with 
\[\alpha + \beta/2 <1\quad\text{and}\quad\alpha + \beta > 1.\]  
All implicit constants depend freely on $\alpha,\beta$.

Let $\phi : \R^d \to [0,\infty)$ be a smooth radial bump function with $\int \phi(x)\,dx =1$ and  $\supp \phi \subseteq {B_{1/2}}$. For $j \in \N$, let $\phi^j : \T^d \to [0,\infty)$ be defined for $|x| \leq 1/2$ by $\phi^j(x) = 2^{dj/2} \phi(2^j x)$ (and extended periodically for other $x$). Define the Gaussian field $a^j : \T^d \to \R$ by $\E a^j =0 $ and
\[\E a^j(x) a^j(y) =  2^{-2(\alpha+1) j} \phi^j * \phi^j(x-y).\]
For $1 \leq k < \ell \leq d$, let $a^j_{k,\ell}$ be iid copies of $a^j$. Then define the Gaussian field $A^j : \T^d \to \R^{d \times d}$ by $(A^j)^t = - A^j$ and for $1 \leq k < \ell \leq d$,
\[A^j_{k\ell} = a^j_{k,\ell}.\]
Define then the Gaussian field $\tilde{u}^j : \T^d \to \R^d$ by
\[\tilde{u}^j_k(x) := \sum_{\ell=1}^d K^{-1} \partial_\ell A^j_{\ell k}(x),\]
where $K \in (0,\infty)$ is a $j$-independent constant defined so that equality holds in Item~\ref{item:normalized} below.
We note the following properties of the Gaussian fields $\tilde{u}^j$, which follow by direct computation, recalling that $\phi$ is radial and $\supp \phi \subseteq B_{1/2}$.

\begin{lemma}
\label{lem:tilde u j properties}
    For $j \in \N$, let $\tilde{u}^j : \T^d \to\R^d$ be the Gaussian field defined above. Then we have that
    \begin{enumerate}
        \item $\E \tilde{u}^j =0$.
        \item $\nabla \cdot \tilde{u}^j =0.$
        \item For all Borel $A,B \subseteq \T^d$, $d(A,B) \geq 2^{-j}$, we have that $\tilde{u}^j|_A \indep \tilde{u}^j|_B$.
        \item For all $y \in \T^d$, $\tilde{u}^j(\cdot + y) \stackrel{d}{=} \tilde{u}^j(\cdot)$.
        \item \label{item:normalized}
        For all $1 \leq \ell,k \leq d$, $\E \tilde{u}^j_\ell(x) \tilde{u}^j_k(x) = 2^{-2 \alpha j} \delta_{\ell k}$.
        \item\label{item:example moment estimates} For all $n \in \N, \ep >0$, there exists $C(d,n,\ep)>0$ such that for all $j \in \N$, $p \geq 1,$
        \[ \|\nabla^n \tilde{u}^j\|_{L^p_\omega \mathcal{C}^0_x(\T^d)} \leq C \sqrt{p} 2^{(n-\alpha + \ep) j}.\]
    \end{enumerate}
\end{lemma}

Let $T^j$ be a decreasing sequence of times such that $\lim_{j\rightarrow \infty} T^j = 0$ and 
\[T^j - T^{j+1} = 2^{-\beta j} \lfloor 2^{\beta j} 2^{-1} Z_{\alpha,\beta}^{-1} 2^{-(2-2\alpha - \beta)j +j^{1/2}}\rfloor,\]
where
\[Z_{\alpha,\beta}:= \sum_{j=0}^\infty 2^{-(2-2\alpha - \beta)j +j^{1/2}} <\infty.\]
The $T^j$ are defined this way so that for some $N^j \in \N$,
\begin{equation}
\label{eq:Tj properties}
\lim_{j\rightarrow \infty} T^j = 0, \quad T^0 \leq \tfrac{1}{2}, \ T^j - T^{j+1} = N^j 2^{-\beta j},\ \text{and}\ C^{-1} 2^{2(\alpha+\beta -1)j +j^{1/2}} - 1 \leq N^j \leq C 2^{2(\alpha+\beta -1)j +j^{1/2}}.
\end{equation}

For all $\ell \in \N$, let $\tilde{u}^{j,\ell}$ be iid copies of $\tilde{u}^j$. Let $\eta : \R \to [0,\infty)$ be a smooth bump function with $\int \eta(t)\,dt =1$ and $\supp \eta \subseteq (0,1)$. Define the Gaussian field $u^j : \R \times \T^d \to \R^d$ by 
\[u^j(t,x) := \sum_{\ell \in \N} \indc_{t \in [T^{j+1}, T^j]} \eta(2^{\beta j}(t-T^{j+1})-\ell) \tilde u^{j,\ell}(x).\]
We then finally define the Gaussian field $U : \R \times \T^d \to \R^d$ by
\[U(t,x) := \sum_{j=0}^\infty u^j(t,x).\]
By construction, $U$ then satisfies the properties stated in Theorem~\ref{thm:sharp refreshing}. The multiscale decomposition, finite range of dependence in time, and incompressibility follow directly from the definition, while the moment estimates follow directly from Item~\ref{item:example moment estimates} in Lemma~\ref{lem:tilde u j properties}.

\subsection{Proof of Theorem~\ref{thm:sharp refreshing} and Theorem~\ref{thm:sharp autonomous}}

In this subsection we state two essential preliminary propositions, Proposition~\ref{prop:main separation estimate} and Proposition~\ref{prop:variance lower bound implies nonunique transport}, and then use them to prove our main non-uniqueness results, Theorems~\ref{thm:sharp refreshing} and~\ref{thm:sharp autonomous}. The proofs of these propositions are deferred to Section~\ref{ss:separation} and Section~\ref{ss:anomalous} respectively.

We consider the following SDE:
\begin{equation}
\label{eq:main SDE}
    \begin{cases}
        dX^\kappa_t(y) = U(t,X^\kappa_t(y)) dt + \sqrt{\kappa} dW_t,\\
        X^\kappa_0(y) =y.
    \end{cases}
\end{equation}

We note that there are two different random elements in~\eqref{eq:main SDE}: there is the random velocity field $U$ and the random noise $W$. These will always be taken to be independent of each other. Throughout the remainder of this section, we will often want to work in a situation where one of $U,W$ is fixed and the other is being integrated over. For that reason, we denote by $\P_U$ (and $\P_W$) the probability measure of $U$ (respectively of $W$); the total joint measure is then $\P_U \otimes \P_W$. By $\Var_W$ we then mean the variance over $W$ for fixed $U$.

\begin{definition}
    For $X$ a torus valued random variable, we denote the \textit{variance of $X$} by $\Var(X)$, which is defined by
    \[\Var(X) = \tfrac{1}{2} \E |X- \tilde X|^2,\]
    where $\tilde X$ is an iid copy of $X$.
\end{definition}

The following is the main quantitative ingredient from which Theorem~\ref{thm:sharp refreshing} and Theorem~\ref{thm:sharp autonomous} will follow. Its proof is deferred to Section~\ref{ss:separation}. 

\begin{proposition}
\label{prop:main separation estimate}
    There exists $C(d,\alpha,\beta)>0$ such that for all $y \in \T^d$, and $j,j_* \in \N$ with $C \leq j < j_*$,
    \[\P_U\big(\Var_W\big(X^{2^{-(2\alpha + \beta)j_*}}_{T^j}(y)\big) \leq 2^{-2j}\big) \leq C2^{-j^{1/2}/8}.\]
\end{proposition}

This can be interpreted as a quantitative statement of the instability of the ODE driven by $U$ under small isotropic perturbations. More precisely, with high probability over $U$, transport by $U$ with arbitrarily small Brownian forcing spreads particles over distances of order $2^{-j}$ by time $T^j$. Thus, as the noise is removed, the solutions to the SDE do not collapse to a single deterministic trajectory. Consequently, the proposition readily implies the following precise form of ODE nonuniqueness.

\begin{corollary}
\label{cor:richardson separation}
    For all $y \in \T^d$ and $h>0$,
    \[\P_U\big(\text{there exists more than one solution to~\eqref{eq:main refreshing ODE 0} on } [0,h]\big) = 1.\]
    Further, for all $y \in \T^d$ and almost surely in $U$, for all $\ep>0$ sufficiently small
    \begin{equation}
    \limsup_{t \to 0} t^{- \frac{1}{2-2\alpha - \beta - \ep}}\mathrm{diam}\{X_t : X \text{ is a solution to~\eqref{eq:main refreshing ODE 0}}\} =\infty.
        \label{eq:diameter growth}
    \end{equation}
\end{corollary}

\begin{proof}
    Fix $y \in \T^d$ and let
    \[D(t) := \mathrm{diam}\{X_t : X \text{ is a solution to~\eqref{eq:main refreshing ODE 0}}\}.\]
    Then we note that---since as $\kappa \to 0$ solutions to~\eqref{eq:main SDE} concentrate on solutions to~\eqref{eq:main refreshing ODE 0} by a straightforward compactness argument---for fixed $U$ we have by the (reverse) Fatou lemma that 
    \[D(T^j)^2 \geq \limsup_{j_* \to \infty} \Var_W\big(X^{2^{-(2\alpha + \beta)j_*}}_{T^j}(y)\big).\]
    Now applying the (reverse) Fatou lemma again, this time over $U$, we thus have that 
    \begin{align*}\P_U\big(D(T^j)^2 &\geq  2^{-2j}\big) \geq  \P_U \big(\limsup_{j_* \to \infty} \Var_W\big(X^{2^{-(2\alpha + \beta)j_*}}_{T^j}(y)\big)  \geq 2^{-2j}\big) 
    \\&\geq  \limsup_{j_* \to \infty} \P_U\big( \Var_W\big(X^{2^{-(2\alpha + \beta)j_*}}_{T^j}(y)\big)  \geq 2^{-2j}\big) 
    \\&\geq 1- C2^{-j^{1/2}/8}.
    \end{align*}
    Thus
    \[\P_U\big(D(T^j) < 2^{-j}\big) \leq C 2^{-j^{1/2}/8}.\]
    By Borel-Cantelli, we have almost surely in $U$ that
    \[\liminf_{j \to \infty} 2^j D(T^j) \geq 1.\]
    Then we note that for any $\ep>0$
    \[T^j \leq C 2^{-(2-2\alpha -\beta)j + j^{1/2}} \leq C 2^{-(2-2\alpha - \beta -
    \ep)j}.\]
    Thus
    \[2^j \leq C (T^j)^{- \frac{1}{2-2\alpha - \beta - \ep}},\]
    and
    \[   \limsup_{t \to 0} t^{- \frac{1}{2-2\alpha - \beta - \ep}}D(t) \geq C^{-1}.\]
    Changing $\ep$ slightly, we conclude~\eqref{eq:diameter growth}.
\end{proof}

We next use Proposition~\ref{prop:main separation estimate} to ensure that the spatially averaged variance has a strictly positive $\limsup$ as $\kappa \to 0$. The following proposition will be the main ingredient to prove the claimed PDE nonuniqueness.

\begin{proposition}
\label{prop:soft fubini variance bound}
     There exists $C>0$ such that for all $j \in \N,$
     \[\P_U\bigg(\limsup_{\kappa \to 0} \int_{\T^d} \Var_W(X^\kappa_{T^j}(y))\,dy =0\bigg) \leq C 2^{-j^{1/2}/8}.\]
\end{proposition}

\begin{proof}
    By Fatou's lemma, 
    \begin{align*}
    \P_U\Big(\limsup_{\kappa \to 0} \int \Var_W(X^\kappa_{T^j}(y))\,dy  =0 \Big)  &\leq  \P_U\Big(\liminf_{\kappa \to 0} \Big\{\int \Var_W(X^\kappa_{T^j}(y))\,dy \leq 2^{-2j-1}\Big\}\Big)
    \\&\leq \liminf_{\kappa \to 0}  \P_U\Big(\int \Var_W(X^\kappa_{T^j}(y))\,dy \leq  2^{-2j-1}\Big)
    \\&\leq \liminf_{j_* \to \infty}  \P_U\Big(\int \Var_W(X^{2^{-(2\alpha + \beta)j_*}}_{T^j}(y))\,dy \leq  2^{-2j-1}\Big),
    \end{align*}
    where after the first inequality we mean the $\liminf_{\kappa \to 0}$ of the sets  $\{\int \Var_W(X^\kappa_{T^j}(y))\,dy \leq 2^{-2j-1}\}$. Then for any $j_* > j$, define the (random) set
    \[A := \{y :  \Var_W(X^{2^{-(2\alpha + \beta)j_*}}_{T^j}(y)) \leq 2^{-2j}\}.\]
    Then by Proposition~\ref{prop:main separation estimate},
    \[\E_U |A| = \int \P_U \big(\Var_W(X^{2^{-(2\alpha + \beta)j_*}}_{T^j}(y)) \leq 2^{-2j}\big)\,dy \leq C 2^{-j^{1/2}/8}.\]
    This implies that
    \[ \P_U\Big(\int \Var_W(X^{2^{-(2\alpha + \beta)j_*}}_{T^j}(y))\,dy \leq 2^{-2j -1}\Big) \leq \P_U\big(|A| \geq 1/2\big) \leq 2 \E_U |A| \leq C 2^{-j^{1/2}/8}.\]
    Combining the above, we get the result.
\end{proof}

The next proposition then gives that the spatially averaged variance lower bound of Proposition~\ref{prop:soft fubini variance bound} is sufficient to get the desired PDE nonuniqueness. The proof, which goes through anomalous dissipation, is given in Section~\ref{ss:anomalous}.

\begin{proposition}
\label{prop:variance lower bound implies nonunique transport}
    Let $v : [0,1] \times \T^d \to \R^d$ be such that $\nabla \cdot v =0$ and $v \in L^\infty([0,1] \times \T^d)$. Let $X^\kappa_t(y)$ denote the solution to~\eqref{eq:main SDE} with $v$ in place of $U$. Suppose for some $t \in (0,1]$
    \[\limsup_{\kappa \to 0} \int_{\T^d} \Var_W(X^\kappa_t(y))\,dy >0.\]
    Then there exists bounded initial data $\phi$ such that
    \begin{equation}
    \label{eq:transport from ad}
    \begin{cases}
        \partial_t f + v \cdot \nabla f =0,\\
        f(0,\cdot) =\phi(\cdot),
    \end{cases}\end{equation}
    admits at least two distinct bounded solutions.
\end{proposition}

The following corollary, giving us a.s.\ PDE nonuniqueness, is then direct from the above two propositions after sending $ j \to\infty$.

\begin{corollary}
\label{cor:nonunique transport}
    \[\P_U\big(U \text{ admits more than one bounded solution to~\eqref{eq:continuity} for some bounded initial data}\big) =1.\]
\end{corollary}

We can then directly conclude Theorem~\ref{thm:sharp refreshing} by the above results.

\begin{proof}[Proof of Theorem~\ref{thm:sharp refreshing}]
    Direct from Corollary~\ref{cor:richardson separation}, Corollary~\ref{cor:nonunique transport}, Lemma~\ref{lem:tilde u j properties}, and the construction of $U$.
\end{proof}

Finally, we prove Theorem~\ref{thm:sharp autonomous} using the dimensional embedding argument, as sketched in Section~\ref{ss:nonunique}.

\begin{proof}[Proof of Theorem~\ref{thm:sharp autonomous}]
    We fix $d \geq 4$, $\alpha \in (0,1/2)$. We then let $V : [0,1] \times \T^{d-1} \to \R^{d-1}$ be the time-dependent field in $d-1$ dimensions constructed above for $\alpha$ and $\beta = 1$, admitting the decomposition $V = \sum_{j=0}^\infty v^j$.

    We note that the $v^j$ naturally define functions $\T^d \to \R^{d-1}$, treating the time coordinate as a periodic coordinate. We then define $u^j : \T^d \to \R^d$ by (calling the first coordinate ``$t$''),
    \[u^j(t,x) := \big(\indc_{j=0}, v^{j}(t,x)\big)\]
    and let
    \[U := \sum_{j=0}^\infty u^j.\]
    We note then that Items (i)--(iv) of Theorem~\ref{thm:sharp autonomous} are essentially direct from the construction of $U$---using Lemma~\ref{lem:tilde u j properties}---except that the spatial range may be $2^{-j+1}$ instead of $2^{-j}$, which is easily fixed by reindexing $j \mapsto j-1 \lor 0$. We thus turn our attention to Items~\ref{item:ode nonunique instantly and spreading auto} and \ref{item:pde nonunique auto}.

    For Item~\ref{item:ode nonunique instantly and spreading auto}, we let $y = (0,z)$ for $z \in \T^{d-1}$. We work on a time interval $h \in (0,1/2]$. We see that the solution to~\eqref{eq:main ODE autonomous data} $X_t$ is exactly of the form
    \[X_t = (t, Z_t),\]
    where $Z_t$ solves
    \[\begin{cases}
        \dot Z_t = \sum_{j=0}^\infty v^{j}(t,Z_t),\\
        Z_0 = z.
    \end{cases}\]
    Thus Item~\ref{item:ode nonunique instantly and spreading auto} follows directly from Corollary~\ref{cor:richardson separation} and the definition of the $v^{j}$.

    We now consider Item~\ref{item:pde nonunique auto}. Let $\chi : \R \to [0,1]$ be a smooth bump function with $\supp \chi \subseteq [-1/4,0]$. We consider initial data to~\eqref{eq:continuity} $\phi_0(y,x)$ of the form
    \[\phi_0(y,x) = \chi(y) \psi_0(x).\]
    We work on the time interval $[0,1/4]$. Then we see that if $\psi: [0,1/4] \times \T^{d-1} \to \R$ solves
    \[\begin{cases}\partial_t \psi + \sum_{j=0}^\infty v^{j}(t,x) \cdot \nabla_x \psi =0,\\ \psi(0,x) = \psi_0(x),\end{cases}\]
    then $\phi: [0,1/4] \times \T^d \to\R$ given by
    \[\phi(t,y,x) := \chi(y-t) \psi(y \lor 0,x)\]
    is a solution to~\eqref{eq:continuity}, where we use that $U|_{[-1/4,0] \times \T^{d-1}}=(1,0)$. Thus, in this case, Item~\ref{item:pde nonunique auto} follows from Corollary~\ref{cor:nonunique transport} and the definition of the $v^{j}$.
\end{proof}

\subsection{Proof of Proposition~\ref{prop:main separation estimate}: separation by quantitative central limit scaling}
\label{ss:separation}

We now prove Proposition~\ref{prop:main separation estimate}, which constitutes the main technical contribution of this section. As discussed in Subsection~\ref{ss:nonunique}, we need to let the noise drive the separation for some initial increment of time, after which the velocity field drives the separation. We will prove the velocity-field-driven separation through comparison with a Gaussian process. For that we need a quantitative version of the Donsker invariance principle (or, equivalently, a process-level CLT), which is provided by the following result. Theorem~\ref{thm:quantitative donsker} is a specialization of the main result of~\cite{gotze_bounds_2008} taking $\ep = 1/2$ and splitting $n$ into $\lfloor n^{1-2/p} (\log n)^2\rfloor$ many blocks of roughly equal size.
\begin{theorem}[{\cite[Theorem 4]{gotze_bounds_2008}}]
\label{thm:quantitative donsker}
    Let $n,d \in \N$ and $p > 2$. Let $(D_j)_{1 \leq j \leq n}$ be independent $\R^d$ valued random variables. We suppose that there exists $\lambda,\Lambda >0$ such that the following holds for all $1 \leq j \leq n$:
    \begin{enumerate}
        \item $\E D_j =0$.
        \item For all $v \in \R^d$, $\E (v \cdot D_j)^2 \geq \lambda |v|^2$.
        \item $\E |D_j|^p \leq \Lambda^{p/2}$.
    \end{enumerate}
    Then there exists $C(p,d, \Lambda/\lambda)>0$ and (on perhaps a larger probability space) mutually independent $\R^d$-valued Gaussian random variables $(Z_j)_{1 \leq j \leq n}$ (though \textit{not independent of the $D_j$}) so that
    \begin{enumerate}
        \item $\E Z_j =0$,
        \item $\E Z_j \otimes Z_j = \E D_j \otimes D_j$,
    \end{enumerate}
    and we have the approximation bound
    \[\E\bigg[\sup_{1\leq k\leq n} \Big|\sum_{j=1}^k (D_j-Z_j)\Big|^p\bigg]^{1/p}\leq C \Lambda^{1/2} n^{1/p}.\]
\end{theorem}

We also need appropriate estimates on the probability the Gaussian process enters a prescribed ball. The following lemma is a slight generalization of standard estimates on random walks, see e.g.~\cite[Sections 4.3 and 6.5]{lawler_random_2010} for analogous arguments.
\begin{lemma}
\label{lem:gaussian random walks}
    Let $d \geq 3$ and let $(Z_\ell)_{\ell \geq 1}$ be a sequence of independent $\R^d$-valued Gaussian random variables with
    \[\E Z_\ell = 0 \quad \text{and} \quad \lambda I \leq \E Z_\ell \otimes Z_\ell \leq \Lambda I.\]
    Let $A_n := \sum_{\ell=1}^n Z_\ell$. Then there exists $C(\Lambda/\lambda,d)>0$ such that for all $x \in \R^d$, $k \geq 1, r >0$,
    \begin{align}
        \P\big(\exists {\ell \geq 1}, A_\ell \in B_r(x)\big)&\leq C \Big(\frac{r}{|x|}\Big)^{d-2}
        \label{eq:probability of getting close to a point}
        \\\P\big(A_k \in B_r(x)\big) &\leq C \Big(\frac{r}{\sqrt{k \Lambda}}\Big)^d.
        \label{eq:density bound Gaussian}
    \end{align}
\end{lemma}

\begin{proof}
    We note that~\eqref{eq:density bound Gaussian} follows directly from the $L^\infty$ bound on the explicit density of $A_k$, using the bounds on the covariance. As such, we focus on~\eqref{eq:probability of getting close to a point}.

    We suppose without loss of generality that $|x| \geq 4r$. We note that
    \begin{align*}
    \E |\{k \geq 1: A_k \in B_{2r}(x)\}|&= \sum_{k\geq 1} \P(A_k \in B_{2r}(x)) 
    \\&\leq C\sum_{k\geq 1} (k\Lambda)^{-d/2} \int_{B_{2r}(x)} \exp\Big({-C^{-1}} k^{-1}\Lambda^{-1} |y|^2\Big)\,dy 
    \\&\leq C \Lambda^{-d/2} r^d \sum_{k \geq 1} k^{-d/2} e^{-C^{-1} \Lambda^{-1} |x|^2  k^{-1}}.
    \end{align*}
    Note then that
    \[\sum_{k\geq1} k^{-d/2} e^{-ak^{-1}} \leq C \int_1^\infty  x^{-d/2} e^{- ax^{-1}}\,dx = Ca^{1-d/2}\int_{a^{-1}}^\infty y^{-d/2} e^{-y^{-1}}\,dy \leq C a^{1-d/2}.\]
    Thus 
    \[  \E |\{k \geq 1: A_k \in B_{2r}(x)\}|\leq C \Lambda^{-1} |x|^2 \Big(\frac{r}{|x|}\Big)^d.\]
    On the other hand, using the Markov structure, for any $n \in \N$,
    \[\E |\{k \geq 1: A_k \in B_{2r}(x)\}| \geq (n+1) \Big(\P\big(\exists k \geq 1, A_k \in B_r(x)\big) \inf_{k \geq 1} \P\big(\sup_{1 \leq \ell \leq n} |A_{k+\ell} - A_k| \leq r\big)\Big),\]
    since the right hand side captures exactly the case that $A_k$ hits $B_r(x)$ and then moves less than $r$ in the next $n$ steps, hence contributes at least $n+1$ many visits to $B_{2r}(x).$ However, by Doob's martingale inequality, 
    \[\P\big(\sup_{1 \leq \ell \leq n} |A_{k+\ell} - A_k| \leq r\big) \geq 1- C\frac{\sqrt{n \Lambda}}{r}.\]
    Thus taking $n = \lfloor C^{-1} \Lambda^{-1} r^2\rfloor$ and combining the three displays above, we see that
\[
        \P\big(\exists k \geq 1, A_k \in B_r(x)\big) \leq C \Lambda r^{-2} \E |\{k \geq 1: A_k \in B_{2r}(x)\}| \leq C  \Big(\frac{r}{|x|}\Big)^{d-2}, \]
    as claimed.
\end{proof}

In addition to comparing to a Gaussian process, we also need to couple our original process to a ``more independent'' process, where the step size does not degenerate if the particles get too close together; see Section~\ref{ss:nonunique} for more discussion. We now construct this process. We call the ``forcing'' path $\Gamma_t$ to emphasize that here it is taken purely deterministically. For Proposition~\ref{prop:main separation estimate}, we will take $\Gamma = \sqrt{\kappa} W$, a scaled Brownian motion. 

Fix $j \in \N$. Let
\[v^j(t,x) := \sum_{\ell \in \N} \eta(2^{\beta j} t-\ell) \tilde u^{j,\ell}(x),\]
with $\tilde u^{j,\ell}$ as defined in the beginning of the section. We view $v^j : \R \times \R^d \to \R^d$, by extending periodically in space. Let $v^{j,1}, v^{j,2}$ be iid copies of $v^j$, and let $x^1, x^2 \in \R^d$. Let $Y^i_t$ solve
\[\begin{cases}
    \dot Y^i_t = v^{j,i}(t,Y^i_t) + \dot \Gamma_t,\\
    Y^i_0 = x^i.
\end{cases}
\]

We now pass the desired estimates on the ``more independent'' process $(Y^1_t, Y^2_t)$ by comparing with a Gaussian process using Theorem~\ref{thm:quantitative donsker} and then using the control on the Gaussian process given by Lemma~\ref{lem:gaussian random walks}.

\begin{lemma}\label{lem:estimates for comparison process}
    Let $j \in \N,x^1, x^2 \in \R^d$, $\Gamma : [0,\infty) \to \R^d$ with $\Gamma_0=0$, $\gamma \in (0,1)$. Suppose that for some $\delta>0$,
    \[\|\Gamma\|_{\mathcal{C}^\gamma_t} \leq 2^{(\gamma \beta -1 - \delta) j}.\]
    Then there exists $ C(\alpha,\beta,\delta,d)>0$ such that if $j \geq C$, $M \geq 1$, letting $N^j$ be as in~\eqref{eq:Tj properties}, we have the bound,
    \begin{align*}&\P_U\big(\inf_{0 \leq t \leq {N^j2^{-\beta j}}} \inf_{\substack{s,r\in [0,t]\\|s-r|\leq 2^{-\beta j}}} |Y^1_s- Y^2_r| \leq 2^{1-j}\big) + \P_U\big(|Y^1_{N^j2^{-\beta j}} - Y^2_{N^j2^{-\beta j}}| \leq M2^{-j})  \\&\qquad\qquad\qquad\qquad\qquad\qquad\qquad\qquad\qquad\qquad\qquad\leq C 2^{-C^{-1} j} + C \frac{2^{-j}}{|x^1 - x^2|} + C M 2^{ -j^{1/2}/2}.
    \end{align*}
\end{lemma}

\begin{proof}
    Let $\tau := 2^{-\beta j}$. Note that for $n \in \N,$
    \begin{equation}
    \label{eq:X i splitting}
    Y^i_{\tau n} - \Gamma_{\tau n}= x^i + \sum_{\ell=1}^n D^i_\ell,\end{equation}
    where 
    \[D^i_\ell := \int_{\tau(\ell-1)}^{\tau \ell} v^{j,i}(s,Y^i_s)\,ds.\]
    Let $Y_\ell : [0,\tau] \to \T^d$ solve the integral equation
    \[Y_{\ell,t}=\int_0^t\eta(\tau^{-1} s) \tilde u^{j,\ell}(Y_{\ell,s})\,ds + \Gamma_{t + (\ell-1)\tau} -\Gamma_{(\ell-1)\tau}  \]
    and let
    \[D_\ell := \int_0^\tau  \eta(\tau^{-1} t) \tilde u^{j,\ell}(Y_{\ell,t})\,dt.\]
    Then note by the translation invariance and refreshing structure of $u^j$, $D^1_\ell\stackrel{d}{=} D^2_\ell \stackrel{d}{=} D_\ell$ and that the $D^i_\ell$ are mutually independent. We seek then to control the $Y^i_t$ process on the discrete times $Y^i_{\tau n}$ by comparing with a Gaussian process using Theorem~\ref{thm:quantitative donsker}, for which we will need control on the moments of the random increments, which have the same law as $D_\ell$. 
    
    We first verify that $D_\ell$ is centered. Define $\Phi^\ell_t: \R^d \to \R^d$ by
    \[\begin{cases}
            \dot \Phi^\ell_t(x) = \eta(\tau^{-1} t) \tilde u^{j,\ell}(\Phi^\ell_t(x)) + \dot \Gamma_{t + (\ell-1)\tau}\\
            \Phi^\ell_0(x)=x,
    \end{cases}\]
    so that $Y_{\ell,t} = \Phi^\ell_t(0).$ Note that the translation invariance of $\tilde u^{j,\ell}$ in law implies that $\tilde u^{j,\ell} \circ \Phi^\ell_t$ is also translation invariant in law. Then we note that 
    \begin{align*}
       \E_U D_\ell &= \int_0^\tau \eta(\tau^{-1} t)   \E_U \tilde u^{j,\ell}(\Phi^\ell_t(0))\,dt 
       \\&= \int_0^\tau \eta(\tau^{-1} t)  \E_U \int_{[0,1]^d} \tilde u^{j,\ell}(\Phi^\ell_t(y))\,dy\,dt 
       \\&=  \int_0^\tau \eta(\tau^{-1} t)  \E_U \int_{(\Phi^\ell_t)^{-1}([0,1]^d)} \tilde u^{j,\ell}(z)\,dz\,dt 
       \\&=  \int_0^\tau \eta(\tau^{-1} t)  \E_U \int_{[0,1]^d} \tilde u^{j,\ell}(z)\,dz\,dt
       \\&=0.
    \end{align*}
    In the above, the second equality follows from the translation invariance in law of $\tilde u^{j,\ell} \circ \Phi^\ell_t$, the third equality by the velocity field being divergence-free hence $\Phi^\ell_t$ volume preserving, the fourth equality by the periodicity of the velocity field (and hence that $\Phi^\ell_t$ defines a map $\T^d \to \T^d$), and the final equality from the centeredness of $\tilde u^{j,\ell}$.

    We then let $A_n \in \R^{2d}$ be given by
    \begin{equation}
        \label{eq:A def}
           A_n =\sum_{\ell=1}^n \big( D^1_\ell, D^2_\ell\big),
    \end{equation}
    so that
    \begin{equation}
        \label{eq:Y in terms of A} 
        (Y^1_{\tau n}, Y^2_{\tau n}) = (x^1, x^2) +  (\Gamma_{\tau n}, \Gamma_{\tau n}) + A_n.
    \end{equation}
    We then denote
    \[R_\ell := \E \big( D^1_\ell, D^2_\ell\big) \otimes \big( D^1_\ell, D^2_\ell\big).\]
    Note that $A_n$ is a sum of independent increments, and we claim that for all $p\geq 2$, there exists $C(p,\alpha,\beta)>0$ such that
    \begin{equation}
    \label{eq:claim about moments}2^{-2(\alpha+\beta) j-1} I \leq R_\ell \leq  2^{-2(\alpha+\beta) j+1}I\quad \text{and}\quad  \Big(\E \big|\big( D^1_\ell, D^2_\ell\big)\big|^p \Big)^{1/p} \leq C 2^{ - (\alpha +\beta)j} ,\end{equation}
    where we mean the bound on $R_\ell$ as a quadratic form.

    For the covariance bound, we control $\E_U D_\ell \otimes D_\ell$, from which the claimed bound directly follows, using that $D^1_\ell \indep D^2_\ell$. We first note that for $t \in [0,\tau]$, 
    \[|Y_{\ell,t}| \leq 2^{-\gamma \beta j} \|\Gamma\|_{\mathcal{C}^\gamma_t} + 2^{-\beta j} \|\tilde u^{j,\ell}\|_{\mathcal{C}^0_x},\]
    so for $1 \leq i, k \leq d$,
    \begin{align*}\E_U D_{\ell,i} D_{\ell, k} &= \E \int_0^\tau \int_0^\tau \eta(\tau^{-1} t)  \eta(\tau^{-1} s) \tilde u^{j,\ell}_i(Y_{\ell,t}) \tilde u^{j,\ell}_k(Y_{\ell,s})\,dt\,ds
    \\&=  \int_0^\tau \int_0^\tau \eta(\tau^{-1} t)  \eta(\tau^{-1} s)  \E_U \tilde u^{j,\ell}_i(0) \tilde u^{j,\ell}_k(0)\,dt\,ds + F_{\ell,i,k}
    \\&= 2^{-2(\alpha+\beta) j} \delta_{i k} + F_{\ell,i,k},
    \end{align*}
    where
    \[F_{\ell,i,k} := \int_0^\tau \int_0^\tau \eta(\tau^{-1} t)  \eta(\tau^{-1} s)  \E_U\big( \tilde u^{j,\ell}_i(Y_{\ell,t}) \tilde u^{j,\ell}_k(Y_{\ell,s})-  \tilde u^{j,\ell}_i(0) \tilde u^{j,\ell}_k(0)\big)\,dt\,ds.\]
    Thus
    \begin{align*}
    |F_{\ell,i,k}| &\leq C 2^{-2\beta j} \E_U \|\tilde u^{j,\ell}\|_{\mathcal{C}^0_x} \|\tilde u^{j,\ell}\|_{\mathcal{C}^1_x} \big(2^{-\gamma \beta j} \|\Gamma\|_{\mathcal{C}^\gamma_t} + 2^{-\beta j} \|\tilde u^{j,\ell}\|_{\mathcal{C}^0_x}\big) 
    \\&\leq C 2^{-2(\beta + \alpha - \ep)j} \big(2^{-\delta j} + 2^{(1+\ep - \beta -\alpha) j} \big) 
    \\&\leq C 2^{-2(\beta + \alpha + C^{-1}) j},
    \end{align*}
    where for the final inequality we take $\ep$ small enough, using that $\alpha + \beta>1$. Then using that $j \geq C(\delta, \alpha,\beta)$, we get the claimed covariance bound for $\E_U D_\ell \otimes D_\ell$ and hence for $R_\ell$.

    For the $p$th moment bound, we take a similar first order expansion:
    \begin{align*}\|D_\ell\|_{L^p_{\omega,U}} &\leq \int_0^\tau \eta(\tau^{-1} t) \|\tilde u^{j,\ell}(0)\|_{L^p_{\omega,U}}\,dt + \int_0^\tau \eta(\tau^{-1} t) \big\|\tilde u^{j,\ell}(Y_{\ell,t}) - \tilde u^{j,\ell}(0)\big\|_{L^p_{\omega,U}}\,dt
    \\&\leq C2^{-(\alpha + \beta) j} + 2^{-\beta j} \|\nabla \tilde u^{j,\ell}\|_{L^{2p}_{\omega,U} \mathcal{C}^0_x} \| 2^{-\gamma \beta j} \|\Gamma\|_{\mathcal{C}^\gamma_t} + 2^{-\beta j} \|\tilde u^{j,\ell}\|_{\mathcal{C}^0_x}\|_{L^{2p}_{\omega,U}}
    \\&\leq  C2^{-(\alpha + \beta) j} + C2^{(-\alpha-\beta + \ep) j} \big( 2^{-\delta j} + 2^{(1-\alpha-\beta) j}\big)
    \\&\leq C 2^{-(\alpha + \beta)j},
    \end{align*}
    where on the last line we use that $\alpha + \beta > 1$ and choose $\ep$ sufficiently small. We thus conclude the claim~\eqref{eq:claim about moments}.

    Theorem~\ref{thm:quantitative donsker} then gives that there exists $(Z_\ell)_{1 \leq \ell \leq N^j}$ independent Gaussians with $\E Z_\ell =0$ and $\E Z_\ell \otimes Z_\ell = R_\ell$ such that for
    \[B_k := \sum_{\ell=1}^k Z_\ell,\]
    we have that
    \[\E_U \big[ \sup_{1 \leq k \leq N^j} |A_k - B_k|^p\big]^{1/p} \leq C(p,d) 2^{-(\alpha + \beta)j} (N^j)^{1/p} \leq C(p,\alpha,\beta,d) 2^{(\ep - \alpha -\beta)j},\]
    where $\ep(p)>0$ can be taken arbitrarily small by choosing $p$ large. By Chebyshev, using that $\alpha + \beta>1$, we get for some $C(\alpha,\beta,d,\delta)>0$,
    \begin{equation}
    \label{eq:quantitative donsker used}
    \P_U\big(  \sup_{1 \leq k \leq N^j} |A_k - B_k| \geq 2^{-j}\big) \leq C 2^{-C^{-1} j}.
    \end{equation}
    We thus have that (the random part of) $(Y^1,Y^2)$ is close to a Gaussian process; we now want to pass estimates on the Gaussian process to control the $(Y^1,Y^2)$ process. Write $B_k = (B^1_k, B^2_k)$ for $B^1_k,B^2_k \in \R^d$. Then by~\eqref{eq:Y in terms of A} and~\eqref{eq:quantitative donsker used},
    \begin{align*}
    \P_U\big(|Y_{N^j 2^{-\beta j}}^1 - Y^2_{N^j2^{-\beta j}}| \leq M 2^{-j}\big) &\leq \P_U\big(B^1_{N^j} - B^2_{N^j} \in B_{4M2^{-j}}(x^1-x^2)\big) 
    \\&\qquad\qquad+ \P_U\big(  \sup_{1 \leq k \leq N^j} |A_k - B_k| \geq 2^{-j}\big) 
    \\&\leq  \P_U\big(B^1_{N^j} - B^2_{N^j} \in B_{4M2^{-j}}(x^1-x^2)\big) + C 2^{-C^{-1} j}.
    \end{align*}
    We next compute that
    \begin{align*}
        \P_U\big(\inf_{0 \leq t \leq {N^j2^{-\beta j}}} \inf_{\substack{s,r\in [0,t]\\|s-r|\leq 2^{-\beta j}}} |Y^1_{s}- Y^2_{r}| \leq 2^{1-j}\big)&\leq \P_U\big(\exists 0 \leq k \leq N^j, B^1_k - B^2_{k} \in B_{2^{3-j}}(x^1-x^2)\big) 
        \\&\qquad\qquad+ \P_U\big(2^{-\beta j}\sup_{0 \leq \ell \leq N^j}\|\tilde u^{j,\ell}\|_{\mathcal{C}^0_x} \geq 2^{-j-1}\big)
    \\&\qquad\qquad+ \P_U\big(  \sup_{1 \leq k \leq N^j} |A_k - B_k| \geq 2^{-j}\big),
    \end{align*}
    where the first term corresponds to the $B_k^i$ getting close together, the second term corresponds to the $Y^i_t$ getting close together in the intermediate times between $\ell \tau$ and $(\ell+1)\tau$---using that as $|s-r| \leq 2^{-\beta j}$, we have that $|\Gamma_s - \Gamma_r| \leq 2^{-\gamma \beta j} \|\Gamma\|_{\mathcal{C}^\gamma_t} \leq 2^{-(1+\delta)j} \leq 2^{-j-1}$---and the final term corresponds to the $B_k$ not tracking well the $A_k$ (and hence that $Y^i_t$).
    For the final term, we use~\eqref{eq:quantitative donsker used}; for the second term we use Chebyshev,
    \[\P_U\big(2^{-\beta j}\sup_{0 \leq \ell \leq N^j}\|\tilde u^{j,\ell}\|_{\mathcal{C}^0_x} \geq 2^{-j-1}\big) \leq 2^{1+(1-\beta)j}\E_U\big|\sup_{0 \leq \ell \leq N^j} \|\tilde u^{j,\ell}\|_{\mathcal{C}^0_x}\big| \leq C 2^{(1 + \ep - \beta - \alpha)j} \leq C 2^{-C^{-1} j},\]
    where we again use that $\alpha + \beta >1$ and the moment bounds on $\|\tilde u^{j,\ell}\|_{\mathcal{C}^0_x}$. Thus we have that
    \[         \P_U\big(\inf_{0 \leq t \leq {N^j2^{-\beta j}}} \inf_{\substack{s,r\in [0,t]\\|s-r|\leq 2^{-\beta j}}} |Y^1_{s}- Y^2_{r}| \leq 2^{1-j}\big)\leq \P_U\big(\exists 0 \leq k \leq N^j, B^1_k - B^2_{k} \in B_{2^{3-j}}(x^1-x^2)\big) + C 2^{-C^{-1}j}.\]
    Combining the above, we have that 
    \begin{align*}
    &\P_U\big(\inf_{0 \leq t \leq {N^j2^{-\beta j}}} \inf_{\substack{s,r\in [0,t]\\|s-r|\leq 2^{-\beta j}}}|Y^1_{s}- Y^2_{r}| \leq 2^{1-j}\big) + \P_U\big(|Y^1_{N^j2^{-\beta j}} - Y^2_{N^j 2^{-\beta j}}| \leq M 2^{-j}) 
    \\&\qquad\leq C2^{- C^{-1} j} + \P_U\big(\exists 0 \leq k \leq N^j, B^1_k - B^2_{k} \in B_{2^{3-j}}(x^1-x^2)\big) 
    \\&\qquad\qquad\qquad\qquad\qquad\qquad\qquad+  \P_U\big(B^1_{N^j} - B^2_{N^j} \in B_{4M2^{-j}}(x^1-x^2)\big)
    \\&\qquad\leq  C 2^{-C^{-1} j} + C \frac{2^{-j}}{|x^1 - x^2|} + C \frac{M 2^{-j}}{ 2^{-(\alpha + \beta)j} \sqrt{N^j}}
    \\&\qquad\leq C 2^{-C^{-1} j} + C \frac{2^{-j}}{|x^1 - x^2|} + C M 2^{ -j^{1/2}/2},
    \end{align*}
    where we use~\eqref{eq:probability of getting close to a point} and~\eqref{eq:density bound Gaussian} of Lemma~\ref{lem:gaussian random walks} for the second to last inequality and the bound on $N^j$ of~\eqref{eq:Tj properties} for the final inequality. This then is exactly the desired conclusion.
\end{proof}

We now write the equation a particle in the velocity field $U$ solves, again with an abstract deterministic forcing $\Gamma$; we will later take $\Gamma = \sqrt{\kappa}W$.

\begin{equation}
    \label{eq:forced ODE}
    \dot X_t = U(t,X_t) + \dot \Gamma_t.
\end{equation}

We now provide the ``one scale'' separation estimate, showing that if particles (with the same forcing) are separated to the correct scale at time $T^{j+1}$, then they will also be separated to the (larger) correct scale at time $T^j$, with suitably high probability. This estimate will then be iterated inductively to conclude Proposition~\ref{prop:main separation estimate}.

\begin{lemma}
\label{lem:one scale separation}
    Let $j \in \N$, $\Gamma : [0,\infty) \to \R^d, \gamma \in (0,1)$. Suppose that for some $\delta>0$,
    \[\|\Gamma\|_{\mathcal{C}^\gamma_t} \leq 2^{(\gamma \beta -1 - \delta) j}.\]
    Then there exists $ C(\alpha,\beta,\delta,d)>0$ such that if $j \geq C$, and $X^1,X^2$ are both solutions to~\eqref{eq:forced ODE} with $X^1_{T^{j+1}}, X^2_{T^{j+1}}$ independent of $U|_{[T^{j+1}, T^j]}$, and
    \[d_{\T^d}(X^1_{T^{j+1}}, X^2_{T^{j+1}}) \geq 2^{-(j+1) + (j+1)^{1/2}/4}.\]
    Then 
    \[\P_U\big(d_{\T^d}\big(X^1_{T^j}, X^2_{T^j}\big) \leq 2^{-j + j^{1/2}/4} \big) \leq C 2^{-j^{1/2}/4}.\]
\end{lemma}

\begin{proof}
    If $d_{\T^d}(X^1_{T^{j+1}}, X^2_{T^{j+1}}) \geq 1/4$, then we have that (for $j \geq C$)
    \[\P_U\big(d_{\T^d}\big(X^1_{T^j}, X^2_{T^j}\big) \leq 2^{-j + j^{1/2}/4} \big) \leq \P_U\big(d_{\T^d}\big(X^1_{T^j}, X^2_{T^j}\big) \leq 1/8\big) \leq \P_U(\|u^j\|_{\mathcal{C}^0_{t,x}} \geq 1/8) \leq C 2^{-C^{-1} j}.\]
    Thus we can assume without loss of generality that $d_{\T^d}(X^1_{T^{j+1}}, X^2_{T^{j+1}}) \leq 1/4.$

    We lift $X^i_t$ to $\R^d$ in the natural way, so that $|X^1_{T^{j+1}} - X^2_{T^{j+1}}| = d_{\T^d}(X^1_{T^{j+1}}, X^2_{T^{j+1}})$ and $X^1_{T^{j+1}} \in [0,1)^d$. Note then that
    \[\P_U\big(d_{\T^d}\big(X^1_{T^j}, X^2_{T^j}\big) \leq 2^{-j + j^{1/2}/4} \big) \leq \P_U\big(|X^1_{T^j}- X^2_{T^j}| \leq 2^{-j + j^{1/2}/4} \big) + \P_U\big(|X^1_{T^j}- X^2_{T^j}| \geq 3/4\big).\]
    Since
    \[|X^1_{T^{j+1}} - X^2_{T^{j+1}}|  \leq 1/4,\]
     we have that 
     \[\P_U\big(|X^1_{T^j}- X^2_{T^j}| \geq 3/4\big) \leq \P_U\big(\|u^j\|_{\mathcal{C}^0_{t,x}} \geq 1/4\big) \leq C2^{-C^{-1}j}.\]
     Thus to conclude, it suffices to bound
     \[\P_U\big(|X^1_{T^j}- X^2_{T^j}| \leq 2^{-j + j^{1/2}/4} \big).\]
     Let $Y^1_t, Y^2_t$ be as in Lemma~\ref{lem:estimates for comparison process} with $x^i = X^i_{T^{j+1}}$ and with forcing $\tilde \Gamma_t := \Gamma_{t + T^{j+1}} - \Gamma_{T^{j+1}}.$ Let $\tilde Y^i_t := Y^i_{t - T^{j+1}}$. The difference then between $(X^1, X^2)$ and $(\tilde Y^1,\tilde Y^2)$ on $[T^{j+1},T^j]$ is that $\tilde Y^1 \indep \tilde Y^2$, while $X^1$ may not be independent of $X^2$. Let
     \begin{align*}
     \tau &:= \inf\big\{ t\geq T^{j+1} : \inf_{\substack{s,r\in [T^{j+1},t]\\|s-r|\leq 2^{-\beta j}}} |X^1_{s} - X^2_{r}|  \leq 2^{1-j}\,\text{or}\, \sup_{\substack{s,r\in [T^{j+1},t]\\|s-r|\leq 2^{-\beta j}}} |X^1_{s} - X^2_{r}|  \geq 3/4\big\},\\
     \\\sigma &:= \inf\big\{ t\geq T^{j+1} : \inf_{\substack{s,r\in [T^{j+1},t]\\|s-r|\leq 2^{-\beta j}}}|\tilde Y^1_{s} - \tilde Y^2_{r}|  \leq 2^{1-j}\,\text{or}\,\sup_{\substack{s,r\in [T^{j+1},t]\\|s-r|\leq 2^{-\beta j}}} |\tilde Y^1_{s} - \tilde Y^2_{r}|  \geq 3/4\big\}.\\
       \end{align*}
     We now argue that
     \begin{equation}\label{eq:stopped equality}(X^1_{\cdot \land \tau}, X^2_{\cdot \land \tau}) \stackrel{d}{=} (\tilde Y^1_{\cdot \land \sigma},\tilde Y^2_{\cdot \land \sigma})
     \end{equation}
     on $[T^{j+1}, T^j]$. Indeed, we first note that $(X^1_{\cdot \land \tau}, X^2_{\cdot \land \tau})$ and $(\tilde Y^1_{\cdot \land \sigma}, \tilde Y^2_{\cdot \land \sigma})$ are $C^0([T^{j+1},T^j])^2$ valued random variables taking values in the open subset
     \[S:=\Big\{(\gamma^1,\gamma^2): \inf_{\substack{s,r\in[T^{j+1},T^j]\\|s-r|\leq 2^{-\beta j}}} |\gamma^1_{s} - \gamma^2_{r}| >2^{-j},\sup_{\substack{s,r\in[T^{j+1},T^j]\\|s-r|\leq 2^{-\beta j}}} |\gamma^1_s - \gamma^2_r|<7/8\Big\}.\]
     To prove they are equal in law, it thus suffices to prove that their laws agree on a generating $\pi$-system for the relative Borel $\sigma$-algebra for $S$.

     For this purpose, for an open set $\mathcal{O}\subset [T^{j+1},T^j]\times \R^d$, let
     \[\Gamma_\mathcal{O}:=\big\{\gamma\in C^0([T^{j+1},T^j]): (t,\gamma(t))\in\mathcal{O}\text{ for every }t\in[T^{j+1},T^j]\big\}.\]
     That is, $\Gamma_\mathcal{O}$ is the set of all continuous paths whose trace lies in the space-time set $\mathcal{O}$. Then, let $P$ be the set of all sets of the form $A\cap (\Gamma_{\mathcal{O}^1}\times \Gamma_{\mathcal{O}^2})$ such that $A\subset C^0([T^{j+1},T^j])^2$ is open, and $\mathcal{O}^i\subset [T^{j+1},T^j]\times \R^d$ are open and additionally satisfy that 
     \[\inf_{(t^i,x^i)\in \mathcal{O}^i,|t^1-t^2|\leq 2^{-\beta j}}|x^1-x^2|>2^{-j},\qquad \sup_{(t^i,x^i)\in \mathcal{O}^i,|t^1-t^2|\leq 2^{-\beta j}}|x^1-x^2|<\frac{7}{8}.\]
     $P$ is clearly a $\pi$-system. Additionally, for any $(\gamma^1,\gamma^2)\in S$, using that the graphs of $(t,\gamma^i_t)$ are compact and that $S$ is open, there exist sets $\mathcal{O}^i$ of the form above such that $\gamma^i\in \Gamma_{\mathcal{O}^i}$. For any open $A\subset C^0([T^{j+1},T^j])^2$, it thus holds that $A\cap (\Gamma_{\mathcal{O}^1}\times \Gamma_{\mathcal{O}^2})\subset A\cap S$. This is exactly the statement that $P$ is a basis for the relative topology of $S$, and hence generates the relative Borel $\sigma$-algebra.

     To conclude~\eqref{eq:stopped equality}, we must thus show that the laws agree on $P$. Fixing such an $A$ and $\mathcal{O}^i$, by the finite range of dependence of $U$, we have that $U|_{\mathcal{O}^1}\indep U|_{\mathcal{O}^2}$. We then note that on the event $(X^1_{\cdot\wedge \tau}, X^2_{\cdot\wedge \tau})\in \Gamma_{\mathcal{O}^1}\times \Gamma_{\mathcal{O}^2}$, $(X^1_{\cdot\wedge \tau},X^2_{\cdot\wedge \tau})$ is solely determined by $(U|_{\mathcal{O}^1},U|_{\mathcal{O}^2})$ and the initial data. As the vector fields governing $\tilde Y^1$ and $\tilde Y^2$ are independent, and $(\tilde Y^1_{\cdot\wedge \sigma},\tilde Y^2_{\cdot\wedge \sigma})$ are solely determined by their restrictions to $\mathcal{O}^1$ and $\mathcal{O}^2$ and the initial data on the event that $(\tilde Y^1_{\cdot\wedge \sigma}, \tilde Y^2_{\cdot\wedge \sigma})\in \Gamma_{\mathcal{O}^1}\times \Gamma_{\mathcal{O}^2}$ as well, these facts imply that
     \[\P((X^1_{\cdot\wedge \tau}, X^2_{\cdot\wedge \tau})\in A\cap (\Gamma_{\mathcal{O}^1}\times \Gamma_{\mathcal{O}^2}))=\P((\tilde Y^1_{\cdot\wedge \sigma},\tilde Y^2_{\cdot\wedge \sigma})\in A\cap (\Gamma_{\mathcal{O}^1}\times \Gamma_{\mathcal{O}^2})),\]
     concluding the desired equality.
     
    Using~\eqref{eq:stopped equality}, we thus find that
    \begin{align*}
    \P_U\big(|X^1_{T^j}- X^2_{T^j}| \leq 2^{-j + j^{1/2}/4} \big) &\leq \P_U\big(|\tilde Y^1_{T^j}- \tilde Y^2_{T^j}| \leq 2^{-j + j^{1/2}/4} \big) + 2\P_U(\sigma \leq T^j)
    \\&\leq \P_U\big(|Y^1_{N^j2^{-\beta j}} - Y^2_{N^j2^{-\beta j}}| \leq  2^{-j + j^{1/2}/4}) 
    \\&\qquad+ 2\P_U\Big( \inf_{\substack{s,r\in[0,N^j2^{-\beta j}]\\|s-r|\leq 2^{-\beta j}}}|Y^1_{s} - Y^2_{r}|  \leq 2^{1-j}\Big)
    \\&\qquad + 2\P_U\big(\|u^j\|_{\mathcal{C}^0_{t,x}} \geq 1/4\big)
    \\&\leq C2^{-C^{-1}j} + C\frac{2^{-j}}{2^{-(j+1) + (j+1)^{1/2}/4}} + C 2^{j^{1/2}/4} 2^{-j^{1/2}/2}
    \\&\leq C 2^{-j^{1/2}/4},
    \end{align*}
    where we use Lemma~\ref{lem:estimates for comparison process} and the assumed separation of $|X^1_{T^{j+1}} - X^2_{T^{j+1}}|$.
\end{proof}

We will use the following lemma to lower bound the variance using the expected squared separation of our two particles given a coupling of their noise paths.

\begin{lemma}
\label{lem:variance compared to coupled}
    Let $X$ be a $\T^d$ valued random variable, $\tilde X$ an iid copy of $X$, and $Y$ identically distributed to $X$ (but not necessarily independent). Then
    \[\E d_{\T^d}(X,Y)^2 \leq 4 \E d_{\T^d}(X,\tilde X)^2.\]
\end{lemma}

\begin{proof}
    Let $Z$ be identically distributed to $X$ and independent of both $X$ and $Y$. Then
    \[\big(\E d_{\T^d}(X,Y)^2\big)^{1/2} \leq \big(\E d_{\T^d}(X,Z)^2\big)^{1/2} + \big(\E d_{\T^d}(Y,Z)^2\big)^{1/2} = 2 \big(\E d_{\T^d}(X,\tilde X)^2\big)^{1/2}.\]
    Squaring, we get the claim.
\end{proof}

Finally, we note the following lemma which is implied by the standard Nash estimate~\cite{nashContinuitySolutionsParabolic1958} on the $L^\infty$ norm of the density and the incompressibility of $U$. This will provide the needed estimate to allow the noise to drive the initial separation.

\begin{lemma}[{\cite{nashContinuitySolutionsParabolic1958}}]
\label{lem:nash estimate}
    There exists $C(d)>0$ such that surely in $U$, for all $t>0,$ $y,z \in \T^d,\kappa\in(0,1],$ and $r>0$,
    \[\P_W\big(X^\kappa_{t}(y) \in B_r(z)\big) \leq C\big(\big( (\kappa t)^{-1/2} + 1\big) r\big)^d.\]
\end{lemma}

We are finally ready to prove Proposition~\ref{prop:main separation estimate}. We build a special coupling of the noise paths of the two particles: independent up to time $T^{j_*}$ and identical after $T^{j_*}$. This allows us to let the noise drive separation up until time $T^{j_*},$ using Lemma~\ref{lem:nash estimate}, and then use the velocity field to drive the remaining separation, inductively using Lemma~\ref{lem:one scale separation}.

\begin{proof}[Proof of Proposition~\ref{prop:main separation estimate}]
    Fix $y \in \T^d$, $j,j_* \in \N$ with $j < j_*$. Let $W^1$ be a standard Brownian motion and let 
    \[W^2_t := \begin{cases} \tilde W_t & t \in [0,T^{j_*}],\\ W^1_t - W^1_{T^{j_*}} + \tilde W_{T^{j_*}} & t \in [T^{j_*},1],\end{cases}\]
     where $\tilde W$ is an independent Brownian motion. For $i=1,2$, we let $X^i$ solve~\eqref{eq:main SDE} with $W^i$ in place of $W$ and $\kappa = 2^{-(2\alpha + \beta)j_*}$.
     
     Our goal is to show that
     \begin{equation}
     \label{eq:coupled variance}
     \P_U \Big(\E_{W_1,W_2} d_{\T^d}\big(X^1_{T^j}, X^2_{T^j}\big)^2 \leq 2^{-2j+j^{1/2}/2-1}\big)  \leq C 2^{-j^{1/2}/8}.
     \end{equation}
     Noting that $X^1_t, X^2_t$ are identically distributed to $X^{2^{-(2\alpha+\beta)j_*}}_t$,~\eqref{eq:coupled variance} together with Lemma~\ref{lem:variance compared to coupled} allows us to conclude, throwing away the additional factor of $2^{j^{1/2}/2}$.
     
     In order to prove~\eqref{eq:coupled variance}, we note that on $[T^{j_*},1]$, $X^1, X^2$ are both solutions to~\eqref{eq:forced ODE} with $\Gamma := \sqrt{\kappa} W^1$ and data independent of $U$. Hence we are in the setting of Lemma~\ref{lem:one scale separation}. Let then $\gamma \in (0,1/2)$ such that $\gamma \beta < 1$ and $\alpha + (\gamma + 1/2) \beta >1$, which is possible as $\alpha + \beta>1$. Then choose $\delta>0$ such that $\alpha + (\gamma +1/2)\beta - \delta >1$. Then for any $j < j_*$, we have that
     \begin{align*}
         &\P_U\big(d_{\T^d}\big(X^1_{T^j}, X^2_{T^j}\big) \leq 2^{-j+ j^{1/2}/4} \text{ and } d_{\T^d}\big(X^1_{T^{j_*}}, X^2_{T^{j_*}}\big) \geq 2^{-j_* + j_*^{1/2}/4}  \big)
         \\&\qquad\leq \sum_{\ell=j}^{j_*-1} \P_U\big(d_{\T^d}\big(X^1_{T^\ell}, X^2_{T^\ell}\big) \leq 2^{-\ell+\ell^{1/2}/4} \text{ and } d_{\T^d}\big(X^1_{T^{\ell+1}}, X^2_{T^{\ell+1}}\big) \geq 2^{-(\ell+1) + (\ell+1)^{1/2}/4}  \big)
         \\&\qquad\leq  C\sum_{\ell=j}^{j_*-1} 2^{-\ell^{1/2}/4}  + \indc_{\|W^1\|_{\mathcal{C}^\gamma_t} \geq  \kappa^{-1/2}2^{(\gamma \beta - 1 - \delta)j_*}}
         \\&\qquad \leq C 2^{-j^{1/2}/8} + \indc_{\|W^1\|_{\mathcal{C}^\gamma_t} \geq  2^{C^{-1}j_*}}.
     \end{align*}
    Thus
\begin{align}
    &\P_U\big(d_{\T^d}\big(X^1_{T^j}, X^2_{T^j}\big) \leq 2^{-j+j^{1/2}/4}\big)
    \notag\\&\qquad \leq \P_U\big(d_{\T^d}\big(X^1_{T^j}, X^2_{T^j}\big) \leq 2^{-j+j^{1/2}/4} \text{ and } d_{\T^d}\big(X^1_{T^{j_*}}, X^2_{T^{j_*}}\big) \geq 2^{-j_* +j_*^{1/2}/4}  \big)
    \notag\\&\qquad\qquad\qquad+ \P_U\big(d_{\T^d}\big(X^1_{T^{j_*}}, X^2_{T^{j_*}}\big) \leq 2^{-j_* +j_*^{1/2}/4}\big)
    \notag\\&\qquad \leq C 2^{-j^{1/2}/8} +  \indc_{\|W^1\|_{\mathcal{C}^\gamma_t} \geq  2^{C^{-1}j_*}} + \P_U\big(d_{\T^d}\big(X^1_{T^{j_*}}, X^2_{T^{j_*}}\big) \leq 2^{-j_* +j_*^{1/2}/4}\big).
    \label{eq:iterated separation bound}
\end{align}
    Then we have by Lemma~\ref{lem:nash estimate} and the definitions of $\kappa$ and $T^{j_*}$,
    \begin{align*} \P_{W^1,W^2}\big(d_{\T^d}\big(X^1_{T^{j_*}}, X^2_{T^{j_*}}\big) \leq 2^{-j_* +j_*^{1/2}/4}\big) &\leq C \big(\kappa^{-1/2} (T^{j_*})^{-1/2} 2^{-j_* +j_*^{1/2}/4}\big)^d 
    \\&\leq C 2^{-dj_*^{1/2}/4}.
    \end{align*}
    Combining this with~\eqref{eq:iterated separation bound}, this shows that
    \begin{align*} &\P_{U,W^1,W^2}\big(d_{\T^d}\big(X^1_{T^j}, X^2_{T^j}\big) \leq 2^{-j+j^{1/2}/4}\big)
    \\&\qquad \leq  C 2^{-j^{1/2}/8} + \P\big(\|W^1\|_{\mathcal{C}^\gamma_t} \geq  2^{C^{-1}j_*}\big) + C 2^{- dj_*^{1/2}/4}
    \\&\qquad \leq  C 2^{-j^{1/2}/8} + C 2^{-C^{-1}j_*} + C 2^{- dj_*^{1/2}/4}
    \\&\qquad \leq C2^{-j^{1/2}/8},
    \end{align*}
    where we use Chebyshev and that $\E \|W^1\|_{\mathcal{C}^\gamma_t} \leq C$ for the second inequality and that $j < j_*$ for the third. We then have that
    \[\P_U\Big( \P_{W_1,W_2} \big(d_{\T^d}\big(X^1_{T^j}, X^2_{T^j}\big) \geq 2^{-j+j^{1/2}/4}\big) \geq 1/2\Big) \geq 1- C 2^{-j^{1/2}/8},\]
    which in turn gives that  
    \[\P_U\Big( \E_{W_1,W_2} \big(d_{\T^d}\big(X^1_{T^j}, X^2_{T^j}\big)\big)^2 \geq 2^{-2j+j^{1/2}/2-1}\Big) \geq 1- C 2^{-j^{1/2}/8}.\]
    This then gives directly~\eqref{eq:coupled variance}, and hence the proposition. 
\end{proof}

\subsection{Proof of Proposition~\ref{prop:variance lower bound implies nonunique transport}: Spontaneous stochasticity, anomalous dissipation, and transport nonuniqueness}

\label{ss:anomalous}

Fix $v : [0,1] \times \T^d \to \R^d$ with $\nabla \cdot v =0$ and $v \in L^\infty([0,1] \times \T^d)$. Let $X^\kappa_t(y)$ denote the solution to~\eqref{eq:main SDE} with $v$ in place of $U$. Fix $T >0$ such that
    \[\limsup_{\kappa \to 0} \int_{\T^d} \Var_W(X^\kappa_T(y))\,dy >0.\]
Let $\bar v : [0,T] \times \T^d \to\R^d$ be defined by 
\[\bar v(t,x):= -v(T-t,x).\]
We will then consider the equation.
\begin{equation}
    \label{eq:main advection diffusion}
    \partial_t f^\kappa - \tfrac{\kappa}{2} \Delta f^\kappa +\bar v \cdot \nabla f^\kappa =0.
\end{equation}

\begin{definition}
    We say $\bar v$ exhibits anomalous dissipation for some data $f_0 \in L^2(\T^d)$ if for $f^\kappa$ solving~\eqref{eq:main advection diffusion} with $f^\kappa(0,\cdot) = f_0(\cdot),$
    \[\limsup_{\kappa \to 0} \kappa \int_0^T \int_{\T^d} |\nabla f^\kappa|^2(t,x)\,dx\,dt >0.\]
\end{definition}

See~\cite[Lemma 2.8]{johansson_anomalous_2024} for a compact proof of the following fluctuation-dissipation relation. We note that in~\cite{johansson_anomalous_2024} this relation is stated for advection-diffusion equations and their (backward in time) stochastic Lagrangian trajectories. Although here our SDE solutions $X^\kappa_t$ are forward in time, by our definition of $\bar v$, these correspond to the backward stochastic trajectories for $f^\kappa$, after a time coordinate change. That is, going backward twice is the same as going forward.

\begin{proposition}[Fluctuation-dissipation relation{~\cite{drivas_lagrangian_2017}}]
\label{prop:fluctuation-dissipation}
    \[\kappa \int_0^T \int_{\T^d} |\nabla f^\kappa|^2(t,x)\,dx\,dt = \int \Var_W\big(f_0(X^\kappa_T(y))\big)\,dy.\]
\end{proposition}

The following result is given as~\cite[Theorem 1.3]{rowan_anomalous_2024} for $L^\infty_t L^2_x$ solutions to~\eqref{eq:transport from ad}. However it is clear from the argument that if the anomalous dissipation is exhibited for some bounded data, we get nonunique bounded solutions to~\eqref{eq:transport from ad}. We note here, as above, we are going backward twice, since the transport nonuniqueness given by~\cite[Theorem 1.3]{rowan_anomalous_2024} is for the time-reversal of the velocity field appearing in the equation~\eqref{eq:main advection diffusion}, which has already been time reversed.
\begin{proposition}[{\cite[Theorem 1.3]{rowan_anomalous_2024}}]
\label{prop:anomalous non-uniqueness equivalence}
    Suppose that $\bar v$ exhibits anomalous dissipation for some data $f_0 \in L^\infty(\T^d)$. Then there exists bounded initial data $\phi$ such that~\eqref{eq:transport from ad} (with velocity field $v$) admits at least two distinct bounded solutions.
\end{proposition}

We can now directly conclude Proposition~\ref{prop:variance lower bound implies nonunique transport}.

\begin{proof}[Proof of Proposition~\ref{prop:variance lower bound implies nonunique transport}]

Noting that
\[\limsup_{\kappa \to 0} \int_{\T^d} \Var_W(X^\kappa_T(y))\,dy  \leq C\limsup_{\kappa \to 0} \sum_{j=1}^d \int_{\T^d} \Var_W(\cos(2\pi X^\kappa_{T,j}(y))) + \Var_W(\sin(2\pi X^\kappa_{T,j}(y)))\,dy,\]
Proposition~\ref{prop:variance lower bound implies nonunique transport} follows directly from Propositions~\ref{prop:fluctuation-dissipation} and~\ref{prop:anomalous non-uniqueness equivalence}.
\end{proof}

\subsection*{Acknowledgments}

MC was supported by the Swiss State Secretariat for Education, Research and Innovation (SERI) under contract number MB22.00034 through the project TENSE. EHC was supported by NSF grant DMS-2342349.

The authors would like to thank Scott Armstrong for advice on an early draft.

\subsection*{AI use statement}

AI tools were used to aid in searching the literature, copy editing, and proofreading. The entire text of the paper was written by the authors.

\appendix

\section{Nonlinear Young integral theory}

\label{appen:young integrals}
We first recall the definition and fundamental estimate for nonlinear Young integrals. See~\cite{catellier_averaging_2016,hu_nonlinear_2017,galeati_noiseless_2021} for a general background on the theory.

\begin{proposition}[{\cite[Proposition 2.4]{hu_nonlinear_2017}}]\label{prop:nonlinear_young}
Let $\gamma,\nu,\rho\in(0,1]$ such that $\gamma+\nu\rho>1$. Then, if $A\in \mathcal{C}^\gamma([0,1],\mathcal{C}^\nu(\R^d))$ and $X\in \mathcal{C}^\rho([0,1]),$ for all $0\leq s<t\leq 1$, the nonlinear Young integral
\[\int_s^t A(dr,X_r):=\lim_{\|\pi\|\rightarrow 0} \sum_{i=0}^{n-1} \big(A(t_{i+1},X_{t_i})-A(t_i,X_{t_i})\big),\]
is well-defined, where above $\pi=\{s=t_0,\dotsc, t_n=t\}$ denotes a partition of $[s,t]$. Moreover, there exists $C(\gamma,\nu,\rho)>0$ such that
\[\bigg|\int_s^tA(dr,X_r)-(A(t,X_s)-A(s,X_s))\bigg|\leq C|t-s|^{\gamma+\nu\rho}\|A\|_{\mathcal{C}^\gamma([s,t],\mathcal{C}^\nu_x)}\|X\|_{\mathcal{C}^\rho([s,t])}^\nu.\]
\end{proposition}

Notably, if $A=\iop^Xv$ for some continuous vector field $v$ and curve $Y$ such that $A$ and $Y$ satisfy the conditions above, then
\[\int_0^t A(ds,Y_s)=\int_0^t v(s,X_s+Y_s)\,ds.\]

We now provide the proof of the mild generalization of~\cite[Theorem 4.8]{galeati_noiseless_2021} we use in the arguments above. The proof proceeds almost identically to that in~\cite[Theorem 4.8]{galeati_noiseless_2021}, although we now allow the operator to be decomposed into two components with different regularities in time.

\begin{proof}[Proof of Proposition~\ref{prop:effective Lipschitz Gronwall}]
First, we rewrite the difference $Z_t:=X^2_t-X^1_t$ as a solution to a nonlinear Young differential equation. We note that
\[X^1_t= y^1+\int_0^t v(s,X^1_s)\,ds+w^1_t=y^1+\int_0^t \iop^{X^1}v(ds,0)+w^1_t,\]
and
\[X_t^2= y^2+\int_0^t v(s,X_s^1+(X_s^2-X_s^1))\,ds+w^2_t=y^2+\int_0^t \iop^{X^1}v(ds,Z_s)+w^2_t.\]
Taking the difference and using the decomposition of $\iop^{X^1}v$ thus yields that
\[Z_t=(y^2-y^1)+\bigg(\int_0^t I^1(ds,Z_s)-I^1(ds,0) \bigg)+\bigg(\int_0^t I^2(ds,Z_s)-I^2(ds,0) \bigg)+(w_t^2-w_t^1).\]
Letting $\tilde{w}_t:=w^2_t-w^1_t$, $A^1(t,x):=I^1(t,x)-I^1(t,0),$ and $A^2(t,x):=I^2(t,x)-I^2(t,0)$ we have thus shown that
\begin{equation}\label{eq:transformed_ode}
Z_t=Z_0+\int_0^tA^1(ds,Z_s)+\int_0^t A^2(ds,Z_s)+\tilde{w}_t.    
\end{equation}
Additionally, by definition, $A^i(t,0)=0$ for all $t\in[0,1]$, $\|A^i\|_{\mathcal{C}^{\gamma_i}_t\mathcal{C}^1_x}\leq \|I^i\|_{\mathcal{C}^{\gamma_i}_t\mathcal{C}^1_x}<\infty$, and $Z\in \mathcal{C}^{\gamma_1}_t$. We have thus transformed the problem of showing stability for the ODE into an \textit{a priori} estimate for~\eqref{eq:transformed_ode}.

For simplicity, let $R^i:=\|A^i\|_{\mathcal{C}^{\gamma_i}_t\mathcal{C}^1_x}$. We now invoke Proposition~\ref{prop:nonlinear_young} to find that there exists $K(\gamma_1,\gamma_2)>0$ such that for all $s,t\in[0,1]$
\begin{align*}
|Z_t-Z_s|&\leq |A^1(t,Z_s)-A^1(s,Z_s)|+|A^2(t,Z_s)-A^2(s,Z_s)|
\\&\qquad+K(R^1|t-s|^{2\gamma_1}+R^2|t-s|^{\gamma_1+\gamma_2}\big)\|Z\|_{\mathcal{C}^{\gamma_1}([s,t])}+|\tilde{w}_t-\tilde{w}_s|\\
&\leq  \big(R^1|t-s|^{\gamma_1}+R^2|t-s|^{\gamma_2}\big) |Z_s|
\\&\qquad+K\big(R^1|t-s|^{2\gamma_1}+R^2|t-s|^{\gamma_1+\gamma_2}\big)\|Z\|_{\mathcal{C}^{\gamma_1}([s,t])}+|t-s|^{\gamma_1}\|\tilde{w}\|_{\mathcal{C}^{\gamma_1}_t},
\end{align*}
where in the second inequality we have used that $A^i(t,0)=A^i(s,0)=0$ to bound the first two terms. Dividing by $|t-s|^{\gamma_1}$ and using that $\gamma_1\leq \gamma_2$, this shows the local-in-time $\mathcal{C}^{\gamma_1}$ estimate
\begin{equation}\label{eq:local-Z-estimate}
\|Z\|_{\mathcal{C}^{\gamma_1}([s,t])}\leq (R^1+R^2|t-s|^{\gamma_2-\gamma_1})\|Z\|_{\mathcal{C}^0([s,t])}+K(R^1|t-s|^{\gamma_1}+R^2|t-s|^{\gamma_2})\|Z\|_{\mathcal{C}^{\gamma_1}([s,t])}+\|\tilde{w}\|_{\mathcal{C}^{\gamma_1}_t}.  
\end{equation}

If $K(R^1|t-s|^{\gamma_1}\vee R^2|t-s|^{\gamma_2})\leq \frac{1}{4}$, after absorbing the $\|Z\|_{\mathcal{C}^{\gamma_1}}$ factor on the right-hand side, this implies that
\begin{equation}\label{eq:Holder_by_L_infty}
\|Z\|_{\mathcal{C}^{\gamma_1}([s,t])}\leq 2((R^1+R^2|t-s|^{\gamma_2-\gamma_1})\|Z\|_{\mathcal{C}^0([s,t])}+\|\tilde{w}\|_{\mathcal{C}^{\gamma_1}_t}). 
\end{equation}
We thus want an estimate on $\|Z\|_{\mathcal{C}^0([0,1])}$.

Using the trivial estimate
\[\|Z\|_{\mathcal{C}^0([s,t])}\leq |Z_s|+|t-s|^{\gamma_1}\|Z\|_{\mathcal{C}^{\gamma_1}([s,t])},\]
~\eqref{eq:Holder_by_L_infty} implies that if $(K\vee 2)(R^1|t-s|^{\gamma_1}\vee R^2|t-s|^{\gamma_2})\leq \frac{1}{4}$, then
\begin{align*}\|Z\|_{\mathcal{C}^0([s,t])}&\leq |Z_s|+2(R^1|t-s|^{\gamma_1}+R^2|t-s|^{\gamma_2})\|Z\|_{\mathcal{C}^0([s,t])}+2|t-s|^{\gamma_1}\|\tilde{w}\|_{\mathcal{C}^{\gamma_1}_t} 
\\&\leq |Z_s|+\tfrac{1}{2}\|Z\|_{\mathcal{C}^0([s,t])}+2|t-s|^{\gamma_1}\|\tilde{w}\|_{\mathcal{C}^{\gamma_1}_t},
\end{align*}
and so
\[\|Z\|_{\mathcal{C}^0([s,t])}\leq 2|Z_s|+4|t-s|^{\gamma_1}\|\tilde{w}\|_{\mathcal{C}^{\gamma_1}_t}.\]
Partitioning $[0,1]$ into $N=1 \lor \lceil(4(K\vee 2) R^1)^{1/\gamma_1}\vee (4(K\vee2) R^2)^{1/\gamma_2}\rceil$ intervals, each of length less than $((4(K\vee 2) R^1)^{1/\gamma_1}\vee (4(K\vee 2)R^2)^{1/\gamma_2})^{-1}$, and iterating this estimate gives that
\[\|Z\|_{\mathcal{C}^0([0,1])}\leq 2^N(|Z_0|+4\|\tilde{w}\|_{\mathcal{C}^{\gamma_1}_t})\leq Ce^{C((R^1)^{1/\gamma_1}\vee(R^2)^{1/\gamma_2})}(|Z_0|+\|\tilde{w}\|_{\mathcal{C}^{\gamma_1}_t})\]
for some $C(\gamma_1,\gamma_2)>0$. Inserting this back into~\eqref{eq:local-Z-estimate}, we have found that
\begin{equation}
\label{eq:local-Z-estimate 2}
\|Z\|_{\mathcal{C}^{\gamma_1}([s,t])}\leq  Ce^{C((R^1)^{1/\gamma_1}\vee(R^2)^{1/\gamma_2})}(|Z_0|+\|\tilde{w}\|_{\mathcal{C}^{\gamma_1}_t}),
\end{equation}
for all $K(R^1|t-s|^{\gamma_1}\vee R^2|t-s|^{\gamma_2})\leq \frac{1}{4}$. Now, using that for any $\Delta \in (0,1]$,
\[\|Z\|_{\mathcal{C}^{\gamma_1}([0,1])}\leq  \Big(\frac{2}{\Delta}\Big)^{1-\gamma_1}\sup_{|t-s|\leq \Delta}\|Z\|_{\mathcal{C}^{\gamma_1}([s,t])},\]
we can take $\Delta = N^{-1}$, use~\eqref{eq:local-Z-estimate 2}, and absorb the algebraic power of $\Delta$ into the exponential to find that 
\[\|Z\|_{\mathcal{C}^{\gamma_1}([0,1])}\leq Ce^{C((R^1)^{1/\gamma_1}\vee(R^2)^{1/\gamma_2})}(|Z_0|+\|\tilde{w}\|_{\mathcal{C}^{\gamma_1}_t}).\]
Unpacking the definition of $Z_0,\tilde{w}$, $R^1$, and $R^2$, this implies the claim.
\end{proof}

Finally, we prove a version of the Bihari--LaSalle inequality that uses the effective H\"older regularity of a vector field. The proof is essentially a combination of the classical Bihari--LaSalle proof and the proof of~\cite[Theorem 4.8]{galeati_noiseless_2021}.

\begin{proof}[Proof of Proposition~\ref{prop:effective Holder Bihari}]
We begin with the same reduction and estimate as in the proof of Proposition~\ref{prop:effective Lipschitz Gronwall}. Letting $Z_t:=X^2_t-X^1_t$, we find that
\[Z_t=Z_0+\int_0^tA(ds,Z_s)+\tilde{w}_t,\]
where $Z_0:=y^2-y^1$, $A(t,x):=\iop^{X^1}v(t,x)-\iop^{X^1}v(t,0)$, and $\tilde{w}_t:=w^2_t-w^1_t$. Letting $R:=\|A\|_{\mathcal{C}^\gamma_t \mathcal{C}^\nu_x}$, Proposition~\ref{prop:nonlinear_young} then also implies that there exists $K(\gamma,\nu)>0$ such that
\[|Z_t-Z_s|\leq R|t-s|^\gamma |Z_s|^\nu+KR|t-s|^{\gamma(1+\nu)}\|Z\|_{\mathcal{C}^\gamma([s,t])}^\nu+|t-s|^\gamma\|\tilde{w}\|_{\mathcal{C}^\gamma_t}.\]
Dividing by $|t-s|^\gamma$ and taking the supremum over $0\leq s,t\leq T\leq 1$ with $t \ne s$, this gives 
\[\|Z\|_{\mathcal{C}^\gamma([0,T])}\leq R\|Z\|_{\mathcal{C}^0([0,T])}^\nu +KRT^{\gamma\nu}\|Z\|_{\mathcal{C}^\gamma([0,T])}^\nu+\|\tilde{w}\|_{\mathcal{C}^\gamma_t}.\]
Using the elementary estimate
\begin{equation}\label{eq:elementary BL}
a\leq b+ca^\nu\Rightarrow a\leq 2b+(2c)^{1/(1-\nu)}
\end{equation}
for $a,b,c\geq 0$, we find that
\[\|Z\|_{\mathcal{C}^\gamma([0,T])}\leq 2R\|Z\|_{\mathcal{C}^0([0,T])}^\nu+(2KR)^{\frac{1}{1-\nu}} T^{\frac{\nu\gamma}{1-\nu}}+2\|\tilde w\|_{\mathcal{C}^\gamma_t}.\]
Inserting this into the trivial estimate on $\|Z\|_{\mathcal{C}^0([0,T])}$ we find
\begin{align*}
\|Z\|_{\mathcal{C}^0([0,T])}&\leq |Z_0|+T^\gamma \|Z\|_{\mathcal{C}^\gamma([0,T])}
\\&\leq |Z_0|+T^\gamma2R\|Z\|_{\mathcal{C}^0([0,T])}^\nu+(2KR)^{\frac{1}{1-\nu}} T^{\frac{\gamma}{1-\nu}}+2T^\gamma\|\tilde w\|_{\mathcal{C}^\gamma_t}.
\end{align*}
Again using~\eqref{eq:elementary BL}, we thus find that
\[\|Z\|_{\mathcal{C}^0([0,T])}\leq C(|Z_0|+T^\gamma \|\tilde w\|_{\mathcal{C}^\gamma_t}+R^{\frac{1}{1-\nu}}T^{\frac{\gamma}{1-\nu}}).\]
Unpacking the definition of $Z_0$ and $\tilde{w}$, and using the bound $R=\|A\|_{\mathcal{C}^\gamma_t\mathcal{C}^\nu_x}\leq \|\iop^{X^1}v\|_{\mathcal{C}^\gamma_t\mathcal{C}^\nu_x}$, this implies the claim.
\end{proof}

\section{Estimates for the examples}
\label{appen:examples}

In this section of the appendix, we provide the proofs of the claimed properties of the random velocity field examples of Section~\ref{s:examples}.

\subsection{Gaussian fields}

We first prove that the autonomous Gaussian field with power law spectrum admits a finite range decomposition satisfying good moment estimates. The proof essentially follows the standard method of decomposing Gaussian fields into finite-range components, see~\cite{bauerschmidt_simple_2013} and \cite[Chapter 3]{bauerschmidt_renormalisation_2019}, plus some direct estimates.

\begin{proof}[Proof of Proposition~\ref{prop:autonomous gaussian}]

We first show the finite range decomposition. By the independence of $\zeta^k$, we have the exact formula
\[R(x):=\mathbb{E}[U^{\alpha,b}(y)\otimes U^{\alpha,b}(y+x)]=\sum_{k\in\Z^d\setminus \{0\}} |k|^{-d-2\alpha} e^{2\pi i k\cdot x}J^{k,b},\]
for the covariance matrix of $U^{\alpha,b}$. It thus suffices to decompose $R(x)$ into a sum $\sum_{j=0}^\infty R^j(x)$ such that for each $j$, $R^j(x)$ is the covariance matrix of a centered Gaussian field.

To this end, fix some radial, nonzero, zero-mean, $\phi\in C^\infty_c(B_{1/4})$ with Fourier transform $\hat\phi(\xi).$  Letting $\phi_t(\cdot):=t^{-d}\phi(\cdot/t)$, we thus have that $\phi_t*\phi_t$ is supported on $B_{t/2}(0)$ and $\hat{(\phi_t*\phi_t)}(\xi)=|\hat \phi(t\xi)|^2$. When $t\leq 1$, we can thus also view $\phi_t*\phi_t$ as a function on $\T^d$.

By scaling and symmetry, for all $k\neq 0$,
\[\int_0^\infty t^{d+2\alpha+2}|k|^2|\hat\phi(t k)|^2\frac{dt}{t}=|k|^{-d-2\alpha}\int_0^\infty r^{d+2\alpha+2}|\hat\phi(r e_1)|^2 \frac{dr}{r}=C|k|^{-d-2\alpha},\]
where $e_1$ is the unit vector, and $C>0$ is some $\alpha$ and $\phi$ dependent constant. Thus, up to normalizing $\phi$, we have that
\[\hat R(k)=\int_0^\infty t^{d+2\alpha+2}|k|^2J^{k,b}|\hat\phi(t k)|^2\frac{dt}{t}.\]
We then define the $R^j$ via their Fourier transforms: $\hat R^j(0)=0$ and for $k\neq 0$
\[\hat R^j(k):=\begin{cases}
    \int_1^\infty t^{d+2\alpha+2}|k|^2J^{k,b}|\hat\phi(t k)|^2\frac{dt}{t}&j=0,\\
    \int_{2^{-j}}^{2^{1-j}} t^{d+2\alpha+2}|k|^2J^{k,b}|\hat\phi(t k)|^2\frac{dt}{t}&j\geq 1.
\end{cases}\]
It thus holds that $R=\sum_{j\geq 0} R^j$.

On the other hand, since $J^{k,b}$ is positive semi-definite and all other terms are nonnegative, $R^j$ is a positive-definite matrix-valued kernel and hence the covariance kernel of a centered stationary Gaussian field $u^j(x)$. As well, taking the inverse Fourier transform, for $j\geq 1$ it holds that
\[R^j(x)=\frac{1}{(2\pi)^2}\int_{2^{-j}}^{2^{1-j}} t^{d+2\alpha+2}(-\Delta I+b\nabla^2)(\phi_t*\phi_t)(x)\frac{dt}{t}.\]
Thus, $R^j(x)$ is supported on $B_{2^{-j}}(0)$. With the Gaussianity of $u^j$, this implies that $u^j|_A\indep u^j|_{B}$ whenever $\mathrm{dist}(A,B)\geq 2^{-j}$.

Summarizing, if the scale fields $\{u^j\}_{j\geq 0}$ are chosen to be mutually independent, we have thus shown Items~\ref{item: autonomous scale decomp} and~\ref{item:autonomous finite range}. We now move on to Item~\ref{item: autonomous moment bounds}.

We prove the moment estimates for $j\geq 1$ as the estimate for $j=0$ follows \textit{mutatis mutandis}. First, for any $m\in\{0,1,2,3\}$ we have that
\begin{align*}
\mathbb{E}[|\nabla^m u^j(x)|^2]\leq C|\nabla^{2m}R^j(0)|&\leq C\int_{2^{-j}}^{2^{1-j}} t^{d+2\alpha+2} |\nabla^{2m}(-\Delta I+b\nabla^2)(\phi_t*\phi_t)|(0)\frac{dt}{t}
\\&\leq C\|\phi\|_{C^{m+2}_x}\int_{2^{-j}}^{2^{1-j}} t^{2(\alpha-m)}\frac{dt}{t}
\\&\leq C 2^{2(m-\alpha)j},
\end{align*}
where $C>0$ depends only on $\alpha$ and $d$ (through $\phi$). By the equivalence of Gaussian moments, this implies that for all $p\geq 1 $
\[\|\nabla^m u^j\|_{\mathcal{C}^0_x L^p_\omega}\leq C\sqrt{p} 2^{(m-\alpha)j}.\]  
As an immediate consequence, we find for any $p,q\geq 1$ that
\begin{equation}\label{eq:mixed exponent estimate}
\|\nabla^m u^j\|_{L^p_\omega L^q_x}\leq \|\nabla^m u^j\|_{L^{p\vee q}_xL^{p\vee q}_\omega}\leq \|\nabla^m u^j\|_{\mathcal{C}^0_x L^{p\vee q}_\omega}\leq C(p\vee q)^{1/2}2^{(m-\alpha)j}.    
\end{equation}

Now, fix $n\in\{0,1,2\}$ and let $q=2d\vee j$. Since $\hat R^j(0)=0$, the spatial average of $\nabla^n u^j$ vanishes almost surely. Hence, by the homogeneous Gagliardo--Nirenberg inequality, there exists $C(d)>0$ such that
\[\|\nabla^n u^j\|_{\mathcal{C}^0_x}\leq  C \|\nabla^n u^j\|_{L^q_x}^{1-\frac{d}{q}}\|\nabla^{n}u^{j}\|_{W^{1,q}_x}^\frac{d}{q}.\]
Taking powers and expectations, we thus find that
\begin{align*}
\|\nabla^n u^j\|_{L^p_\omega \mathcal{C}^0_x}&\leq  C \mathbb{E}\Big[\|\nabla^n u^j\|_{L^q_x}^{p-\frac{pd}{q}}\|\nabla^nu^{j}\|_{W^{1,q}_x}^\frac{pd}{q}\Big]^{1/p}
\\&\leq  C  \|\nabla^n u^j\|_{L^p_\omega L^q_x}^{1-d/q} \|\nabla^n u^j\|_{L^p_\omega W^{1,q}_x}^{d/q},
\end{align*}
where in the last line we have used H\"older's inequality. Inserting~\eqref{eq:mixed exponent estimate}, we have that
\[\|\nabla^n u^j\|_{L^p_\omega \mathcal{C}^0_x}\leq C(p\vee q)^{1/2} 2^{dj/q}  2^{(n-\alpha)j}\leq C(p^{1/2}+j^{1/2})2^{(n-\alpha)j},\]
where we have used our choice of $q$, and $C>0$ only depends on $\alpha$ and $d$, thus establishing Item~\ref{item: autonomous moment bounds}.

Finally, we note that if $b=1$, then Item~\hyperref[item:div free]{(i)} is immediate. Additionally, $U^{\alpha,b}(x)$ is a nondegenerate Gaussian random variable with the same law for all $x$, thus
\[\limsup_{\eps\rightarrow 0} \sup_{x\in\mathbb T^d}\eps^{-d}\P(|U^{\alpha,b}(x)|\leq \eps)<\infty.\]
Since $U^{\alpha,b}\in C^{\alpha-}_x$ almost surely, Lemma~\ref{lem:sufficient condition for good zero set} thus implies that for all $\alpha'<\alpha$,
\[\lim_{\eps\rightarrow 0} \eps^{-\alpha'd}|Z_\eps|=0,\]
as claimed in Item~\hyperref[item:zero level set]{(ii)}. This concludes the proof of the proposition.
\end{proof}

Next, we prove the finite range decomposition and moment estimates for the refreshing Gaussian field. The proof follows similarly to that of Proposition~\ref{prop:autonomous gaussian}, except now we do the finite range decomposition in time as opposed to space.

\begin{proof}[Proof of Proposition~\ref{prop:refreshing gaussian}]
We again first show the finite range decomposition. By stationarity and symmetry, it suffices to consider nonnegative
time increments. For $t\geq0$, by the independence of $\zeta^{k,\beta}$, we again have an exact formula
\[R(t,x):=\mathbb{E}[U^{\alpha,\beta,b}(s,y)\otimes U^{\alpha,\beta,b}(s+t,y+x)]=\sum_{k\in\Z^d\setminus \{0\}} |k|^{-d-2\alpha} e^{-|k|^\beta t} e^{2\pi i k\cdot x} J^{k,b},\]
for the covariance matrix of $U^{\alpha,\beta,b}$. It thus again also suffices to decompose $R(t,x)$ into a sum $\sum_{j=0}^\infty R^j(t,x)$ such that for each $j$, $R^j(t,x)$ is the covariance matrix of a centered Gaussian field.

To this end, let 
\[\phi_M(t):=\frac{((1-|t|)\vee 0)^{M}}{M!}\]
for some integer $M\geq1$ such that $\beta(M+1)> 2\alpha$. We then note that the function $\phi_M=\phi_1^M/M!$ is positive definite on $\R$, since $\phi_1=\indc_{[-1/2,1/2]}*\indc_{[-1/2,1/2]}$. Additionally, for all $k\neq 0$, after changing variables, it holds that
\begin{align*}
\int_0^\infty |k|^{\beta(M+1)} s^Me^{-|k|^\beta s}\phi_{M}(t/s)\,ds&=\frac{e^{-|k|^\beta t} }{M!}\int_0^\infty |k|^{\beta(M+1)}r^Me^{-|k|^\beta r}\,dr
\\&=\frac{e^{-|k|^\beta t} }{M!}\Gamma(M+1)
\\&=e^{-|k|^\beta t}.
\end{align*}
We thus have that
\[\hat R(t,k)=|k|^{-d-2\alpha} J^{k,b}\int_0^\infty |k|^{\beta(M+1)} s^Me^{-|k|^\beta s}\phi_{M}(t/s)\,ds.\]
We then define the $R^j$ via their Fourier transforms: $\hat R^j(t,0)=0$ and for $k\neq 0$
\[\hat R^j(t,k):=\begin{cases}
    |k|^{-d-2\alpha} J^{k,b}\int_{2^{-\beta}}^\infty |k|^{\beta(M+1)} s^Me^{-|k|^\beta s}\phi_{M}(t/s)\,ds&j=0,\\
    |k|^{-d-2\alpha} J^{k,b}\int_{2^{-\beta (j+1)}}^{2^{-\beta j}} |k|^{\beta(M+1)} s^Me^{-|k|^\beta s}\phi_{M}(t/s)\,ds&j\geq 1.
\end{cases}\]
It thus holds that $R=\sum_{j\geq 0} R^j$.

On the other hand, since $\phi_M(t/s)$ is positive definite in time, all coefficients are non-negative, and $J^{k,b}$ is positive semidefinite, $R^j$ is a positive-definite matrix-valued kernel and hence the covariance kernel of a centered stationary Gaussian field $u^j(t,x)$. Additionally, since $\phi_M(t/s)=0$ whenever $t\geq s$, for $t\geq0$ we have
$R^j(t,x)=0$ whenever $t\geq2^{-\beta j}$. Thus, by the Gaussianity of $u^j$, if $A,B\subset [0,1]$ are such that $\dist(A,B)\geq 2^{-\beta j}$ then $u^j|_{A\times \T^d}\indep u^j|_{B\times \T^d}.$

Summarizing, if the scale fields $\{u^j\}_{j\geq 0}$ are chosen to be mutually independent, we have thus shown Items~\ref{item: refreshing scale decomp} and~\ref{item:refreshing finite range}. We now move on to Item~\ref{item: refreshing moment bounds}.

We prove the estimates for $j\geq 1$ as the estimate for $j=0$ follows \textit{mutatis mutandis}. For $m\in\{0,1,2,3\}$, $x\in\T^d$, and $t\geq s$,
\begin{align*}
&\mathbb{E}[|\nabla^m u^j(t,x)-\nabla^mu^j(s,x)|^2]
\\&\quad\leq 2|\nabla^{2m}R^j(0,0)-\nabla^{2m}R^j(t-s,0)|
\\&\quad\leq C\sum_{k\in\Z^d\setminus \{0\}} |k|^{-d+2(m-\alpha)} \int_{2^{-\beta (j+1)}}^{2^{-\beta j}} |k|^{\beta(M+1)} \tau^Me^{-|k|^\beta \tau}\Big|\frac{1}{M!}-\phi_M((t-s)/\tau)\Big|\,d\tau
\\&\quad \leq C2^{\beta j}|t-s|\sum_{k\in\Z^d\setminus \{0\}} |k|^{-d+2(m-\alpha)} \int_{2^{-\beta (j+1)}}^{2^{-\beta j}} |k|^{\beta(M+1)} \tau^Me^{-|k|^\beta \tau}\,d\tau,
\end{align*}
where in the last inequality we have used that $|\frac{1}{M!}-\phi_M(r)|\leq r$. Next, we note that
\begin{align*}
\int_{2^{-\beta (j+1)}}^{2^{-\beta j}} |k|^{\beta(M+1)} \tau^Me^{-|k|^\beta \tau}\,d\tau&=(|k|^\beta2^{-\beta j})^{M+1}\int_{2^{-\beta}}^{1} r^Me^{-|k|^\beta 2^{-\beta j} r}\,dr
\\&\leq C (|k|/2^j)^{\beta(M+1)} e^{-2^{-\beta }(|k|/2^j)^\beta},
\end{align*}
where $C>0$ only depends on $d$ and $\beta$. Combining the two displays above we find that
\begin{align*}&\mathbb{E}[|\nabla^m u^j(t,x)-\nabla^m u^j(s,x)|^2]
\\&\qquad\qquad\qquad\leq C|t-s| 2^{2(m-\alpha+\beta/2)j}2^{-dj}\sum_{k\in\Z^d\setminus \{0\}}|k/2^j|^{-d+2(m-\alpha)+\beta(M+1)} e^{-2^{-\beta }(|k|/2^j)^\beta}.\end{align*}
Next, we note that our definition of $M$ ensures that $2(m-\alpha)+\beta(M+1)>0$, thus
\[\lim_{j\rightarrow \infty} 2^{-dj}\sum_{k\in\Z^d\setminus \{0\}}|k/2^j|^{-d+2(m-\alpha)+\beta(M+1)} e^{-2^{-\beta }(|k|/2^j)^\beta}=\int_{\R^d} |z|^{-d+2(m-\alpha)+\beta(M+1)} e^{-2^{-\beta} |z|^\beta }\,dz<\infty.\]
This implies the sum on the left-hand side above is uniformly bounded over $j$, hence
\[\mathbb{E}[|\nabla^m u^j(t,x)-\nabla^m u^j(s,x)|^2]\leq C |t-s| 2^{2(m-\alpha+\beta/2)j}.\]
By the equivalence of Gaussian moments, we have found that there exists $C(d,\alpha,\beta)>0$ such that for all $p\geq 1$ and $m\in\{0,1,2,3\}$
\begin{equation}
\label{eq:holder omega estimate}
\|\nabla^m u^j\|_{\mathcal{C}^{\frac{1}{2}}_t\mathcal{C}^0_xL^p_\omega }\leq C\sqrt{p} 2^{(m-\alpha+\beta/2)j}.
\end{equation}
Additionally, proceeding almost identically but instead considering $\E[|\nabla^m u^j(t,x)|^2]$, we find that 
\begin{equation}\label{eq: tx omega estimate}
\|\nabla^m u^j\|_{\mathcal{C}^0_{t,x} L^p_\omega}\leq C  \sqrt{p} 2^{(m-\alpha)j}.    
\end{equation}

Now, fix $n\in\{0,1,2\}$ and $p\geq1$. Since
$\hat R^j(t,0)=0$, the spatial average of $u^j(t,\cdot)$ vanishes
for every $t$. Hence, for every $\gamma,\nu\in(0,1]$,
\begin{equation}\label{eq:refreshing split}
\|\nabla^n u^j\|_{L^p_\omega\mathcal C^0_{t,x}}
\leq
\|\nabla^n u^j(0,\cdot)\|_{L^p_\omega\mathcal C^0_x}
+
C\|\nabla^n u^j\|_
{L^p_\omega\mathcal C^\gamma_t\mathcal C^\nu_x}
\end{equation}
for a $C(d)>0$.

We first estimate the initial-time term. By~\eqref{eq: tx omega estimate}, for $m\in\{0,1,2,3\}$,
\[
\|\nabla^m u^j(0,\cdot)\|_{\mathcal C^0_xL^p_\omega}
\leq C\sqrt p\,2^{(m-\alpha)j}.
\]
Applying the same spatial interpolation argument as in the proof of
Proposition~\ref{prop:autonomous gaussian}, we obtain
\begin{equation}\label{eq:refreshing initial}
\|\nabla^n u^j(0,\cdot)\|_{L^p_\omega\mathcal C^0_x}
\leq
C(p^{1/2}+j^{1/2})\,2^{(n-\alpha)j}.
\end{equation}
It remains to estimate the second term in~\eqref{eq:refreshing split}.

Let $q:=p\vee j\vee (d+3)\vee8$, and set $r=\frac4q$ and $s=\frac{d+3}{q}$ so that $r\leq1/2$ and $s\leq1$. Now, interpolating in space between the estimates for $m=n$ and $m=n+1$
in~\eqref{eq: tx omega estimate}, Lemma~\ref{lem:iterated holder interpolation}  gives
\[\|\nabla^n u^j\|_{\mathcal C^0_t\mathcal C^s_xL^q_\omega}\leq C\sqrt q\,2^{(n+s-\alpha)j}.\]
Similarly, spatial interpolation with~\eqref{eq:holder omega estimate}
gives
\[\|\nabla^n u^j\|_{\mathcal C^{1/2}_t\mathcal C^s_xL^q_\omega}\leq C\sqrt q\,2^{(n+s-\alpha+\beta/2)j}.\]
Interpolating these two estimates in time, with interpolation
parameter $2r=8/q\leq1$, then yields
\[\|\nabla^n u^j\|_{\mathcal C^r_t\mathcal C^s_xL^q_\omega}\leq C\sqrt q\,2^{(n+s-\alpha+r\beta)j}.\]
Applying Lemmas~\ref{lem:alpha-holder_embedding} and~\ref{lem:sobolev}, with intermediate exponents $r_0=3/q$ and $s_0=(d+2)/q$, we obtain
\[\|\nabla^n u^j\|_
{L^q_\omega\mathcal C^{1/q}_t\mathcal C^{1/q}_x}
\leq
C\sqrt q\,2^{(n-\alpha)j}
2^{\frac{d+3+4\beta}{q}j}\leq C(p^{1/2}+j^{1/2}) 2^{(n-\alpha)j},\]
for $C(d,\alpha,\beta)>0$. Note here that by our choice of $q$, the embedding constants are uniform over $p$. Combining this estimate with~\eqref{eq:refreshing split} and~\eqref{eq:refreshing initial}, we conclude Item~\ref{item: refreshing moment bounds}, and thus the proposition.
\end{proof}

\subsection{Poisson-driven fields}

We only prove Proposition~\ref{prop: autonomous poisson}. The refreshing case is proved almost identically, with Poisson point processes on $[0,1]$ rather than on $\T^d$.

\begin{proof}[Proof of Proposition~\ref{prop: autonomous poisson}]

Mutual independence of the fields $u^j$ follows from independence of
the Poisson processes and the $f^j_y$. Centeredness follows from the centeredness of the $f^j_y$. We thus continue to the proof of Item~\ref{item: poisson autonomous finite range}.

We note that for any subset $A\subset \T^d$, letting $\tilde{A}:=\{x\in\T^d: \mathrm{dist}(x,A)<2^{-(j+1)}\}$, $u^j|_{A}$ is a measurable function of $\Pi^j\cap\tilde{A}$ and $\{f^j_y:y\in \tilde{A}\}$. Thus, if $\mathrm{dist}(A,B)\geq 2^{-j}$, $u^j|_{A}$ and $u^j|_{B}$ are independent since $\tilde{A}$ and $\tilde B$ are disjoint.

Next, we prove the moment estimate Item~\ref{item:poisson autonomous moment bounds}. First we note that there exists a constant $C(d)>0$ such that for any $j\geq 0$, $\T^d$ is covered by $L\leq C 2^{dj}$ balls, $\{B_k\}_{1\leq k\leq L}$, each of radius $2^{-j}$. Fixing $j\geq 0$, for $n\in\{0,1,2\}$ it thus holds that
\begin{equation}\label{eq:trivial sup estimate}
\|\nabla^nu^j\|_{\mathcal{C}^0_x}= \max_{1\leq k\leq L}\|\nabla^nu^j\|_{\mathcal{C}^0(B_k)}.
\end{equation}
Letting $\tilde{B_k}:=\{y:\mathrm{dist}(y,B_k)\leq 2^{-j}\}$, by the definition of $u^j$, we then have
\[\|\nabla^nu^j\|_{\mathcal{C}^0_x(B_k)}\leq \sum_{y\in\Pi^j\cap\tilde B_k}\|\nabla^nf^j_y\|_{\mathcal{C}^0_x}.\]
Let $N_k:=|\Pi^j\cap\tilde B_k|$ be the number of points in the Poisson process that lie in $\tilde B_k$. Then taking expectations and conditioning on $N_k$,
\begin{align*}
\mathbb{E}\Big[\Big(\sum_{y\in\Pi^j\cap\tilde B_k}\|\nabla^nf^j_y\|_{\mathcal{C}^0_x}\Big)^p\Big]&=\sum_{m=0}^\infty \mathbb{E}\Big[\Big(\sum_{y\in\Pi^j\cap\tilde B_k}\|\nabla^nf^j_y\|_{\mathcal{C}^0_x}\Big)^p
\mid N_k=m\Big]\P(N_k=m),
\\&\leq \sum_{m=0}^\infty (mM 2^{(n-\alpha-d/p)j})^p\P(N_k=m)
\\&\leq (M 2^{(n-\alpha-d/p)j})^p \E[N_k^p]
\\&\leq C^p(M 2^{(n-\alpha-d/p)j})^p(\lambda+p)^p,
\end{align*}
where in the last line we've used that 
\[\E[N_k^p]\leq C^p(\lambda+p)^p\]
for a constant $C(d)>0$. Combining the two math displays above, we have thus found that
\[ \|\nabla^nu^j\|_{L^p_\omega \mathcal{C}_x^0(B_k)}\leq C M(\lambda+p) 2^{(n-\alpha-d/p)j}.\]
Inserting this estimate into~\eqref{eq:trivial sup estimate}, this shows that
\[\|\nabla^nu^j\|_{L^p_\omega \mathcal{C}^0_x}^p\leq \mathbb{E}\Big[\max_{1\leq k\leq L}\|\nabla^nu^j\|_{\mathcal{C}^0_x(B_k)}^p\Big]\leq \sum_{1\leq k\leq L} \|\nabla^nu^j\|^p_{L^p_\omega \mathcal{C}^0_x(B_k)}\leq C^pLM^p(\lambda+p)^p 2^{(n-\alpha-d/p)jp}.\]
Thus, using that $L\leq C 2^{dj}$,
\[\|\nabla^nu^j\|_{L^p_\omega \mathcal{C}^0_x}\leq CM(\lambda+p) 2^{(n-\alpha)j},\]
for a constant $C(d)>0$ as claimed.

As Item~\hyperref[item:poisson div free]{(i)} follows directly, we move on to the proof of Item~\hyperref[item: poisson zero level set]{(ii)}. For the sake of convenience, let
\[K^j:=\sup_{\substack{x,y\in\T^d\\d_{\T^d}(x,y)\leq 2^{-(j+2)}}}\bigg\|\frac{d\mu^j_{y,x}(z)}{dz}\bigg\|_{L^\infty(\R^d)}.\]
Now, fix $x\in\T^d$ and let $J_x:=\inf\{j\geq 0: |\Pi^j\cap B_{2^{-(j+2)}}(x)|> 0\}.$ We thus have that
\begin{equation}\label{eq:stopping time estimate}
 \P(J_x=j)=(1-e^{-\lambda})e^{-\lambda j}   
\end{equation}
where we have used the independence of $\Pi^\ell$ over $\ell$ and that
\[\P(\Pi^\ell\cap B_{2^{-(\ell+2)}}(x)=\emptyset)=e^{-\lambda}.\]
The equality~\eqref{eq:stopping time estimate} implies that $J_x<\infty$ almost surely. Conditioning on $J_x$, we thus find that for any $\eps>0$
\begin{align*}
\P(|U(x)|\leq \eps)&=\sum_{j=0}^\infty \P(|U(x)|\leq \eps\mid J_x=j) \P(J_x=j)
\\&\leq (1-e^{-\lambda})\sum_{j=0}^\infty e^{-\lambda j}\P(|U(x)|\leq \eps\mid J_x=j).
\end{align*}
However, if $J_x=j$ then there exists some point $y\in\Pi^j$ such that $d_{\T^d}(y,x)\leq 2^{-(j+2)}$. Then, on this conditional event, one can write $U(x)$ as a sum of independent random variables, at least one of which has density bounded by $K^j$. As the $L^\infty$ norm is contractive under convolution, this implies that 
\[\P(|U(x)|\leq \eps\mid J_x=j)\leq C K^j \eps^{d},\]
for some $C(d)>0$, hence
\[\P(|U(x)|\leq \eps)\leq  C(d) \eps^d (1-e^{-\lambda })\sum_{j=0}^\infty e^{-\lambda j}K^j.\]
The assumption~\eqref{eq:density condition} guarantees that
\[\sum_{j=0}^\infty e^{-\lambda j}K^j<\infty,\]
thus
\[\limsup_{\eps\rightarrow 0} \sup_{x\in\mathbb T^d} \eps^{-d}\P(|U(x)|\leq \eps)<\infty.\]
Since $U$ is almost surely in $C^{\alpha-}_x$, Lemma~\ref{lem:sufficient condition for good zero set} implies that for all $\alpha'<\alpha$, it almost surely holds that
\[\lim_{\eps\rightarrow 0} \eps^{-\alpha'd}|Z_\eps|=0,\]
concluding the proof of Item~\hyperref[item: poisson zero level set]{(ii)}, and thus the proposition.
\end{proof}

{\small
\bibliographystyle{alpha}
\bibliography{cleanreferences}
}

\end{document}